\documentclass[letterpaper]{amsart}

\usepackage[letterpaper, hmarginratio=1:1]{geometry}
 \usepackage{hyperref} 
\usepackage{amsmath}
\usepackage{amsthm}
\usepackage{amsfonts}
\usepackage{amssymb}
\usepackage{tikz}
\usepackage{bbm}
\usepackage[mathscr]{euscript}
\usepackage{microtype}
 \usepackage{cleveref}
 \hypersetup{pageanchor=false}
\usepackage[noadjust]{cite}
\numberwithin{equation}{section}

\DeclareMathOperator{\Id}{Id}

\DeclareMathOperator{\Tr}{Tr}

\DeclareMathOperator{\b|}{\boldsymbol{|}}

\DeclareMathOperator{\semicircle}{sc}

\def\R{\mathbb{R}}
\def\P{\mathbb{P}}
\def\E{\mathbb{E}}

\def\one{{\mathbbm{1}}}
\def\eps{\varepsilon}

\def\Im{\operatorname{Im}}
\def\Re{\operatorname{Re}}
\newcommand{\bma}{\begin{bmatrix}}
\newcommand{\ema}{\end{bmatrix}}

\def\bet{\begin{thm}}
\def\eet{\end{thm}}
\def\bel{\begin{lem}}
\def\eel{\end{lem}}
\def\bas{\begin{ass}}
\def\eas{\end{ass}}
\def\bec{\begin{cor}}
\def\eec{\end{cor}}
\def\bed{\begin{defn}}
\def\eed{\end{defn}}
\def\bep{\begin{prop}}
\def\eep{\end{prop}}
\def\beq{\begin{equation}}
\def\eeq{\end{equation}}
\def\bea{\begin{equation*}}
\def\eea{\end{equation*}}

\def\bex{\begin{ex}}
\def\eex{\end{ex}}
\def\bp{\begin{proof}}
\def\ep{\end{proof}}
\def\X{\mathbf X}
\def\one{\textbf{1}}

\newtheorem{thm}{Theorem}[section]
\newtheorem{prop}[thm]{Proposition}
\newtheorem{lem}[thm]{Lemma}
\newtheorem{cor}[thm]{Corollary}

\theoremstyle{remark}
\newtheorem{rem}[thm]{Remark}
\theoremstyle{definition}
\newtheorem{definition}[thm]{Definition}

\title{GOE Statistics for L\'{e}vy Matrices} 

\author{Amol Aggarwal, Patrick Lopatto, and Horng-Tzer Yau}

\thanks{A.A. is partially supported by the NSF Graduate Research Fellowship under grant  DGE-1144152. P.L. is partially supported by NSF grants DMS-1606305 and DMS-1855509, and the NSF Graduate Research Fellowship Program under Grant DGE-1144152. H.-T. Y. is partially supported by NSF grants DMS-1606305 and DMS-1855509, and a Simons Investigator award.}

\begin{document}	
	
\begin{abstract}
	
	We establish eigenvector delocalization and bulk universality for L\'{e}vy matrices, which are real, symmetric, $N \times N$ random matrices $\textbf{H}$ whose upper triangular entries are independent, identically distributed $\alpha$-stable laws. First, if $\alpha \in (1, 2)$ and $E \in \mathbb{R}$ is bounded away from $0$, we show that every eigenvector of $\textbf{H}$ corresponding to an eigenvalue near $E$ is completely delocalized and that the local spectral statistics of $\textbf{H}$ around $E$ converge to those of the Gaussian Orthogonal Ensemble as $N$ tends to $\infty$. Second, we show for almost all $\alpha \in (0, 2)$, there exists a constant $c(\alpha) > 0$ such that the same statements hold if $|E| < c (\alpha)$. 
	\end{abstract}

\maketitle

\tableofcontents

\section{Introduction} 

\label{Introduction}

The spectral analysis of random matrices has been a topic of intense study since Wigner's pioneering investigations \cite{wigner1955characteristic} in the 1950s. Wigner's central thesis asserts that the spectral statistics of random matrices are universal models for highly correlated systems. A concrete realization of his vision, the \emph{Wigner--Dyson--Mehta conjecture}, states that the bulk local spectral statistics of an $N \times N$ real symmetric (or complex Hermitian) Wigner matrix should become independent of the laws of its entries as $N$ tends to $\infty$; see Conjecture 1.2.1 and Conjecture 1.2.2 of \cite{mehta2004random}. This phenomenon is known as \emph{bulk universality}. 

Over the past decade, a framework based on resolvent estimates and Dyson Brownian motion has been developed to establish this conjecture and extend its conclusion to a wide class of matrix models. These include Wigner matrices \cite{bourgade2016fixed,erdos2010bulk,erdos2011universality,erdos2012local,erdos2017dynamical,erdos2012bulk,huang2015bulk,johansson2012universality,johansson2001universality,lee2014necessary,tao2012random,tao2011random}, correlated random matrices \cite{ajanki2016local,che2017universality}, random graph models \cite{bauerschmidt2017bulk,erdos2013spectral,erdos2012spectral,huang2015spectral,huang2015bulk}, general Wigner-type matrices \cite{ajanki2017singularities,ajanki2017universality}, certain families of band matrices \cite{bourgade2018survey, bourgade2017universality, bourgade2018random, bourgade2019random, yang2018random}, and various other models. All these models require that the variance of each matrix entry is finite, an assumption already present in the original universality conjectures \cite{mehta2004random}. The moment assumption required for the bulk universality of Wigner matrices has been progressively improved, and universality is now known to hold for matrix entries with finite $(2 + \varepsilon)$-th moments \cite{aggarwal2019bulk,erdos2012spectral}. 

While finite variance might seem to be the natural assumption for the Wigner--Dyson--Mehta conjecture, in 1994 Cizeau and Bouchaud \cite{cizeau1994theory} asked to what extent local eigenvalue statistics and related phenomena remain universal once the finite variance constraint is removed. Their work was motivated by heavy-tailed phenomena in physics \cite{bouchaud1995more, sornette2006critical}, including the study of spin glass models with power-law interactions \cite{cizeau1993mean}, and applications to finance \cite{galluccio1998rational, bouchaud1997option, bun2017cleaning, bouchaud1998taming, bouchaud2009financial,laloux1999noise,laloux2000random}. Recent work has also shown the appearance of heavy-tailed spectral statistics in neural networks \cite{martin2018implicit,martin2019heavy,mahoney2019traditional}.

The authors of \cite{cizeau1994theory} proposed a family of symmetric random matrix models, called L\'{e}vy matrices, whose entries are random variables in the domain of attraction of an $\alpha$-stable law.\footnote{When $\alpha <2$, we recall that the densities of such laws decay asymptotically like $x^{-\alpha - 1} \, dx$. In particular, they have infinite second moment.}   Based on numerical simulations, they predicted that bulk universality should still hold in certain regimes when $\alpha < 2$. In particular, for $\alpha <1$  they proposed that the local statistics of L\'{e}vy matrices should exhibit a sharp phase transition from GOE at small energies to Poisson at large energies. 

Such a transition is called a \emph{mobility edge} (also known as a \emph{Anderson transition} or \emph{Mott transition}, depending on the physical context) and is a principal concept in the pathbreaking work of the physicists Anderson and Mott on metal--insulator transitions in condensed matter physics \cite{abrahams201050,anderson1958absence,anderson1978local, mott1949basis,mott1987mobility}. It is widely believed to exist in the context of random Schr\"odinger operators, particularly in the Anderson tight-binding model \cite{abou1973selfconsistent, abou1974self, aizenman2011absence, aizenman2013resonant, bapst2014large}, but rigorously establishing this statement has remained a fundamental open problem in mathematical physics for decades. While localization and Poisson statistics at large energies in the Anderson model have been known since the 1980s, even the existence of a delocalized phase with GOE local statistics has not been rigorously verified for any model exhibiting a conjectural mobility edge \cite{frohlich1983absence,minami1996local,aizenman1993localization,spencer1988localization,gol1977pure,ding2018localization,li2019anderson}. As explained below, L\'{e}vy matrices provide one of the few examples of a random matrix ensemble for which such a transition is also believed to appear. Consequently, the predictions of \cite{cizeau1994theory} have attracted significant attention among mathematicians and physicists over the past 25 years \cite{auffinger2009poisson,arous2008spectrum,benaych2014central,benaych2014central,belinschi2009spectral,biroli2007top,bordenave2011spectrum,bordenave2017delocalization,bordenave2013localization,burda2006random,soshnikov2004poisson,tarquini2016level,biroli2007extreme}.

The work \cite{cizeau1994theory} further analyzed the large $N$ limiting profile for the \emph{empirical spectral distribution} of a L\'{e}vy matrix $\textbf{H}$, defined by $\mu_{\textbf{H}} = N^{-1} \sum_{j = 1}^N \delta_{\lambda_j}$, where $\lambda_1, \lambda_2, \ldots , \lambda_N$ denote the eigenvalues of $\textbf{H}$. They predicted that $\mu_{\textbf{H}}$ should converge to a deterministic, explicit measure $\mu_{\alpha}$ as $N$ tends to $\infty$, which was later proven by Ben Arous and Guionnet \cite{arous2008spectrum}. This measure $\mu_{\alpha}$ is absolutely continuous with respect to the Lebesgue measure on $\mathbb{R}$ and therefore admits a density $\varrho_{\alpha}$, which is symmetric and behaves as $\varrho_{\alpha} (x) \sim \frac{\alpha}{2 x^{\alpha + 1}}$ as $x$ tends to $\infty$ \cite{arous2008spectrum,belinschi2009spectral,bordenave2011spectrum}. In particular, $\varrho_{\alpha}$ is supported on all of $\mathbb{R}$ and has an $\alpha$-heavy tail. This is in contrast with the limiting spectral density for Wigner matrices, given by the \emph{semicircle law},
\begin{flalign}
\label{rhodefinition}
\varrho_{\semicircle} (x) = (2 \pi)^{-1} \one_{|x| < 2} \sqrt{4 - x^2},
\end{flalign}

\noindent which is compactly supported on $[-2, 2]$. 

Two other phenomena of interest are eigenvector delocalization and local spectral statistics. Associated with any eigenvalue $\lambda_k$ of $\textbf{H}$ is an eigenvector $\textbf{u}_k = (u_{1k}, u_{2k}, \ldots , u_{Nk}) \in \mathbb{R}^N$, normalized such that $\| \textbf{u}_k \|_2^2 = \sum_{i = 1}^N u_{ik}^2 = 1$. If $\textbf{H} = \textbf{GOE}_N$ is instead taken from the Gaussian Orthogonal Ensemble\footnote{This is defined to be the $N \times N$ real symmetric random matrix $\textbf{GOE}_N = \{ g_{ij} \}$, whose upper triangular entries $g_{ij}$ are mutually independent Gaussian random variables with variances $2 N^{-1}$ if $i = j$ and $N^{-1}$ otherwise.} (GOE), then the law of $\textbf{u}_k$ is uniform on the $(N - 1)$-sphere, and so $\max_{1 \le i \le N} \big| u_{ik} \big| \le N^{\delta - 1 / 2}$ holds with high probability for any $\delta > 0$. This bound is referred to as \emph{complete eigenvector delocalization}. The \emph{local spectral statistics} of $\textbf{H}$ concern the behavior of its neighboring eigenvalues close to a fixed energy level $E \in \mathbb{R}$.

  \begin{figure}
 \centering
 \begin{tikzpicture}
 \filldraw[lightgray] (4,0) rectangle (8,4);
  \filldraw[lightgray] (0,0) parabola (4,2) -- (4,0) -- cycle;
  \draw (0,0) rectangle (8,4);
  \draw[very thick] (0,0) parabola (4,4);
   \node[scale=1.25] at (2,2) {$E_\alpha$};
   \node at (-.5,0 ) {$0$};
   \node at (-.5,4) {$\infty$};
   \node at ( 0 , -.5 ) {$0$};
   \node at ( 4 , -.5 ) {$1$};
    \node at ( 8 , -.5 ) {$2$};
     \node[scale=1.25] at ( 4 , -1.25 ) {$\alpha \in (0, 2)$};
      \node[scale=1.25] at ( -1 , 2 ) {$E$};
     \node[scale=1,align=center] at (1.5, 3) {Poisson/\\Localized};
     \node[scale=1,align=center] at (6, 2) {GOE/\\Delocalized};
     \node at (9,0) {}; 
   
 \end{tikzpicture}
 \caption{Phase diagram. The thick line indicates the location of the conjectural mobility edge, which separates the localized phase from the delocalized phase. The gray area indicates the scope of our results.}
 \end{figure}
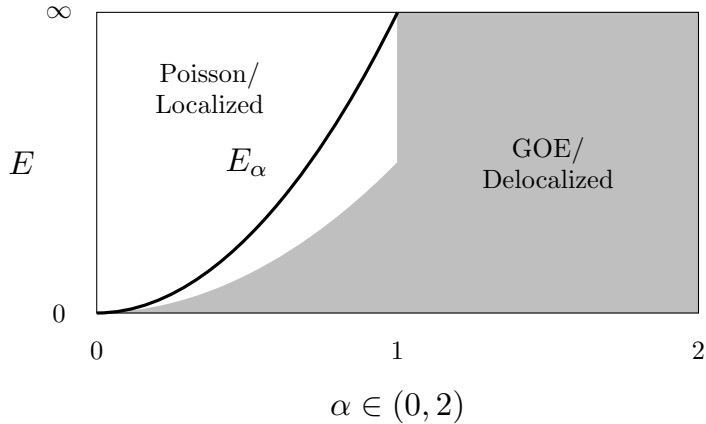

The main predictions of \cite{cizeau1994theory} were certain transitions in the eigenvector behavior and local spectral statistics of L\'{e}vy matrices. Their  predictions  are not fully consistent with the recent work  \cite{tarquini2016level} by Tarquini, Biroli, and Tarzia, based on the supersymmetric method. The latter predictions can be summarized  as follows.

\begin{description}
	\item[A ($1 \le \alpha < 2$)] \label{alpha12eigenvector}  All eigenvectors of $\textbf{H}$ corresponding to finite eigenvalues are completely delocalized. Further, for any $E \in \mathbb{R}$, the local statistics of $\textbf{H}$ near $E$  converge to those of the GOE as $N$ tends to $\infty$.

	\item[B ($0< \alpha < 1$)] \label{alpha01eigenvector}  There exists a \emph{mobility edge} $E_{\alpha}$ such  that  (i) if $|E| < E_{\alpha}$ then the local statistics of $\textbf{H}$ near $E$  converge to those of the GOE and all eigenvectors in this region are completely delocalized;  (ii) if $|E| > E_{\alpha}$, then the local statistics of $\textbf{H}$ near $E$  converge to those of a Poisson point process and all eigenvectors in this region are localized. 
\end{description}

The earlier predictions of  \cite{cizeau1994theory} are different: \textbf{A' {($1 \le \alpha < 2$):}} There are two regions: (i) for sufficiently small energies, the eigenvectors are completely delocalized and the local statistics are GOE; (ii) for sufficiently large energies, the eigenvectors are weakly localized according to a power law decay, and the local statistics are given by certain non-universal laws that converge to Poisson statistics in the infinite energy limit; \textbf{B' (${0< \alpha < 1}$):} essentially the same as prediction \textbf{B} above except that in the delocalized region the eigenvectors  were predicted to only be partially delocalized, in that a positive proportion of the mass is completely delocalized and a positive proportion of the mass is completely localized. In addition, \cite{cizeau1994theory} proposes an equation for the mobility edge $E_{\alpha}$; a much simpler (but equivalent) version of this equation was predicted in \cite{tarquini2016level}.

The problem of rigorously establishing this mobility edge remains open. In fact, there have been no previous mathematical results on local statistics for L\'{e}vy matrices, in any regime. However, partial results on eigenvector (de)localization were established by Bordenave and Guionnet in \cite{bordenave2017delocalization,bordenave2013localization}. If $1 < \alpha < 2$, they showed that almost all eigenvectors $\textbf{u}_k$ satisfy $\max_{1 \le i \le N} |u_{ik}| < N^{\delta - \rho}$ for any $\delta > 0$ with high probability, where $\rho = \frac{\alpha - 1}{\max \{ 2 \alpha, 8 - 3 \alpha \}}$ \cite{bordenave2013localization}. For almost all $\alpha \in (0, 2)$, they also proved the existence of some $c = c(\alpha)$ such if $\textbf{u}_k$ is an eigenvector of $\textbf{H}$ corresponding to an eigenvalue $\lambda_k \in [-c, c]$, then $\max_{1 \le i \le N} |u_{ik}| < N^{\delta - \alpha / (4 + 2 \alpha)}$ for any $\delta > 0$ with high probability \cite{bordenave2017delocalization}. These estimates remain far from the complete delocalization bounds that have been established in the Wigner case. Furthermore, if $0 < \alpha < \frac{2}{3}$ and $\textbf{G} (z) = \big\{ G_{ij} (z) \big\} = (\textbf{H} - z)^{-1}$, then they showed that $\mathbb{E} \big[ \big( \Im G_{11} (z) \big)^{\alpha / 2} \big] = O (\eta^{\alpha / 2 - \delta})$ for any $\delta > 0$ if $\Re z$ is sufficiently large and $\eta = \Im z \gg N^{- (2 + \alpha) / (4 \alpha + 12)}$, which implies eigenvector localization in a certain weak sense at large energy \cite{bordenave2013localization}.

In this paper, we  establish complete delocalization and bulk universality for L\'{e}vy matrices for all energies in any fixed compact interval away from  $E = 0$ if $1 < \alpha < 2$. In addition, for $0 < \alpha  <  2$ outside a (deterministic) countable set, we prove that there exists $ \widetilde E_{\alpha}$ such that complete delocalization and bulk universality hold for all energies in $[- \widetilde E_{\alpha},  \widetilde E_{\alpha}] $. These results  establish the   prediction \textbf{A} of \cite{tarquini2016level}  essentially completely for $1 < \alpha < 2$ and also the existence of the GOE regime for $0 < \alpha  < 1$, with completely delocalized eigenvectors. Before describing these results in more detail, we recall the \emph{three-step strategy} for establishing bulk universality of Wigner matrices developed in \cite{erdos2010bulkMRL, erdos2011universality, erdos2010wegner, erdos2012local, erdos2017dynamical, tao2011random} (see \cite{erdos2011survey, benaych2016lectures} or the book \cite{erdos2017dynamical} for a survey).

\label{localcircle} The first step is to establish a local law for $\textbf{H}$, meaning that the spectral density of $\textbf{H}$ asymptotically follows that of its deterministic limit on small scales of order nearly $N^{-1}$. The second step is to consider a Gaussian perturbation $\textbf{H} + t^{1 / 2} \textbf{GOE}_N$ of $\textbf{H}$, for some small $t$, and then use the local law to show that the local statistics of the perturbed matrix are universal. The third step is to compare the local statistics of $\textbf{H}$ and its perturbed variant $\textbf{H} + t^{1 / 2} \textbf{GOE}_N$, and show that they are asymptotically the same. The comparison of the local statistics can be most efficiently obtained by comparing the entries of the resolvents of the ensembles; this is often referred to as a Green's function (resolvent) comparison theorem \cite{erdos2012bulk}.

There are two issues with adapting this framework to the heavy-tailed setting. First, we do not know of a direct way 
to establish a local law for the $\alpha$-stable matrix $\textbf{H}$ on the optimal scale of roughly $N^{-1}$. Second, justifying the removal of the Gaussian  perturbation in the third step has intrinsic problems since the entries of $\textbf{H}$ have divergent variances (and possibly divergent expectations).

To explain the first problem, we introduce some notation. We recall the \emph{Stieltjes transform} of the empirical spectral distribution $\mu_{\textbf{H}}$ is defined by the function 
\begin{flalign}
\label{mn}
m_N = m_N (z) = m_{N, \textbf{H}} (z) = \displaystyle\frac{1}{N} \displaystyle\sum_{j = 1}^N \displaystyle\frac{1}{\lambda_j - z} = \displaystyle\frac{1}{N} \Tr \big( \textbf{H} - z \big)^{-1}, 
\end{flalign}

\noindent for any $z \in \mathbb{H}$. Since $\mu_{\textbf{H}}$ converges weakly to $\mu_{\alpha}$ as $N$ tends to $\infty$, one expects $m_N (z)$ to converge to $ m_{{\alpha}} (z) = \int_{\mathbb{R}}  ({x - z})^{-1}  \varrho_{\alpha} (x)\, dx$. The imaginary part of the Stieltjes transform represents the convolution of the empirical spectral distribution with an approximate identity, the Poisson kernel, at scale $\eta = \Im z$. Hence, control of the Stieltjes transform at scale $\eta$ can be thought of as control over the eigenvalues averaged over windows of size approximately $\eta$. 

A \emph{local law} for $\textbf{H}$ is an estimate on $\big| m_N (z) - m_{{\alpha}} (z) \big|$ when $\eta = \Im z$ scales like $N^{-1+\eps}$. The typical procedure \cite{erdos2010wegner,erdos2009local,erdos2013spectral,erdos2013local,erdos2012bulk,lee2014necessary} for establishing a local law relies on a detailed understanding of the \emph{resolvent} of $\textbf{H}$, defined to be the $N \times N$ matrix $ \textbf{G} (z) = \big( \textbf{H} - z \big)^{-1} = \big\{ G_{ij} (z) \big\} $. Indeed, since $m_N(z) = N^{-1} \Tr \textbf{G}(z)$, it suffices to estimate the diagonal entries of $\textbf{G}$. In many of the known finite variance cases, (almost) all of the entries $G_{ij}$ converge to a deterministic quantity in the large $N$ limit. 

This is no longer true in the heavy-tailed setting, where the limiting resolvent entries are instead random away from the real axis \cite{bordenave2011spectrum}.  While the  idea  that  the  resolvent entries should satisfy  a  self-consistent  equation (which has been a central concept in proving local laws for Wigner matrices \cite{erdos2009local}) is still  applicable  to the current  setting \cite{arous2008spectrum,belinschi2009spectral,bordenave2017delocalization,bordenave2013localization}, the  random  nature  of these resolvent entries poses many difficulties  in  analyzing the resulting self-consistent equation. This presents serious difficulties  in applying previously developed methods to establish a local law for $\alpha$-stable matrices on the optimal scale. While local laws on intermediate scales $\eta \gg N^{-1 / 2}$ were established for such matrices in \cite{bordenave2017delocalization,bordenave2013localization} if $\alpha$ is sufficiently close to two, the value of $\eta$ allowed in these estimates deteriorates to $1$ as $\alpha$ decreases to zero. 

For the second problem, all existing  methods  of comparing two  matrix ensembles $\textbf{H}$ and $\widetilde{\textbf{H}}$ \cite{aggarwal2019bulk,bourgade2016fixed,erdos2013spectral,erdos2012spectral,lee2014necessary,huang2015bulk,tao2012random,tao2011randomarxiv,tao2011random} involve Taylor expanding the matrix entries of their resolvents to order at least three and then matching the expectations of the first and second order terms of this expansion, which correspond to the first and second moments of the matrix entries. If  the entries of $\textbf{H}$ and $\widetilde{\textbf{H}}$ are heavy-tailed, then all second and higher moments  of these matrix entries diverge, and this expansion is meaningless.   

These two difficulties are in fact intricately related, and our solution to them consists of the following steps.  

1.  We first  rewrite the matrix as $\textbf{H} = \textbf{X} + \textbf{A}$, where $\textbf{A}$ consists of the ``small terms'' of $\textbf{H}$ that are bounded by $N^{-\nu}$ in magnitude for some constant $0 < \nu < \frac{1}{\alpha}$, and $\textbf{X}$ consists of the remaining terms. With this cutoff, the moments of the entries of $\textbf{A}$ behave as those of (a scale of) the adjacency matrix of an Erd\H{o}s--R\'eyni  graph with a certain ($N$-dependent) edge density as defined in \cite{erdos2013spectral}. We prove a comparison theorem for the resolvent entries of $\textbf{H} = \textbf{X} + \textbf{A}$ and those of $\textbf{V}_t = \textbf{X} + t^{1 / 2} \textbf{GOE}_N$, where $\textbf{GOE}_N$ is independent from $\textbf{X}$. The parameter $t \sim N^{\nu (\alpha - 2)}$ will be chosen so that the variances of matrix entries of  $ t^{1 / 2} \textbf{GOE}_N$ and $\textbf{A}$ match. By construction,  $\textbf{A}$ and $\textbf{X}$ are symmetric, so the first and third moments of the matrix entries vanish. Hence in the comparison argument, the problem is reduced to considering the second and fourth order terms.

Notice that  $\textbf{A}$ and $\textbf{X}$ are dependent, so the previous heuristic cannot be applied directly. To remove their correlation, we expand upon a procedure introduced in \cite{aggarwal2019bulk} to produce a three-level decomposition of $\textbf{H}$. By conditioning on the decomposition into large and small field regions, $\textbf{A}$ and $\textbf{X}$ are independent and a version of the comparison theorem can be proven. 

2.  From the work of \cite{landon2017convergence}, the GOE component in $ \textbf{V}_t$ improves the regularity of the initial data $\textbf{V}_0$, which is a manifestation of the parabolic regularization of the Dyson Brownian motion flow. Roughly speaking, if the spectral density of $\textbf{V}_0$  is bounded above and below at a scale $\eta \le  N^{-\delta} t$, then  the following three properties for $\textbf{V}_t$ hold: (i) universality of local statistics, (ii) complete eigenvector delocalization, and (iii) a local law at all scales up to $\eta \ge N^{\delta - 1}$ for any $\delta > 0$ \cite{bourgade2017eigenvector,erdos2017universality,landon2017convergence,landon2019fixed}. The fundamental input for this method  is an intermediate local law for $\textbf{X}$ on a scale $\eta \ll N^{\nu (\alpha - 2)} \sim t $.  The existing intermediate local laws for heavy tailed matrices established in \cite{bordenave2017delocalization,bordenave2013localization} are unfortunately only valid on scales larger than this critical scale when $\alpha$ is close to one. Our second main result  is to improve these  laws to scales below $N^{\nu (\alpha - 2)}$. Our method uses self-consistent equations for the resolvent entries and special tools developed in \cite{bordenave2013localization} for L\'{e}vy matrices. Note that the resolvent entries of $\textbf{X}$ are random and self-consistent equations for them are difficult to work with. Still, we are able to derive effective upper bounds on the diagonal resolvent entries of $\textbf{X}$, which enable us to improve the intermediate local laws to scales below $N^{\nu (\alpha - 2)}$.

3.  Combining steps 1 and 2, we are able to transport the three properties for  $\mathbf{V}_t$ to  our original matrix $\textbf{H}$. Recall that  in the standard  comparison theorem, resolvent bounds on the optimal  scale are required on both ensembles. Since our initial estimates on the resolvent of the original matrix $\textbf{H}$ are far from on the optimal scale, a different approach is required. In particular, it is known that one can induct on the scale $\eta$ to transfer resolvent estimates from one ensemble to another using the comparison method \cite{knowles2017anisotropic}. Although technical estimates must be supplied, the upshot of this step is that all three properties for $\mathbf{V}_t$ hold for the  original matrix $\textbf{H}$. The eigenvector delocalization and universality of local statistics constitute our main results. For the sake of brevity, we will not pursue the local law on the optimal scale of approximately $N^{-1}$, since it will not be needed to prove our main results.

The remainder of this article is organized as follows. In \Cref{Eigenvalues} we explain our results in more detail. In \Cref{Outline2} we state the comparison between $\textbf{H}$ and $\textbf{V}_t$, as well as the intermediate local laws for $\textbf{X}$ in the $\alpha \in (1, 2)$ case and the small energy $\alpha \in (0, 2)$ case. Then, assuming these estimates, we establish our results (given by \Cref{alpha12},  and \Cref{alpha02}). In \Cref{Compare} we establish the comparison between $\textbf{H}$ and $\textbf{V}_t$. In \Cref{Model2} and \Cref{ProofEstimates} we establish the intermediate local law on $\textbf{X}$ at all energies away from $0$ when $\alpha \in (1, 2)$. In \Cref{LocalTail2} and \Cref{EstimatesSmall} we show a similar intermediate local law on $\textbf{X}$, but at sufficiently small energies and for almost all $\alpha \in (0, 2)$.

\subsection*{Acknowledgments}

The authors thank Jiaoyang Huang and Benjamin Landon for helpful conversations. We are also grateful to the anonymous referees for valuable suggestions on an earlier version of this paper.

\section{Results}

\label{Eigenvalues}

We fix parameters $\alpha \in (0, 2)$ and $\sigma > 0$. A random variable $Z$ is a \textit{$(0, \sigma)$ $\alpha$-stable law} if 
\begin{flalign}
\label{betasigmaalphalaw}
\mathbb{E} \big[ e^{\mathrm{i} t Z} \big] = \exp \big( - \sigma^{\alpha} |t|^{\alpha} \big), \quad \text{for all $t \in \mathbb{R}$}.
\end{flalign}

While many previous works have considered only matrices whose entires are distributed as $\alpha$-stable laws, the methods of this work apply to a fairly broad range of symmetric power-law distributions. We now define the entry distributions we consider in this paper (see the end of this section for discussion on extensions to more general ones). 
\begin{definition}
	
	\label{momentassumption} Let $Z$ be a $(0, \sigma)$ $\alpha$-stable law with
	\begin{flalign}
	\label{stable}
	\sigma = \left( \displaystyle\frac{\pi}{2 \sin \big( \frac{\pi \alpha}{2} \big) \Gamma (\alpha)} \right)^{1 / \alpha} > 0.
	\end{flalign} 
	Let $J$ be a symmetric\footnote{By \emph{symmetric}, we mean that $J$ has the same distribution as $-J$.} random variable (not necessarily independent from $Z$) such that $\E [J^2] < \infty$, $Z+J$ is symmetric, and 
\begin{flalign}
\label{probabilityxij}
\frac{C_1}{\big( |t| + 1 \big)^\alpha} \le \P \big[ |Z + J | \ge t   \big] \le \frac{C_2}{\big( |t| + 1 \big)^\alpha}, \quad \text{for each $t \ge 0$ and some constants $C_1, C_2 > 0$.}
\end{flalign} 

	\noindent Denoting $\mathfrak{z} = Z + J$, the symmetry of $J$ and the condition $\mathbb{E} [J^2] < \infty$ are equivalent to imposing a coupling between $\mathfrak{z}$ and $Z$ such that $\mathfrak{z} - Z$ is symmetric and has finite variance, respectively. 
	
	For each positive integer $N$, let $\{ H_{ij} \}_{1 \le i \le j \le N}$ be mutually independent random variables that each have the same law as $N^{-1 / \alpha} (Z + J) = N^{-1 / \alpha} \mathfrak{z}$. Set $H_{ij} = H_{ji}$ for each $i, j$, and define the $N \times N$ random matrix $\textbf{H} = \textbf{H}_N = \{ H_{ij} \} = \{ H_{i, j}^{(N)} \}$, which we call	an \emph{$\alpha$-L\'{e}vy matrix}.  
	
\end{definition}

The $N^{-1 / \alpha}$ scaling in the $H_{ij}$ is different from the more standard $N^{-1 / 2}$ scaling that occurs in the entries of Wigner matrices. This is done in order to make the typical row sum of $\textbf{H}$ of order one. Furthermore, the explicit constant $\sigma$ \eqref{stable} was chosen to make our notation consistent with that of previous works, such as \cite{arous2008spectrum,bordenave2017delocalization,bordenave2013localization}, but can be altered by rescaling $\textbf{H}$. 

It was shown as Theorem 1.1 of \cite{arous2008spectrum} that, as $N$ tends to $\infty$, the empirical spectral distribution of $\textbf{H}$ converges to a deterministic measure $\mu_{\alpha}$. This is the (unique) probability distribution $\mu$ on $\mathbb{R}$ whose Stieltjes transform $\int_{\mathbb{R}} (x -z )^{-1} \, d \mu (x)$ is equal to the function $m_{{\alpha}} (z)$, which can be explicitly described as follows. Denote the upper half plane by $\mathbb{H} = \left\{ z \in \mathbb{C}\colon \Im z > 0\right\}$ and its image under multiplication by $-\mathrm{i}$ by $\mathbb{K} = \left\{ z \in \mathbb{C}\colon \Re z > 0\right\}$. For any $z \in \mathbb{H}$, define the functions $\varphi = \varphi_{\alpha,z}\colon \mathbb{K} \rightarrow \mathbb{C}$ and $\psi = \psi_{\alpha,z}\colon \mathbb{K} \rightarrow \mathbb{C}$ by
\begin{flalign}
\label{psi}
\varphi_{\alpha,z}(x) = \frac{1}{\Gamma(\alpha/2)} \int_0^\infty t^{\alpha/2 -1} e^{\mathrm{i} t z} e^{-\Gamma(1-\alpha/2) t^{\alpha/2} x}\, dt, \quad \psi_{\alpha, z}(x) = \int_0^\infty e^{\mathrm{i} t z} e^{-\Gamma( 1 - \alpha/2) t^{\alpha/2} x}\, dt, 
\end{flalign}

\noindent for any $x \in \mathbb{K}$. For each $z \in \mathbb{H}$ there exists a unique solution $ y(z) \in \mathbb{K}$ to the equation $y(z) = \varphi_{\alpha,z} \big( y(z) \big)$, so let us define
\begin{flalign} 
\label{stieltjespsi}
m_\alpha(z) = \mathrm{i} \psi_{\alpha,z} \big( y(z) \big).
\end{flalign}

\noindent The probability density function of the measure $\mu_{\alpha}$ is given by $\varrho_{\alpha}$, which is defined by setting 
\begin{flalign*} 
\varrho_{\alpha} (E) = \frac{1}{\pi} \lim_{\eta \rightarrow 0} \Im m_{{\alpha}} (E + \mathrm{i} \eta), \quad \text{for each $E \in \mathbb{R}$.}
\end{flalign*}

The term \emph{bulk universality} refers to the phenomenon that, in the bulk of the spectrum, the correlation functions of an $N \times N$ random matrix should converge to those of an $N \times N$ GOE matrix in the large $N$ limit.\footnote{Since the latter ensemble of matrices is exactly solvable through the framework of orthogonal polynomials and Pfaffian point processes, its correlation functions can be evaluated explicitly in the large $N$ limit. We will not state these results here, but they can be found in Chapter 6 of \cite{mehta2004random} or Chapter 3.9 of \cite{anderson2010introduction}.} This is explained more precisely through the following definitions.

\begin{definition}
	
	\label{correlation}
	
	Let $N$ be a positive integer and $\textbf{H}$ be an $N \times N$ real symmetric random matrix. Denote by $p_{\textbf{H}}^{(N)} (\lambda_1, \lambda_2, \ldots , \lambda_N)$ the density of the symmetrized joint eigenvalue distribution of $\textbf{H}$.\footnote{In particular, with respect to the symmetrized density, $\lambda_1, \lambda_2, \dots , \lambda_N$ are exchangeable random variables. Such a density exists because each entry distribution of the random matrix has a density.} For each integer $k \in [1, N]$, define the \emph{$k$-th correlation function} of $\textbf{H}$ by 
	\begin{flalign*}
	p_{\textbf{H}}^{(k)} (x_1, x_2, \ldots , x_k) = \displaystyle\int_{\mathbb{R}^{N - k}} p_{\textbf{H}}^{(N)} (x_1, x_2, \ldots , x_k, y_{k + 1}, y_{k + 2}, \ldots , y_N) \displaystyle\prod_{j = k + 1}^N d y_j. 
	\end{flalign*} 
\end{definition}

\begin{definition}
	
	\label{gapscorrelations} 
	
	Let $\{ \textbf{H} = \textbf{H}_N \}_{N \in \mathbb{Z}_{\ge 1}}$ be a set of matrices, $ \{ \varrho = \varrho_N \}_{N \in \mathbb{Z}_{\ge 1}}$ be a set of a probability density functions, and $E \in \mathbb{R}$ be a fixed real number. We say that \emph{the correlation functions of $\textbf{\emph{H}}$ are universal at energy level $E$ with respect to $\varrho$} if, for any positive integer $k$ and compactly supported smooth function $F \in \mathcal{C}_0^{\infty} (\mathbb{R}^k)$, we have that 
	\begin{flalign}
	\label{universality2}
	\begin{aligned}
	\displaystyle\lim_{N \rightarrow \infty} \Bigg| \displaystyle\int_{\mathbb{R}^k} & F (\textbf{a}) \bigg( p_{\textbf{H}_N}^{(k)} \Big( E + \displaystyle\frac{\textbf{a}}{N \varrho_N (E)} \Big)  - p_{\textbf{GOE}_N}^{(k)} \Big( \displaystyle\frac{\textbf{a}}{N \varrho_{\semicircle} (0)} \Big) \bigg) d \textbf{a} \Bigg| = 0,
	\end{aligned}
	\end{flalign}
	
	\noindent where $d \textbf{a}$ denotes the Lebesgue measure on $\mathbb{R}^k$ and we recall that $\varrho_{\semicircle}$ was defined by \eqref{rhodefinition}.
	
\end{definition}

Now we can state our main results. In what follows, we set $\| \textbf{v} \|_{\infty} = \max_{j \in [1, d]}	|v_j|$ for any $\textbf{v} = (v_1, v_2, \ldots , v_d) \in \mathbb{R}^d$. 

\begin{thm} 
	
	\label{alpha12} 
	
	Let $\textbf{\emph{H}}$ denote an $N \times N$ $\alpha$-stable matrix, as in \Cref{momentassumption}. Suppose that $\alpha \in (1, 2)$, and fix some compact interval $K \subset \mathbb{R} \setminus \{ 0 \}$.
	
	\begin{enumerate} 
		
		\item \label{eigenvectors12} 
		
		For any $\delta > 0$ and $D > 0$, there exists a constant $C = C(\alpha, \delta, D, K) > 0$ such that 
		\begin{flalign*}
		\mathbb{P} \left[ \displaystyle\max \big\{ \| \textbf{\emph{u}} \|_{\infty}: \textbf{\emph{H}} \textbf{\emph{u}} = \lambda \textbf{\emph{u}}, \| \textbf{\emph{u}} \|_2 = 1, \lambda \in K \big\} > N^{\delta - 1 / 2} \right] < C N^{-D}.
		\end{flalign*}
		
		\item \label{eigenvalues12}
		
		Fix some $E \in K$. Then the correlation functions of $\textbf{\emph{H}}$ are universal at energy level $E$ with respect to $\varrho_\alpha$, as in \Cref{gapscorrelations}. 
		
	\end{enumerate}
	
\end{thm}

\begin{thm}
	
	\label{alpha02}
	
	Let $\textbf{\emph{H}}$ denote an $\alpha$-stable matrix, as in \Cref{momentassumption}. There exists a countable set $\mathcal{A} \subset (0, 2)$ with no accumulation points in $(0, 2)$ such that for any $\alpha \in (0, 2) \setminus \mathcal{A}$, there exists a constant $c = c (\alpha) > 0$ such that the following holds. 
	
	\begin{enumerate}
		
		\item \label{eigenvectors02}  For any $\delta > 0$ and $D > 0$, there exists a constant $C = C(\alpha, \delta, D) > 0$ such that 
		\begin{flalign*}
		\mathbb{P} \left[ \displaystyle\max \big\{ \| \textbf{\emph{u}} \|_{\infty}: \textbf{\emph{H}} \textbf{\emph{u}} = \lambda \textbf{\emph{u}}, \| \textbf{\emph{u}} \|_2 = 1, \lambda \in [-c, c] \big\} > N^{\delta - 1 / 2} \right] < C N^{-D}.
		\end{flalign*}

		\item \label{eigenvalues02} Fix $E \in [-c, c]$. Then, the correlation functions of $\textbf{\emph{H}}$ are universal at energy level $E$ with respect to $\varrho_\alpha$, as in \Cref{gapscorrelations}. 
		
	\end{enumerate}
	
\end{thm}

The above two theorems comprise the first complete eigenvector delocalization and bulk universality results for a random matrix model whose entries have infinite variance. For  $\alpha \in (1, 2)$, \Cref{alpha12} completely settles the bulk universality and complete eigenvector delocalization for all energies (except if $\alpha \in \mathcal{A}$ and $E = 0$), consistent with prediction \textbf{A} in \Cref{Introduction}. 
When $\alpha < 1$, \Cref{alpha02}  can be viewed as establishing  a lower bound on the mobility edge in prediction \textbf{B}. 

Let us make four additional comments about the results above. First, although they are only stated for real symmetric matrices, they also apply to complex Hermitian random matrices. In order to simplify notation later in the paper, we only pursue the real case here. 

Second, the exceptional set $\mathcal{A}$ of $\alpha$ to which \Cref{alpha02} does not apply should be empty. Its presence is due to the fact that we use results of \cite{bordenave2017delocalization} stating that certain deterministic, $\alpha$-dependent fixed point equations can be inverted when $\alpha \notin \mathcal{A}$.

Third, our conditions in Definition \ref{momentassumption} allow for the entries of $\textbf{H}$ to be not exactly $\alpha$-stable, but they are not optimal. Although our methods currently seem to require the symmetry of $J$ and $Z + J$, they can likely be extended to establish our main results under weaker moment constraints on $J$. In particular, they should apply assuming only this symmetry, \eqref{probabilityxij}, and that $\mathbb{E}[|J|^{\beta}] < \infty$, for some fixed $\beta > \alpha$. Pursuing this improvement would require altering the statements and proofs of \eqref{exponentestimate}, \Cref{sdiff}, and \Cref{exponentialexpectationx} below (with the primary effort being in the former).\footnote{If $\alpha \in (0, 2) \setminus \mathcal{A}$, then it only suffices to modify the statements and proofs of \Cref{sdiff}, and \Cref{exponentialexpectationx}.} However, for the sake of clarity and brevity, we do not develop this further here. 

Fourth, local statistics of a random matrix $\textbf{H}$ are also quantified through \emph{gap statistics}. For some fixed (possibly $N$-dependent) integer $i$ and uniformly bounded integer $m \ge 0$, these statistics provide the joint distribution of the gaps between the (nearly) neighboring eigenvalues ${\big\{ N (\lambda_j - \lambda_k) \big\}_{|j - i|, |k - i| \le m}}$. Our methods can be extended to establish universality of gap statistics of L\'{e}vy matrices, by replacing the use of Proposition~\ref{universalityperturbation2} below with Theorem 2.5 of \cite{landon2017convergence}, but we do not pursue this here.

\section{Proofs of Delocalization and Bulk Universality} 

\label{Outline2} 

In this section we establish the theorems stated in \Cref{Eigenvalues} assuming some results that will be proven in later parts of this paper. For the remainder of this paper, all matrices under consideration will be real and symmetric. 

Throughout this section, we fix a compact interval $K \subset \mathbb{R}$ and parameters $\alpha, b, \nu, \rho > 0$ satisfying
\begin{flalign}
\label{alphanurho}
\alpha \in (0, 2), \qquad \nu = \frac{1}{\alpha} - b > 0, \qquad 0 < \rho < \nu < \frac{1}{2}, \qquad \displaystyle\frac{1}{4 - \alpha} < \nu < \displaystyle\frac{1}{4 - 2 \alpha}, \qquad \alpha \rho < (2 - \alpha) \nu.
\end{flalign}

Viewing $\alpha \in (0, 2)$ as fixed, one can verify that it is possible to select the remaining parameters $b, \nu, \rho > 0$ such that the conditions \eqref{alphanurho} all hold. The reason for these constraints will be explained in \Cref{ComparisonOutline}. The proofs of \Cref{alpha12} and \Cref{alpha02} will proceed through the following three steps.

\begin{enumerate}
	
	\item First we define a matrix $\textbf{X}$ obtained by setting all entries of $\textbf{H}$ less than $N^{-\nu}$ to zero, and we establish an intermediate local law for $\textbf{X}$ on a certain scale $\eta = N^{-\varpi}$. We will see in \Cref{LocalLaw} that if $\alpha \in (1, 2)$ we can take $\varpi < \nu$, and if $\alpha \le 1$ we can take $\varpi < \frac{1}{2}$. 
	
	\item Next we study $\textbf{V} =\textbf{V}_t = \textbf{X} + t^{1 / 2} \textbf{W}$, for a GOE matrix $\textbf{W}$ and $t \sim N^{\nu (\alpha - 2)}$. The results of \cite{bourgade2017eigenvector,landon2019fixed} imply that if the Stieltjes transform and diagonal resolvent entries of $\textbf{X}$ are bounded on some scale $\eta_0 \ll t$, then all resolvent entries of $\textbf{V}$ are bounded by $N^{\delta}$ on the scale $\eta \sim N^{\delta-  1}$ for any $\delta > 0$, and bulk universality holds for $\textbf{V}$. In particular, this does not require that the resolvent entries of $\textbf{X}$ approximate a deterministic quantity. Thus, the previously mentioned local law for $\textbf{X}$ (which takes place on scale $N^{-\varpi}$, which is less than $t \sim N^{(\alpha - 2) \nu}$) implies that the resolvent entries of $\textbf{V}$ are bounded by $N^{\delta}$ when $\eta = N^{\delta - 1}$, and that the local statistics of $\textbf{V}$ are universal. 
	
	\item Then we establish a comparison theorem between the resolvent entries of $\textbf{H}$ and $\textbf{V}$. Combining this with the estimates on the resolvent entries of $\textbf{V}$ from the previous step, this allows us to conclude that that the resolvent entries of $\textbf{H}$ are bounded by $N^{\delta}$ on the scale $\eta = N^{\delta - 1}$, implying complete eigenvector delocalization for $\textbf{H}$. Further combining this comparison with bulk universality for $\textbf{V}$ will also imply bulk universality for $\textbf{H}$. 
	
\end{enumerate}

We will implement the first, second, and third steps outline above in \Cref{LocalLaw}, \Cref{Eigenvectors}, and \Cref{EigenvectorsEigenvaluesProof}, respectively.

\begin{rem}

In the above outline we use \cite{landon2019fixed} to prove the strongest form of convergence of local statistics, which is given by \eqref{universality2}. However, if one is content to establish this convergence after averaging the eigenvalues over a small interval of size $N^{\delta - 1}$ (known as \emph{averaged energy universality}), then one can instead use Theorem 2.4 of the shorter work \cite{landon2017convergence}. Moreover, if one is only interested in proving complete delocalization for the eigenvectors of $\textbf{H}$, then it suffices to instead use Theorem 2.1 and Proposition 2.2 of \cite{bourgade2017eigenvector}.

\end{rem}

\subsection{The Intermediate Local Laws}

\label{LocalLaw}

In this section we introduce a \textit{removal} variant, denoted by $\textbf{X}$, of our $\alpha$-stable matrix $\textbf{H}$, given by \Cref{partialstable} and \Cref{abremovedmatrix} below. Then, we state two intermediate local laws for $\textbf{X}$, depending on whether $\alpha \in (1, 2)$ or $\alpha \in (0, 2)$. These are given by \Cref{localalpha12} and \Cref{localsmallalpha3}, respectively.

\begin{definition}
	
	\label{partialstable} 
	
	Fix constants $\alpha$ and $b$ satisfying \eqref{alphanurho}, and let $Z$, $J$, and $\mathfrak{z} = Z + J$ be as in \Cref{momentassumption}. Let $Y = \mathfrak{z}\one_{|\mathfrak{z}| \le N^b}$, and let $X = \mathfrak{z} - Y$. We call $X$ the \textit{$b$-removal of a deformed $(0, \sigma)$ $\alpha$-stable law}. 
	
\end{definition}

\begin{definition}
	
	\label{abremovedmatrix}
	
	Let $\{ X_{ij} \}_{1 \le i \le j \le N}$ be mutually independent random variables that each have the same law as $N^{-1 / \alpha} X$, where $X$ is given by \Cref{partialstable}. Set $X_{ij} = X_{ji}$ for each $1 \le j < i \le N$, and define the $N \times N$ matrix $\textbf{X} = \{ X_{ij} \}$. We call $\textbf{X}$ a \textit{$b$-removed $\alpha$-L\'{e}vy matrix}. For any complex number $z \in \mathbb{H}$, define the resolvent $\textbf{R} = \textbf{R} (z) = \{ R_{ij} \}_{1 \le i, j \le N} = (\textbf{X} - z)^{-1}$. Further denote $m = m_N = m_N (z) = N^{-1} \Tr \textbf{R}$, and also set $z = E + \mathrm{i} \eta$ with $E, \eta \in \mathbb{R}$ and $\eta > 0$. 
	
	Now, let $\{ Z_{ij} \}_{1 \le i \le j \le N}$ and $\{ J_{ij} \}_{1 \le i \le j \le N}$ mutually independent random variables that have the same laws as $N^{-1/\alpha} Z$ and $N^{-1/\alpha} J$, respectively, where $Z$ and $J$ are as in \Cref{momentassumption}.  Let $\{ H_{ij} \}_{1 \le i \le j \le N}$ be mutually independent random variables such that $H_{ij} = Z_{ij} + J_{ij}$. Couple each $H_{ij}$ with $X_{ij}$ so that $X_{ij} = H_{ij} - H_{ij} \one_{N^{1 / \alpha} |H_{ij}| \le N^b}$. Set $H_{ij} = H_{ji}$ for each $1 \le j < i \le N$, and define the $N \times N$ matrix $\textbf{H} = \{ H_{ij} \}$. The matrix $\textbf{H}$ is an $\alpha$-L\'{e}vy matrix that is coupled with $\textbf{X}$, and we refer to this coupling as the \emph{removal coupling}. For any $z \in \mathbb{H}$, let $ \textbf{G} (z) = \big\{ G_{ij} (z) \big\} = (\textbf{H} - z)^{-1}$. 
	
\end{definition} 

Now let us state intermediate local laws for the removal matrix $\textbf{X}$ at all energies away from $0$ when $\alpha \in (1, 2)$ (given by \Cref{localalpha12} below) and at sufficiently small energies for almost all $\alpha \in (0, 2)$ (given by \Cref{localsmallalpha3} below). The scale at which the former local law will be stated is $\eta = N^{-\varpi}$ for some $\varpi \in \big( (2 - \alpha) \nu, \nu \big)$, and the scale at which the latter will be is $\eta = N^{\delta - 1 / 2}$ for any $\delta > 0$. These should not be optimal and do not match that at which local laws were proven in finite variance cases, which is $\eta = N^{\delta - 1}$ \cite{aggarwal2019bulk,ajanki2016local,ajanki2017universality,bauerschmidt2017local,che2017universality,erdos2011survey,erdos2013spectral,erdos2013local,erdos2010wegner,erdos2017dynamical,erdos2009local,erdos2012bulk,huang2015spectral}, but they will suffice for our purposes. In fact, one can establish a local law on this optimal scale by combining  \Cref{gvcompare} and \Cref{estimategijrijtij} with \Cref{localalpha12} and \Cref{localsmallalpha3}, but we will not pursue this here. 

The below result will be established in \Cref{LocalEstimate12}. 

\bet

\label{localalpha12}

Fix $\alpha, b, \nu > 0$ satisfying \eqref{alphanurho}. Assume that $\alpha \in (1, 2)$ and $K \subset \mathbb{R} \setminus \{ 0 \}$. Let $\varpi$ be such that 
\begin{flalign*} 
(2 - \alpha) \nu < \varpi < \nu,
\end{flalign*}

\noindent and define the domain 
\begin{flalign}
\label{dcdelta}
\mathcal{D}_{K, \varpi, C} = \left\{ z = E + \mathrm{i} \eta \in \mathbb{H}: \quad E \in K, \quad N^{- \varpi} \le \eta \le C \right\}, 
\end{flalign}

\noindent There exists a small constant $\varkappa = \varkappa(\alpha, b, \nu, \varpi, K) > 0$ and large constants $\mathfrak{B} = \mathfrak{B} (\alpha) > 0$ and $C = C(\alpha, b, \nu, \varpi, K) > 0$ such that 
\begin{flalign}
\label{c1y} 
\begin{aligned}
\mathbb{P} \left[ \displaystyle\sup_{z \in \mathcal{D}_{K, \varpi, \mathfrak{B}}} \big| m_N(z) - m_\alpha(z) \big| > \displaystyle\frac{C}{N^{\varkappa}} \right] & \le  C \exp\left( - \frac{(\log N)^2}{C}  \right), \\
\mathbb{P}\left[ \displaystyle\sup_{z \in \mathcal{D}_{K, \varpi, \mathfrak{B}}} \displaystyle\max_{1 \le j \le N} \big| R_{jj}(z) \big| > C (\log N)^{30/(\alpha-1)} \right] & \le  C \exp\left( - \frac{(\log N)^2}{C}  \right). 
\end{aligned}
\end{flalign}

\eet

\Cref{localalpha12} is similar to Theorem 3.5 of \cite{bordenave2013localization}, but there are several differences. For appropriate choices of constants satisfying constraints \eqref{alphanurho}, we control the Stieltjes transform for $\eta \gg N^{-1/2}$, which essentially equals the scale achieved for $\alpha \in \big( \frac{8}{5}, 2 \big)$ in \cite{bordenave2013localization} and improves the scale ${\eta \gg N^{-\alpha / (8 - 3 \alpha)}}$ achieved for $\alpha \in \big( 1, \frac{8}{5} \big)$ in \cite{bordenave2013localization}. The latter improvement is important for our work because it permits us to access the the critical scale $t \sim N^{(\alpha - 2) \nu}$ for all $\alpha \in (1,2)$. This would not have been possible for $\alpha$ near 1 using the scales achieved in \cite{bordenave2013localization}. \Cref{localalpha12} also asserts estimates on the diagonal resolvent entries $R_{jj} (z)$, which are crucial for our main results but were not estimated in \cite{bordenave2013localization} for any $\alpha$. Finally, in Theorem 3.5 of \cite{bordenave2013localization}, a finite, non-explicit set of energies must be excluded, while we need only exclude the energy $0$. 


Next let us state the intermediate local law for $\textbf{X}$ at sufficiently small energies when $\alpha \in (0, 2) \setminus \mathcal{A}$, which is a consequence of \Cref{localsmallalpha} (and \Cref{s1zomegammualpha}), stated in \Cref{LocalTheorem} below. 

\begin{thm}
	
	\label{localsmallalpha3}
	
	There exists a countable set $\mathcal{A} \subset (0, 2)$, with no accumulation points in $(0, 2)$, that satisfies the following property. Fix $\alpha$ and $b$ satisfying \eqref{alphanurho}, set $\theta = \frac{(b - 1 / \alpha) (2 - \alpha)}{20}$, and let $\delta \in (0, \theta)$. Define the domain 
	\begin{flalign}
	\label{dcdelta2}
	\mathcal{D}_{C, \delta} = \left\{ z = E + \mathrm{i} \eta \in \mathbb{H}: \quad E \le \displaystyle\frac{1}{C}, \quad   N^{\delta - 1 / 2} \le \eta \le \displaystyle\frac{1}{C} \right\}.
	\end{flalign}
	
	\noindent Then there exists a large constant $C = C(\alpha, b, \delta) > 0$ such that 
	\begin{flalign}
	\label{mnzs1zomega2}
	\begin{aligned} 
	& \mathbb{P} \left[ \displaystyle\sup_{z \in \mathcal{D}_{C, \delta}} \Big| m_N (z) - m_{{\alpha}} (z) \Big| > \displaystyle\frac{1}{N^{\alpha \delta / 8}} \right] < C \exp \left( - \displaystyle\frac{(\log N)^2}{C} \right),
	\end{aligned}
	\end{flalign}
	
	\noindent and
	\begin{flalign}
	\label{rijestimate2}
	\mathbb{P} \left[ \displaystyle\sup_{z \in \mathcal{D}_{C, \delta}} \displaystyle\max_{1 \le j \le N} \big| R_{jj} (z) \big| > (\log N)^C \right] < C \exp \left( -\displaystyle\frac{(\log N)^2}{C} \right).  
	\end{flalign}

\end{thm}

\Cref{localsmallalpha3} is similar to Proposition 3.2 of \cite{bordenave2017delocalization}, except that it also bounds the diagonal resolvent entries $R_{jj} (z)$. Furthermore, \Cref{localsmallalpha3} estimates the Stieltjes transform $m_N (z)$ for smaller values of $\eta = \Im z \gg N^{-1 / 2}$ than in Proposition 3.2 of \cite{bordenave2017delocalization}, which requires $\eta \gg N^{-\alpha/(2+\alpha)}$. This improvement is again important for us to access the the critical scale $t \sim N^{(\alpha - 2) \nu}$ for all $\alpha \in (0, 2)$.

\subsection{Estimates for \texorpdfstring{$\textbf{V}$}{}} 

\label{Eigenvectors}

In this section we implement the second step of our outline, in which we define a matrix $\textbf{V} = \textbf{X} + t^{1 / 2} \textbf{W}$, establish that its resolvent entries are bounded by $N^{\delta}$ on scale $N^{\delta - 1}$, and show that its local statistics are universal. 

Recall that $\alpha, b, \nu, \rho > 0$ are parameters satisfying \eqref{alphanurho}, and define $t = t (\rho, \nu)$ by the conditional expectation
\begin{flalign}
\label{t}
t = N \mathbb{E} \Big[ H_{11}^2  \one_{ |H_{11}| < N^{-\nu}} \big| |H_{11}| < N^{-\rho} \Big] = \displaystyle\frac{N \mathbb{E} \big[ H_{11}^2  \one_{ |H_{11}| < N^{-\nu}} \big]}{\mathbb{P} \big[ |H_{11}| < N^{-\rho} \big]}. 
\end{flalign}

We require the following lemma that provides large $N$ asymptotics for $t$; with the definitions of \eqref{alphanurho}, it in particular implies $t = o(1)$. Its proof will be given in \Cref{ProbabilityEstimates} below. 

\bel 

\label{t0estimate} 

There exist a small constant $c = c(\alpha, \nu, \rho) > 0$ and a large constant $C = C(\alpha, \nu, \rho) > 0$ such that
\begin{flalign}
\label{c1c2t0}
c N^{(\alpha -2) \nu} \le t \le C N^{(\alpha -2) \nu}.
\end{flalign}
\eel

Now let us define a matrix $\textbf{V}$ that we will compare to $\textbf{H}$. 

\begin{definition}
	
	\label{vmatrix}
	
	Define the $N \times N$ random matrix $\textbf{V} = \{ v_{ij} \} = \textbf{X} + t^{1 / 2} \textbf{W}$, where $t$ is given by \eqref{t}; $\textbf{X}$ is the removal matrix from \Cref{abremovedmatrix}; and $\textbf{W} = \{ w_{ij} \}$ is an $N \times N$ GOE matrix independent from $\textbf{X}$. For any $z \in \mathbb{H}$, let $\textbf{T} = \textbf{T} (z) =  \{ T_{ij} (z) \} = (\textbf{V} - z)^{-1}$. 
	
\end{definition}

Now we would like to bound the entries of $\textbf{T}$ and show that bulk universality holds for $\textbf{V}$. To do this, we first require the following definition from \cite{landon2017convergence}, which defines a class of initial data on which Dyson Brownian motion is well-behaved. 

\begin{definition}[{\cite[Definition 2.1]{landon2017convergence}}]
	
	\label{eta0rregular}
	
	Let $N$ be a positive integer, let $\textbf{H}_0$ be an $N \times N$ matrix, and set $m_0 (z) = N^{-1} \Tr \big( \textbf{H}_0 - z \big)^{-1}$. Fix $E_0 \in \mathbb{R}$, $\delta \in (0, 1)$, and $\gamma > 0$ independently of $N$. Let $\eta_0$ and $r$ be two ($N$-dependent) parameters satisfying $N^{\delta - 1} \le \eta_0$ and $N^{2 \delta} \eta_0 < r \le 1$. Define
	\begin{flalign}
	\label{definitionde0reta0}
	\begin{aligned} 
	\mathcal{D} (E_0, r, \eta_0, \gamma) & = \left\{ z = E + \mathrm{i} \eta \in \mathbb{H}: \quad E \in \big[ E_0 - r, E_0 + r \big], \quad \eta \in [\eta_0, \gamma] \right\}. 
	\end{aligned} 
	\end{flalign} 
	
	\noindent Although $\mathcal{D} (E_0, r, \eta_0, \gamma)$ in the above definition depends on $\delta$ through the choice of $\eta_0$, we omit this from the notation.
	
	We say that $\textbf{H}_0$ is \emph{$(\eta_0, \gamma, r)$-regular with respect to $E_0$} if there exists a constant $A > 1$ (independent of $N$) such that
	\begin{flalign}
	\label{m0h0estimate}
	\| \textbf{H}_0 \| \le N^A, \qquad \displaystyle\frac{1}{A} < \displaystyle\sup_{z \in \mathcal{D} (E_0, r, \eta_0, \gamma)} \Im m_0 (z) \le A.
	\end{flalign}

\end{definition}

Now let $N$ be a positive integer and $\textbf{H}_0$ denote an $N \times N$ matrix. Recall that $\textbf{W} = \{ w_{ij} \}$ is an $N \times N$ GOE matrix (which we assume to be independent from $\textbf{H}_0$), and define $\textbf{H}_s = \textbf{H}_0 + s^{1 / 2} \textbf{W}$ for each $s > 0$. For each $z \in \mathbb{H}$, let $\textbf{G}_s = \textbf{G}_s (z) = \big\{ G_{ij} (s, z) \big\} = \big( \textbf{H}_s - z \big)^{-1}$. 

If $\textbf{H}_0$ is $(\eta_0, \gamma, r)$-regular and $s$ is between $\eta_0$ and $r$, then the following proposition estimates the entries of $\textbf{G}_s (E + \mathrm{i} \eta)$, when $\eta$ can be nearly of order $N^{-1}$, in terms of estimates on the diagonal entries of $\textbf{G}_0 (E + \mathrm{i} \eta_0)$. Its proof will appear in \Cref{GtEstimate} and is based on results of \cite{bourgade2017eigenvector,landon2017convergence}.

\begin{prop}
	
	\label{mestimategijestimate}
	
	Adopt the notation of \Cref{eta0rregular}, and let $B \in \big( 1, \frac{1}{\eta_0} \big)$ be an $N$-dependent parameter. Assume that $\textbf{\emph{H}}_0$ is $(\eta_0, \gamma, r)$-regular with respect to $E_0$ and that $\max_{1 \le j \le N} \big| G_{jj} (0, z) \big| \le B$ for all $z \in \mathcal{D} (E_0, r, \eta_0, \gamma)$. Let $s \in \big( N^{\delta} \eta_0, N^{-\delta} r)$. Then, for any $D > 1$ there exists a large constant $C = C(\delta, D) > 0$ such that
	\begin{flalign*}
	\mathbb{P} \left[ \displaystyle\sup_{z \in \mathfrak{D}} \displaystyle\max_{1 \le i, j \le N} \big| G_{ij} (s, z) \big| > N^{\delta} B \right] \le C N^{-D},
	\end{flalign*} 
	
	\noindent where we have abbreviated $\mathfrak{D} = \mathcal{D} (E_0, \frac{r}{2}, N^{\delta - 1}, \gamma - \frac{r}{2})$.	
\end{prop}

Now we can bound the entries of $\textbf{T}$.

\begin{cor}
	
	\label{estimatetij1201}
	
	Let $\alpha, b, \nu, \rho > 0$ satisfy \eqref{alphanurho}. For given $E_0\in \R$ and $\delta, \gamma, r > 0$, we abbreviate $\mathfrak{D} = \mathcal{D} \big( E_0, \frac{r}{2}, N^{\delta - 1}, \gamma - \frac{r}{2} \big)$ (as in \eqref{definitionde0reta0}). 
	
	\begin{enumerate}
		
		\item \label{tij12} If $\alpha \in (1, 2)$ and $K \subset \mathbb{R} \setminus \{ 0 \}$ is a compact interval, let $\gamma$ denote the constant $\mathfrak{B} = \mathfrak{B} (\alpha)$ from \Cref{localalpha12}. Let $E_0 \in K$ and $\delta, r > 0$ be constants (independent of $N$) such that $[E_0 - r, E_0 + r] \subset K$ and $r < \gamma$. Then, for any $D > 0$ there exists a large constant $C = C(\alpha, \nu, \rho, \delta, D, K) > 0$ such that
		\begin{flalign}
		\label{tijalpha1201}
		 \mathbb{P} \left[ \displaystyle\sup_{z \in \mathfrak{D}} \displaystyle\max_{1 \le i, j \le N}   \big| T_{ij} (z) \big| > N^{\delta} \right] \le C N^{-D}.
		\end{flalign}
		
		\item \label{tij02} If $\mathcal{A} \subset (0, 2)$ is as in \Cref{localsmallalpha3} and $\alpha \in (0, 2) \setminus \mathcal{A}$, then let $\gamma = \frac{1}{2C}$, where the constant $C$ is from \Cref{localsmallalpha3}. Further let $E_0 \in \mathbb{R}$ and $r \in (0, \gamma)$ be constants (independent of $N$) such that $[E_0 - r, E_0 + r] \subset \big[ - 2 \gamma, 2 \gamma \big]$. Then, for any $\delta, D > 0$, there exists a large constant $C = C(\alpha, \nu, \rho, \delta, D) > 0$ such that \eqref{tijalpha1201} holds. 
	
	\end{enumerate}
	
\end{cor}

\begin{proof}
	
	We assume $\alpha \in (1, 2)$, since the case $\alpha \in (0, 2) \setminus \mathcal{A}$ is entirely analogous. By \Cref{localalpha12}, there exist large constants $\mathfrak{B} = \mathfrak{B} (\alpha) > 0$ and $C = C(\alpha, b, \varpi, \delta, D, K) > 0$ such that
	\begin{flalign}
	\label{rjjzestimatealpha12}
	\mathbb{P} \left[ \displaystyle\sup_{z \in \mathcal{D}_{K, \varpi, \mathfrak{B}}} \displaystyle\max_{1 \le j \le N} \big| R_{jj} (z) \big| > N^{\delta / 4} \right] < C \exp \left( - \displaystyle\frac{(\log N)^2}{C} \right),
	\end{flalign}
	
	\noindent for any $(2 - \alpha) \nu < \varpi < \nu$, where we recall the definition of $\mathcal{D}_{K, \varpi, \mathfrak{B}}$ from \eqref{dcdelta}. Furthermore, observe (after increasing $C$ if necessary) that $\mathbb{P} \big[ \| \textbf{X} \| > N^{(2D + 3) / \alpha} \big] \le C N^{-2D}$, since $\alpha < 2$ and the probability that the magnitude of a given entry of $\textbf{H}$ is larger than $N^{(2 D + 1) / \alpha}$ is at most $C N^{-2 D - 2}$. 
	
	Therefore, we may apply \Cref{mestimategijestimate} with that $\textbf{H}_0$ equal to our $\textbf{X}$; that $\eta_0$ equal to our $N^{-\varpi}$; that $t$ equal to our $t$, which is defined by \eqref{t} and satisfies $t \sim N^{(\alpha - 2) \nu}$ by \Cref{t0estimate}; that $\delta$ to be sufficiently small, so that it is less than our $\frac{\delta}{4}$ and $\frac{1}{4} \big( \varpi - (2 - \alpha) \nu \big)$ (if $\alpha$ were in $(0, 2)$, then we would require that $\delta$ be less than $\frac{1}{4} \big( \frac{1}{2} - (2 - \alpha) \nu \big)$ instead); that $E_0$ equal to the $E_0$ here; that $\gamma$ equal to our $\mathfrak{B}$; that $r$ equal to the $\min \big\{ r, \frac{\mathfrak{B}}{4} \big\}$ here; and that $A$ sufficiently large. Under this choice of parameters, $\textbf{G}_t = \textbf{T}$, so \Cref{mestimategijestimate} implies \eqref{tijalpha1201}.
\end{proof}

We will next show that the local statistics of $\textbf{V}$ are universal, which will follow from the results of \cite{erdos2017universality,landon2017convergence,landon2019fixed} together with the intermediate local laws \Cref{localalpha12} and \Cref{localsmallalpha3}. Specifically, the results of \cite{erdos2017universality, landon2017convergence, landon2019fixed} state that, if we start with a $(\eta_0, \gamma, r)$-regular matrix (recall \Cref{eta0rregular}) and then add an independent small Gaussian component of order greater than $\eta_0$ but less than $r$, then the local statistics of the result will asymptotically coincide with those of the GOE. To state this more precisely, we must introduce the free convolution \cite{biane1997free} of a probability distribution with the semicircle law. 

Fix $N \in \mathbb{Z}_{> 0}$ and an $N \times N$ matrix $\textbf{A}$. For each $s \ge 0$, define $\textbf{A}^{(s)} = \textbf{A} + s^{1 / 2} \textbf{W}$, where $\textbf{W}$ is an $N \times N$ GOE matrix. For any $z \in \mathbb{H}$, also define $m^{(s)} (z)  = N^{-1} \Tr \big( \textbf{A}^{(s)} - z \big)^{-1}$ to be the Stieltjes transform of the ($N$-dependent) empirical spectral density of $\textbf{A}^{(s)}$, which we denote by $\rho^{(s)} (x) = \pi^{-1} \lim_{\eta \rightarrow 0} \Im m^{(s)} (E + \mathrm{i} \eta)$. 

The following proposition establishes the universality of correlation functions of the random matrix $\textbf{M}^{(s)}$, assuming that $\textbf{M}$ is regular in the sense of \Cref{eta0rregular}.

\begin{prop}[{\cite[Theorem 2.2]{landon2019fixed}}]
	
	\label{universalityperturbation2}
	
	Fix some $\delta \in (0, 1)$ and $\gamma > 0$, let $N$ be a positive integer, and let $r \in (0, N^{-\delta})$ and $\eta_0 \in (N^{\delta - 1}, 1)$ be $N$-dependent parameters satisfying $\eta_0 < N^{-2 \delta} r$. Let $\textbf{\emph{M}}$ be an $N \times N$ matrix, and assume that $\textbf{\emph{M}}$ is $(\eta_0, \gamma, r)$-regular with respect to some fixed $E \in K$. Then, for any $s \in \big( N^{\delta} \eta_0, N^{-\delta} r \big)$, the correlation functions of $\textbf{\emph{M}}^{(s)}$ are universal at energy level $E$ with respect to $\rho^{(s)}$, as in \Cref{gapscorrelations}.
	
\end{prop}

Using \Cref{universalityperturbation2}, one can deduce the following result. In what follows, we recall the matrices $\textbf{X}$ and $\textbf{V} = \textbf{X}^{(t)}$ from \Cref{abremovedmatrix} and \Cref{vmatrix}, respectively (where $t$ was given by \eqref{t}).

\begin{prop}
	
	\label{universalityperturbation3}

	Assume $\alpha \in (1, 2)$ and $K \subset \mathbb{R} \setminus \{ 0 \}$, and let $E \in K$. Then the correlation functions of \textbf{\emph{V}} are universal at energy level $E$ with respect to $\varrho_\alpha$, as in \Cref{gapscorrelations}. Moreover, the same statement holds if $\mathcal{A}$ and $C$ are as in \Cref{localsmallalpha3}, $\alpha \in (0, 2) \setminus \mathcal{A}$, and $E \subset \big[ - \frac{1}{2C}, \frac{1}{2C} \big]$.
	
\end{prop}

To establish this proposition, one conditions on $\textbf{X}$ and uses its intermediate local law (\Cref{localalpha12} or \Cref{localsmallalpha3}) and \Cref{t0estimate} to verify the assumptions of \Cref{universalityperturbation2}. Then, the latter proposition implies that the correlation functions of $\textbf{V}$ are universal at $E$ with with respect to $\rho^{(s)}$. The remaining difference between universality with respect to $\rho^{(s)} (E)$ and the desired result is in the scaling in \eqref{universality2}. Specifically, one must approximate the factors of $\rho^{(s)} (E)$ by $\varrho_{\alpha} (E)$ in \Cref{gapscorrelations}. This approximation can be justified using the intermediate local law (\Cref{localalpha12} or \Cref{localsmallalpha3}) for $\textbf{X}$ through a very similar way to what was explained in Lemma 3.3 and Lemma 3.4 of \cite{huang2015bulk}. Thus, we omit further details.

\subsection{Proofs of \Cref{alpha12} and \Cref{alpha02}}

\label{EigenvectorsEigenvaluesProof} 

In this section we establish \Cref{alpha12} and \Cref{alpha02}. This will proceed through a comparison between the resolvent entries of $\textbf{H}$ and $\textbf{V}$ (from \Cref{vmatrix}). In \Cref{CompareEstimates}, we state this comparison; we will provide a heuristic for its proof in \Cref{ComparisonOutline}, and the result will be established in detail in \Cref{Compare}. We will then in \Cref{HStatistics} use the comparison to deduce eigenvector delocalization and bulk universality for $\textbf{H}$ from the corresponding results for $\textbf{V}$ established in \Cref{Eigenvectors}.

\subsubsection{The Comparison Theorem}

\label{CompareEstimates}

To formulate our specific comparison statement, we require a certain way of decomposing the matrix $\textbf{H}$ so that the elements of this decomposition remain largely independent. A less general version of this procedure was described in \cite{aggarwal2019bulk} under different notation to establish bulk universality for Wigner matrices whose entries have finite $(2 + \varepsilon)$-th moment. This is done through the following two definitions.

\begin{definition}
	
	\label{chipsi}
	
	Let $\psi$ and $\chi$ be independent Bernoulli $0-1$ random variables defined by 
	\begin{flalign*}
	\mathbb{P} \big[ \psi = 1 \big] = \mathbb{P} \big[ |H_{ij}| \ge N^{-\rho} \big], \qquad \mathbb{P} \big[ \chi = 1 \big] = \displaystyle\frac{\mathbb{P} \big[ |H_{ij}| \in [N^{-\nu}, N^{-\rho}) \big]}{\mathbb{P} \big[ |H_{ij}| < N^{-\rho} \big]}.
	\end{flalign*} 
	
	\noindent In particular, $\psi$ has the same law as the indicator of the event that $|H_{ij}| \ge N^{-\rho}$. Similarly, $\chi$ has the same law as the indicator of the event that $|H_{ij}| \ge N^{-\nu}$, conditional on $|H_{ij}| < N^{-\rho}$. 
	
	Additionally, let $a$, $b$, and $c$ be random variables such that 
	\begin{flalign*}
	\P(a_{ij} \in I ) = & \frac{\P \Big[ H_{ij} \in  (-N^{-\nu}, N^{-\nu})  \cap I \Big]}{\P \big[ |H_{ij}| < N^{-\nu}) \big]}, \quad \P [c_{ij} \in I] = \frac{\P \Big[ H_{ij} \in \big( (-\infty, -N^{-\rho}] \cup [N^{-\rho}, \infty) \big) \cap I \Big] }{\P \big[ |H_{ij}| \ge N^{-\rho} \big]}, \\
	& \qquad \quad \P(b_{ij} \in I ) = \frac{\P \Big[ H_{ij} \in \big( (- N^{-\rho}, - N^{-\nu}]  \cup [N^{-\nu}, N^{-\rho}) \big)  \cap I \Big]}{\P \big[ |H_{ij}| \in  [N^{-\nu}, N^{-\rho}) \big]},
	\end{flalign*} 
	
	\noindent for any interval $I\subset \R$. Again, $a$ has the same law as $H_{ij}$ conditional on $|H_{ij}| < N^{-\nu}$; similar statements hold for $b$ and $c$.

\end{definition}

Observe that if $a$, $b$, $c$, $\psi$, and $\chi$ are mutually independent, then $H_{ij}$ has the same law as $(1 - \psi) (1 - \chi) a + (1 - \psi) \chi b + \psi c$ and $X_{ij}$ has the same law as $(1 - \psi) b + \psi c$. Thus, although the random variables $H_{ij} \one_{|H_{ij}| \ge N^{-\rho}}$, $H_{ij} \one_{N^{-\nu} \le |H_{ij}| < N^{-\rho}}$, and $H_{ij} \one_{|H_{ij}| < N^{-\nu}}$ are correlated, this decomposition expresses their dependence through the Bernoulli random variables $\psi$ and $\chi$. 

\begin{definition}
	
	\label{hsumabc}
	
	For each $1 \le i \le j \le N$, let $a_{ij}$, $b_{ij}$, $c_{ij}$, $\psi_{ij}$, and $\chi_{ij}$ be mutually independent random variables whose laws are given by those of $a$, $b$, $c$, $\psi$, and $\chi$ from \Cref{chipsi} respectively. For each $1 \le j < i \le N$, define $a_{ij} = a_{ji}$ by symmetry, and similarly for each $b_{ij}$, $c_{ij}$, $\psi_{ij}$, and $\chi_{ij}$.  Let $\mathbb{P}$ and $\mathbb{E}$ denote the probability measure and expectation with respect to the joint law of these random variables, respectively. 
	
	Now for each $1 \le i, j \le N$, set
	\begin{flalign}
	\label{abc}
	A_{ij} = (1 - \psi_{ij}) (1 - \chi_{ij}) a_{ij}, \quad B_{ij} = (1 - \psi_{ij}) \chi_{ij} b_{ij}, \quad C_{ij} = \psi_{ij} c_{ij},
	\end{flalign}
	
	\noindent and define the four $N \times N$ matrices $\textbf{A} = \{ A_{ij} \}$, $\textbf{B} = \{ B_{ij} \}$, $\textbf{C} = \{ C_{ij} \}$, and $\Psi = \{ \psi_{ij} \}$.
	
	Sample $\textbf{H}$ and $\textbf{X}$ by setting $\textbf{H} = \textbf{A} + \textbf{B} + \textbf{C}$ and $\textbf{X} = \textbf{B} + \textbf{C}$. We will commonly refer to $\Psi$ as the \emph{label} of $\textbf{H}$ (or of $\textbf{X}$). Defining $\textbf{H}$ and $\textbf{X}$ in this way ensures that they have the same laws as in \Cref{momentassumption} and \Cref{abremovedmatrix}, respectively. Furthermore, this sampling induces a coupling between $\textbf{H}$ and $\textbf{X}$, which coincides with the removal coupling of \Cref{abremovedmatrix}. 
\end{definition}

To state our comparison results, we require some additional notation. Define $\textbf{A}$, $\textbf{B}$, $\textbf{C}$, $\textbf{H}$, and $\textbf{X}$ as in \Cref{hsumabc}, and let $\textbf{W} = \{ w_{ij} \}$ be an independent $N \times N$ GOE matrix. Recalling the parameter $t$ from \eqref{t}, define for each $\gamma \in [0, 1]$ the $N \times N$ random matrices  
\begin{flalign*}
\textbf{H}^{\gamma} = \big\{ H_{ij}^{\gamma} \big\} = \gamma \textbf{A} + \textbf{X} + (1 - \gamma^2)^{1 / 2} t^{1 / 2} \textbf{W}, \qquad \textbf{G}^{\gamma} = \big\{ G_{ij}^{\gamma} \big\} = \big( \textbf{H}^{\gamma} - z \big)^{-1}. 
\end{flalign*}

\noindent Observe in particular that $\textbf{H}^{0} = \textbf{V}$, $\textbf{G}^0 = \textbf{T}$, $\textbf{H}^1 = \textbf{H}$, and $\textbf{G}^1 = \textbf{G}$, where we recall the matrices $\textbf{V}$ and $\textbf{T}$ from \Cref{vmatrix}. Our comparison result will approximate the entries of $\textbf{G}^{\gamma}$ by those of $\textbf{G}^0$ for any $\gamma \in [0, 1]$, after conditioning on $\Psi$ and assuming it to be in an event with high probability with respect to $\mathbb{P}$. 

So, it will be useful to consider the laws of $\textbf{H}$ and $\textbf{X}$ conditional on their label $\Psi$. This amounts to conditioning on which entries of $\textbf{H}$ are at least $N^{-\rho}$. For any $N \times N$ symmetric $0-1$ matrix $\Psi$, let $\mathbb{P}_{\Psi}$ and $\mathbb{E}_{\Psi}$ denote the probability measure and expectation with respect to the joint law of the random variables $\big\{ a_{ij}, b_{ij}, c_{ij}, \psi_{ij}, \chi_{ij} \}$ from \Cref{hsumabc} conditional on the event that $\{ \psi_{ij} \}$ is equal to $\Psi$. This induces a probability measure and expectation on the $\textbf{H}^{\gamma}$ and $\textbf{G}^{\gamma}$, denoted by $\mathbb{P}_{\Psi}$ and $\mathbb{E}_{\Psi}$, respectively. 

It will also useful for us to further condition on a single $\chi_{ij}$. Thus, for any $\chi \in \{ 0, 1 \}$ and $1 \le p, q \le N$, let $\mathbb{P}_{\Psi} \big[ \cdot \big| \chi_{pq} \big] = \mathbb{P}_{\Psi} \big[ \cdot \big| \chi_{pq} = \chi \big]$ denote the probability measure $\mathbb{P}_{\Psi}$ after additionally conditioning on the event that $\chi_{pq} = \chi$, and let $\mathbb{E}_{\Psi} \big[ \cdot \big| \chi_{pq} \big] = \mathbb{E}_{\Psi} \big[ \cdot \big| \chi_{pq} = \chi \big]$ denote the associated expectation. Observe in particular that $\mathbb{E}^{\chi} \big[ \mathbb{E}_{\Psi} [ \cdot | \chi_{pq}] \big] = \mathbb{E}_{\Psi} \big[ \cdot \big]$, where $\mathbb{E}^{\chi}$ denotes the expectation with respect to the Bernoulli $0-1$ random variable $\chi$ from \Cref{chipsi}. 

The following theorem, which will be a consequence of \Cref{derivativeestimate} stated in \Cref{CompareOutline} below, provides a way to compare conditional expectations of $\textbf{G}^0$ to those of $\textbf{G}^{\gamma}$ for any $\gamma \in [0, 1]$.

\bet

\label{gvcompare} 

Let $\alpha, b, \rho, \nu$ satisfy \eqref{alphanurho}, and fix a positive integer $m$. Then, there exist (sufficiently small) constants $\varepsilon = \varepsilon (\alpha, \nu, \rho, m) > 0$ and $\omega = \omega (\alpha, \nu, \rho, m) > 0$ such that the following holds. Let $N$ be a positive integer. For each integer $j \in [1, m]$, fix real numbers $E_j \in \mathbb{R}$ and $\eta_j > N^{-2}$, and denote $z_j = E_j + \mathrm{i} \eta_j$ for each $j \in [1, m]$. Furthermore, let $F: \mathbb{R}^m \rightarrow \mathbb{R}$ be a function such that
\begin{flalign} 
\label{fc0c0}
\displaystyle\sup_{\substack{0 \le |\mu| \le d \\ |x_j| \le 2 N^{\varepsilon}}} \big| F^{(\mu)} (x_1, \ldots , x_m) \big| \le N^{C_0 \varepsilon}, \qquad \displaystyle\sup_{\substack{0 \le |\mu| \le d \\ |x_j| \le 2 N^2}} \big| F^{(\mu)} (x_1, \ldots , x_m) \big| \le N^{C_0},
\end{flalign}

\noindent for some real numbers $C_0, d > 0$. Here $\mu = (\mu_1, \mu_2, \ldots , \mu_m)$ is an $m$-tuple of nonnegative integers, $|\mu| = \sum_{j = 1}^m \mu_j$, and $F^{(\mu)} =  \prod_{j = 1}^m \big( \frac{\partial}{\partial x_j} \big)^{\mu_j} F$. Assume that $d > d_0 (\alpha, \nu, \rho, m, C_0)$ is sufficiently large. For any symmetric $0-1$ matrix $\Psi$ and complex number $z$, define the quantities $\mathfrak{J} = \mathfrak{J} (\Psi)$ and $Q_0 = Q_0 (\varepsilon, z_1, z_2, \ldots , z_m, \Psi)$ and the event $\Omega_0 = \Omega_0 (\varepsilon, z)$ by 
\beq
\label{e:jdef} 
\mathfrak{J} = \displaystyle\max_{0 \le |\mu| \le d} \displaystyle\sup_{\substack{1 \le i_s, j_s \le N \\ 0 \le \gamma \le 1}}  \mathbb{E}_{\Psi} \bigg[ \Big| F^{(\mu)} \big( \Im G_{i_1 j_1}^\gamma, \ldots ,  \Im G_{i_m j_m}^\gamma \big) \Big| \bigg], 
\eeq

\noindent and 
\begin{flalign} 
\label{omegak1} 
\begin{aligned}
\Omega_0 = \Bigg\{ \displaystyle\sup_{\substack{1 \le i, j \le N \\ 0 \le \gamma \le 1}} \big| G^\gamma_{ij} (z) \big| \le N^{\varepsilon} \Bigg\}, \qquad Q_0 = 1 - \displaystyle\sum_{j = 1}^m \P_{\Psi} \big[ \Omega_0 (z_j) \big].
\end{aligned}
\end{flalign} 

\noindent Now let $\Psi$ be a symmetric $0-1$ random matrix with at most $N^{1 + \alpha \rho + \varepsilon}$ entries equal to $1$. Then, there exists a large constant $C = C(\alpha, \nu, \rho, m) > 0$ such that
\begin{flalign}
\begin{aligned}
\label{gvcompareestimate}
\displaystyle\sup_{0 \le \gamma \le 1} \bigg| \mathbb{E}_{\Psi} \Big[ F \big( \Im G_{a_1 b_1}^{\gamma}, \ldots , \Im G_{a_m b_m}^{\gamma}\big)  \Big] -  \mathbb{E}_{\Psi} \Big[ F \big( & \Im G_{a_1 b_1}^0, \ldots , \Im G_{a_m b_m}^0 \big)  \Big] \bigg| \\
& < C N^{-\omega} (\mathfrak{J} + 1) + C Q_0 N^{C + C_0}, 
\end{aligned}
\end{flalign}

\noindent for any indices $1 \le a_1, a_2, \ldots , a_m, b_1, b_2, \ldots , b_m \le N$. The same estimate \eqref{gvcompareestimate} holds if some of the $\Im G_{a_j b_j}^0$ and $\Im G_{a_j b_j}^{\gamma}$ are replaced by $\Re G_{a_j b_j}^0$ and $\Re G_{a_j b_j}^{\gamma}$, respectively.
\eet

Although the conditioning on the label $\Psi$ might notationally obscure the statement of \Cref{gvcompare}, we will see in \Cref{DecreaseEta} that this particular statement of the result will be useful for the proof of \Cref{decreaseetasum} below. Additionally, we note the constants $\varepsilon$, $\omega$, and $d_0$ from \Cref{gvcompare} are explicit; see \eqref{omegadefinition} and \eqref{c0d} for their values in the case $m = 1$.

\subsubsection{Eigenvector Delocalization and Bulk Universality for \texorpdfstring{$\textbf{\emph{H}}$}{}}

\label{HStatistics}

In this section we establish \Cref{alpha12} and \Cref{alpha02}. We first show that the resolvent entries of $\textbf{H}$ are bounded by $N^{\delta}$ on the nearly optimal scale $\eta = N^{\delta - 1}$ for arbitrarily small $\delta > 0$. 

\begin{thm}
	
	\label{estimategijrijtij} 
	
	In both regimes \eqref{tij12} and \eqref{tij02} in \Cref{estimatetij1201}, we have for sufficiently large $C = C(\alpha, \nu, \rho, \delta, D, K) > 0$ that 
		\begin{flalign}
		\label{gijalpha1201}
		\mathbb{P} \left[ \displaystyle\sup_{0 \le \gamma \le 1} \displaystyle\sup_{z \in \mathfrak {D}} \displaystyle\max_{1 \le i, j \le N} \big| G_{ij}^{\gamma} (z) \big| > N^{\delta} \right] \le C N^{-D}. 
		\end{flalign}
	
\end{thm}

\Cref{estimategijrijtij} is a consequence of \Cref{estimatetij1201} and the following comparison result, which allows one to deduce bounds on the entries of $\textbf{G}^{\gamma}$ from bounds on those of $\textbf{T}$; the latter result will be established using \Cref{gvcompare} in \Cref{DecreaseEta} below.

\begin{prop} 
	
	\label{decreaseetasum}
	
	Assume that $\alpha, b, \nu, \rho > 0$ satisfy \eqref{alphanurho}, and recall that $K \subset \mathbb{R}$ is a compact interval. Fix $\varsigma \ge 0$, and suppose that for each $\delta > 0$ and $D > 0$ there exists a constant $C = C(\alpha, \rho, \nu, \delta, D, K)$ such that
	\begin{flalign}
	\label{tijestimate}
	\P \Bigg[ \displaystyle\sup_{\eta \ge N^{\varsigma - 1}} \displaystyle\sup_{E \in K} \displaystyle\max_{1 \le i,j \le N} \big| T_{ij}(E + \mathrm{i} \eta ) \big| \ge N^{\delta} \Bigg] \le C N^{-D}. 
	\end{flalign}
	
	\noindent Then, for each $\delta > 0$ and $D > 0$ there exists a large constant $A =  A( \alpha, \rho, \nu, \delta, D, K)$ such that
	\begin{flalign}
	\label{gijestimate}
	\P \Bigg[ \displaystyle\sup_{0 \le \gamma \le 1} \displaystyle\sup_{\eta \ge N^{\varsigma - 1}} \displaystyle\sup_{E \in K} \displaystyle\max_{1 \le i,j \le N} \big| G_{ij}^{\gamma} (E + \mathrm{i} \eta ) \big| \ge N^{\delta} \Bigg] \le A N^{-D}. 
	\end{flalign}
\end{prop}

Now we can establish \Cref{alpha12} and \Cref{alpha02}.

\begin{proof}[Proof of \Cref{alpha12} and \Cref{alpha02}]
	
	It is known from Corollary 3.2 of \cite{erdos2012bulk} that complete eigenvector delocalization of the form given by the first parts of \Cref{alpha12} and \Cref{alpha02} follows from bounds on the resolvent entries $\big| G_{ij} (z) \big| = \big| G_{ij}^1 (z) \big|$ of the form \eqref{gijalpha1201}. Thus, the first parts of \Cref{alpha12} and \Cref{alpha02} follow from \Cref{estimategijrijtij}. 

To establish the second parts of these two theorems, fix a positive integer $m$, and let $z_1, z_2, \ldots , z_m \in \mathbb{C}$ be such that $\Im z_j \ge \frac{1}{N^2}$ for each $j \in [1, N]$. Furthermore, if we are in the setting of \Cref{alpha12} then we additionally impose that each $\Re z_j \in K$: if we are in the setting of \Cref{alpha02}, then we require that each $\big| \Re z_j \big| < \frac{1}{2C}$, where $C$ is from \Cref{localsmallalpha3}. We now apply \Cref{gvcompare} with $F (x_1, x_2, \ldots , x_m) = \prod_{i = 1}^m x_i$. 

Then, \Cref{estimategijrijtij} implies that the quantity $Q_0$ from \Cref{gvcompare} is bounded above by $N^{-D}$ for any $D > 0$ if $N$ is sufficiently large. Furthermore, that theorem and the deterministic bounds $|T_{ij}|, |G_{ij}| \le N^2$ (due to \eqref{gijeta} below) imply that for each $\delta > 0$ there exists a constant $C = C(\delta)$ such that the quantity $\mathfrak{J} (\Psi)$ from \eqref{e:jdef} is bounded by $C N^{\delta}$. Also observe from \eqref{probabilityxij} and the Chernoff bound that there exists a large constant $C > 0$ such that
\begin{flalign}
\label{hijnumber}
\begin{aligned}
\P \bigg[ \Big| \big\{ (i,j) \colon |H_{ij}| \in [N^{-\rho}, \infty) \big\}   \Big| \notin \Big[ \displaystyle\frac{ N^{1+ \alpha \rho}}{C}, C N^{1+ \alpha \rho}  \Big] \bigg] < C e^{- N / C}.
\end{aligned}
\end{flalign}

Thus, the probability that the matrix $\Psi$ from \Cref{gvcompare} has more than $N^{1 + \alpha \rho + \varepsilon}$ entries equal to one is bounded by $c^{-1} e^{-cN}$ for some constant $c > 0$. On this event, we apply the deterministic bounds $|T_{ij}|, |G_{ij}| \le N^2$. Off of this event, we apply \eqref{gvcompareestimate} (averaged over all $(a_1, a_2, \ldots , a_m) = (b_1, b_2, \ldots , b_m)$ in $[1, N]$) and then average over $\Psi$ conditional on the event that $\Psi$ has at most $N^{1 + \alpha \rho + \varepsilon}$ entries equal to one. Combining these estimates implies
\begin{flalign}
\label{gtestimate}
\begin{aligned}
& \Bigg| \E \bigg[ N^{-m} \displaystyle\prod_{j = 1}^m \Im \Tr \textbf{G} (z_j) - N^{-m} \displaystyle\prod_{j = 1}^m \Im \Tr \textbf{T} (z_j)  \bigg] \Bigg|  \le   C N^{ - c},
\end{aligned}
\end{flalign} 

\noindent after increasing $C$ and decreasing $c$ if necessary. It is known from Theorem 6.4 of \cite{erdos2012bulk} that a comparison of this form implies that the correlation functions of $\textbf{G}$ and $\textbf{T}$ asymptotically coincide. Now the universality of the correlation functions for $\textbf{H}$ at energy level $E$ follows from the corresponding statement for $\textbf{V}$, given by \Cref{universalityperturbation3}.
\end{proof}

\section{Comparison Results} 

\label{Compare}

In this section we establish \Cref{gvcompare}. First, we provide in \Cref{ComparisonOutline} a heuristic proof of this result. Next, assuming \Cref{gvcompare}, we use it to establish \Cref{decreaseetasum} in \Cref{EtaDecrease}. We then outline the proof of \Cref{gvcompare} in \Cref{CompareOutline} and implement this outline in the remaining sections: \Cref{REstimatesU}, \Cref{Degree134}, and \Cref{Degree2}.

\subsection{A Heuristic for the Comparison}

\label{ComparisonOutline}

Here we provide a heuristic for the estimate \eqref{gvcompareestimate} if ${a = i = b}$ for some $i \in [1, N]$. Conditioning on $\Psi$ (and abbreviating $\mathbb{E}_{\Psi}$ as $\mathbb{E}$ here for brevity), we obtain
\begin{flalign*}
\partial_\gamma  \E \big[ G^\gamma_{ii} \big] = \sum_{1 \le j, k \le N} \E \Bigg[ G^\gamma_{ij} \bigg( A_{jk} - \frac{ \gamma t^{1 / 2}}{(1 - \gamma^2)^{1 / 2}}   w_{jk} \bigg)  G^\gamma_{ki} \Bigg]. 
\end{flalign*}

Now let us consider two cases. The first is the ``large field case,'' meaning that $\psi_{jk} = 1$ (which implies that $A_{jk} = 0 = B_{jk}$ and $|H_{jk}| \ge N^{-\rho}$). Recall the formula for Gaussian integration by parts (see, for example, Appendix A.4 of \cite{talagrand2010mean}): for a differentiable function $F\colon \mathbb R \rightarrow \mathbb R$ subject to a mild growth condition, and a centered Gaussian $g$, $\E \left[ g F(g) \right ] = \E\left[ g^2 \right] \E  \left[F'(g)\right]$. We integrate by parts with respect to the Gaussian random variable $x = N^{1 / 2} w_{jk}$, which is centered and has variance one. This yields  
\begin{flalign*}
 \frac { \gamma  }  {(1 - \gamma^2)^{1 / 2} } \left( \displaystyle\frac{t}{N} \right)^{1 / 2}  \E \big[ G^\gamma_{ij}   x G^\gamma_{ki} \big] =   \displaystyle\frac{\gamma t}{N} \E \big[ G^\gamma_{ij}  G^\gamma_{kk}  G^\gamma_{ki} + \cdots \big],
\end{flalign*}

\noindent where the additional terms are degree three monomials in the $G_{ij}^{\gamma}$. Assuming that each $\big| G_{ij}^{\gamma} \big|$ is bounded, and using \Cref{t0estimate} and the fact that the number of pairs $(j, k)$ for which $\psi_{jk} = 1$ is essentially bounded by $N^{\alpha \rho + 1}$, we can bound the total contribution of these terms by a multiple of
\begin{flalign*}
t N^{-1} N^{\alpha \rho + 1} \sim N^{\nu(\alpha - 2) + \alpha \rho}.  
\end{flalign*}

The second is the ``small field case,'' meaning that $\psi_{jk}=0$ (so $|H_{jk}| < N^{-\rho}$). Recall that $A_{jk} = a_{jk} (1 - \chi_{jk})$ and $B_{jk} = b_{jk} \chi_{jk}$, and abbreviate $a_{jk} = a$, $b_{jk} = b$, $\chi_{jk} = \chi$, and $w_{jk} = w$. Letting $\textbf{U}^{\gamma} = \big\{ U_{jk}^{\gamma} \big\}$ denote the resolvent of $\textbf{H}$ whose $(j, k)$ and $(k, j)$ entries are set to zero, we can expand the $(j, k)$-entry of $\textbf{G}^\gamma$ to obtain 
\begin{flalign*}
\E \Bigg[ & G^\gamma_{i k } \bigg( (1-\chi) a - \frac { \gamma t^{1 / 2} w }  {(1 - \gamma^2)^{1 / 2}} \bigg) G^\gamma_{j i } \Bigg] \\
& = \E \Bigg[  \bigg( (1-\chi) a + \frac{ \gamma t^{1 / 2} w}{(1 - \gamma^2)^{1 / 2}}  \bigg) \bigg(  \gamma  (1-\chi) a + \chi b + (1 - \gamma^2)^{1 / 2} t^{1 / 2} w \bigg) \big( U_{i j}^{\gamma} U^\gamma_{kk }  U^\gamma_{j i } + \cdots \big) \Bigg] \\
& \quad + \E \Bigg[ \bigg( (1-\chi) a + \frac{ \gamma t^{1 / 2} w}{(1 - \gamma^2)^{1 / 2}}  \bigg) \bigg(  \gamma  (1-\chi) a + \chi b + (1 - \gamma^2)^{1 / 2} t^{1 / 2} w \bigg)^3 \big( U^{\gamma} \cdots U^{\gamma} + \cdots \big) \Bigg] \\
& = \gamma \mathbb{E} \Big[ (1 - \chi) a^2 - tw^2 \Big] \E [   U_{i j}^{\gamma} U^\gamma_{kk }  U^\gamma_{j i } + \cdots ] \\
& \quad + \gamma \E \Big[ \gamma^2 (1-\chi) a^4 + 3 \gamma^2 (1 - \chi) tw^2 a^2 + 3 \chi t w^2 b^2 + (1 - \gamma^2) t^2 w^4 \Big] \mathbb{E} [U^{\gamma} \cdots U^{\gamma} + \cdots ],
\end{flalign*}

\noindent where the additional terms refer to polynomials in the entries of $\textbf{U}$. To deduce the first equality, we used the fact that terms not involving a factor of $\gamma (1-\chi) a + \chi b + (1 - \gamma^2)^{1 / 2} t^{1 / 2} w$ (first order terms) and those involving $\big(  \gamma a (1-\chi) + b \chi + (1 - \gamma^2)^{1 / 2} t^{1 / 2} w \big)^2$ (third order terms) vanish because $a$, $b$, and $w$ are symmetric and $\textbf{U}$, $a$, $b$, $w$, and $\chi$ are mutually independent.  

From the choice of $t$, we have that
\begin{flalign*}
\gamma \E \big[ (1-\chi) a^2 - t w^2 \big] = 0.
\end{flalign*}

\noindent Hence the second order terms vanish if $\psi_{jk} = 0$. Assuming that the entries of $\textbf{U}^{\gamma}$ are bounded, we can also estimate the sum of all fourth order terms by a multiple of
\begin{flalign*}
N^2  \E \Big[ (1-\chi) a^4 + (1 - \chi) t w^2 a^2 + \chi t w^2 b^2 + t^2 w^4 \Big]  & \le  N^{\nu(\alpha - 4) +1} +  N^{(\rho+ \nu)(\alpha - 2)} + N^ {2\nu(\alpha - 2)} \\
&\le  N^{\nu(\alpha -4) + 1} +   N^{-r},
\end{flalign*}

\noindent for some $r > 0$. Here, we used \eqref{probabilityxij}, \Cref{t0estimate}, and the facts that $(1 - \chi) a = H_{ij} \one_{|H_{ij}| < N^{- \nu}}$ and $\chi b = H_{ij} \one_{N^{-\nu} \le |H_{ij}| < N^{-\rho}}$ to deduce that $\mathbb{E} \big[ (1 - \chi) a^4 \big] \sim N^{\nu (\alpha - 4) - 1}$, $\mathbb{E} \big[ (1 - \chi) a^2] \sim N^{(\alpha - 2) \nu - 1}$, and $\mathbb{E} [ \chi b^2] \sim N^{\rho (\alpha - 2) - 1}$. 

Hence the total contribution from the second and fourth order terms is bounded by a multiple of
\begin{flalign*}
N^{\nu(\alpha - 2) + \alpha \rho}+ N^{\nu (\alpha - 4) + 1} +   N^{-r}. 
\end{flalign*}

\noindent For this to tend to $0$, we require 
\begin{flalign*}
\nu > \frac 1 { 4-\alpha }, \qquad \alpha \rho < (2 - \alpha ) \nu, \qquad 0 < \rho < \nu,
\end{flalign*}

\noindent where the last restriction is by definition. This recovers a number of the constraints imposed by \eqref{alphanurho}. To motivate the others, recall that the local law for $\textbf{X}$ was proved for any scale $N^{-\varpi}$ with $(2 - \alpha) \nu < \varpi < \nu < \frac{1}{2}$ for $\alpha \in (1, 2)$ (\Cref{localalpha12}) and at the scale $N^{\delta-1/2}$ for almost all $\alpha \in (0, 2)$ in the small energy regime (\Cref{localsmallalpha3}). In order to apply the results on Dyson Brownian motion from \Cref{Eigenvectors}, we need the scale of these local laws to be smaller than $t \sim N^{\nu(\alpha - 2)}$. For $\alpha \in (1, 2)$, this condition is guaranteed. For $\alpha \le 1$, this requires $\nu  < \frac{1}{4 - 2 \alpha}$, which is the remaining condition in \eqref{alphanurho}.

\subsection{Improving the Scale}

\label{EtaDecrease} 

In this section we establish \Cref{decreaseetasum} in \Cref{DecreaseEta}, after stating some preliminary estimates and identities in \Cref{ProbabilityEstimates}.

\subsubsection{Estimates and Identities}

\label{ProbabilityEstimates}

In this section we state several identities and estimates that will be used throughout this article. We first recall that, for any square matrices $\textbf{M}$ and $\textbf{K}$ of the same dimension, we have the resolvent identity
\begin{flalign}
\label{abidentity}
\textbf{K}^{-1} - \textbf{M}^{-1} = \textbf{K}^{-1} \big( \textbf{M} - \textbf{K} \big) \textbf{M}^{-1}. 
\end{flalign}

\noindent Furthermore, for any symmetric matrix $\textbf{M}$ and $z = E + \mathrm{i} \eta \in \mathbb{H}$ with $E, \eta \in \mathbb{R}$, we have the deterministic estimate (see equation (8.34) of \cite{erdos2017dynamical})
\begin{flalign}
\label{gijeta}
\big| K_{ij} \big| \le \frac{1}{\eta}, \quad \text{where $\textbf{K} = \{ K_{ij} \} = (\textbf{M} - z)^{-1}$.}
\end{flalign}

\noindent Moreover, observe from \eqref{probabilityxij} and the fact that $H_{ij}$ has the same law as $N^{-1 / \alpha}( Z+J)$ that
\begin{flalign}
\label{probabilityxij2}
\displaystyle\frac{C_1}{N t^{\alpha} + 1 } \le \mathbb{P} \big[ |H_{ij}| \ge t \big] \le \displaystyle\frac{C_2}{N t^{\alpha} + 1 }, \quad \text{for any $t > 0$.}
\end{flalign}

Using \eqref{probabilityxij2}, we can establish the following lemma, which bounds moments of truncations of $H_{ij}$. As a consequence, we deduce \Cref{t0estimate}. 

\bel

\label{l:truncated}

Fix $R \ge N^{-1 / \alpha}$ and let $s_{ij} = H_{ij}\textbf{\emph{1}}_{|H_{ij}| < R}$. For any positive real number $p>\alpha$, we have that $c N^{-1} R^{p - \alpha} \le \mathbb{E} \big[ |s_{ij}|^p \big] \le C N^{-1} R^{p-\alpha}$, for a small constant $c = c (\alpha, p, C_2) > 0$ and a large constant $C = C(\alpha, p, C_1) > 0$.  \eel

\begin{proof} 
	From \eqref{probabilityxij2}, we have that 
	\begin{flalign*}
	\mathbb{E} \big[ |s_{ij}|^p \big] = p \displaystyle\int_0^R s^{p - 1} \mathbb{P} \big[ |H_{ij}| \ge s \big] ds \le \displaystyle\frac{C_1 p}{N} \displaystyle\int_0^R s^{p - 1 - \alpha} ds = \displaystyle\frac{C_1 p R^{p - \alpha}}{N (p - \alpha)}, 
	\end{flalign*}
	
	\noindent which establishes the upper bound in the lemma. To establish the lower bound, observe from \eqref{probabilityxij2} and the bound $R \ge N^{-1 / \alpha}$ that 
	\begin{flalign*}
	\mathbb{E} \big[ |s_{ij}|^p \big] = p \displaystyle\int_0^R s^{p - 1} \mathbb{P} \big[ |H_{ij}| \ge s \big] ds & \ge \displaystyle\frac{C_1 p}{N} \displaystyle\int_{R / 2}^R \displaystyle\frac{ds}{s^{\alpha + 1 - p} + N^{-1} s^{1 - p}} \\
	& \ge \displaystyle\frac{C_1 p}{5N} \displaystyle\int_{R / 2}^R s^{p - \alpha - 1} ds = \displaystyle\frac{C_1 p (1 - 2^{\alpha - p}) R^{p - \alpha}}{5 N (p - \alpha)}.
	\end{flalign*} 
\end{proof}

\begin{proof}[Proof of \Cref{t0estimate}]
	
	From \Cref{l:truncated} applied with $R = N^{-\nu}$ and $p = 2$, we deduce the existence of constants $C = C(\alpha, C_1) > 0$ and $c = c (\alpha, C_2) > 0$ such that $c N^{(\alpha - 2) \nu - 1} \le \mathbb{E} \big[ H_{11}^2  \one_{|H_{11}| < N^{-\nu}} \big] \le C  N^{(\alpha - 2) \nu - 1}$. Combining this with the fact that $\mathbb{P} \big[ |H_{11}| < N^{-\rho} \big] \ge \frac{1}{2}$ for sufficiently large $N$ (due to \eqref{probabilityxij2}), we deduce the lemma.
\end{proof}

We close this section with the following lemma, which bounds the conditional moments of the random variables $A_{ij}$ and $B_{ij}$ from \Cref{hsumabc}.

\bel

\label{l:truncatedc} 

Let $p > \alpha$. There exists a large constant $C = C (\nu, \rho, p)$ such that, for any indices $1\le i, j \le N$, we have that
\begin{flalign}
\label{papb}
\mathbb{E}_{\Psi} \big[ |A_{ij}|^p \big| \chi_{ij} \big] \le C N^{\nu(\alpha - p) - 1}, \qquad  \mathbb{E}_{\Psi} \big[ |B_{ij}|^p \big] \le CN^{\rho(\alpha - p) - 1}.
\end{flalign}
\eel

\begin{proof}
	
	Let us first establish the bound on $\mathbb{E}_{\Psi} \big[ |B_{ij}|^p \big]$. There are two cases to consider, depending on the entry $\psi \in \{ 0, 1 \}$. If $\psi_{ij} = 1$, then $B_{ij} = 0$ and thus \eqref{papb} holds. If $\psi_{ij} = 0$, then there exists a constant $C = C(\rho, p) > 0$ such that
	\begin{flalign*} 
	\mathbb{E}_{\Psi} \big[ |B_{ij}|^p \big] = \frac{ \mathbb{E} \big[ |B_{ij}|^p \big]}{\P \big[ \psi_{ij} = 0  \big]} \le \frac{ \mathbb{E} \big[ | H_{ij}|^p \one_{| H_{ij}| \le N^{-\rho}} \big]}{\P \big[ |H_{ij}| \le N^{-\rho} \big]} \le CN^{\rho(\alpha - p) - 1}, 
	\end{flalign*} 
	
	\noindent where to deduce the last estimate above we used \Cref{l:truncated} and the fact that $\mathbb{P} \big[ |H_{ij}| > N^{-\rho} \big] > \frac{1}{2}$ for sufficiently large $N$ (due to \eqref{probabilityxij2}). This yields the second estimate in \eqref{papb}. 
	
	Through a very similar procedure, we deduce after increasing $C$ if necessary that $\mathbb{E} \big[ |a_{ij} |^p \big] \le C N^{\nu (\alpha - p) - 1}$, where $a_{ij}$ has the same law as the random variable $a$ given in \Cref{chipsi}. Now the first estimate in \eqref{papb} follows from the deterministic bound $|A_{ij}| \le |a_{ij}|$.
\end{proof}

\subsubsection{Proof of \Cref{decreaseetasum}}

\label{DecreaseEta}

In this section we establish \Cref{decreaseetasum}, assuming \Cref{gvcompare} holds. To that end, recall the definitions of the matrices $\textbf{G}^{\gamma} (z)$ for any $\gamma \in [0, 1]$, and define 
\begin{flalign}
\mathfrak{P} (\delta, \eta) = \mathbb{P} \Bigg[  \displaystyle\max_{\substack{1 \le i,j \le N \\ 0 \le \gamma \le 1}} \big| G_{ij}^{\gamma} (E + \mathrm{i} \eta ) \big| > N^{\delta} \Bigg],
\end{flalign}

\noindent for any $E \in \mathbb{R}$, $\eta \ge N^{\varsigma - 1}$, and $\delta > 0$. Moreover, fix $\varepsilon$ and $\omega$ as in \Cref{gvcompare}, choosing $k=1$ in that theorem, and let $\sigma = \frac{\varepsilon}{4}$. We omit the dependence of $\alpha, b, \rho, \nu, \varepsilon, \omega,$ and $k$ in the notation for the constants appearing in the following lemma and view them as fixed parameters. 

We begin with the following lemma. 

\begin{lem}
	
	\label{pdeltaetaincrease}
	
	Adopt the notation and assumptions of \Cref{decreaseetasum}. For any $\delta > 0$ and integer $D > 0$, there exists a large constant $C = C (\delta, D)$ such that $\mathfrak{P} (\delta, \eta) \le C N^{C} \mathfrak{P} \big( \frac{\varepsilon}{2}, N^{\sigma} \eta \big) + C N^{-D}$ for all $\eta \ge N^{\varsigma - 1}$. 
\end{lem}

\begin{proof} 
	
	Let $p = \big\lceil \frac{D + 30}{\delta} \big\rceil$, and define $F_p (x) = |x|^{2p} + 1$. Observe that there exists a constant $C_p$, only dependent on $p$ (and therefore only dependent on $\delta$ and $D$) such that 
	\begin{flalign}
	\label{fpafp}
	\big| F_p^{(a)} (x) \big| \le C_p F_p (x), \quad \text{for all $x \in \mathbb{R}$ and $a \in \mathbb{Z}_{\ge 0}$.} 
	\end{flalign}	
	
	Now we apply \Cref{gvcompare} with $F(x) = F_p (x)$. Observe that the $C_0$ from that theorem can be taken to be $4p$ and that $d$ can also be taken to be bounded by constant multiple of $p$ (where the implicit constants depend $\nu$, $\rho$, and $\varepsilon$, although in the future we will not mention the dependence on these parameters, since they are already fixed). In view of \eqref{gvcompareestimate} and \eqref{fpafp}, there exists a large constant $B_p$ (only dependent $p$) such that 
	\begin{flalign}
	\label{fpexpectation1} 
	\E_{\Psi} \Big[ F_p \big( \Im G_{ab}^{\gamma} (z) \big) \Big] \le \E_{\Psi} \Big[ F_p \big( \Im T_{ab} (z) \big) \Big] + B_p \big( N^{-\omega} \mathfrak{J}_p ( \Psi ) + Q_0 ( \varepsilon, \Psi ) N^{B_p} + 1 \big),
	\end{flalign}
	
	\noindent for any $0-1$ symmetric $N \times N$ matrix $\Psi$ with at most $N^{1 + \alpha \rho + \varepsilon}$ entries equal to $1$, where 
	\begin{flalign*} 
	\mathfrak{J}_p \big( \Psi \big) =  \displaystyle\sup_{\substack{ 1 \le i, j \le N \\ 0 \le \gamma \le 1}} \E_{\Psi} \big[ F_p (\Im G^\gamma_{ij})  \big], \quad Q_0 \big( \varepsilon, z, \Psi \big) = \mathbb{P}_{\Psi} \Bigg[ \displaystyle\max_{\substack{1 \le i, j \le N \\ 0 \le \gamma \le 1}} \big| G_{ij}^{\gamma} (z) \big| > N^{\varepsilon} \Bigg].
	\end{flalign*}
	
	Now observe that taking the supremum over all $1 \le a, b \le N$ and $0 \le \gamma \le 1$ on the left side of \eqref{fpexpectation1} yields $\mathfrak{J}_p \big( \Psi \big)$. Therefore, 
	\begin{flalign}
	\label{fpexpectation2} 
	(1 - B_p N^{-\omega}) \mathfrak{J}_p \big( \Psi \big) \le \displaystyle\max_{1 \le a, b \le N} \E_{\Psi} \Big[ F_p \big( \Im T_{ab} (z) \big)  \Big] + B_p \left( Q_0 \big( \varepsilon, z, \Psi \big) N^{B_p} + 1 \right).
	\end{flalign}
	
	We now take the expectation of \eqref{fpexpectation2} over $\Psi$. On the event when there are at most $N^{1 + \alpha \rho + \varepsilon}$ entries equal to $C$ in $\Psi$, we apply \eqref{fpexpectation2}. The complementary event has probability bounded by $c^{-1} e^{-cN}$, for some constant $c > 0$, due to \eqref{hijnumber}; on this event, we apply the deterministic bounds $F_p \big( \Im G_{ab}^{\gamma} \big) \le N^{5p}$ and $F_p ( \Im T_{ab} ) \le N^{5p}$. Combining these estimates and fact that $B_p N^{-\omega} < \frac{1}{2}$ for sufficiently large $N$ yields that
	\begin{flalign}
	\label{fpexpectation3}
	\mathfrak{J}_p \le N^2 \displaystyle\max_{1 \le a, b \le N} \E \Big[ F_p \big( \Im T_{ab} (z) \big) \Big] + B_p \left( N^{B_p} \mathfrak{P} ( \varepsilon, \eta) +  1 \right) + B_p N^{B_p} e^{-cN}, 
	\end{flalign} 
	
	\noindent where 	
	\begin{flalign*}
	\mathfrak{J}_p =  \displaystyle\max_{\substack{ 1 \le i, j \le N \\ 0 \le \gamma \le 1}} \E \big[ F_p (\Im G^\gamma_{ij}) \big].
	\end{flalign*} 
	
	After increasing $B_p$ again if necessary, we find from \eqref{tijestimate} that $\E \big[ F_p \big( \Im T_{ab} (z) \big) \big] \le B_p N$ for any $1 \le a, b \le N$. Inserting this into \eqref{fpexpectation3} yields  
	\begin{flalign}
	\label{fpexpectation4}
	\mathfrak{J}_p \le B_p N^3 + B_p N^{B_p} \mathfrak{P} ( \varepsilon, \eta) \le B_p N^3 + B_p N^{B_p} \mathfrak{P} ( \varepsilon - \sigma, N^{\sigma} \eta),
	\end{flalign}
	
	\noindent where in the second estimate above we have used the fact that $\mathfrak{P} (\varepsilon, \eta) \le \mathfrak{P} (\varepsilon - \sigma, N^{\sigma} \eta)$, which follows from the bound 
	\begin{flalign*}
	\displaystyle\max \left\{ \displaystyle\max_{1 \le i, j \le N} \big| G_{ij}^{\gamma} (E + \mathrm{i} \eta) \big|, 1 \right\} \le R \displaystyle\max \left\{ \displaystyle\max_{1 \le i, j \le N} \big| G_{ij}^{\gamma} (E + \mathrm{i} R \eta) \big|, 1 \right\}, \quad \text{for any $R > 0$}, 
	\end{flalign*}
	
	\noindent given as Lemma 2.1 of \cite{bauerschmidt2017local}. Applying \eqref{fpexpectation4}, a Markov estimate, the fact that $\sigma = \frac{\varepsilon}{4}$, the fact that $p \delta \ge D + 30$, we see for any $i,j$ that 
	\begin{multline}
		\label{probabilitygijgammaestimate1}
	\displaystyle\sup_{\gamma \in [0, 1]} \mathbb{P} \left[ \displaystyle\max_{1 \le i, j \le N} \big| \Im G_{ij}^{\gamma} (z) \big| > N^{\delta / 2} \right] \\ \le \frac{\E\left[ F_p\left( \Im  G_{ij}^\gamma\right)  \right]}{F_p( N^{\delta/2})}  <  \frac{\mathfrak{J}_p}{ N^{p\delta}}< B_p N^{- D - 25} + B_p N^{B_p} \mathfrak{P} \left( \displaystyle\frac{\varepsilon}{2}, N^{\sigma} \eta \right).
	\end{multline}

	 Applying a union bound over the $i,j$ and the same reasoning with $\Im G_{ab}^{\gamma}$ replaced by $\Re G_{ab}^{\gamma}$, we deduce the estimate
	\begin{flalign}
	\label{probabilitygijgammaestimate}
	\displaystyle\sup_{\gamma \in [0, 1]} \mathbb{P} \left[ \displaystyle\max_{1 \le i, j \le N} \big| G_{ij}^{\gamma} (z) \big| > N^{\delta / 2} \right] < B_p N^{- D - 20} + B_p N^{B_p + 2} \mathfrak{P} \left( \displaystyle\frac{\varepsilon}{2}, N^{\sigma} \eta \right).
	\end{flalign}

	\noindent Now the proposition follows from applying a union bound in \eqref{probabilitygijgammaestimate} over all $\gamma \in [0, 1] \cap N^{-20} \mathbb{Z}$, then extending these range of $\gamma$ to all of $[0, 1]$ through the deterministic estimate  
	\begin{flalign*} 
	\big| G_{ij}^{\gamma} (z) - G_{ij}^{\gamma'} (z) \big| \le 2 |\gamma - \gamma'|^{1 / 2} N^6 \left( 1 +  \max_{1 \le i, j \le N} |w_{ij}| \right),
	\end{flalign*} 
	
	\noindent due to \eqref{abidentity}, \eqref{gijeta}, the fact that $\eta \ge N^{-2}$, and the bound $\mathbb{P} \big[ |w_{ij}| > 2 \big] < e^{-N}$.
\end{proof}

We can now establish \Cref{decreaseetasum}. 

\begin{proof}[Proof of \Cref{decreaseetasum}]
	
	Set $\kappa = \big\lceil \frac{1 - \varsigma}{\sigma} \big\rceil$. We first claim that, for any integers $D>0$  and $k \in [-1, \kappa]$, there exists a constant $C = C( D, k ) >0$ such that $\mathfrak{P} \big( \frac{\eps}{2}, N^{-k \sigma} \big) < C N^{-D}$. 
	
	We proceed by induction on $k$; the base case $k= -1$ is trivial, because \eqref{gijeta} implies that $\mathfrak{P} \big( \frac{\eps}{2}, \eta \big) = 0$ for any $\eta \in (1, N^{\sigma}]$. For the induction step, suppose the claim holds for $k = m \in [0, \kappa - 1]$. Applying \Cref{pdeltaetaincrease} yields
	\begin{flalign}
	\label{pm1d}
	\mathfrak{P} \left( \displaystyle\frac{\varepsilon}{2}, N^{- (m+1) \sigma}\right) \le C N^{C} \mathfrak{P} \left( \frac{\eps}{2}, N^{-m \sigma} \right) + C N^{-D}.
	\end{flalign}
	By the induction hypothesis, there exists a constant $C' = C' (D, k) > 0$ such that $\mathfrak{P} \big( \frac{\eps}{2}, N^{-m \sigma} \big) < C' N^{-C - D}$. Inserting this into \eqref{pm1d} yields $\mathfrak{P} \big( \frac{\eps}{2}, N^{- (m+1) \sigma}\big)  \le (C + C') N^{-D}$, which completes the induction. 
	
	Now fix $\delta, D>0$. For any $\eta \ge N^{\varsigma - 1}$, applying \Cref{pdeltaetaincrease} shows there exist constants $B = B(\delta, D) > 0$ and $C = C(\delta, D) > 0$ such that
	\begin{flalign}\label{e:unioninput}
	\mathfrak{P} \left(\delta, \eta \right) \le C N^{C} \mathfrak{P} \left( \frac{\eps}{2} , N^{\sigma} \eta \right) + C N^{-D} \le B N^{-D},
	\end{flalign}
	where we used the fact that $N^{\sigma} \eta \ge N^{-\kappa \sigma}$, the bound $\mathfrak{P} \big( \frac{\varepsilon}{2}, N^{-\kappa \sigma} \big) < C N^{- C - D}$, and the monotonicity of $\mathfrak{P} \left(\delta, \eta \right)$ in $\eta$.

	Now let $\mathcal{D}$ denote the set of $z \in \mathbb{H}$ of the form $E + \mathrm{i} \eta$, where $E \in K$ and $\eta \ge N^{\varsigma - 1}$ are both of the form $\frac{k}{N^{10}}$ for some integer $k$. Then a union bound applied to \eqref{e:unioninput} shows that 
	\begin{flalign}
	\label{bddeltaprobabilityestimate}
	\mathbb{P} \left[ \displaystyle\sup_{z \in \mathcal{D}} \displaystyle\sup_{\gamma \in [0, 1]} \displaystyle\max_{1 \le i, j \le N} \big| G_{ij}^{\gamma} (z) \big| \ge N^{\delta} \right] \le \displaystyle\frac{B_{\delta, D + 25}}{N^D}.
	\end{flalign}
	
	\noindent Now from \eqref{bddeltaprobabilityestimate} and the deterministic estimate $\big| G_{ij} (z) - G_{ij} (z') \big| \le N^6 |z - z'|$ we deduce that
	\begin{flalign}
	\label{bddeltaprobabilityestimate2}
	\mathbb{P} \left[ \displaystyle\sup_{\eta \ge N^{\varsigma - 1}} \displaystyle\sup_{\gamma \in [0, 1]} \displaystyle\max_{1 \le i, j \le N} \big| G_{ij}^{\gamma} (z) \big| \ge 2 N^{\delta} \right] \le \displaystyle\frac{B_{\delta, D + 25}}{N^D}.
	\end{flalign}
	
	\noindent Thus \eqref{gijestimate} follows by setting $\gamma = 0$ in \eqref{bddeltaprobabilityestimate2}. 
\end{proof}

\subsection{Outline of the Proof of \Cref{gvcompare}}

\label{CompareOutline} 

For the remainder of this section, we assume that $m = 1$ in \Cref{gvcompare}, and we abbreviate $z_1 = z$, $a_1 = a$, and $b_1= b$. Since the proof of \Cref{gvcompare} for $m > 1$ is entirely analogous, it is omitted. However, in \Cref{Proofm1} we briefly outline how to modify the proof in this case. 

Observe that 
\begin{flalign}
\label{psif}
\begin{aligned}
\displaystyle\frac{\partial}{\partial \gamma} \mathbb{E}_{\Psi} \Big[ F \big( \Im G_{ab}^{\gamma} \big) \Big]  =\displaystyle\sum_{1 \le p, q \le N} \mathbb{E}_{\Psi} \Bigg[ \Im (G_{ap}^{\gamma} G_{qb}^{\gamma}) \left( A_{pq} - \displaystyle\frac{\gamma t^{1 / 2} w_{pq}}{(1 - \gamma^2)^{1 / 2}} \right)  F' \big( \Im G_{ab}^{\gamma} \big) \Bigg],
\end{aligned}
\end{flalign}

\noindent and so it suffices to establish the following proposition. 

\begin{prop}
	
	\label{derivativeestimate}
	
	Adopt the notation of \Cref{gvcompare}. Then there exists a large constant $C = C(\alpha, \nu, \rho) > 0$ such that, for sufficiently large $N$,  
	\begin{flalign}
	\label{estimatesumderivative}
	\begin{aligned}
	\displaystyle\sum_{1 \le p, q \le N} \Bigg| \mathbb{E}_{\Psi} \bigg[ \Im (G_{ap}^{\gamma} G_{qb}^{\gamma}) \Big( & A_{pq} - \displaystyle\frac{\gamma t^{1 / 2} w_{pq}}{(1 - \gamma^2)^{1 / 2}} \Big)  F' \big( \Im G_{ab}^{\gamma} \big) \bigg] \Bigg| \\
	& \le \displaystyle\frac{C}{(1 - \gamma^2)^{1 / 2}} \big( N^{-\omega} (\mathfrak{J} + 1) + Q_0 N^{C + C_0} \big).
	\end{aligned}
	\end{flalign}
	
\end{prop}

To establish \Cref{derivativeestimate}, we estimate each summand on the right side of \eqref{psif}. Thus, in what follows, let us fix some integer pair $(p, q) \in [1, N] \times [1, N]$. 

If $\textbf{G}^{\gamma}$ were independent from $A_{pq}$ and $w_{pq}$, then each expectation on the left side of \eqref{estimatesumderivative} would be equal to zero, from which the proposition would follow. Since this independence does not hold, we will approximate $\textbf{G}^{\gamma}$ with a matrix that is independent from $A_{pq}$ and $w_{pq}$ (after conditioning on $\Psi$, as we will do throughout the proof of \Cref{derivativeestimate}) and estimate the error incurred by this replacement. 

In fact, it will be useful to introduce two matrices. The first will be independent from $w_{pq}$ but not quite independent from $A_{pq}$ (although it will be independent from $A_{pq}$ after additionally conditioning on $\chi_{pq}$); the second will be independent from both $w_{pq}$ and $A_{pq}$

More specifically, we define the $N \times N$ matrices $\textbf{D} = \textbf{D}^{\gamma, p, q} = \{ D_{ij} \} = \big\{ D_{ij}^{\gamma, p, q} \big\}$ and $\textbf{E} = \textbf{E}^{\gamma, p, q} = \big\{ E_{ij} \big\} = \big\{ E_{ij}^{\gamma, p, q} \big\}$ by setting $D_{ij} = H_{ij}^{\gamma} = E_{ij}$ if $(i, j) \notin \big\{ (p, q), (q, p) \big\}$ and 
\begin{flalign*}
D_{pq} = D_{qp} = X_{pq} = B_{pq} + C_{pq}, \qquad E_{pq} = E_{qp} = C_{pq}. 
\end{flalign*}

We also define the $N \times N$ matrices $\Gamma = \Gamma^{\gamma, p, q} = \{ \Gamma_{ij} \} = \big\{ \Gamma_{ij}^{\gamma, p, q} \big\} = \textbf{H}^{\gamma} - \textbf{D}$ and $\Lambda = \Lambda^{\gamma, p, q} = \{ \Lambda_{ij} \} = \big\{ \Lambda_{ij}^{\gamma, p, q} \big\} = \textbf{D} - \textbf{E}$, so that 
\begin{flalign}
\label{gammalambda}
\Gamma_{ij} = \gamma \Theta_{ij} + (1 - \gamma^2)^{1 / 2} \Phi_{ij}, \qquad \Lambda_{ij} = B_{pq} \one_{(i, j) \in \{ (p, q), (q, p) \}},
\end{flalign}

\noindent where 
\begin{flalign}
\label{athetatw}
\Theta_{ij} = A_{ij} \one_{(i, j) \in \{ (p, q), (q, p) \}}, \qquad \Phi_{ij} = t^{1 / 2} w_{ij} \one_{(i, j) \in \{ (p, q), (q, p) \}}.
\end{flalign}

\noindent In addition, we define the resolvent matrices  
\begin{flalign}
\label{rresolventu} 
\begin{aligned}
\textbf{R} = \textbf{R}^{\gamma, p, q} = \{ R_{ij} \} = \big\{ R_{ij}^{\gamma, p, q} \big\} = (\textbf{D} - z)^{-1}, \qquad \textbf{U} = \textbf{U}^{\gamma, p, q} = \{ U_{ij} \} = \big\{ U_{ij}^{\gamma, p, q} \big\}= (\textbf{E} - z)^{-1}.
\end{aligned}
\end{flalign}

\begin{rem} 
	
	\label{2l1l2independentabc}
	
	Observe that, after conditioning on $\Psi$, the matrices $\Gamma$ and $\Lambda$ are both independent from $\textbf{U}$. After further conditioning on $\chi_{pq}$, the matrices $\Theta = \{ \Theta_{ij} \}$, $\Phi = \{ \Phi_{ij} \}$, and $\textbf{R}$ become mutually independent.
\end{rem}

We would first like to replace the entries $G_{ij}^{\gamma}$ in the $(p, q)$ summand on the left side of \eqref{estimatesumderivative} with the entries $R_{ij} = R_{ij}^{\gamma, p, q}$. To that end, we set
\begin{flalign}
\label{rgammapsi}
\begin{aligned}
\xi_{ij} = \xi_{ij} (\gamma) = \big(\textbf{G}^{\gamma} - \textbf{R} \big)_{ij} =  \big( - \textbf{R} \Gamma \textbf{R} + (\textbf{R} \Gamma)^2 \textbf{R} - (\textbf{R} \Gamma)^3 \textbf{G}^{\gamma} \big)_{ij}, \qquad \zeta_{ij} = \Im \xi_{ij}, 
\end{aligned}
\end{flalign} 

\noindent for any $1 \le i, j \le N$, where the third equality in \eqref{rgammapsi} follows from the resolvent identity \eqref{abidentity}. We abbreviate $\zeta = \zeta_{ab}	$. 

By a Taylor expansion, there exists some $\zeta_0 \in [ \Im G_{ab}^{\gamma}, \Im R_{ab} ]$ such that
\begin{flalign}
\label{sqf} 
\begin{aligned} 
F' (\Im G_{ab}^{\gamma} ) = F' (\Im R_{ab} + \zeta)  & = F^{(1)} (\Im R_{ab}) + \zeta F^{(2)} (\Im R_{ab}) + \frac{\zeta^2 }{2} F^{(3)} (\Im R_{ab})  + \frac{\zeta^3 }{6} F^{(4)} (\zeta_0),
\end{aligned}
\end{flalign} 

\noindent where $F^{(i)} (x) = \frac{\partial^i F}{\partial x^i} (x)$ for any $i \in \mathbb{Z}_{\ge 0}$ and $x \in \mathbb{R}$. 

Using \eqref{athetatw}, \eqref{rgammapsi} and \eqref{sqf}, we deduce that the $(p, q)$-summand on the left side of \eqref{estimatesumderivative} can be expanded as a finite sum of (consisting of less than $2^{22}$) monomials in $\Theta_{pq}$ and $\Phi_{pq}$, whose coefficients depend on the entries of $\textbf{G}^{\gamma}$ and $\textbf{R}$. We call such a monomial of \emph{degree $k$} (or a \emph{$k$-th order term}) if it is of total degree $k$ in $\Theta_{pq}$ and $\Phi_{pq}$. We will estimate the $(p, q)$-summand on the left side of \eqref{estimatesumderivative} by bounding the expectation of each such monomial, which will be done in the following sections.

Before proceeding, let us fix an integer pair $(p, q) \in [1, N] \times [1, N]$ throughout this the remainder of section. It will also be useful for us to define some additional parameters that will be fixed throughout this section. In what follows, we define the positive real numbers $\omega > \varepsilon > 0$ through 
\begin{flalign}
\label{omegadefinition}
\begin{aligned}
& \varepsilon = \displaystyle\frac{\alpha}{100} \displaystyle\min \left\{ (4 - \alpha) \nu - 1, (2 - \alpha) \nu - \alpha \rho, \nu - \rho, \displaystyle\frac{\rho}{2}, 1 \right\},\\
& \omega = \displaystyle\min \big\{ (\alpha - 2 \varepsilon) \rho - 15 \varepsilon, (2 - \alpha) \nu - \alpha \rho - 15 \varepsilon, (4 - \alpha) \nu - 1 - 10 \varepsilon, (4 - 2 \alpha) \nu - 15 \varepsilon \big\}.
\end{aligned}
\end{flalign}

\noindent Moreover, let us fix integers $\vartheta, d > 0$ such that 
\begin{flalign}
\label{c0d} 
\vartheta (\rho - 2 \varepsilon ) >  C_0 \varepsilon + 3, \qquad d > 3 \vartheta + 5.
\end{flalign}

The remainder of this section is organized as follows. We will estimate the contribution to the left side of \eqref{estimatesumderivative} resulting from the first, third, and higher degree terms in \Cref{Degree134}, and we will estimate the contribution from the second degree terms in \Cref{Degree2}. However, we first require estimates on the entries of $\textbf{R}$ and $\textbf{U}$ (from \eqref{rresolventu}), which will be provided in \Cref{REstimatesU}. We then outline the modifications necessary in the proof of \Cref{gvcompare} in \Cref{Proofm1}.

\subsection{Estimating the Entries of  \texorpdfstring{$\textbf{R}$}{} and \texorpdfstring{$\textbf{U}$}{} }

\label{REstimatesU}

Recall that the event $\Omega_0$ from \eqref{omegak1} bounds the entries of $\textbf{G}^{\gamma}$. In this section we will provide similar estimates on the entries of $\textbf{R}$ and $\textbf{U}$ on an event slightly smaller than $\Omega_0$. More specifically, define 
\begin{flalign*} 
\Omega_1 = \Omega_1 (\rho) = \bigg\{ & \displaystyle\max_{1 \le i, j \le N} |w_{ij}| \le N^{-\rho} \bigg\}, \qquad  \Omega = \Omega (\rho, \varepsilon, z) = \Omega_0 \cap \Omega_1,  \\
&  Q = 1 - \mathbb{P}_{\Psi} \big[ \Omega (\rho, \varepsilon, z) \big] = \mathbb{P}_{\Psi} [\Omega^c],
\end{flalign*}

\noindent where $\Omega^c$ denotes the complement of $\Omega$. Since $\rho < \frac{1}{2}$ and $w_{ij}$ is a Gaussian random variable with variance at most $\frac{2}{N}$, there exists small constant $c = c(\rho) > 0$ such that 
\begin{flalign}
\label{probabilityomega1estimate}
1 - \mathbb{P} \big[ \Omega_1 \big] < e^{-cN^c}.
\end{flalign}

 Thus, it suffices to establish	\eqref{estimatesumderivative} with $Q_0$ there replaced by $Q$. The following lemma estimates $|R_{ij}|$ and $|U_{ij}|$ on the event $\Omega$. 

\bel\label{l:RUresolvent}

For $N$ sufficiently large, we have that
\begin{flalign}
\label{estimater2u} 
\textbf{\emph{1}}_{\Omega} \sup_{1 \le i, j \le N} \big| R_{ij} \big| \le 2N^{\varepsilon}, \qquad \textbf{\emph{1}}_{\Omega} \sup_{1 \le i,j \le N} \big| U_{ij}(z) \big| \le 2N^{\varepsilon}.
\end{flalign}
\eel

\begin{proof}
	We only establish the second estimate (on $|U_{ij}|$) in \eqref{estimater2u}, since the proof of the first is entirely analogous. Let us also restrict to the event $\Omega$, since the lemma holds off of $\Omega$.
	
	Recall from the resolvent identity \eqref{abidentity} and the definitions \eqref{gammalambda}, \eqref{athetatw}, and \eqref{rresolventu} that 
	\begin{flalign}
	\label{uijsij1}
	\textbf{U} - \textbf{G}^{\gamma} = \displaystyle\sum_{j = 1}^s \big( \textbf{G}^{\gamma} (\Gamma + \Lambda) \big)^j \textbf{G}^{\gamma} + \big( \textbf{G}^{\gamma} (\Gamma + \Lambda)  \big)^{s + 1} \textbf{U}, 
	\end{flalign}                 
	
	\noindent for any integer $s > 0$. 
	
	Now, set $s = \big\lceil \frac{2}{\rho - 2 \varepsilon} \big\rceil$, which is positive by \eqref{omegadefinition}. Observe that $\one_{\Omega} \max_{1 \le i, j \le N} \big| G_{ij}^{\gamma} \big| \le N^{\varepsilon}$ and that the only nonzero entries of $\one_{\Omega} (\Gamma + \Lambda)$ are $\one_{\Omega} (\Gamma + \Lambda)_{pq}$ and $\one_{\Omega} (\Gamma + \Lambda)_{qp}$, which satisfy 
	\begin{flalign}
	\label{omegagammalambda} 
	\one_{\Omega} (\Gamma + \Lambda)_{pq} = \one_{\Omega} (\Gamma + \Lambda)_{qp} \le N^{-\rho} + t^{1 / 2} |w_{pq}| \one_{\Omega_1} \le 2 N^{-\rho}.
	\end{flalign} 
	
	\noindent Thus, \eqref{uijsij1} yields
	\begin{flalign} 
	\label{uijsij} 
	\one_{\Omega} \big| U_{ij}  - G_{ij}^{\gamma} \big| \le \displaystyle\sum_{j = 1}^s (4 N^{2 \varepsilon - \rho})^j  + (4 N^{\varepsilon - \rho})^{(s + 1)} \displaystyle\max_{1 \le i', j' \le N} |U_{i' j'}| \le 1, 
	\end{flalign}
	
	\noindent if $N$ is sufficiently large, where we have also used the deterministic estimate $\big| U_{i' j'} \big| \le \eta^{-1} \le N^2$. Now the estimate \eqref{estimater2u} on $|U_{ij}|$ follows from \eqref{uijsij}, the choice of $s$, and the fact that $\one_{\Omega} \big| G_{ij}^{\gamma} \big| \le N^{\varepsilon}$. 
\end{proof}

We also require the following lemma, which states that we can approximate quantities near $\big| F^{(k)} (\Im R_{ab}) \big|$ and $\big| F^{(k)} (\Im U_{ab}) \big|$ in terms of derivatives of $F^{(k)} \big( \Im G_{ab}^{\gamma} \big)$.

\bel\label{l:jp}

 Let $\varphi \in \mathbb{R}$ be either such that $\varphi \in \big[ \Im G_{ab}^{\gamma}, \Im U_{ab} \big]$ or $\varphi \in \big[ \Im G_{ab}^{\gamma}, \Im R_{ab} \big]$. Then there exists a large constant $C = C(\vartheta) > 0$ such that, for any integer $k \ge 0$, we have that 
\begin{flalign}
\label{frusestimate}
\begin{aligned}
& \textbf{\emph{1}}_{\Omega} \big|F^{(k)} (\varphi) \big| \le C \textbf{\emph{1}}_{\Omega} \displaystyle\sum_{j = 0}^{\vartheta} N^{(2 \varepsilon - \rho) j} \big| F^{(k + j)} (\Im G_{ab}^{\gamma}) \big| + \displaystyle\frac{C}{N^3}. 
\end{aligned}
\end{flalign}

\noindent Moreover, if $\varphi \in \big[ \Im G_{ab}^{\gamma}, \Im R_{ab} \big]$, then 
\begin{flalign}
\label{frusestimate2}
\begin{aligned}
& \textbf{\emph{1}}_{\Omega} \big|F^{(k)} (\varphi) \big| \le C \textbf{\emph{1}}_{\Omega} \displaystyle\sum_{j = 0}^{\vartheta} N^{(2 \varepsilon - \rho) j} \big| F^{(k + j)} (\Im R_{ab}) \big| + \displaystyle\frac{C}{N^3}. 
\end{aligned}
\end{flalign}

\eel
\begin{proof}
	
	The proof of this lemma will be similar to that of \Cref{l:RUresolvent}. We only establish \eqref{frusestimate} when $\varphi \in \big[ \Im G_{ab}^{\gamma}, \Im U_{ab} \big]$, since the proofs of \eqref{frusestimate2} and of \eqref{frusestimate} when $\varphi \in \big[ \Im G_{ab}^{\gamma}, \Im R_{ab} \big]$ are entirely analogous. 
	
	Through a Taylor expansion, we have that 
	\begin{flalign}
	\label{fuabsab} 
	F^{(k)} ( \varphi) - F^{(k)} \big( \Im G_{ab}^{\gamma} \big) = \displaystyle\sum_{j = 1}^{\vartheta}  \frac{\Upsilon^j}{j!} F^{(j + k)} \big( \Im G_{ab}^{\gamma} \big)    + \frac{\Upsilon^{\vartheta + 1}}{(\vartheta + 1)!} F^{(\vartheta + k)} (\Upsilon_1),
	\end{flalign}
	
	\noindent where $\Upsilon_1 \in \big[ \Im G_{ab}^{\gamma}, \varphi \big]$, and $\Upsilon = \varphi - \Im G_{ab}^{\gamma}$, which satisfies 
	\begin{flalign}
	\label{suab}
	|\Upsilon| \le \big| \Im U_{ab} - \Im G_{ab}^{\gamma}| = \Big| \Im \big( \textbf{U} (\Gamma + \Lambda) \textbf{G}^{\gamma} \big)_{ab} \Big|,
	\end{flalign}
	
	\noindent where in \eqref{suab} we used the resolvent identity \eqref{abidentity} and the definition \eqref{gammalambda} of $\Gamma$ and $\Lambda$. 
	
	Recalling that $\Gamma + \Lambda$ has only two nonzero entries, both of which are at most $2 N^{-\rho}$ on $\Omega$ (due to \eqref{omegagammalambda}), and further recalling that the entries of $\textbf{G}^{\gamma}$ and $\textbf{U}$ are bounded by $2 N^{\varepsilon}$ on $\Omega$ (due to \Cref{l:RUresolvent}), we deduce that $\one_{\Omega} \big| \Upsilon \big| \le 16 N^{2 \varepsilon - \rho} \one_{\Omega}$. Inserting this and the first estimate of \eqref{fc0c0} into \eqref{fuabsab}, we deduce the existence of a constant $C = C(\vartheta) > 0$ such that
	\begin{flalign}
	\label{fuabsab1} 
	\one_{\Omega} \Big| F^{(k)} \big( \Im U_{ab} \big) - F^{(k)} \big( \Im G_{ab}^{\gamma} \big) \Big| \le C \one_{\Omega} \displaystyle\sum_{j = 1}^{\vartheta} N^{(2 \varepsilon - \rho) j} \big| F^{(j + k)} (\Im G_{ab}^{\gamma}) \big| + C N^{(\vartheta + 1) (2 \varepsilon - \rho) + C_0 \varepsilon}.
	\end{flalign}
	
	\noindent 	Now the second estimate in \eqref{frusestimate} follows from \eqref{fuabsab1} and the fact \eqref{c0d} that $(\rho - 2 \varepsilon) \vartheta > C_0 \varepsilon + 3$. 
\end{proof}

\subsection{The First, Third, and Higher Order Terms}

\label{Degree134}

In this section we show that the expectations of the first and third order terms in the expansion of \eqref{estimatesumderivative} are equal to $0$ through \Cref{l:vanish}, and we also estimate the higher order terms through \Cref{l:final1} and \Cref{estimategijgammaproduct}. 

Observe that any degree one or degree three term appearing in the expansion of the $(p, q)$-summand on the left side of \eqref{estimatesumderivative} (using \eqref{rgammapsi} and \eqref{sqf}) contains either zero or two factors of $\Gamma$. The following lemma indicates that the expectation of any such term is equal to $0$.

\bel\label{l:vanish} 

For any integers $1 \le i, j \le N$ and $k \in \{ 0, 1, 2 \}$, define $\xi_{ij}^{(k)} =  \big( (- \textbf{\emph{R}} \Gamma)^k \textbf{\emph{R}} \big)_{ij}$. 

Let $M$ be a (possibly empty) product of $s \ge 0$ of the $\xi_{ij}^{(k)}$, so that $M = \prod_{r = 1}^s \xi_{i_r j_r}^{(k_r)}$ for some $1 \le i_r, j_r \le N$ and $k_r \in \{ 0, 1, 2 \}$. If $\sum_{r = 1}^s k_r$ is even (in particular, if it is either $0$ or $2$) and $m \in \{ 1, 2, 3 \}$, then 
\begin{flalign}
\label{13expectation} 
\mathbb{E}_{\Psi} \Bigg[F^{(m)} (\Im R_{ab}) \bigg( A_{pq} - \displaystyle\frac{\gamma t^{1 / 2} w_{pq}}{(1 - \gamma^2)^{1 / 2}} \bigg) M \Bigg] = 0. 
\end{flalign}

\noindent The same estimate \eqref{13expectation} holds if some of the $\xi_{i_r j_r}^{(k_r)}$ are replaced by $\Re \xi_{i_r j_r}^{(k_r)}$ or $\Im \xi_{i_r j_r}^{(k_r)}$ in the definition of $M$. 

\eel

\begin{proof}
	
	First observe from the symmetry of the random variables $H_{ij}$ that $\mathbb{E}_{\Psi} \big[ A_{pq}^m \big| \chi_{pq} \big] = 0 = \mathbb{E}_{\Psi} \big[ w_{pq}^m  \big| \chi_{pq} \big]$ for any odd integer $m > 0$. Now, recall from \Cref{2l1l2independentabc} that $A_{pq}$, $w_{pq}$, and $\textbf{R}$ are mutually independent after conditioning on $\chi_{pq}$ and $\Psi$. Therefore, 
	\begin{flalign}
	\label{13expectationxi}
	\begin{aligned}
	\mathbb{E}_{\Psi} & \Bigg[F^{(k)} (\Im R_{ab}) \bigg( A_{pq} - \displaystyle\frac{\gamma t^{1 / 2} w_{pq}}{(1 - \gamma^2)^{1 / 2}} \bigg) M \Bigg]  \\
	& = \mathbb{E}^{\chi} \Bigg[ \mathbb{E}_{\Psi} \bigg[F^{(k)} (\Im R_{ab}) \Big( A_{pq} - \displaystyle\frac{\gamma t^{1 / 2} w_{pq}}{(1 - \gamma^2)^{1 / 2}} \Big) M \bigg| \chi_{pq} \bigg] \Bigg] = 0,
	\end{aligned}
	\end{flalign} 
	
	\noindent where we have used the fact that the term inside the first expectation in the middle of \eqref{13expectationxi} is a linear combination of products of expressions that each either contain a term of the form $\mathbb{E} \big[ A_{pq}^m \big| \chi_{pq} \big]$ or $\mathbb{E} \big[ w_{pq}^m \big| \chi_{pq} \big]$ for some odd integer $m > 0$ (by \eqref{gammalambda}, \eqref{athetatw}, and the fact that $\sum_{r = 1}^s k_r$ is even), and each of these expectations is equal to $0$. This establishes \eqref{13expectation}. 
\end{proof}

Now let us consider the fourth and higher order terms that can occur in \eqref{estimatesumderivative} through the expansions \eqref{rgammapsi} and \eqref{sqf}. Two types of such terms can appear. The first is when the final term in \eqref{sqf} appears, giving rise to a factor of $\zeta^3 F^{(4)} (\zeta_0)$. The second is when $\zeta^3 F^{(4)} (\zeta_0)$ does not appear and instead the term is a product of $F^{(m)} (\Im R_{ab})$ (for some $1 \le m \le 3$) with at most four expressions of the form $(-\textbf{R} \Gamma)^k \textbf{R}$ or $(-\textbf{R} \Gamma)^k \textbf{G}^{\gamma}$ (and their real or imaginary parts).  

The following lemma addresses terms of the first type.

\bel\label{l:final1} There exists a large constant $C = C (\alpha, \nu, \rho, \vartheta) > 0$ such that
\begin{flalign}
\label{sumxi1zeta1}
\begin{aligned}
\mathbb{E}_{\Psi} \Bigg[ \bigg| & \Im (G_{ap}^{\gamma} G_{qb}^{\gamma}) \Big( A_{pq} - \displaystyle\frac{\gamma t^{1 / 2} w_{pq}}{(1 - \gamma^2)^{1 / 2}} \Big) \zeta^3 F^{(4)} \big( \zeta_0 \big) \bigg| \Bigg]  \\
& \le \displaystyle\frac{C N^{10 \varepsilon}}{(1 - \gamma^2)^{1 / 2}} \left(  N^{(\alpha - 4) \nu - 1 } \mathfrak{J} + \displaystyle\frac{t^2 \mathfrak{J}}{N^2} + Q N^{C_0 + 10} + \displaystyle\frac{1}{N^3} \right) . 
\end{aligned}
\end{flalign} 
\eel

\begin{proof}
	
	We first establish an estimate that holds off of the event $\Omega$. In this case, to bound the left side of \eqref{sumxi1zeta1}, we use the deterministic facts that $|G_{ij}^{\gamma}|, |R_{ij}|, \zeta\le \eta^{-1} \le N^2$ and $|A_{ij}| < 1$, which implies from \eqref{fc0c0} that $\big| F(\Im R_{ij}) \big| \le N^{C_0}$. This yields for sufficiently large $N$  
	\begin{flalign}
	\label{deterministicxiestimate}
	\begin{aligned}  
	 \mathbb{E}_{\Psi} \Bigg[ \bigg| \Im & (G_{ap}^{\gamma} G_{qb}^{\gamma}) \Big ( A_{pq} - \displaystyle\frac{\gamma t^{1 / 2} w_{pq}}{(1 - \gamma^2)^{1 / 2}} \Big) \zeta^3 F^{(4)} \big( \zeta_0 \big) \one_{\Omega^c} \bigg| \Bigg] \\
	& \le N^{C_0 + 10}  \mathbb{E}_{\Psi} \bigg[ \one_{\Omega^c} + \displaystyle\frac{\gamma \mathbb{E} \big[ |w_{pq}| \one_{\Omega^c} \big] }{(1 - \gamma^2)^{1 / 2}}  \bigg] \le \displaystyle\frac{4 N^{C_0 + 10} Q}{(1 - \gamma^2)^{1 / 2}}.
	\end{aligned}
	\end{flalign}

	Next we first work on the event $\Omega$. To that end, observe from \eqref{abidentity}, \eqref{rgammapsi}, and \Cref{l:RUresolvent} that
	\begin{flalign*}
	| \zeta| \one_{\Omega} \le \Big| \big( \textbf{G}^{\gamma} \Gamma \textbf{R} \big)_{ab} \Big| \one_{\Omega} = \Big( \big| G_{ap}^{\gamma} \Gamma_{pq} R_{qb} \big| + \big| G_{aq}^{\gamma} \Gamma_{qp} R_{pb} \big| \Big) \one_{\Omega} \le 8 N^{2 \varepsilon} |\Gamma_{pq}| \one_{\Omega}.
	\end{flalign*}
	
	\noindent Furthermore, since $\zeta_0 \in \big[ \Im G_{ab}^{\gamma}, \Im R_{ab} \big]$, \eqref{frusestimate2} yields that 
	\begin{flalign} 
	\label{xiestimatesum1} 
	\begin{aligned}
	 \mathbb{E}_{\Psi} & \Bigg[ \bigg| \Im (G_{ai}^{\gamma} G_{jb}^{\gamma}) \Big ( A_{pq} - \displaystyle\frac{\gamma w_{ij}}{(1 - \gamma^2)^{1 / 2}} \Big) \zeta^3 F^{(4)} \big( \zeta_0 \big) \one_{\Omega} \bigg| \Bigg] \\
	& \quad \le \displaystyle\frac{512 N^{10 \varepsilon}}{(1 - \gamma^2)^{1 / 2}} \mathbb{E}_{\Psi} \bigg[  |\Gamma_{pq}|^3 \big( |A_{pq}| + t^{1 / 2} |w_{ij}| \big) \big| F^{(4)} (\zeta_0) \big| \one_{\Omega} \bigg] \\
	& \quad \le \displaystyle\frac{C N^{10 \varepsilon}}{(1 - \gamma^2)^{1 / 2}} \displaystyle\sum_{j = 0}^{2 \vartheta} N^{(2 \varepsilon - \rho) j} \mathbb{E}_{\Psi} \Bigg[ \Big| F^{(j + 4)} \big( \Im R_{ab} \big) \Big|  \big( |A_{pq}| + t^{1 / 2} |w_{pq}| \big)^4  \Bigg] + \displaystyle\frac{C}{N^3} 
	\end{aligned}
	\end{flalign} 
	
	\noindent for some constant $C = C(\vartheta) > 0$. To estimate the right side of \eqref{xiestimatesum1}, we condition on $\chi_{pq}$, and apply \Cref{2l1l2independentabc} to deduce that 
	\begin{flalign*} 
 	\mathbb{E}_{\Psi} & \Bigg[ \Big| F^{(j + 4)} \big( \Im R_{ab} \big) \Big|  \big( |A_{pq}| + t^{1 / 2} |w_{pq}| \big)^4  \Bigg] \\
 	& \le 8 \mathbb{E}^{\chi} \Bigg[ \mathbb{E}_{\Psi} \bigg[ \Big| F^{(j + 4)} \big( \Im R_{ab} \big) \Big| \big( |A_{pq}|^4 + t^2 |w_{pq}|^4 \big) \bigg| \chi_{pq} \bigg] \Bigg] \\
 	& = 8 \mathbb{E}^{\chi} \Bigg[ \mathbb{E}_{\Psi} \bigg[ \Big| F^{(j + 4)} \big( \Im R_{ab} \big) \Big| \bigg| \chi_{pq} \bigg] \mathbb{E}_{\Psi} \bigg[ \big( |A_{pq}|^4 + t^2 |w_{pq}|^4 \big) \bigg| \chi_{pq} \bigg] \Bigg].
	\end{flalign*}

	\noindent Then \Cref{l:truncated} (with $p = 4$) and the fact that $\mathbb{E} \big[ |w_{pq}|^4 \big] \le \frac{60}{N^2}$ yields after enlarging $C = C(\alpha, \nu, \rho, \vartheta)$ that 
	\begin{flalign} 
	\label{xiestimatesum12} 
	\begin{aligned}
	\mathbb{E}_{\Psi} & \Bigg[ \Big| F^{(j + 4)} \big( \Im R_{ab} \big) \Big|  \big( |A_{pq}|^4 + t^2 |w_{pq}|^4 \big)  \Bigg] \\
	& \le C \left( N^{(\alpha - 4) \nu - 1} + \displaystyle\frac{t^2}{N^2} \right) \mathbb{E}^{\chi} \Bigg[ \mathbb{E}_{\Psi} \bigg[ \Big| F^{(j + 4)} \big( \Im R_{ab} \big) \Big| \bigg| \chi_{pq} \bigg] \Bigg]  \le C \left( N^{(\alpha - 4) \nu - 1} + \displaystyle\frac{t^2}{N^2} \right) \mathfrak{J} + \displaystyle\frac{C}{N^3}, 
	\end{aligned}
	\end{flalign}
	
	\noindent where we used \eqref{frusestimate} to deduce the last estimate.
	
	Now \eqref{sumxi1zeta1} follows from applying \eqref{deterministicxiestimate} off of $\Omega$ and \eqref{xiestimatesum1} and \eqref{xiestimatesum12} on $\Omega$.
\end{proof}

The following lemma addresses the higher order terms of the second type. Its proof is very similar to that of \Cref{l:final1} and is therefore omitted.

\begin{lem}
	
	\label{estimategijgammaproduct}
	
	Recall the definitions of the $\xi_{ij}^{(k)}$ for $k \in \{ 0, 1, 2 \}$ from \Cref{l:vanish}, and further define $\xi_{ij}^{(3)} = \big( (- \textbf{\emph{R}} \Gamma)^3 \textbf{\emph{G}}^{\gamma} \big)_{ij}$, for each $1 \le i, j \le N$. 
	
	There exists a large constant $C = C(\alpha, \nu, \rho, \vartheta) > 0$ such that the following holds. Let $M$ be a product of $s \in \{ 1, 2, 3, 4 \}$ of the $\xi_{ij}^{(k)}$, so that $M = \prod_{r = 1}^s \xi_{i_r j_r}^{(k_r)}$ for some $1 \le i_r, j_r \le N$ and $k_r \in \{ 1, 2, 3 \}$. If $\sum_{r = 1}^s k_r \ge 3$ and $m \in \{ 1, 2, 3 \}$, then  
	\begin{flalign}
	\label{estimateproduct4ijk}
	\begin{aligned}
	 \mathbb{E}_{\Psi} & \Bigg[  |M| \bigg| \Big( A_{pq}  - \displaystyle\frac{\gamma t^{1 / 2} w_{pq}}{(1 - \gamma^2)^{1 / 2}} \Big) F^{(m)} \big( \Im R_{ab} \big) \bigg| \Bigg]  \le \displaystyle\frac{C N^{16 \varepsilon}}{(1 - \gamma^2)^{1 / 2}} \left( N^{\nu (\alpha - 4) - 1} \mathfrak{J} + \frac{t^2 \mathfrak{J}}{N^2}  + \displaystyle\frac{1}{N^3} \right).
	\end{aligned}
	\end{flalign} 
	
	\noindent The same estimate holds if some of the $\xi_{i_r j_r}^{(0)}$ are replaced by $G_{ij}^{\gamma}$. 
\end{lem}

\subsection{Terms of Degree  \texorpdfstring{$2$}{}}

\label{Degree2}

In this section we estimate the contribution of terms of degree two to the $(p, q)$-summand of the left side of \eqref{estimatesumderivative}. In \Cref{EstimateTerms2} we will state this bound use it to establish \Cref{derivativeestimate}; we will then establish this estimate in \Cref{Second}.

\subsubsection{Estimates on the Degree Two Terms}

\label{EstimateTerms2}

In this section we bound the contribution of the second order terms to the $(p, q)$-summand left side of \eqref{estimatesumderivative}. There are two types of terms to consider. The first corresponds to when the factor of $\zeta F'' (\Im R_{ab})$ appears in the expansion \eqref{sqf} for $F' (\Im G_{ab}^{\gamma})$, and the second corresponds to when either $\Im (-\textbf{R} \Gamma \textbf{R})_{ap}$ or $\Im (-\textbf{R} \Gamma \textbf{R})_{qb}$ appears in the expansion \eqref{rgammapsi} for $\Im G_{ij}^{\gamma}$. Both such terms are estimated through the following proposition.

\begin{prop}
	
	\label{gawestimatef2}
	
	Define 
	\begin{flalign*}
	\mathfrak{E}_1 =  N^{4 \varepsilon + (\alpha - 2) \rho - 2} t \mathfrak{J} + N^{C_0 + 6} t Q  + \displaystyle\frac{t}{N^2}, \qquad \mathfrak{E}_2 = N^{\alpha \rho + 3 \varepsilon - 1} t \mathfrak{J} . 
	\end{flalign*}
	
	\noindent Then, there exists a large constant $C = C(\alpha, \nu, \rho, \vartheta) > 0$ such that  
	\begin{flalign}
	\label{rgammarestimate21} 
	\mathbb{E}_{\Psi} \Bigg[ \Im \big( (\textbf{\emph{R}} \Gamma \textbf{\emph{R}})_{ap} R_{qb} \big) \bigg( A_{pq} - \displaystyle\frac{\gamma t^{1 / 2} w_{pq}}{(1 - \gamma^2)^{1 / 2}} \bigg) F' (\Im R_{ab}) \Bigg] \le C \big( \mathfrak{E}_1 + \mathfrak{E}_2 (\psi_{pq} + \textbf{\emph{1}}_{p = q}) \big),
	\end{flalign} 
	
	\noindent and similarly if $(\textbf{\emph{R}} \Gamma \textbf{\emph{R}})_{ap} R_{qb}$ is replaced by $(\textbf{\emph{R}} \Gamma \textbf{\emph{R}})_{qb} R_{ap}$. Moreover, 
	\begin{flalign}
	\label{rgammarestimate22}
	\mathbb{E}_{\Psi} \Bigg[ \Im (R_{ap} R_{qb}) \bigg( A_{pq} - \displaystyle\frac{\gamma t^{1 / 2} w_{pq}}{(1 - \gamma^2)^{1 / 2}} \bigg) \Im (\textbf{\emph{R}} \Gamma \textbf{\emph{R}})_{ab} F'' (\Im R_{ab}) \Bigg] \le C \big( \mathfrak{E}_1 + \mathfrak{E}_2 (\psi_{pq} + \textbf{\emph{1}}_{p = q}) \big)
	\end{flalign} 
	
\end{prop}

We can now establish \Cref{gvcompare} assuming \Cref{gawestimatef2}.

\begin{proof}[Proof of \Cref{derivativeestimate} Assuming \Cref{gawestimatef2}]
	
	As mentioned previously, through \eqref{rgammapsi} and \eqref{sqf}, the right side of \eqref{estimatesumderivative} expands into a sum of expectations of degrees one, two, three, four, and higher. By \Cref{l:vanish}, we deduce that the expectation of each term of degree one or three in this expansion is equal to $0$. Furthermore, summing \Cref{l:final1} and \Cref{estimategijgammaproduct} over all $N^2$ possibilities for $(p, q)$ yields the existence of a constant $C = C(\alpha, \nu, \rho) > 0$ such that the sum of the fourth and higher order terms is bounded by 
	\begin{flalign}
	\label{terms4estimate}
	\frac{C}{(1 - \gamma^2)^{1 / 2}} N^{16 \varepsilon} \left( N^{\nu (\alpha - 4) + 1} \mathfrak{J} + t^2 \mathfrak{J} +Q N^{C_0 + 10} + \frac{1}{N} \right) < \frac{C}{(1 - \gamma^2)^{1 / 2} N^{\omega}}  \big( \mathfrak{J} + 1 + Q N^{C_0 + 11} \big),
	\end{flalign}
	
	\noindent here, we used the definition \eqref{omegadefinition} of $\omega$ and recalled that $t \sim N^{(\alpha - 2) \nu}$ from \Cref{t0estimate}. Next, summing \Cref{gawestimatef2} over all $N^2$ possibilities for $(p, q)$ and using the fact that $\Psi$ has at most $N^{1 + \alpha \rho + \varepsilon}$ entries equal to $1$, we estimate the second order terms by 
	\begin{flalign}
	\label{terms2estimate} 
	C N^{4 \varepsilon} \left( N^{(\alpha - 2) \rho} t \mathfrak{J} + N^{\alpha \rho} t \mathfrak{J} + t + N^{C_0 + 6} Q + \frac{1}{N} \right) < C N^{-\omega} \big( \mathfrak{J} + 1 + N^{C_0 + 7} Q \big),
	\end{flalign}
	
	\noindent after increasing $C$ if necessary. We have again used the definition \eqref{omegadefinition} of $\omega$ and recalled that $t \sim N^{(\alpha - 2) \nu}$. 
	
	Now the proposition follows from summing the contributions from \eqref{terms4estimate} and \eqref{terms2estimate} and using \eqref{probabilityomega1estimate} to replace $Q$ with $Q_0$ (up to an additive error that decays exponentially in $N$). 
\end{proof}

\subsubsection{Proof of \Cref{gawestimatef2}} 

\label{Second}

In this section we establish \Cref{gawestimatef2}. In fact, we will only establish the first estimate \eqref{rgammarestimate21} of that proposition, since the proof of the second estimate \eqref{rgammarestimate22} is entirely analogous. 

To that end, we will first through \Cref{rureplace2} estimate the error incurred be replacing all entries of $\textbf{R}$ on the left side of \eqref{rgammarestimate21} with those of $\textbf{U}$. Then, using the mutual independence of $\textbf{U}$, $A_{pq}$, and $w_{pq}$ conditional on $\Psi$ (recall \Cref{2l1l2independentabc}) and the definition \eqref{t} of $t$, we will deduce \Cref{gawestimatef2}. 

In order to implement the replacement, first observe that, since $A_{pq} (\textbf{R} - \textbf{U}) = 0$ by \eqref{rresolventu},
\begin{flalign*}
\Im \Big(  \big( \textbf{R} \Gamma \textbf{R} \big)_{ap} R_{qb} \Big) A_{pq}  F' \big( \Im R_{ab} \big)  =  \Im \Big(  \big( \textbf{U} \Gamma \textbf{U} \big)_{ap} U_{qb} \Big) A_{pq}  F' \big( \Im U_{ab} \big).
\end{flalign*}
Now write
 \begin{flalign}
 \label{e:lasttwolines}
 \begin{aligned}
 \Im & \big( ( \textbf{R} \Gamma \textbf{R} )_{ap} R_{qb} \big) \bigg( \displaystyle\frac{\gamma t^{1 / 2} w_{pq}}{(1 - \gamma^2)^{1 / 2}} \bigg) F' (\Im R_{ab}) \\
 & = \Im \big( ( \textbf{U} \Gamma \textbf{U} )_{ap} U_{qb} \big) \bigg( \displaystyle\frac{\gamma t^{1 / 2} w_{pq}}{(1 - \gamma^2)^{1 / 2}} \bigg) F' ( \Im U_{ab} ) \\
 & \quad +   \bigg( \displaystyle\frac{\gamma t^{1 / 2} w_{pq}}{(1 - \gamma^2)^{1 / 2}} \bigg) \Big( \Im \big( ( \textbf{R} \Gamma \textbf{R} )_{ap} R_{qb} \big) F' \big( \Im R_{ab} \big)  -   \Im \big( ( \textbf{U} \Gamma \textbf{U})_{ap} U_{qb} \big) F' ( \Im U_{ab} ) \Big).
 \end{aligned}
\end{flalign}
Using $\Gamma_{ij} = \gamma \Theta_{ij} + (1 - \gamma^2)^{1 / 2} \Phi_{ij}$ and $A_{pq} (\textbf{R} - \textbf{U}) = 0$ again,  and recalling from \eqref{gammalambda} and \eqref{athetatw} that
\begin{flalign*}
\Theta_{ij} = A_{ij} \one_{(i, j) \in \{ (p, q), (q, p) \}}, \qquad \Phi_{ij} = t^{1 / 2} w_{ij} \one_{(i, j) \in \{ (p, q), (q, p) \}},
\end{flalign*}
we can compute the last line in \eqref{e:lasttwolines} to find the terms with $\Theta_{ij}$ factors vanish, leaving
 \begin{flalign*}
  \bigg( \displaystyle\frac{\gamma t^{1 / 2} w_{pq}}{(1 - \gamma^2)^{1 / 2}} \bigg) \Big( \Im \big( ( \textbf{R} \Gamma \textbf{R} )_{ap} R_{qb} \big) F' \big( \Im R_{ab} \big)  -   \Im \big( ( \textbf{U} \Gamma \textbf{U})_{ap} U_{qb} \big) F' ( \Im U_{ab} ) \Big) = - \gamma t w^2_{pq} \mathfrak{Y},
 \end{flalign*}
  \noindent where 
\begin{flalign} 
\label{uyr}
\mathfrak{Y} = \Im  \big( U_{ap} U_{qp} U_{qb} + U_{aq} U_{pp} U_{qb} \big)  F' \big( \Im U_{ab} \big) - \Im \big( R_{ap} R_{qp} R_{qb} + R_{aq} R_{pp} R_{qb}  \big) F' \big( \Im R_{ab} \big) .
\end{flalign}
In total,
\begin{flalign}
\label{ufrf2}
\begin{aligned}
\Im \Big( & \big( \textbf{R} \Gamma \textbf{R} \big)_{ap} R_{qb} \Big) \bigg( A_{pq} - \displaystyle\frac{\gamma t^{1 / 2} w_{pq}}{(1 - \gamma^2)^{1 / 2}} \bigg) F' \big( \Im R_{ab} \big) \\
& \qquad = \Im \Big( \big( \textbf{U} \Gamma \textbf{U} \big)_{ap} U_{qb} \Big) \bigg( A_{pq} - \displaystyle\frac{\gamma t^{1 / 2} w_{pq}}{(1 - \gamma^2)^{1 / 2}} \bigg) F' \big( \Im U_{ab} \big) + \gamma t w_{pq}^2 \mathfrak{Y},
\end{aligned}
\end{flalign}

\noindent and so we would like to estimate $\big| \mathbb{E}_{\Psi} [ \gamma t w_{pq}^2 \mathfrak{Y} ] \big|$. This will be done through the following lemma.

\begin{lem} 
	
	\label{rureplace2}
	
	There exists a large constant $C = C (\alpha, \rho, \varepsilon, \vartheta) > 0$ such that
	\begin{flalign}
	\label{gammagamma1}
	\begin{aligned} 
	&  \Big| \mathbb{E}_{\Psi} \big[ \gamma t w_{pq}^2 \mathfrak{Y} \big] \Big| \le C N^{4 \varepsilon + (\alpha - 2) \rho - 2} t \mathfrak{J} + \displaystyle\frac{Ct}{N^2} + C N^{C_0 + 6} t Q. 
	\end{aligned} 
	\end{flalign}
	
\end{lem}

\begin{proof} 
	
	Since $w_{pq}$ is independent from $\textbf{R}$ and $\textbf{U}$, and since $\mathbb{E} [w_{pq}^2] = \frac{1}{N}$, we have that $\mathbb{E}_{\Psi} \big[ \gamma t w_{pq}^2 \mathfrak{Y} \big] = \gamma t N^{-1} \mathbb{E}_{\Psi} \big[ \mathfrak{Y} \big]$, and so it suffices to show that
	\begin{flalign}
	\label{uyr12}
	\begin{aligned}
	\bigg| \mathbb{E}_{\Psi} \Big[\Im  \big( U_{ap} U_{qp} U_{qb}  \big)  F' \big( \Im U_{ab} \big) & - \Im \big( R_{ap} R_{qp} R_{qb}  \big) F' \big( \Im R_{ab} \big) \Big] \bigg| \\
	& \quad <  C N^{4 \varepsilon + (\alpha - 2) \rho - 1} \mathfrak{J} + \displaystyle\frac{C}{N} + C N^{C_0 + 6} Q,
	\end{aligned}
	\end{flalign}
	
	\noindent and the same estimate if $\Im (U_{ap} U_{qp} U_{qb})$ and $\Im (R_{ap} R_{qp} R_{qb})$ are replaced by $\Im (U_{aq} U_{pp} U_{qb})$ and $\Im (R_{aq} R_{pp} R_{qb})$, respectively. We will only show \eqref{uyr12}, since the proof of the second statement is entirely analogous. 
	
	To that end, recall that \eqref{abidentity} and the definitions \eqref{gammalambda} and \eqref{rresolventu} yield 
	\begin{flalign}
	\label{rurur}
	\textbf{R} = \textbf{U} - \textbf{U} \Lambda \textbf{U} + \textbf{U} \Lambda \textbf{U} \Lambda \textbf{R}. 
	\end{flalign}
	
	
	
	Furthermore, we find from a Taylor expansion 
	\begin{flalign}
	\label{kfru}
	F^{(k)} (\Im R_{ab} )  - F^{(k)} (\Im U_{ab}) = \sum_{j=1}^{\vartheta} \frac{1}{j!} \kappa^j F^{(j + k)}(\Im U_{ab}) + \frac{1}{(\vartheta + 1)!} \kappa^{\vartheta + 1} F^{(\vartheta + 1)}(\kappa_1), 
	\end{flalign}
	
	\noindent where 
	\begin{flalign}
	\label{rabuab}
	\kappa = \Im (R_{ab} - U_{ab}) = - \Im (\textbf{U} \Lambda \textbf{R})_{ab},
	\end{flalign}
	
	\noindent by \eqref{abidentity} and \eqref{rresolventu}, and $\kappa_1  \in (\Im R_{ab}, \Im U_{ab})$.
	
	Applying \eqref{rurur} and \eqref{kfru}, we find that
	\begin{flalign} 
	\label{rpsiu} 
	\begin{aligned}
	\Im \big( R_{ap} R_{qp} R_{qb}  \big) F' \big( \Im R_{ab} \big) & =  \left(\sum_{j=0}^{\vartheta} \frac{\kappa^j}{j!}  F^{(j + 1)}(\Im U_{ab}) + \frac{\kappa^{\vartheta + 1} }{(\vartheta + 1)!} F^{(\vartheta +1)}(\kappa_1)\right) \\
	& \quad \times \Im \big( (\textbf{U} - \textbf{U} \Lambda \textbf{U} + \textbf{U} \Lambda \textbf{U} \Lambda \textbf{R})_{ap} ( \textbf{U} - \textbf{U} \Lambda \textbf{U} +  \textbf{U} \Lambda \textbf{U} \Lambda \textbf{R})_{qp} \\
	& \qquad \qquad \quad \times ( \textbf{U} -  \textbf{U} \Lambda \textbf{U} + \textbf{U} \Lambda \textbf{U} \Lambda \textbf{R} )_{qb} \big).
	\end{aligned} 
	\end{flalign} 
	
	Using \eqref{rabuab} to express $\kappa$ in terms of $\Lambda$ and expanding the right side of \eqref{rpsiu} yields a sum of monomials, each of which contains a product of $\Lambda$ factors. Any such monomial with $u$ factors of $\Lambda$ will be called an \emph{order $u$ monomial}. Observe that there is only one order $0$ monomial on the right side of \eqref{rpsiu}, which is $F' (\Im U_{ab}) U_{ap} U_{qp} U_{qb}$. We would like to estimate the other, higher order, monomials on the right side of \eqref{rpsiu}. 
	
	We first consider the monomials of order $1$. Observe that any such monomial is a product of $\Lambda_{pq}$ with terms of the form $F^{(j + 1)} (\Im U_{ab})$ and $U_{ij}$. Furthermore, recall from \Cref{2l1l2independentabc} that $\Lambda$ is independent from $\textbf{U}$ (conditional on $\Psi$). Thus, the symmetry of the entries of $\textbf{H}$ (and therefore the entries of $\Lambda$) implies that  
	\begin{flalign} 
	\label{1thetaexpectation} 
	\mathbb{E}_{\Psi} \big[ M \big] = 0, \quad \text{for any monomial $M$ of order $1$}. 
	\end{flalign} 
	
	Next let us estimate monomials of order $u$ with $2 \le u \le \vartheta$ on the event $\Omega$. Any such monomial is a product of $\Lambda_{pq}^u$ with a term of the form $F^{(k)} (\Im U_{ab})$ and at most $2u$ entries of $\textbf{U}$ or $\textbf{R}$; \Cref{l:RUresolvent} implies that the latter terms are all bounded by $2 N^{\varepsilon}$ on the event $\Omega$. Thus, if $M$ is a monomial of order $2 \le u \le \vartheta$, we have for some $1 \le k \le \vartheta$ that 
	\begin{flalign} 
	\label{mestimate2} 
	\begin{aligned}
	\mathbb{E}_{\Psi} \big[ \one_{\Omega} |M| \big] & \le 4^u N^{2 \varepsilon u} \mathbb{E}_{\Psi} \Big[ \big| F^{(k)} (\Im U_{ab}) \big| |\Lambda_{pq}|^u \Big] \\
	& = 4^u N^{2 \varepsilon u} \mathbb{E}_{\Psi} \Big[ \big| F^{(k)} (\Im U_{ab}) \big| \Big] \mathbb{E}_{\Psi} \big[ |\Lambda_{pq}|^u \big], 
	\end{aligned}
	\end{flalign}
	
	\noindent for some $j \le \vartheta$, where we have used the fact from \Cref{2l1l2independentabc} that $\textbf{U}$ and $\Lambda$ are independent (after conditioning on $\Psi$). 
	
	Now, recalling from \eqref{gammalambda} that $\big| \Lambda_{pq} \big| = B_{pq} \le |H_{pq}| \one_{|H_{pq}| \le N^{-\rho}}$, and applying \Cref{l:truncatedc}, the first estimate in \eqref{frusestimate}, \eqref{mestimate2}, and the definition \eqref{e:jdef} of $\mathfrak{J}$ yields the existence of a constant $C = C(\rho, \vartheta) > 0$ 
	\begin{flalign} 
	\label{2thetaexpectation} 
	\mathbb{E}_{\Psi} \big[ \one_{\Omega} |M|  \big] \le C N^{4 \varepsilon + (2 \varepsilon - \rho) (u - 2) + (\alpha - 2) \rho - 1} \mathfrak{J}, \quad \text{for any monomial $M$ of order $2 \le u \le \vartheta$}. 
	\end{flalign} 
	
	The final monomials to estimate on $\Omega$ are those of order $u$, with $u \ge \vartheta + 1$. Since \Cref{l:RUresolvent} implies that $\one_{\Omega} \big| \kappa_1 \big| \le 2 N^{\varepsilon}$, we find from the first estimate of \eqref{fc0c0} that $\one_{\Omega} \big| F^{(k + 1)} (\kappa_1) \big| \le N^{C_0 \varepsilon}$ for $0 \le k \le \vartheta$. Moreover, \Cref{l:RUresolvent} and the first estimate of \eqref{fc0c0} imply that $\one_{\Omega} \big| F^{(k + 1)} (\Im U_{ab}) \big| \le N^{C_0 \varepsilon}$ for any $0 \le k \le \vartheta$. Combining these estimates, the fact that any monomial of order $u$ is a product of $\Lambda_{pq}^u$ with one term of the form $F^{(k + 1)} (\Im U_{ab})$ or $F^{(k + 1)} (\kappa_1)$ and at most $2u$ entries of $\textbf{U}$ and $\textbf{R}$, and the fact that $(\rho - 2 \varepsilon) \vartheta \ge C_0 \varepsilon + 3$ implies the existence of a constant $C = C(\alpha, \rho, \varepsilon, \vartheta) > 0$ such that
	\begin{flalign} 
	\label{largeorderchi} 
	\mathbb{E}_{\Psi} \big[ \one_{\Omega} |M| \big] \le \displaystyle\frac{C}{N^3}, \quad \text{for any monomial $M$ of order $u \ge \vartheta + 1$}. 
	\end{flalign} 
	
	Off of the event $\Omega$, we apply the estimate 
	\begin{flalign} 
	\label{rpsirkfupsiu}
	\begin{aligned} 
	\bigg| \mathbb{E}_{\Psi} \Big[\Im  \big( U_{ap} U_{qp} U_{qb}  \big)  F' \big( \Im U_{ab} \big) & - \Im \big( R_{ap} R_{qp} R_{qb}  \big) F' \big( \Im R_{ab} \big) \Big] \bigg| \\
	& \qquad \le 2 N^{C_0 + 6} \mathbb{E}_{\Psi} \big[ \Phi_{pq}^2 \big] \le 2 N^{C_0 + 6},
	\end{aligned} 
	\end{flalign}
	
	\noindent where we have used the fact that the entries of $\textbf{R}$ and $\textbf{U}$ are bounded by $\eta^{-1} \le N^2$, and also the second estimate in \eqref{fc0c0}. 
	
	Now the lemma follows from applying \eqref{1thetaexpectation}, \eqref{2thetaexpectation}, and \eqref{largeorderchi} on $\Omega$, and applying \eqref{rpsirkfupsiu} off of $\Omega$.
\end{proof}

We can now establish Proposition \eqref{gawestimatef2}. 


\begin{proof}[Proof of \Cref{gawestimatef2}]
	
	Let us only establish \eqref{rgammarestimate21}, since the proof of \eqref{rgammarestimate22} is entirely analogous. 
	
	To that end, observe from \eqref{ufrf2} and \Cref{rureplace2} that for some $C = C(\alpha, \nu, \rho, \vartheta) > 0$ we have that 
	\begin{flalign}
	\label{estimatergammarraprqb2}
	\begin{aligned}
	\Bigg| \mathbb{E}_{\Psi} \bigg[ & \Im \Big( (\textbf{R} \Gamma \textbf{R} )_{ap} R_{qb} \big) \Big( A_{pq} - \displaystyle\frac{\gamma t^{1 / 2} w_{pq}}{(1 - \gamma^2)^{1 / 2}} \Big) F' (\Im R_{ab}) \bigg] \Bigg| \\
	& \le \Bigg| \mathbb{E}_{\Psi} \bigg[ \Im \Big( (\textbf{U} \Gamma \textbf{U} )_{ap} U_{qb} \big) \Big( A_{pq} - \displaystyle\frac{\gamma t^{1 / 2} w_{pq}}{(1 - \gamma^2)^{1 / 2}} \Big) F' (\Im U_{ab}) \bigg] \Bigg| + C \mathfrak{E}_1.
	\end{aligned}
	\end{flalign}
	
	\noindent Now, since $A_{pq}$, $w_{pq}$, and $\textbf{U}$ are mutually independent conditional on $\psi_{pq}$, and since $A_{pq}$ and $w_{pq}$ are symmetric we have from the definition \eqref{gammalambda} of $\Gamma$ that 
	\begin{flalign}
	\label{ugammauestimate2}
	\begin{aligned}
	\mathbb{E}_{\Psi} \bigg[ \Im \Big( (\textbf{U} \Gamma \textbf{U} )_{ap} U_{qb} \big) & \Big( A_{pq} - \displaystyle\frac{\gamma t^{1 / 2} w_{pq}}{(1 - \gamma^2)^{1 / 2}} \Big) F' (\Im U_{ab}) \bigg] \\
	& = \gamma \mathbb{E}_{\Psi} [ A_{pq}^2 - t w_{pq}^2 ] \mathbb{E}_{\Psi} \big[ \Im (U_{ap} U_{qp} U_{qb} + U_{aq} U_{pp} U_{qb}) F' (\Im U_{ab}) \big].
	\end{aligned} 
	\end{flalign}
	
	Now there are three cases to consider. If $\psi_{pq} = 0$ and $p \ne q$, then $\mathbb{E} [w_{pq}^2] = \frac{1}{N}$, so by the definition \eqref{t} of $t$ we have that 
	\begin{flalign}
	\label{apqtwpq}
	\mathbb{E}_{\Psi} [A_{pq}^2 - t w_{pq}^2] = \mathbb{E} \big[ H_{ij}^2 \one_{|H_{ij}| < N^{-\nu}} \big| |H_{ij}| < N^{-\rho} \big] - \displaystyle\frac{t}{N} = 0,
	\end{flalign}
	
	\noindent in which case the left side of \eqref{ugammauestimate2} is zero. 
	
	If $\psi_{pq} = 1$, then $A_{pq} = 0$ and $\mathbb{E} [w_{pq}^2] \le \frac{2}{N}$, so  
	\begin{flalign}
	\label{expectationpsiapqtwpq}
	\begin{aligned}
	\Big| \mathbb{E}_{\Psi} & [ A_{pq}^2 - t w_{pq}^2 ] \mathbb{E}_{\Psi} \big[ \Im (U_{ap} U_{qp} U_{qb} + U_{aq} U_{pp} U_{qb}) F' (\Im U_{ab}) \big] \Big| \\
	& \le \displaystyle\frac{2t}{N} \mathbb{E}_{\Psi} \Big[ \big( | U_{ap} U_{qp} U_{qb}| + |U_{aq} U_{pp} U_{qb} | \big) \big| F' (\Im U_{ab}) \big| \Big] \le \displaystyle\frac{4t}{N} \big( 8 N^{3 \varepsilon} \mathfrak{J} + N^6 Q \big),
	\end{aligned}
	\end{flalign} 
	
	\noindent where we have used \Cref{l:RUresolvent} to bound $\max_{1 \le i, j \le N} |U_{ij}|$ by $2 N^{\varepsilon}$ on $\Omega$ and \eqref{gijeta} and the fact that $\eta \ge N^{-2}$ to bound it off of $\Omega$. 
	
	Similarly, if $\psi_{pq} = 0$ and $p = q$, then $\mathbb{E} [w_{pq}^2] = \frac{2}{N}$ and so similar reasoning as applied in \eqref{apqtwpq} yields $\mathbb{E}_{\Psi} [A_{pq}^2 - t w_{pq}^2] = - \frac{t}{N}$, and so we again deduce that \eqref{expectationpsiapqtwpq} holds. 
	
	Now the proposition follows from summing \eqref{estimatergammarraprqb2}, \eqref{ugammauestimate2}, and either \eqref{apqtwpq} if $\psi_{pq} = 0$ and $p \ne q$ or \eqref{expectationpsiapqtwpq} if $\psi_{pq} = 0$ or $p = q$. 
	\end{proof}

\subsection{Outline of the Proof of \Cref{gvcompare} for \texorpdfstring{$m > 1$}{}}

\label{Proofm1} 

Let us briefly outline the modifications required in the above proof of \Cref{gvcompare} in the case $m > 1$. Then, the analog of \eqref{psif} becomes 
\begin{flalign*}
& \displaystyle\frac{\partial}{\partial \gamma} \mathbb{E}_{\Psi} \Big[ F \big( \Im G_{a_1 b_1}^{\gamma}, \ldots , \Im G_{a_m b_m}^{\gamma} \big) \Big] \\
& \quad = \displaystyle\sum_{k = 1}^m \displaystyle\sum_{1 \le p, q \le N} \mathbb{E}_{\Psi} \Bigg[ \Im (G_{a_k p}^{\gamma} G_{q b_k}^{\gamma}) \left( A_{pq} - \displaystyle\frac{\gamma t^{1 / 2} w_{pq}}{(1 - \gamma^2)^{1 / 2}} \right)  \partial_k F \big( \Im G_{a_1 b_1}^{\gamma}, \ldots , \Im G_{a_m b_k}^{\gamma} \big) \Bigg],
\end{flalign*}

\noindent and so we must show for each integer $k \in [1, m]$ that
\begin{flalign}
\label{psif2}
\begin{aligned}
\displaystyle\sum_{1 \le p, q \le N} \Bigg| \mathbb{E}_{\Psi} \bigg[  \Im (G_{a_k p}^{\gamma} G_{q b_k}^{\gamma}) \Big( A_{pq} - \displaystyle\frac{\gamma t^{1 / 2} w_{pq}}{(1 - \gamma^2)^{1 / 2}} & \Big) \partial_k F \big( \Im G_{a_1 b_1}^{\gamma}, \ldots , \Im G_{a_m b_k}^{\gamma} \big) \bigg] \Bigg| \\
& < \displaystyle\frac{C}{(1 - \gamma^2)^{1 / 2}} \big( N^{-\omega} (\mathfrak{J} + 1) + Q_0 N^{C + C_0} \big),
\end{aligned}
\end{flalign}

\noindent for some constants $\omega = \omega (\alpha, \nu, \rho, m) > 0$ and $C = C(\alpha, \nu, \rho, m) > 0$. 

Following \eqref{sqf}, for fixed $k \in [1, m]$ we then expand $\partial_k F (\Im G_{a_1 b_1}^{\gamma}, \ldots , \Im G_{a_m b_m}^{\gamma})$ as a degree three polynomial in the $\zeta_j = \Im \xi_{a_j b_j}$, whose lower (at most second) degree coefficients are derivatives of $F (\Im R_{a_1 b_1}, \ldots , \Im R_{a_m b_m})$. The degree three coefficients of this polynomial are fourth order derivatives of $F$, evaluated at some $\big( \zeta_{0; 1}, \ldots , \zeta_{0; m} \big)$ with $\zeta_{0; j} \in [\Im G_{a_j b_j}^{\gamma}, \Im R_{a_j b_j}]$. Inserting this expansion into \eqref{psif2}, one can show using \Cref{l:vanish} that the resulting first and third order terms in \eqref{psif2} will have expectation equal to $0$. Following the proofs of \Cref{l:final1} and \Cref{estimategijgammaproduct}, the fourth and higher order terms in this expansion can further be estimated by $C (1 - \gamma^2)^{-1 / 2} \big( N^{-\omega} (\mathfrak{J} + 1) + Q_0 N^{C + C_0} \big)$, for some $\omega = \omega (\alpha, \nu, \rho, m) > 0$ and $C = C(\alpha, \nu, \rho, m) > 0$. 

Let us make analogous estimates on the second order terms by following the content in \Cref{Second}. In particular, the analog of \eqref{ufrf2} becomes 
\begin{flalign}
\begin{aligned} 
\label{ufrf3} 
\Im \Big( & \big( \textbf{R} \Gamma \textbf{R} \big)_{a_k p} R_{q b_k} \Big) \bigg( A_{pq} - \displaystyle\frac{\gamma t^{1 / 2} w_{pq}}{(1 - \gamma^2)^{1 / 2}} \bigg) \partial_k F ( \Im R_{a_1 b_1}, \ldots , \Im R_{a_m b_m} ) \\
& = \Im \Big( \big( \textbf{U} \Gamma \textbf{U} \big)_{a_k p} U_{q b_k} \Big) \bigg( A_{pq} - \displaystyle\frac{\gamma t^{1 / 2} w_{pq}}{(1 - \gamma^2)^{1 / 2}} \bigg) \partial_k F ( \Im U_{a_1 b_1}, \ldots , \Im U_{a_m b_m} ) + \gamma t w_{pq}^2 \mathfrak{Y}_k,
\end{aligned}
\end{flalign}

\noindent where 
\begin{flalign*}
\mathfrak{Y}_k = & \Im  \big( U_{a_k p} U_{qp} U_{qb_k} + U_{a_k q} U_{pp} U_{qb_k} \big)  \partial_k F ( \Im U_{a_1 b_1}, \ldots , U_{a_m b_m}) \\
& \qquad - \Im \big( R_{a_k p} R_{qp} R_{qb_k} + R_{a_k q} R_{pp} R_{q b_k}  \big) \partial_k F ( \Im R_{a_1 b_1}, \ldots , \Im R_{a_m b_m} ) .
\end{flalign*}

\noindent As in \Cref{rureplace2}, $\big| \mathbb{E}_{\Psi} [\gamma t w_{pq}^2 \mathfrak{Y}_k] \big|$ can be bounded by $C N^{-\omega} (\mathfrak{J} + 1) + C Q_0 N^{C + C_0}$. Following \eqref{ugammauestimate2}, the expectation of the first term on the right side of \eqref{ufrf3} is equal to $0$ if $\psi_{pq} = 0$ and $p \ne q$ (by \eqref{apqtwpq}), and so the using the proof of \eqref{expectationpsiapqtwpq} the total of this expectation over all $(p, q) \in [1, N]^2$ can be bounded by $C N^{-\omega} (\mathfrak{J} + 1) + C Q_0 N^{C + C_0}$. 

Thus, the second order terms in the expansion of the left side \eqref{psif2} can also be bounded by $C \big( N^{-\omega} (\mathfrak{J} + 1) + Q_0 N^{C + C_0} \big)$, which verifies \eqref{psif2} and therefore establishes \Cref{gvcompare}.

\section{Intermediate Local Law for \texorpdfstring{$\alpha \in (1, 2)$}{}} 

\label{Model2}

In this section we establish \Cref{localalpha12}, which provides a local law for $\textbf{X}$ (recall \Cref{abremovedmatrix}) at almost all energies $E$ for $\alpha \in (1, 2)$. We begin by formulating an alternative version of this local law in \Cref{LocalEstimate12} and showing that it implies \Cref{localalpha12}. Its proof is deferred until \Cref{ProofEstimates}; the remainder of this section consists of preparatory material. In \Cref{Resolvent1} we recall some preliminary identities and estimates. In \Cref{OutlineLocal12} we provide an outline of the previous work and of our proof. Finally, we conclude in \Cref{afpe} with a statement for an approximate fixed point equation (given by \Cref{approxfixedpoint}), which will be established in \Cref{Equations12Proof}.

In what follows we fix parameters $\alpha, b, \nu > 0$ satisfying \eqref{alphanurho} and $\alpha \in (1, 2)$. We recall the functions $\varphi_{\alpha, z}$, $\psi_{\alpha, z}$, $y(z)$, and $m_{{\alpha}} (z)$ from \eqref{psi} and \eqref{stieltjespsi}; the removal matrix $\textbf{X}$ and its resolvent $\textbf{R}$ from \Cref{abremovedmatrix}; that $m_N (z) = N^{-1} \Tr \textbf{R}$; and the domain $\mathcal{D}_{K, \varpi, C}$ from \eqref{dcdelta}. Furthermore, for each $s > 0$ we denote by $\mathbb{K}_s \subset \mathbb{C}$ the set of $z \in \mathbb{C}$ of the form $r e^{\mathrm{i} \theta}$, with $r \in \mathbb{R}_{\ge 0}$ and $-\frac{\pi s}{2} \le \theta \le \frac{\pi s}{2}$. 

\subsection{An Alternative Intermediate Local Law}

\label{LocalEstimate12}

Through an inductive procedure that has been applied several times for Wigner matrices (see the book \cite{erdos2017dynamical} and references therein), \Cref{localalpha12} will follow from the following result.

\begin{thm}

\label{localz1z0}

Adopt the notation and hypotheses of \Cref{localalpha12}. For each $z \in \mathbb{H}$, define the event 
\begin{flalign}
\label{omega}
\begin{aligned}
\Omega (z) &  = \left\{ \big| m_N(z) - m_\alpha(z) \big| \le \displaystyle\frac{1}{N^{\varkappa }} \right\}  \cap \left\{ \displaystyle\max_{1 \le j \le N} \big| R_{jj} (z) \big| \le	 (\log N)^{30/(\alpha-1)} \right\} \\
& \qquad \cap \left\{ \displaystyle\max_{1 \le j \le N} \bigg| \mathbb{E} \Big[ \big( -\mathrm{i} R_{jj} (z) \big)^{\alpha / 2} \Big] - y(z)  \bigg| \le \displaystyle\frac{1}{N^{\varkappa }} \right\} .
\end{aligned}
\end{flalign}

\noindent Then, for sufficiently large $N$, there exist large constants $C = C(\alpha, b, \nu, \varpi, K) > 0$ and $\mathfrak{B} = \mathfrak{B} (\alpha) > 0$ such
\begin{flalign}
\label{omegacestimate}
\mathbb{P} \big[ \Omega (z)^c \big] \le C \exp \left( - \displaystyle\frac{(\log N)^2}{C} \right) \quad \text{if $\Im z = \mathfrak{B}$.}
\end{flalign} 

\noindent Further, suppose that $z_0, z \in \mathcal{D}_{K, \varpi, \mathfrak{B}}$ satisfy $\Re z = \Re z_0$ and $\Im z_0 - \frac{1}{N^5} \le \Im z \le \Im z_0$. If $\mathbb{P} \big[ \Omega (z_0)^c \big] \le \frac{1}{N^{20}}$, then
\begin{flalign}
\label{omegaz1z}
\mathbb{P} \big[ \textbf{\emph{1}}_{\Omega(z)} < \textbf{\emph{1}}_{\Omega(z_0)} \big] \le C \exp \left( - \displaystyle\frac{(\log N)^2}{C} \right).
\end{flalign}
for large enough $N$.
\end{thm}

\begin{proof}[Proof of \Cref{localalpha12} Assuming \Cref{localz1z0}]

	Let $\mathfrak{B}$ be as in \Cref{localz1z0}, and let $K = [u, v]$. Now let $A = \big\lfloor N^5 (v - u) \big\rfloor$ and let $B = \big\lfloor N^5 (\mathfrak{B} - N^{-\varpi}) \big\rfloor$. For each integer $j \in [0, A]$ and $k \in [0, B]$, let $z_{j, k} = u + \frac{j}{N^5} + \mathrm{i} \big( \mathfrak{B} - \frac{k}{N^5} \big)$. 
	
	Then, by induction on $M \in [0, B]$, there exists a large constant $C = C (\alpha, b, \nu, \varpi, K) > 0$ such that
	\begin{flalign}
	\label{omegaabm} 
	\mathbb{P} \left[ \bigcup_{j = 0}^A \bigcup_{k = 0}^M \Omega (z_{j, k})^c \right] \le C (M + 1) \exp \left( - \displaystyle\frac{(\log N)^2}{C} \right).
	\end{flalign}
	
	Now, the theorem follows from \eqref{omegaabm}; the deterministic estimate $\big| R_{ij} (z) - R_{ij} (z_0) \big| < \frac{1}{N}$ and $\big| m_N (z) - m_N (z_0) \big| < \frac{1}{N}$ for $z_0, z \in \mathcal{D}_{[u, v], \delta, \mathfrak{B}}$ with $|z - z_0| < \frac{1}{N^5}$ (due to \eqref{abidentity}, \eqref{gijeta}, and the fact that $\eta \ge \frac{1}{N}$); and the deterministic estimate $\big|  m_{{\alpha}} (z) - m_{{\alpha}} (z_0) \big| \le \frac{1}{N}$ for $z_0$ and $z$ subject to the same conditions (which holds since $m_{{\alpha}}$ is the Stieltjes transform of the probability measure $\mu_{\alpha}$). 
\end{proof}

\subsection{Identities and Estimates}

\label{Resolvent1}

In this section we recall several facts that will be used throughout the proof of \Cref{localz1z0}. In particular, we recall several resolvent identities and related bounds in \Cref{Resolvent}, and we recall several additional estimates in \Cref{Estimates}.

\subsubsection{Resolvent Identities and Estimates} 

\label{Resolvent}

In this section we collect several resolvent identities and estimates that will be used later. 

In what follows, for any index set $\mathcal{I} \subset \{ 1, 2, \ldots , N \}$, let $\textbf{X}^{(\mathcal{I})}$ denote the $N \times N$ matrix formed by setting the $i$-th row and column of $\textbf{X}$ to zero for each $i \in \mathcal{I}$. Further denote $\textbf{R}^{(\mathcal{I})} = \big \{ R_{jk}^{(\mathcal{I})} \big\} = \big( \textbf{X}^{(\mathcal{I})} - z \big)^{-1}$. If $\mathcal{I} = \{ i \}$, we abbreviate $\textbf{X}^{(\{ i \})} = \textbf{X}^{(i)}$, $\textbf{R}^{(\{ i \})} = \textbf{R}^{(i)}$, and $R_{jk}^{(\{ i \})} = R_{jk}^{(i)}$. Observe that $\mathbb{E} [m_N] = \mathbb{E} [R_{jj}]$, for any $j \in [1, N]$, due to the fact that all entries of $\textbf{X}$ are identically distributed.

\begin{lem}
	
	Let $\textbf{\emph{H}} = \{ H_{ij} \}$ be an $N \times N$ real symmetric matrix, $z \in \mathbb{H}$, and $\eta = \Im z$. Denote $\textbf{\emph{G}} = \{ G_{ij} \} = (\textbf{\emph{H}} - z)^{-1}$. 
	
	\label{matrixidentities} 
	
	\begin{enumerate}
		\item{ We have the Schur complement identity, which states for any $i \in [1, N]$ that
			\begin{flalign}
			\label{gii2}
			\displaystyle\frac{1}{G_{ii}} = H_{ii} - z - \sum_{\substack{1 \le j, k \le N \\ j, k \ne i }} H_{ij} G_{jk}^{(i)} H_{ki}. 
			\end{flalign} }

		\item{Let $\mathcal{I} \subset [1, N]$. For any $j \in [1, N] \setminus \mathcal{I}$, we have the Ward identity
			\begin{flalign}
			\label{gij2}
			\displaystyle\sum_{k \in [1, N] \setminus \mathcal{I}} \big| G_{jk}^{(\mathcal{I})} \big|^2 = \displaystyle\frac{\Im G_{jj}^{(\mathcal{I})}}{\eta}. 
			\end{flalign}}

	\end{enumerate}
	
\end{lem}

The estimates \eqref{gii2} and \eqref{gij2} can be found as (8.8) and (8.3) in the book \cite{erdos2017dynamical}, respectively. 

Observe that \eqref{abidentity}, \eqref{gijeta}, and the estimate (which holds for any $x, y \in \mathbb{C}$ and $p \in \mathbb{R}$)
\begin{flalign} 
\label{powerdiff} 
|x^p - y^p| \le |p| |x - y| \big( |x|^{p - 1} + |y|^{p - 1} \big), 
\end{flalign} 

\noindent implies that 
\begin{flalign}
\label{rijrijestimate}
\begin{aligned}
\big| R_{ij}^{(\mathcal{I})}  (z_1)^p - R_{ij}^{(\mathcal{I})} (z_0)^p \big| & \le |p| \big| R(z_1) - R(z_2) \big| \left( \displaystyle\frac{1}{(\Im z_0)^{p - 1}} + \displaystyle\frac{1}{(\Im z_1)^{p - 1}} \right) \\
& \le 2 |p| |z_1 - z_2| \left( \displaystyle\frac{1}{(\Im z_0)^{p + 1}} + \displaystyle\frac{1}{(\Im z_1)^{p + 1}} \right) N.
\end{aligned}
\end{flalign}

\noindent For each subset $\mathcal{I} \subset \{ 1, 2, \ldots , N \}$ and $i \notin \mathcal{I}$, define 
\begin{flalign}
\label{sitisi}
\begin{aligned}
S_{i, \mathcal{I}} & = \displaystyle\sum_{j \notin \mathcal{I} \cup \{ i \}} X_{ij}^2 R_{jj}^{(\mathcal{I} \cup \{ i \})}, \qquad    T_{i, \mathcal{I}} = X_{ii} - U_{i, \mathcal{I}}, \qquad \mathfrak{S}_{i, \mathcal{I}}  = \sum_{j \notin \mathcal{I} \cup \{ i \}} Z^2_{ij} R^{(\mathcal{I} \cup \{i \})}_{jj},
\end{aligned}
\end{flalign}

\noindent where we recall the entries $H_{ij}$ of $\textbf{H}$ are coupled with the entries $X_{ij}$ of $\textbf{X}$ through the removal coupling of \Cref{abremovedmatrix}, which also defined the $Z_{ij}$, and where 
\begin{flalign}
\label{uiidefinition}
U_{i, \mathcal{I}} = \displaystyle\sum_{\substack{j, k \notin \mathcal{I} \cup \{ i \} \\ j \ne k}} X_{ij} R_{jk}^{(\mathcal{I} \cup \{ i \})} X_{ki}.
\end{flalign}

\noindent If $\mathcal{I}$ is empty, we denote $S_i = S_{i, \mathcal{I}}$, $\mathfrak{S}_i = \mathfrak{S}_{i, \mathcal{I}}$, $T_i = T_{i, \mathcal{I}}$, and $U_i = U_{i, \mathcal{I}}$. The Schur complement identity \eqref{gii2} can be restated as 
\begin{flalign} 
\label{gii1}
R_{ii} = \displaystyle\frac{1}{T_i - z - S_i}.
\end{flalign}

\noindent Observe that since the matrix $\Im \textbf{R}^{(\mathcal{I})}$ is positive definite and each $X_{ii}$ is real, we have that 
\begin{flalign}
\label{siui} 
\Im S_{i, \mathcal{I}}\ge 0,\qquad \Im \mathfrak{S}_{i, \mathcal{I}} \ge 0, \qquad \Im (S_{i, \mathcal{I}} - T_{i, \mathcal{I}}) = \Im (S_{i, \mathcal{I}} + U_{i, \mathcal{I}}) \ge 0.
\end{flalign}

\subsubsection{Additional Estimates}

\label{Estimates} 

In this section we collect several estimates that mostly appear as (sometimes special cases of) results in \cite{bordenave2017delocalization,bordenave2013localization}. The first states that Lipschitz functions of the resolvent entries concentrate around their expectation and appears as Lemma C.3 of \cite{bordenave2013localization} (with the $f$ there replaced by $Lf$ here), which was established through the Azuma--Hoeffding estimate.

\begin{lem}[{\cite[Lemma C.3]{bordenave2013localization}}]
	
	\label{fgjjnear}
	
	Let $N$ be a positive integer, and let $\textbf{\emph{A}} = \{ a_{ij} \}_{1 \le i, j \le N}$ be an $N \times N$ real symmetric random matrix such that the $i$-dimensional vectors $A_i = (a_{i1}, a_{i2}, \ldots , a_{ii} )$ are mutually independent for $1 \le i \le N$. Let $z = E + \mathrm{i} \eta \in \mathbb{H}$, and denote $\textbf{\emph{B}} = \{ B_{ij} \} = (\textbf{\emph{A}} - z)^{-1}$. Then, for any Lipschitz function $f$ with Lipschitz norm $L$, we have that 
	\begin{flalign*}
	\mathbb{P} \Bigg[ \bigg| \displaystyle\frac{1}{N} \displaystyle\sum_{j = 1}^N f (B_{jj}) - \displaystyle\frac{1}{N} \displaystyle\sum_{j = 1}^N \mathbb{E} \big[ f(B_{jj}) \big] \bigg| \ge t \Bigg] \le 2 \exp \left( - \displaystyle\frac{N \eta^2 t^2}{8 L^2} \right).
	\end{flalign*}
\end{lem}

\noindent By setting $f(x) = x$ or $f(x) = \Im x$, $L = 1$, and $t = 4 (N \eta^2)^{-1 / 2} \log N$ in \Cref{fgjjnear}, we obtain
\begin{flalign} 
\label{mnexpectationnearmnimaginary} 
\begin{aligned}
\mathbb{P} \bigg[ \Big|  m_N (z) - \mathbb{E} \big[  m_N (z) \big] \Big| & > \frac{4 \log N}{(N \eta^2)^{1 / 2}}\bigg] \le 2 \exp \big( - (\log N)^2 \big), \\
\mathbb{P} \bigg[ \Big| \Im m_N (z) - \mathbb{E} \big[ \Im m_N (z) \big] \Big| & > \frac{4 \log N}{(N \eta^2)^{1 / 2}}\bigg] \le 2 \exp \big( - (\log N)^2 \big). 
\end{aligned}
\end{flalign}

The next lemma can be deduced from \Cref{fgjjnear} by choosing $f$ to be a suitably truncated variant of $x^{\alpha/2}$. It can be found as Lemma C.4 of \cite{bordenave2013localization}, with their $\gamma$ equal to our $\frac{\alpha}{2}$.

\begin{lem}[{\cite[Lemma C.4]{bordenave2013localization}}]
	
	\label{expectationfnear2}
	
	Adopt the notation of \Cref{fgjjnear}, and fix $\alpha \in (0, 2)$. Then there exists a large constant $C = C(\alpha) > 0$ such that, for any $t > 0$,
	\begin{flalign*}
	\mathbb{P} \Bigg[ \bigg| \displaystyle\frac{1}{N} \displaystyle\sum_{j = 1}^N ( - \mathrm{i} B_{jj} )^{\alpha / 2} - \displaystyle\frac{1}{N} \displaystyle\sum_{j = 1}^N \mathbb{E} \big[ (- \mathrm{i} B_{jj})^{\alpha / 2} \big] \bigg| \ge t \Bigg] \le 2 \exp \left( - \displaystyle\frac{ N (\eta^{\alpha / 2} t)^{4 / \alpha}}{C} \right). 
	\end{flalign*}
	
\end{lem}

The following, which is a concentration result for linear combinations of Gaussian random variables, follows from Bernstein's inequality and \eqref{gijeta}.

\begin{lem}
	
	\label{randomvariablenearexpectation2}

	Let $(y_1, y_2, \ldots , y_N)$ be a Gaussian random vector whose covariance matrix is given by $\Id$, and for each $1 \le j \le N$ let 
	\begin{flalign*}
	f_j = \big( \Im R_{jj}^{(i)}  \big)^{\alpha / 2} |y_j|^{\alpha}, \qquad g_j = \big( \Im R_{jj}^{(i)} \big)^{\alpha / 2} \mathbb{E} \big[ |y_j|^{\alpha} \big]. 
	\end{flalign*}
	
	\noindent Then, there exists a large constant $C > 0$ such that
	\begin{flalign*}
	\mathbb{P} & \left[ \bigg| \displaystyle\frac{1}{N} \displaystyle\sum_{j = 1}^N (f_j - g_j) \bigg| > \displaystyle\frac{C (\log N)^4}{N^{1 / 2} \eta^{\alpha / 2}} \right] < C \exp \left( - \displaystyle\frac{	(\log N)^2}{C} \right),
	\end{flalign*}
	
	\noindent where the probability is with respect to $(y_1, y_2, \ldots , y_N)$ and conditional on $\textbf{\emph{X}}^{(i)}$. 
	
\end{lem}

The following two results state that the diagonal resolvent entries of $\textbf{R}$ are close to those of $\textbf{R}^{(i)}$ on average. The first appears as Lemma 5.5 of \cite{bordenave2017delocalization} and was established by inspecting the singular value decomposition of $\textbf{R} - \textbf{R}^{(i)}$ (one could alternatively use the interlacing of eigenvalues between $\textbf{R}^{(i)}$ and $\textbf{R}$) and then applying H\"older's inequality. Estimates of this type for $r=1$ have appeared previously, for example as (2.7) of \cite{erdos2009local}.

\begin{lem}[{\cite[Lemma 5.5]{bordenave2017delocalization}}]
	
	\label{rankperturbation} 
	
	For any $r \in (0, 1]$, we have the deterministic estimate 
	\begin{flalign}
	\label{gijgijr1}
	\displaystyle\frac{1}{N} \displaystyle\sum_{j = 1}^N \big| R_{jj} - R_{jj}^{(i)} \big|^r \le \displaystyle\frac{4}{(N \eta)^r}. 
	\end{flalign}

\end{lem} 

\begin{cor}
	
	\label{rankperturbation2}
	
	For any $r \in [1, 2]$, we have the deterministic estimate 
	\begin{flalign}
	\label{gijgijr2} 
	\displaystyle\frac{1}{N} \displaystyle\sum_{j = 1}^N \big| R_{jj} - R_{jj}^{(i)} \big|^r \le \displaystyle\frac{8}{N \eta^r}. 
	\end{flalign}
	
\end{cor}

\begin{proof} 
	
	The estimate \eqref{gijeta} together with the bound $|a - b|^{r - 1} \le |a|^{r - 1} + |b|^{r - 1}$ for any $a, b \in \mathbb{C}$ yields 
	\begin{flalign}
	\label{gijgijr}
	\big| R_{jj} - R_{jj}^{(i)} \big|^r \le \big| R_{jj} - R_{jj}^{(i)} \big| \Big( |R_{jj}|^{r - 1}+ \big| R_{jj}^{(i)} \big|^{r - 1} \Big) \le 2 \eta^{1 - r} \big| R_{jj} - R_{jj}^{(i)} \big|.
	\end{flalign} 
	
	\noindent Now combining \eqref{gijgijr} with the $r = 1$ case of \Cref{rankperturbation} yields \eqref{gijgijr2}. 
\end{proof}

We also recall the following Lipschitz estimate for the functions $\varphi_{\alpha, z}$ and $\psi_{\alpha, z}$ (see \eqref{psi}), which appears as Lemma 3.6 in \cite{belinschi2009spectral}.

\bel[{\cite[Lemma 3.4]{bordenave2013localization}}]
\label{mapping}

There exists a large constant $c = c (\alpha)$ such that the following holds. For any $z\in \mathbb{H}$, the functions $\varphi_{\alpha,z}$ and $\psi_{\alpha, z}$ (see \eqref{psi}) are Lipschitz with constants $c_\varphi=c(\alpha)|z|^{-\alpha}$ and $c_\psi=c(\alpha)|z|^{-\alpha/2}$ on $\mathbb{K}_{\alpha / 2}$ and $\mathbb{K}_1$, respectively. 

\eel

We conclude this section with the following proposition (which is reminiscent of Lemma 3.2 of \cite{bordenave2013localization}) that bounds the quantity $T_i$ from \eqref{sitisi}. 

\begin{prop}
	
	\label{tiestimateprobability2}
	
	Recall the definition of $T_i = T_i (z)$ from \eqref{sitisi}. There exists a large constant $C = C(\alpha) > 0$ such that for any $t \ge 1$ we have that
	\begin{flalign}
	\label{tiprobability}
	\mathbb{P} \Bigg[ |T_i| \ge \displaystyle\frac{ C t }{(N \eta^2)^{1 / 2}}   \Bigg] \le \displaystyle\frac{C}{t^{\alpha / 2}}. 
	\end{flalign}
\end{prop}

\begin{proof}
	
	First, \eqref{probabilityxij2} yields the existence of a large constant $C(\alpha) > 0$ such that 
	\begin{flalign}
	\label{hiiprobability}
	\mathbb{P} \left[ |X_{ii}| \ge \displaystyle\frac{t}{(N \eta^2)^{1 / 2}} \right] \le \displaystyle\frac{C(N \eta^2)^{\alpha / 2}}{N t^{\alpha}} \le \displaystyle\frac{C}{t^{\alpha}}. 
	\end{flalign}
	
	\noindent Now, from a Markov estimate, we have for any $s > 0$ that
	\begin{flalign}
	\label{uiestimateprobability}
	\begin{aligned}
	\mathbb{P} \left[ |U_i| \le \displaystyle\frac{t}{(N \eta^2)^{1 / 2}} \right] & \le \displaystyle\frac{N \eta^2}{t^2} \mathbb{E} \Bigg[ \bigg|  \sum_{1 \le j \ne k \le N} X_j R_{jk}^{(i)} X_k \bigg|^2 \displaystyle\prod_{j = 1}^N \one_{|X_j| \le s} \Bigg] + \displaystyle\sum_{j = 1}^N \mathbb{P} \big[ |X_j| \le s \big] \\
	& \le \displaystyle\frac{N \eta^2}{t^2}\mathbb{E} \Bigg[ \bigg|  \sum_{\substack{1 \le j \ne k \le N \\ 1 \le j' \ne k' \le N}} X_j X_k \overline{X_{j'}} \overline{X_{k'}} R_{jk}^{(i)} \overline{R_{j' k'}}^{(i)} \displaystyle\prod_{j = 1}^N \one_{|X_j| \le s} \Bigg] + \displaystyle\frac{C}{s^{\alpha}} \\
	& = \displaystyle\frac{ 2 N \eta^2}{t^2}  \sum_{1 \le j \ne k \le N} \big| R_{jk}^{(i)} \big|^2 \mathbb{E} \big[ |X_j|^2 \one_{|X_j| \le s} \big]^2 + \displaystyle\frac{C}{s^{\alpha}} \\
	& \le \displaystyle\frac{8 C^2 s^{4 - 2 \alpha} \eta^2}{(2 - \alpha)^2 t^2 N} \displaystyle\sum_{1 \le j \ne k \le N} \big| R_{jk}^{(i)} \big|^2 + \displaystyle\frac{C}{s^{\alpha}} \le \displaystyle\frac{C^3 s^{4 - 2 \alpha}}{t^2} + \displaystyle\frac{C}{s^{\alpha}}, 
	\end{aligned}
	\end{flalign}
	
	\noindent after increasing $C$ if necessary, where we have abbreviated $X_j = X_{ij}$ for each $j \in [1, N]$, used \eqref{probabilityxij2}, the independence and symmetry of the $\{ X_j \}$, the fact that 
	\begin{flalign*}
	\mathbb{E} \big[ |X_j|^2 \one_{|X_j| \le s} \big] = 2 \displaystyle\int_0^s u \mathbb{P} \big[ |X_j| \ge u \big] du \le \displaystyle\frac{2C}{N} \displaystyle\int_0^s u^{1 - \alpha} du = \displaystyle\frac{2 C s^{2 - \alpha}}{(2 - \alpha) N},
	\end{flalign*}
	
	\noindent \eqref{gij2}, and \eqref{gijeta}. Setting $s = t^{1 / 2}$ in \eqref{uiestimateprobability} yields 
	\begin{flalign}
	\label{uiestimateprobability2}
	\mathbb{P} \left[ |U_i| \le \displaystyle\frac{t}{(N \eta^2)^{1 / 2}} \right]  \le \displaystyle\frac{C^3}{t^{\alpha}} + \displaystyle\frac{C}{t^{\alpha / 2}}. 
	\end{flalign}
	
	\noindent Now the lemma follows from the second identity in \eqref{sitisi}, \eqref{hiiprobability}, and \eqref{uiestimateprobability2}. 
\end{proof}

\begin{rem}
	
	\label{alphatiestimate}
	
	The proof of \Cref{tiestimateprobability2} does not require that $\alpha \in (1, 2)$ or that $E = \Re z$ is bounded away from $0$. Instead, it only uses that the entries of $N^{1 / \alpha} \textbf{X}$ are symmetric random variables satisfying \eqref{probabilityxij} and that $\Im z = \eta$. Thus, we will also use \Cref{tiestimateprobability2} in the proof of the local law in the case $\alpha \in (0, 2) \setminus \mathcal{A}$, which appears in \Cref{LocalTail2}. 
\end{rem}

\subsection{Outline of Proof}

\label{OutlineLocal12}

In preparation for the next section, we briefly outline the method used in \cite{bordenave2013localization} to prove a local law on intermediate scales, and also the way in which we improve on this method. Recalling the notation of \Cref{Resolvent}, we begin with the identity \eqref{gii1}. 

Approximating $T_i \approx \mathbb{E} [T_i] = 0$ and replacing each $X_{ij}$ with $h_{ij}$, we find that $R_{ii} \approx  (- \mathrm{i} z - \mathrm{i} \mathfrak{S}_i)^{-1}$. The identity $x^{-s} = \Gamma (s)^{-1} \int_0^\infty t^{s - 1} e^{-xt}\, dt$ then yields for any $s > 0$
\begin{flalign}
\label{approximateentryr}
\E \big[ (-\mathrm{i} R_{ii})^s \big]  \approx \displaystyle\frac{1}{\Gamma (s)} \int_0^\infty t^{s - 1} \E \Bigg[ \exp  \bigg( \mathrm{i}t z + \mathrm{i} t \sum_{j \ne i} R_{jj }^{(i)} h^2_{ij}   \bigg)  \Bigg]\, dt.
\end{flalign}

To linearize the exponential appearing in the integrand on the right side of \eqref{approximateentryr}, we use the fact that, for a standard Gaussian random variable $g$, $\E \big[ \exp ( \mathrm{i} z g) \big] = \exp \big( -\frac{z^2}{2} \big)$. Together with the mutual independence of the $\{ h_{ij} \}$, this yields
\begin{flalign}
\label{smomentrii}
\begin{aligned}
\E [ (-\mathrm{i} R_{ii})^s ] & \approx \displaystyle\frac{1}{\Gamma (s)} \int_0^\infty t^{s - 1} e^{\mathrm{i} t z} \displaystyle\prod_{j \ne i} \E \Bigg[ \exp  \bigg(  \mathrm{i} \big( - 2 t \mathrm{i} R_{jj }^{(i)}  \big)^{1 / 2} h_{ij} g_j \bigg)  \Bigg] \, dt \\
& \approx \displaystyle\frac{1}{\Gamma (s)} \int_0^\infty t^{s - 1} e^{\mathrm{i}tz} \E \Bigg[ \exp \bigg(  - \frac{\sigma^{\alpha} (2t)^{\alpha / 2}}{N} \sum_{j=1}^N \left(- \mathrm{i} R_{jj} \right)^{\alpha/2} |g_j|^\alpha  \bigg) \Bigg]  \, dt,  
\end{aligned}
\end{flalign}

\noindent where we used the explicit formula \eqref{betasigmaalphalaw} for the characteristic function of an $\alpha$-stable random variable and the $\textbf{g} = (g_1, g_2, \ldots , g_N)$ is an $N$-dimensional Gaussian random variable with covariance given by $\Id$. 

Approximating $|g_j|^{\alpha} \approx \mathbb{E} \big[ |g_j|^{\alpha} \big]$, using the identities
\begin{flalign}
\label{yjalphaidentity}
\mathbb{E} \big[ |g_j|^{\alpha} \big] = \displaystyle\frac{\Gamma (\alpha)}{2^{\alpha / 2 - 1} \Gamma \big( \frac{\alpha}{2} \big)}, \quad \text{and} \quad	\Gamma \left( \frac{\alpha}{2} \right) \Gamma \left( 1 - \frac{\alpha}{2} \right) = \displaystyle\frac{\pi}{\sin \big( \frac{\pi \alpha}{2} \big)},
\end{flalign}

\noindent recalling the definition of $\varphi_{\alpha, z}$ and $\psi_{\alpha, z}$ from \eqref{psi}, and applying \eqref{smomentrii} first with $s = \frac{\alpha}{2}$ and then with $s = 1$, we deduce  
\begin{flalign*} 
Y(z)  \approx \varphi_{\alpha,z} \big( Y(z) \big), \qquad X(z) \approx \psi_{\alpha,z} \big( Y(z) \big),
\end{flalign*}

\noindent where $X(z) = \mathbb{E} \big[ - \mathrm{i} R_{jj} (z) \big]$ and $Y(z) = \E \left[ ( - \mathrm{i} R_{jj} (z) )^{\alpha/2} \right]$.

Since the equation $Y(z) = \varphi_{\alpha, z} \big( Y(z) \big)$ is known \cite{arous2008spectrum} to have a unique fixed point $y(z)$, we expect from the previous two approximations that there is a global limiting measure ${m_\alpha = \mathrm{i} \psi_{\alpha,z}(y(z))}$; this matches with \eqref{stieltjespsi}. 

To obtain an intermediate local law for this measure, one must additionally quantify the error incurred from the above approximations. Among the primary sources of error here is the approximation $R_{ii} \approx \big( - \mathrm{i} z - \mathrm{i} \mathfrak{S}_i \big)^{-1}$. This not only requires that $|T_i|$ be small, but also that $\big| \mathfrak{S}_i + z \big|$ and $\big| S_i - T_i + z \big|$ (which is the denominator of $R_{ii}$) be bounded below. By analyzing certain Laplace transforms for quadratic forms in heavy-tailed random variables, the work \cite{bordenave2013localization} bounded these denominators by $\eta^{2 / \alpha - 1}$. This bound does not account for the true behavior of these resolvent entries (which should be bounded by $N^{\delta}$ for any $\delta > 0$), which causes the loss in scale of the intermediate local law established in \cite{bordenave2013localization} for $\alpha$ closer to one.

Thus, the improvement we seek will be to lower bound these denominators by $(\log N)^{- 30 / (\alpha - 1)}$; see \Cref{lambdaestimate} below. This will both yield nearly optimal bounds on the diagonal resolvent entries $R_{jj}$ and also allow us to establish an intermediate local law on the smaller scale $\eta = N^{-\varpi}$. Let us mention that the latter improvement (on the scale) is in fact necessary for us to implement our method. Indeed, if for instance $\alpha$ is near one, the results of \cite{bordenave2013localization} establish an intermediate local law for $\textbf{H}$ on scale approximately $\eta \gg N^{-1 / 5}$. However, in order for us to apply the flow results of \cite{bourgade2017eigenvector,erdos2017universality,landon2017convergence,landon2019fixed} we require $\eta < t$, and to apply our comparison result given by \Cref{gvcompare}, we require $ t \le N^{1 / (\alpha - 4)} \sim N^{-1 / 3}$. Hence in this case we require a local law for \textbf{X} on a scale $\eta  \ll N^{-1 / 3}$, and this is the scale accessed by \Cref{localalpha12}.

We do not know of a direct way to improve such a local law to the nearly optimal scale $\eta = N^{\delta - 1}$, which is necessary to establish complete eigenvector delocalization and bulk universality. However, one can instead access such estimates for \textbf{H} by combining our current local law for $\textbf{X}$ on scale $\eta^{-\varpi}$ with the comparison result given by \Cref{gvcompare} applied to $\textbf{V}_t$, for which the estimates hold on the optimal scale by the regularizing effect of Dyson Brownian motion. 

\subsection{Approximate Fixed Point Equations}

\label{afpe}

In light of the outline from \Cref{OutlineLocal12}, let us define the quantities
\begin{flalign}
\label{xy12}
X(z) &= \E \big[  - \mathrm{i} R_{jj} (z) \big], \qquad  Y(z) = \E \big[ ( - \mathrm{i} R_{jj} (z) )^{\alpha/2} \big],
\end{flalign}

\noindent which are independent of the index $j$, since the entries of $\textbf{X}$ are identically distributed.

Throughout this section and the next, we use the notation of \Cref{localalpha12} and set the parameters $\theta = \theta (\alpha, b, \nu) > 0$ and $\delta = \delta (\alpha, b, \nu, \varpi) > 0$ by 
\begin{flalign}
\label{deltaomega}
\theta = \displaystyle\frac{2 - \alpha}{50}, \qquad \delta = \displaystyle\frac{1}{10} \min \left\{ \theta, \nu - \varpi, \varpi - (2 - \alpha) \nu, \displaystyle\frac{1}{2} - \varpi \right\}.
\end{flalign}

As mentioned in \Cref{OutlineLocal12}, let us now define an event on which the denominators of $R_{jj} (z)$ and $\big( -z - S_j (z) \big)^{-1}$ are bounded below. To that end, for any $z \in \mathbb{H}$, we define
\begin{flalign}
\label{lambdaevent}
\begin{aligned}
\Lambda (z) & =  \left\{ \displaystyle\min_{1 \le j \le N} \Im \big( S_j + z \big) \ge (\log N)^{-30/(\alpha-1)} \right\} \cap \left\{ \displaystyle\min_{1 \le j \le N} \Im \big( \mathfrak{S}_j + z \big) \ge (\log N)^{-30/(\alpha-1)} \right\} \\
& \qquad \cap \left\{ \displaystyle\min_{1 \le j \le N} \Im \big( S_j - T_j + z \big) \ge (\log N)^{-30/(\alpha-1)} \right\}.
\end{aligned}
\end{flalign}

Assuming that $\mathbb{P} \big[ \Lambda (z)^c \big]$ has very small probability, the following proposition provides an approximate fixed point equation for $Y(z)$, as explained in \Cref{OutlineLocal12}. Its  proof will be provided in \Cref{Equations12Proof}. 
	
\begin{prop} 

\label{approxfixedpoint}

Adopt the notation and hypotheses of \Cref{localalpha12} and recall the parameters $\delta$ and $\theta$ defined in \eqref{deltaomega}. Let $z \in \mathcal{D}_{K, \varpi, \mathfrak{B}}$ for some compact interval $K \subset \mathbb{R}  \setminus \{ 0 \}$ and some $\mathfrak{B} > 0$. If $\mathbb{P} \big[ \Lambda (z)^c \big] < \frac{1}{N^{10}}$, then there exists a large constant $C = C(\alpha, b, \delta, \varepsilon) > 0$ such that
\begin{flalign}
\label{yyxyestimate}
\begin{aligned}
\Big| Y(z) - \varphi_{\alpha,z} \big( Y(z) \big) \Big|  & \le  C (c_{\varphi} + C) (\log N)^{100 /(\alpha-1)} \left( \displaystyle\frac{1}{(N \eta^2)^{\alpha / 8}} + \displaystyle\frac{1}{N^{2 \theta}} \right), \\
\Big| X(z) - \psi_{\alpha,z} \big( Y (z) \big) \Big|  & \le  C (c_{\psi} + C) (\log N)^{100 /(\alpha-1)} \left( \displaystyle\frac{1}{(N \eta^2)^{\alpha / 8}} + \displaystyle\frac{1}{N^{2 \theta}} \right),
\end{aligned}
 \end{flalign}

\noindent where $c_{\varphi} = c_{\varphi} (\alpha, z)$ and $c_{\psi} = c_{\psi} (\alpha, z)$ are given by \Cref{mapping}. 
\end{prop}

\section{Proof of \Cref{localz1z0}}

\label{ProofEstimates}

In this section we establish \Cref{localz1z0} in \Cref{61proof} after bounding the probability $\mathbb{P} \big[ \Lambda (z)^c \big]$ in \Cref{LambdaProbability}.

\subsection{Estimating \texorpdfstring{$\mathbb{P} \big[ \Lambda (z)^c \big]$}{}}

\label{LambdaProbability} 

In this section, we provide a estimate for $\mathbb{P} \big[ \Lambda (z)^c \big]$, given by \Cref{lambdaestimate}. Due to the Schur complement formula, this proposition implies optimal bounds on the resolvent entries that were not present in the previous work \cite{bordenave2013localization}. These bounds will in turn allow us to establish the local law on an improved scale.

In \Cref{LambdaProbability} we prove \Cref{lambdaestimate}, assuming \Cref{lambda1lambda2} and \Cref{lambda3} below. These propositions are then established in \Cref{ProofSLower} and \Cref{ProofSTLower}, respectively.

\begin{prop}

\label{lambdaestimate} 

Assume that $z \in \mathcal{D}_{K, \varpi, \mathfrak{B}}$ for some $\mathfrak{B} > 0$ and that $\varepsilon \le \mathbb{E} \big[ \Im m_N (z) \big] \le \frac{1}{\varepsilon}$, for some $\varepsilon > 0$. Then, there exists a large constant $C = C(\alpha, b, \delta, \varepsilon) > 0$ such that 
\begin{flalign*}
\P\left[ \Lambda (z)^c \right]  & \le  C \exp\left( - \frac{( \log N)^2}{C } \right).
\end{flalign*}		

\end{prop}

\subsubsection{A Heuristic for the Proof}

\label{OutlineEstimateR}

We now briefly outline our argument for the lower bound on $\Im (S_i + z)$. The Schur complement formula reads $R_{ii} = (T_i - z - S_i)^{-1}$. For the purposes of this outline, let us assume that $R_{ii} \approx (-z - S_i)^{-1}$, so that a lower bound on $\Im S_i$ implies an upper bound on $|R_{ii}|$.

Letting $\textbf{A}$ denote the diagonal $(N - 1) \times (N - 1)$ matrix whose entries are given by $\Im R_{jj}^{(i)}$ with $j \ne i$, we find that $\Im S_i = \langle X, \textbf{A} X \rangle$, where we defined the $(N - 1)$-dimensional vector $\textbf{X} = (X_{ij})_{j \ne i}$. Thus we obtain from \eqref{betasigmaalphalaw} that, if $Y = (y_1, y_2, \ldots , y_{N - 1})$ denotes an $(N - 1)$-dimensional Gaussian random variable whose covariance is given by $\Id$, then for any $t > 0$
 \begin{flalign*}
 \E \Bigg[  \exp \bigg( - \frac{t^2 }{2} \langle \textbf{A} X, X \rangle \bigg) \Bigg] = \mathbb{E} \Big[ \exp \big( \mathrm{i} t \langle \textbf{A}^{1 / 2} X, Y \rangle \big) \Big] \approx \mathbb{E} \Bigg[ \exp \bigg( - \displaystyle\frac{c |t|^{\alpha} \| \textbf{A}^{1 / 2} Y \|_{\alpha}^{\alpha}}{N} \bigg) \Bigg],
 \end{flalign*} 
 
 \noindent for some constant $c > 0$. Assuming that $\| \textbf{A}^{1 / 2} Y \|_{\alpha}^{\alpha}$ concentrates around its expectation, we obtain after replacing $t$ by $t \sqrt{2}$ and altering $c$ that 
\begin{flalign}
\label{oeq} 
\E \big[  \exp ( - t^2 \Im S_i ) \big]  <  \exp \Bigg(   - \frac{c  |t|^\alpha}{N} \sum_{j \ne i} \left|  \Im R_{jj}^{(i)} \right|^{\alpha/2} \Bigg).
\end{flalign}

\noindent Since $\frac{\alpha}{2} < 1$, we have 
\begin{flalign}
\label{sumrjjrjjalpha}
 \frac{1}{N}\sum_{j \ne i} \left|  \Im R_{jj}^{(i)} \right|^{\alpha/2} \ge \Im m_N^{(i)} \left( \max_{j \ne i} \big| \Im R^{(i)}_{jj} \big| \right)^{\alpha/2-1},
 \end{flalign}
 where $m_N^{(i)} = N^{-1} \Tr \textbf{R}^{(i)}$. If $t$ is chosen such that 
\begin{flalign*}
\frac{|t|^{\alpha} } {N}  \sum_j \left|\Im R_{jj}^{( i )}\right|^{\alpha/2} = (\log N)^2,
\end{flalign*}

\noindent then the right side of \eqref{oeq}  is very  small. Hence \eqref{oeq} implies, using Markov's inequality, that 
\begin{align*}
\P\left[\Im S_{i} \le  t^{-2} \right] = \P\left[t^2 \Im S_{i} \le  1  \right] \le \exp \big( - c ( \log N )^2   \big).
\end{align*}

\noindent Therefore, using the definition of $t$ and ignoring logarithmic factors, we have with high probability that
\begin{equation*}
\Im S_i \ge  \left( \frac{1}{N}\sum_{j \ne i} \left|  \Im R_{jj}^{(i)} \right|^{\alpha/2} \right)^{2/\alpha}.
\end{equation*}
Proceeding using \eqref{sumrjjrjjalpha} yields
\begin{flalign}
\label{siimaginaryrjj}
\Im S_{i} \ge   \left( \Im m_N^{(i)} \left( \max_{j \ne i} \left| \Im R^{(i)}_{jj} \right| \right)^{\alpha/2-1} \right)^{2/\alpha} \approx  \big| \Im  m_N \big|^{2/\alpha}   \left(\max_{1 \le j \le N} |R_{jj}| \right)^{1-2/\alpha}. 
\end{flalign}

\noindent Since $|R_{ii}| \le \Im S_{i}^{-1}$, this suggests that
\begin{flalign*}
|R_{ii}| \le  \big| \Im  m_N \big|^{- 2/\alpha}   \max_{1 \le j \le N} |R_{jj}|^{2/\alpha - 1},
\end{flalign*}

\noindent with high probability. Assuming that $\Im  m_N$ is bounded below and taking maximum over $i \in [1, N]$, this yields 
\begin{flalign*}
\displaystyle\max_{1 \le j \le N} |R_{jj}  | \le C \Big( \displaystyle\max_{1 \le j \le N}  |R_{jj} | \Big)^{2/\alpha -1},
\end{flalign*}

\noindent for some constant $C > 0$. Thus, since  $\frac{2}{\alpha} -1 < 1$ (this is where we use $\alpha > 1$), this implies an upper bound on each $|R_{jj}|$ with high probability.

\subsubsection{Proof of \Cref{lambdaestimate}}

\label{EstimateLowerR1}

An issue with the outline from \Cref{OutlineEstimateR} is in \eqref{siimaginaryrjj}, where we claimed that $\max_{j \ne i} \big| R_{jj}^{(i)} \big| \approx \max_j |R_{jj}|$. So, to implement this outline more carefully, we will instead proceed by showing that if one can bound the entries of $\textbf{R}^{\left(\mathcal{I}\right)}$ for each $|\mathcal{I}| = k$ (recall \Cref{Resolvent}) by some large $\vartheta > 0$, then we can bound the entries of $\textbf{R}^{\left(\mathcal{J}\right)}$ for each $|\mathcal{J}| = k - 1$ by $\vartheta^{2 / \alpha - 1} (\log N)^{20}$. In particular, if $\alpha > 1$, then $\frac{2}{\alpha} - 1 < 1$, so we can repeat this procedure approximately $(\log \log N)^2$ times to obtain nearly optimal estimates on the entries of $\textbf{R}$. 

To that end, we will define generalized versions of the event $\Lambda$. Fix the integer $M = \big\lceil (\log \log N)^2 \big\rceil$, and for each $0 \le k \le M$, define the positive real numbers $\varsigma_0, \varsigma_1, \ldots , \varsigma_M$ by 
\begin{flalign}
\label{varsigma}
\varsigma_M = \eta, \quad \text{and} \quad \varsigma_k = \varsigma_{k + 1}^{2 / \alpha - 1} (\log N)^{-20}, \quad \text{for each $0 \le k \le M - 1$}.
\end{flalign}

\noindent Observe that $\varsigma_0 \ge (\log N)^{- 25 /(\alpha - 1)}$ for sufficiently large $N$ since $\alpha \in (1, 2)$. 

Now, for each subset $\mathcal{I} \subset \{1, 2, \ldots , N \}$ with $|\mathcal{I}| = k \le M$, define the three events
\begin{flalign}
\label{siitiisiievent}
\begin{aligned}
\Lambda_{S, i, \mathcal{I}} (z) & = \left\{ \Im \big( S_{i, \mathcal{I}} (z) + z \big) \ge \varsigma_k \right\}, \qquad \Lambda_{\mathfrak{S}, i, \mathcal{I}} (z) = \left\{  \Im \big( \mathfrak{S}_{i, \mathcal{I}} (z) + z \big) \ge \varsigma_k \right\}, \\
& \qquad \quad \Lambda_{T, i, \mathcal{I}} (z)  = \left\{ \Im \big( S_{i, \mathcal{I}} (z) - T_{i, \mathcal{I}} (z) + z \big) \ge \varsigma_k	 \right\}.
\end{aligned}
\end{flalign}

\noindent Furthermore, for each $0 \le u \le M$, define the event 
\begin{flalign*}
\Lambda^{(u)} (z)  = \bigcap_{k = u}^M \bigcap_{\substack{\mathcal{I} \subset \{ 1, 2, \ldots , N \} \\ |\mathcal{I}| = k}} \bigcap_{i \notin \mathcal{I}}  \big( \Lambda_{S, i, \mathcal{I}} (z) \cap \Lambda_{\mathfrak{S}, i, \mathcal{I}} (z) \cap \Lambda_{T, i, \mathcal{I}} (z)\big). 
\end{flalign*}

\noindent The following propositions estimate the probabilities of the events $\Lambda_{S, i, \mathcal{I}}$, $\Lambda_{\mathfrak{S}, i, \mathcal{I}}$, and $\Lambda_{T, i, \mathcal{I}}$. We will establish \Cref{lambda1lambda2} in \Cref{ProofSLower} and \Cref{lambda3} in \Cref{ProofSTLower}.

\begin{prop} 
	
	\label{lambda1lambda2}
	
	Assume that $z \in \mathcal{D}_{K, \varpi, \mathfrak{B}}$, for some $\mathfrak{B} > 0$, and that $\varepsilon \le \mathbb{E} \big[ \Im m_N (z) \big] \le \frac{1}{\varepsilon}$, for some $\varepsilon > 0$. Then, there exists a large constant $C = C(\alpha, b, \delta, \varepsilon) > 1$ such that the following holds. For any integer $u \in [0, M - 1]$, any subset $\mathcal{I} \subset \{1, 2, \ldots , N \}$ with $|\mathcal{I}| = u$, and any $i \notin \mathcal{I}$, we have that
	\begin{flalign}
	\label{siestimate}
	\begin{aligned}
	\mathbb{P} \big[ \Lambda_{S, i, \mathcal{I}} (z)^c \big] \le \mathbb{P} \big[ \Lambda^{(u + 1)} (z)^c \big] + C \exp \left( -\displaystyle\frac{(\log N)^2}{C} \right), \\
	\mathbb{P} \big[ \Lambda_{\mathfrak{S}, i, \mathcal{I}} (z)^c \big] \le \mathbb{P} \big[ \Lambda^{(u + 1)} (z)^c \big] + C \exp \left( -\displaystyle\frac{(\log N)^2}{C} \right). 
	\end{aligned}
	\end{flalign}
	
\end{prop}

\begin{prop}
	
	\label{lambda3}
	
	Adopt the notation of \Cref{lambda1lambda2}. Then, there exists a large constant $C = C(\alpha, b, \delta, \varepsilon) > 1$ such that the following holds. For any integer $u \in [0, M - 1]$, any subset $\mathcal{I} \subset \{1, 2, \ldots , N \}$ with $|\mathcal{I}| = u$, and any $i \notin \mathcal{I}$, we have that
	\begin{flalign}
	\label{sitiestimate}
	\mathbb{P} \big[ \Lambda_{T, i, \mathcal{I}} (z)^c \big] \le \mathbb{P} \big[ \Lambda^{(u + 1)} (z)^c \big] + C \exp \left( - \displaystyle\frac{(\log N)^2}{C} \right). 
	\end{flalign}
	
\end{prop}

\noindent Assuming \Cref{lambda1lambda2} and \Cref{lambda3}, we can establish \Cref{lambdaestimate}.

\begin{proof}[Proof of \Cref{lambdaestimate} Assuming \Cref{lambda1lambda2} and \Cref{lambda3}]
	
	Applying a union bound over all $\mathcal{I} \subset \{1, 2, \ldots , n \}$ with $|\mathcal{I}| = u$ and $i \notin \mathcal{I}$ in \eqref{siestimate} and \eqref{sitiestimate} yields (with $C$ as in those estimates) 
	\begin{flalign}
	\label{lambdauestimatelambdau1}
	\mathbb{P} \big[ \Lambda^{(u)} (z)^c \big] \le 3 N^{u + 2} \mathbb{P} \big[ \Lambda^{(u + 1)} (z)^c \big] + 3 C N^{u + 2} \exp \left(- \displaystyle\frac{(\log N)^2}{C} \right). 
	\end{flalign}
	
	\noindent The estimate \eqref{gijeta} implies that $\Lambda^{(M)} (z)$ holds deterministically, so \eqref{lambdauestimatelambdau1} and induction on $u$ yields 
	\begin{flalign}
	\label{lambdauestimate}
	\mathbb{P} \big[ \Lambda^{(u)} (z)^c \big] \le (3 CN)^{(M + 2) (M - u + 1)} \exp \left( - \displaystyle\frac{(\log N)^2}{C} \right),
	\end{flalign}
	
	\noindent for each $0 \le u \le M$. Since $M = \big\lfloor (\log \log N)^2 \big\rfloor$, it follows from \eqref{lambdauestimate} (after increasing $C$ if necessary) that $\mathbb{P} \big[ \Lambda^{(0)} (z)^c \big] \le C \exp \big( - C^{-1} (\log N)^2 \big)$, from which the proposition follows since $\Lambda^{(0)} (z) \subseteq \Lambda (z)$. 
\end{proof}

\subsubsection{Proof of \Cref{lambda1lambda2}}

\label{ProofSLower}

In this section we establish \Cref{lambda1lambda2}. Before doing so, we require the following estimate on the Laplace transform for quadratic forms of removals of stable laws, which is an extension of Lemma B.1 of \cite{bordenave2013localization} to removals of stable laws; this lemma will be established in \Cref{StableCompare}.

\begin{lem}
	
	\label{quadraticlaw2}
	
	Let $\alpha \in (0, 2)$, $\sigma > 0$ be real, $0 < b < \frac{1}{\alpha}$ be reals, and $N$ be a positive integer. Let $\widetilde{X}$ be a $b$-removal of a deformed $(0, \sigma)$ $\alpha$-stable law (recall \Cref{partialstable}), and let $X = (X_1, X_2, \ldots , X_N)$ be mutually independent random variables, each having the same law as $N^{-1 / \alpha} \widetilde{X}$. Let $\textbf{\emph{A}} = \{ a_{ij} \}$ be an $N \times N$ nonnegative definite, symmetric matrix; $\textbf{\emph{B}} = \{ B_{ij} \} = \textbf{\emph{A}}^{1 / 2}$; and $Y = (y_1, y_2, \ldots , y_N)$ be an $N$-dimensional centered Gaussian random variable (independent from $X$) with covariance matrix given by $\Id$. Then, 	 
	\begin{flalign*}
	\mathbb{E} & \bigg[ \exp \Big( - \displaystyle\frac{t^2}{2} \langle \textbf{\emph{A}} X, X \rangle \Big) \bigg] \\
	& = \mathbb{E} \Bigg[ \exp \bigg( - \displaystyle\frac{\sigma^{\alpha} |t|^{\alpha} \| \textbf{\emph{B}} Y \|_{\alpha}^{\alpha}}{N} \bigg) \Bigg] \exp \bigg( O \Big( t^2 N^{ (2 - \alpha) (b - 1 / \alpha) - 1} (\log N) \Tr \textbf{\emph{A}} \Big) \bigg)  + N e^{- (\log N)^2 / 2},
	\end{flalign*}
	
	\noindent where, for any vector $w = (w_1, w_2, \ldots , w_N) \in \mathbb{C}^N$ and $r > 0$, we define $\| w \|_r = \big( \sum_{j = 1}^N |w_j|^r \big)^{1 / r}$. 
\end{lem}

Now we can prove \Cref{lambda1lambda2}. 

\bp[Proof of \Cref{lambda1lambda2}] 

Since the proofs of the two estimates in \eqref{siestimate} are very similar, we only establish the first one (on $\Im S_{i, \mathcal{I}}$). For notational convenience we assume that $i = N$ and $\mathcal{I} = \{ N - u, N - u + 1, \ldots , N - 1 \}$. Denote $\mathcal{J} = \mathcal{I} \cup \{ N \}$, and set 
\begin{flalign}
\label{siitiisiievent2}
\Lambda_{\mathcal{J}} (z) = \bigcap_{j \notin \mathcal{J}} \big( \Lambda_{S, j, \mathcal{J}} (z) \cap \Lambda_{\mathfrak{S}, j, \mathcal{J}} (z) \cap \Lambda_{T, j, \mathcal{J}} (z) \big) \subseteq \Lambda^{(u + 1)} (z). 
\end{flalign}

In what follows, let $\mathcal{G}$ denote the event on which 
\begin{flalign}
\label{eventg}
\left| \displaystyle\frac{1}{N} \displaystyle\sum_{j = 1}^{N - u - 1} \Im R_{jj}^{(\mathcal{J})} (z) - \mathbb{E} \big[ \Im m_N (z) \big] \right| > \displaystyle\frac{4 (u + 1)}{(N - u - 1) \eta} + \displaystyle\frac{4 	\log N}{N \eta^2}.
\end{flalign}

Observe that \eqref{gijgijr1} (applied with $r = 1$) and the second estimate in \eqref{mnexpectationnearmnimaginary} imply that $\mathbb{P} \big[ \mathcal{G} \big] \le 2 \exp \big( - (\log N)^2 \big)$. Now let us apply \Cref{quadraticlaw2} with $X = (X_{Nj})_{1 \le j \le N - u - 1}$ and $\textbf{A} = \{ A_{ij} \}$ given by the $(N - u - 1)\times (N - u - 1)$ diagonal matrix whose $(j, j)$-entry is equal to $A_{jj} = \Im R^{(\mathcal{J})}_{jj}$. Then $\Im S_{N, \mathcal{I}} = \langle X, \textbf{A} X \rangle$, so taking $t = (2\log 2)^{1/2} \varsigma_u^{- 1 /2}$ in \Cref{quadraticlaw2} yields from a Markov estimate that for sufficiently large $N$,
\begin{flalign} 
\label{1siestimate}
\begin{aligned}
\P \big[ &  \Im S_{N, \mathcal{I}} \le \varsigma_u \one_{\Lambda^{(u + 1)} (z)} \big] \\
& \le  2 \E \Bigg[\exp \bigg(  - \frac{t^2}{2} \langle {\textbf A} X , X \rangle \bigg) \one_{\Lambda_{\mathcal{J}} (z)} \one_{\mathcal{G}} \Bigg] + 2 \mathbb{P} [\mathcal{G}] \\
& =  2 \E \Bigg[ \mathbb{E} \bigg[ \exp \Big(  - \frac{t^2}{2} \langle {\textbf A} X , X \rangle \Big)  \bigg| \{ X_{jk} \}_{j, k \notin \mathcal{J}} \bigg] \one_{\Lambda_{\mathcal{J}} (z)} \one_{\mathcal{G}} \Bigg] + 2 \mathbb{P} [\mathcal{G}] \\
& \le 2 \E \Bigg[  \exp \bigg( - \frac{\sigma^{\alpha} \| {\textbf {A}}^{1/2} Y \|_\alpha^\alpha}{2 (N - u - 1) \varsigma_u^{\alpha / 2}}   \bigg)  \exp \Big( O \big( \varsigma_u^{-1} N^{(2 - \alpha) (b - 1 / \alpha)  -1} \Tr {\textbf A}   \big) \Big) \one_{\Lambda_{\mathcal{J}} (z)} \one_{\mathcal{G}} \Bigg] \\
& \qquad   + 6 N \exp\left( - \frac{( \log N)^2}{4} \right),
\end{aligned}		
\end{flalign} 

\noindent where $Y = (y_1, y_2, \ldots , y_{N - u - 1})$ is an $(N - u - 1)$-dimensional Gaussian vector whose covariance is given by $\Id$. On the right side of the equality in \eqref{1siestimate}, the inner expectation is over the $\{ X_{jk} \}$ with either $i \in \mathcal{J}$ or $j \in \mathcal{J}$, conditional on the remaining $\{ X_{jk} \}$; the outer expectation is over these remaining $\{ X_{jk} \}$ (with $j, k \notin \mathcal{J}$).

To estimate the terms on the right side of \eqref{1siestimate}, first observe from the definition \eqref{eventg} of the event $\mathcal{G}$ that $\one_{\mathcal{G}} \Tr \textbf{A} < \mathbb{E} \big[ \Im m_N (z) \big] + N^{-\delta}$ for sufficiently large $N$. Applying this, our assumption $\mathbb{E} \big[ \Im m_N (z) \big] \le \frac{1}{\varepsilon}$, and the fact that $\varsigma_u \ge \varsigma_{M - 1} \ge \eta^{2 / \alpha - 1} (\log N)^{-20}$ yields for sufficiently large $N$
\begin{flalign}
\label{exponentestimate} 
\begin{aligned}
\one_{\mathcal{G}} \varsigma_u^{-1} N^{(2 - \alpha) (b - 1 / \alpha)  -1} \Tr {\textbf A}  & \le 2 \varepsilon^{-1} \eta^{1 - 2 / \alpha} N^{(2 - \alpha) (b - 1 / \alpha)} (\log N)^{20} \\
& \le 2 \varepsilon^{-1} N^{(2 - \alpha) (b - 1 / \alpha + \varpi / \alpha)} (\log N)^{20} \le N^{5 \delta (\alpha - 2)}, 
\end{aligned}
\end{flalign}

\noindent where we have recalled that $\eta \ge N^{-\varpi}$ and used \eqref{alphanurho} and \eqref{deltaomega}. Inserting \eqref{exponentestimate} into \eqref{1siestimate} yields the existence of a large constant $C = C(\alpha, b, \delta, \varepsilon) > 0$ such that
\begin{flalign} 
\label{2siestimate}
\begin{aligned}
\P \big[  \Im S_{N, \mathcal{I}} & < \varsigma_u \one_{\Lambda_{\mathcal{J}} (z)} \big] \le C \E \Bigg[  \exp \bigg( - \frac{ \| {\textbf {A}}^{1/2} Y \|_\alpha^\alpha}{C N \varsigma_u^{\alpha / 2}}   \bigg) \one_{\Lambda_{\mathcal{J}} (z)} \Bigg] + C \exp\left( - \frac{( \log N)^2}{C } \right)  .
\end{aligned}
\end{flalign}

\noindent Therefore it suffices to lower bound $N^{-1} \| \textbf{A}^{1 / 2} Y \|_{\alpha}^{\alpha}$. To that end, we apply \Cref{randomvariablenearexpectation2} to deduce (after increasing $C$ if necessary) that
\begin{flalign*}
\begin{aligned}
\P \Bigg[ \bigg| \frac{1}{N} & \sum_{j = 1}^{N - u - 1} \big( \Im R^{(\mathcal{J})}_{jj} \big)^{\alpha/2} |y_j|^\alpha - \frac{1}{N} \sum_{j = 1}^{N - u -1} \big( \Im R^{(\mathcal{J})}_{jj} \big)^{\alpha/2} \E \big[ |y_j|^\alpha \big] \bigg| > \frac{C (\log N)^4 }{N^{1 / 2} \eta^{\alpha/2} }  \Bigg] \\
& \le C \exp\left( - \frac{(\log N)^2}{C} \right),
\end{aligned}
\end{flalign*} 

\noindent from which we find (again, after increasing $C$ if necessary) that
\begin{flalign}
\label{avestimatelower1}
\P \Bigg[ \frac{ \| \textbf{A}^{1 / 2} Y \|_{\alpha}^{\alpha}}{N} < \frac{1}{C N} \sum_{j = 1}^{N - u - 1} \big( \Im R_{jj}^{(\mathcal{J})} \big)^{\alpha/2} - \frac{C (\log N)^4 }{N^{1 / 2} \eta^{\alpha/2} }  \Bigg] \le C \exp\left( - \frac{(\log N)^2}{C} \right).
\end{flalign}

\noindent Now, observe that by \eqref{gii2} and the definition \eqref{siitiisiievent2} of the event $\Lambda_{S, N, \mathcal{J}} (z)$ that 
\begin{flalign*} 
\one_{\Lambda_{\mathcal{J}} (z)} \big| R_{jj}^{(\mathcal{J})} (z) \big| = \one_{\Lambda_{\mathcal{J}} (z)} \big| S_{j, \mathcal{J}} (z) - T_{j, \mathcal{J}} (z) + z \big|^{-1} \le \varsigma_{u + 1}^{-1},
\end{flalign*} 

\noindent for each $j \notin \mathcal{J}$. Therefore,
\begin{flalign}
\label{rijalphaestimate}
\displaystyle\frac{\one_{\Lambda_{\mathcal{J}} (z)}}{N} \sum_{j = 1}^{N - u - 1} \big( \Im R_{jj}^{(\mathcal{J})} \big)^{\alpha/2} \ge \displaystyle\frac{\varsigma_{u + 1}^{1 - \alpha / 2} \one_{\Lambda_{\mathcal{J}} (z)}}{N} \sum_{j = 1}^{N - u - 1} \Im R_{jj}^{(\mathcal{J})}.
\end{flalign}

\noindent Since we furthermore have by \eqref{gijgijr1} (applied with $r = 1$) and \eqref{gijeta} that 
\begin{flalign}
\label{rjjjestimate}
\displaystyle\frac{\one_{\Lambda_{\mathcal{J}} (z)}}{N} \displaystyle\sum_{j = 1}^{N - u - 1} \Im R_{jj}^{(\mathcal{J})} (z) \ge \one_{\Lambda_{\mathcal{J}} (z)} \left( m_N (z) - \displaystyle\frac{4 (u + 1)}{(N - u - 1) \eta} \right), 
\end{flalign}

\noindent it follows from the second estimate in \eqref{mnexpectationnearmnimaginary} and the assumption $\mathbb{E} \big[ \Im m_N (z) \big] \ge \varepsilon$ that 
\begin{flalign}
\label{rjjsumestimatelower1}
\mathbb{P} \left[ \displaystyle\frac{\one_{\Lambda_{\mathcal{J}} (z)}}{N} \displaystyle\sum_{j = 1}^{N - u - 1} \Im R_{jj}^{(\mathcal{J})} (z)  \le \frac{\varepsilon \one_{\Lambda_{\mathcal{J}} (z)}}{2} \right] \le 2 \exp \big( - (\log N)^2 \big),
\end{flalign}

\noindent for sufficiently large $N$. Inserting \eqref{rijalphaestimate} and \eqref{rjjsumestimatelower1} into \eqref{avestimatelower1} (upon observing that $\varsigma_{u + 1}^{1 - \alpha / 2} \ge \eta^{1 - \alpha / 2} \ge \frac{1}{N^{1 / 2 - \delta} \eta^{\alpha / 2}}$) yields

\begin{flalign}
\label{rijalphaestimate2}
\P \Bigg[ \frac{\| \textbf{A}^{1 / 2} Y \|_{\alpha}^{\alpha}}{N} <  \displaystyle\frac{\varepsilon \varsigma_{u + 1}^{1 - \alpha / 2} \one_{\Lambda_{\mathcal{J}} (z)} }{C} \Bigg] \le C \exp\left( - \frac{(\log N)^2}{C} \right).
\end{flalign}

\noindent again after increasing $C$ if necessary. Therefore, inserting \eqref{rijalphaestimate2} into \eqref{2siestimate} yields 
\begin{flalign*}
\P \big[ \Im S_{N, \mathcal{I}} & < \varsigma_u \one_{\Lambda_{\mathcal{J}} (z)} \big] \le C \E \Bigg[  \exp \bigg( - \frac{\varepsilon \varsigma_{u + 1}^{1 - \alpha / 2}}{C \varsigma_u^{\alpha / 2}}   \bigg) \Bigg] + 2 C \exp \left( - \displaystyle\frac{(\log N)^2}{C} \right), 
\end{flalign*}

\noindent from which the proposition follows since $\varsigma_u^{\alpha / 2} = \varsigma_{u + 1}^{1 - \alpha / 2} (\log N)^{-10 \alpha}$ (due to \eqref{varsigma}). 
\ep

\subsubsection{Proof of \Cref{lambda3}}

\label{ProofSTLower}

In this section we establish \Cref{lambda3}. We first require the following lemma that will be established in \Cref{EstimatesVariable}. 

\begin{lem} 
	
	\label{lowersum1} 
	
	Let $N$ be a positive integer and $ 0 < r < 2 < a \le 4$ be positive real numbers. Denote by $\textbf{\emph{w}} = (w_1, w_2, \ldots , w_N)$ a centered $N$-dimensional Gaussian random variable with covariance $U_{ij} = \mathbb{E} [w_i w_j]$ for each $1 \le i, j \le N$.	
	Define $V_j = \mathbb{E} [w_j^2]$ for each $1 \le j \le N$, and define 
	\begin{flalign*} 
	U = \displaystyle\frac{1}{N} \sum_{1 \le i, j \le N} U_{ij}^2, \quad V = \displaystyle\frac{\mathbb{E} \big[ \| \textbf{\emph{w}} \|_2^2 \big]}{N} = \displaystyle\frac{1}{N} \sum_{j = 1}^N V_j, \quad \mathcal{X} = \displaystyle\frac{1}{N} \displaystyle\sum_{j = 1}^N V_j^{a / 2}, \quad p = \frac{a - r}{a - 2}, \quad q = \frac{a - r}{2 - r}. 
	\end{flalign*}

	\noindent If $V > 100 (\log N)^{10} U^{1 / 2}$, then there exists a large constant $C = C(a, r) > 0$ such that	
	\begin{flalign*}
	\mathbb{P} \left[ \displaystyle\frac{\| \textbf{\emph{w}} \|_r^r}{N}  \le \displaystyle\frac{V^p}{C \big( \mathcal{X} (\log N)^8 \big)^{p / q}} \right] \le C \exp \left( - \displaystyle\frac{(\log N)^2}{2}\right).
	\end{flalign*}

\end{lem}
 
Observe that \Cref{lowersum1} is a certain type of H\"{o}lder estimate for correlated Gaussian random variables. The exponents $p$ and $q$ in that lemma come from such a bound (see \eqref{wjpwja}). With this lemma, we can now establish \Cref{lambda3}.

\begin{proof}[Proof of \Cref{lambda3}] 

For notational convenience we assume that $i = N$ and $u = 0$ (in which case $\mathcal{I}$ is empty); in what follows, we abbreviate the event 
\begin{flalign}
\label{siitiisiievent3}
\Lambda_N (z) = \bigcap_{j = 1}^{N - 1} \big( \Lambda_{S, j, \{ N \}} (z) \cap \Lambda_{\mathfrak{S}, j, \{ N \}} (z) \cap \Lambda_{T, j, \{ N \}} (z) \big).
\end{flalign} 

Now let us apply \Cref{quadraticlaw2} with $X = (X_{Nj})_{1 \le j \le N - 1}$ and the $(N - 1) \times (N - 1)$ matrix $\textbf{A} = \{ A_{ij} \}$, where we define $A_{ij} = \Im R_{ij}^{(N)}$ for $1 \le i, j \le N - 1$ (the superscript refers to the removal of the $N$th row.) Then $\Im (S_{N, \mathcal{I}} - T_{i, \mathcal{I}}) = \langle X, \textbf{A} X \rangle$. Therefore, taking $t = (2\log 2)^{1/2} \varsigma_0^{- 1 /2}$ in \Cref{quadraticlaw2} yields by following the beginning of the proof of \Cref{lambda1lambda2} until \eqref{2siestimate} the existence of a large constant $C = C(\alpha, b, \delta, \varepsilon) > 0$ such that
\begin{flalign} 
\label{2sitiestimate}
\begin{aligned}
\P \big[ \Im (S_N & - T_N) < \varsigma_0 \one_{\Lambda^{(1)} (z)} \big] \le C \E \Bigg[  \exp \bigg( - \frac{\| {\textbf {A}}^{1/2} Y \|_\alpha^\alpha}{C N \varsigma_0^{\alpha / 2}}   \bigg) \one_{\Lambda_N (z)} \Bigg] + C \exp\left( - \frac{( \log N)^2}{C } \right),
\end{aligned}
\end{flalign}

\noindent where $Y = (y_1, y_2, \ldots , y_{N - 1})$ is an $(N - 1)$-dimensional centered Gaussian random variable whose covariance is given by $\Id$. 
	
	Now let us apply \Cref{lowersum1} with $w_i = (\textbf{A}^{1 / 2} Y)_i$, $r = \alpha$, and $a = 4 - \alpha$. Then we find that $p = 2 = q$, $V_j = \Im R_{jj}^{(N)} (z)$, and $U_{jk} = \Im R_{jk}^{(N)} (z)$ for each $1 \le j, k \le N - 1$. We must next estimate the quantities $V$, $\mathcal{X}$, and $U$ from that lemma.

	To that end, observe from \eqref{gij2} and \eqref{gijeta} that 
	\begin{flalign}
	\label{uuijestimateupper}
	U \le \displaystyle\frac{4}{N^2} \displaystyle\sum_{1 \le i, j \le N - 1} \big| \Im R_{ij}^{(N)} (z) \big|^2 \le \displaystyle\frac{4}{N^2 \eta} \displaystyle\sum_{j = 1}^{N - 1} \Im R_{jj}^{(N)} (z) \le \displaystyle\frac{4}{N \eta^2}.
	\end{flalign} 
	
	\noindent Furthermore, since \eqref{gijgijr1} (with $r = 1$) and \eqref{gijeta} together imply \eqref{rjjjestimate}, we obtain from the second estimate in \eqref{mnexpectationnearmnimaginary}, the assumption $\mathbb{E} \big[ \Im m_N (z) \big] \ge \varepsilon$, and the fact that $V \ge N^{-1} \sum_{j = 1}^{N - 1} \Im R_{jj}^{(N)} (z)$ that 
	\begin{flalign}
	\label{rjjsumestimatelower}
	\mathbb{P} \left[  V \le \frac{\varepsilon }{2} \one_{\Lambda_N (z)} \right] \le 2 \exp \big( - (\log N)^2 \big),
	\end{flalign} 
	
	\noindent for sufficiently large $N$, which in particular by \eqref{uuijestimateupper} implies that 
	\begin{flalign}
	\label{uvestimate}
	\mathbb{P} \big[ V \le 100 (\log N)^{10} U^{1 / 2} \big] \le 2 \exp \big( - (\log N)^2 \big).
	\end{flalign}
	
	\noindent To upper bound $\mathcal{X}$, first observe from \eqref{gii2} and the definition \eqref{siitiisiievent3} of the event $\Lambda_N (z)$ that $\big| R_{jj}^{(N)} (z) \big| \one_{\Lambda_N (z)} \le \varsigma_{1}^{-1}$. Therefore, for sufficiently large $N$,
	\begin{flalign}
	\label{rijnrij2alpha}
	\mathcal{X} \one_{\Lambda_N (z)} = \displaystyle\frac{\one_{\Lambda_N (z)}}{N - 1} \displaystyle\sum_{j = 1}^{N - 1} \big( \Im R_{jj}^{(N)} (z) \big)^{2 - \alpha / 2} \le \displaystyle\frac{2 \one_{\Lambda_N (z)}}{N \varsigma_{1}^{1 - \alpha / 2}} \displaystyle\sum_{j = 1}^{N - 1} \Im R_{jj}^{(N)} (z).	
	\end{flalign}
	
	\noindent Therefore, \eqref{rijnrij2alpha}, \eqref{gijgijr1} (applied with $r = 1$), the second estimate in \eqref{mnexpectationnearmnimaginary}, and the assumption that $\mathbb{E} \big[ \Im m_N (z) \big] \le \frac{1}{\varepsilon}$ imply that for sufficiently large $N$ 
	\begin{flalign}
	\label{vestimateupper}
	\begin{aligned}
	\mathbb{P} \left[ \mathcal{X} \one_{\Lambda_N (z)} > \displaystyle\frac{4}{\varepsilon \varsigma_{1}^{1 - \alpha / 2}} \right] \le 2 \exp \big( - (\log N)^2 \big).
	\end{aligned} 
	\end{flalign}
	
	\noindent Now \eqref{rjjsumestimatelower}, \eqref{uvestimate}, \eqref{vestimateupper}, and \Cref{lowersum1} yield (after increasing $C$ if necessary) that 
	\begin{flalign}
	\label{avlowerestimate}
	\mathbb{P}\left[ \displaystyle\frac{\| \textbf{A}^{1 / 2} Y \|_{\alpha}^{\alpha}}{N} \le \displaystyle\frac{\varepsilon^3  \varsigma_{1}^{1 - \alpha / 2}\one_{\Lambda_N (z)}}{C (\log N)^8}\right] \le C \exp \left( - \displaystyle\frac{(\log N)^2}{2} \right) 
	\end{flalign}
	
	\noindent Inserting \eqref{avlowerestimate} into \eqref{2sitiestimate}, we obtain (again after increasing $C$ if necessary) that
	\begin{flalign*}
	\P \big[ \Im (S_N & - T_N) < \varsigma_0 \one_{\Lambda_N (z)} \big] \\
	&  \le C   \exp \bigg( - \displaystyle\frac{\varepsilon^3 \varsigma_1^{1 - \alpha / 2}}{C \varsigma_0^{\alpha / 2} (\log N)^8} \bigg) + C \exp\left( - \frac{( \log N)^2}{C } \right) + C \exp \left( - \displaystyle\frac{(\log N)^2}{2} \right) ,
	\end{flalign*}
	
	\noindent from which we deduce the proposition since $\varsigma_0^{\alpha / 2} = \varsigma_{1}^{1 - \alpha / 2} (\log N)^{- 10 \alpha}$ (due to \eqref{varsigma}). 
	\end{proof}

\subsubsection{Proof of \Cref{lowersum1}}

\label{EstimatesVariable} 

In this section we establish \Cref{lowersum1}. Before doing so, however, we require the following (likely known) estimate for sums of squares of correlated Gaussian random variables.

\begin{lem}
	
	\label{sumlower}
	
	Let $N$ be a positive integer, and let $\textbf{\emph{g}} = (g_1, g_2, \ldots , g_N)$ denote an $N$-dimensional centered Gaussian random variable with covariance matrix $\textbf{\emph{C}} = \{ c_{ij} \}$. Define $\textbf{\emph{a}} = (a_1, a_2, \ldots , a_N) \in \mathbb{R}_{\ge 0}$ by $a_j^2 = c_{jj}$ for each $j \in [1, N]$. Then, for sufficiently large $N$,
	\begin{flalign*}
	\mathbb{P} \Bigg[ \big| \| \textbf{\emph{g}} \|_2^2 - \| \textbf{\emph{a}} \|_2^2 \big| \ge 50 (\log N)^{10} \bigg( \displaystyle\sum_{1 \le j, k \le N} c_{jk}^2 \bigg)^{1 / 2} \Bigg] \le \exp \big( - (\log N)^2 \big). 
	\end{flalign*}
\end{lem}

\begin{proof}
	
	Let $\textbf{w} = (w_1, w_2, \ldots , w_N)$ be an $N$-dimensional centered Gaussian random variable with covariance matrix given by $\Id$. Let $\textbf{D}$ and $\textbf{U}$ be diagonal and orthogonal matrices, respectively, such that $\textbf{C} = \textbf{U} \textbf{D} \textbf{U}^{-1}$. Then $\textbf{g}$ has the same law as $\textbf{U} \textbf{D}^{1 / 2} \textbf{w}$, which implies that $\| \textbf{g} \|_2^2$ has the same law as $\sum_{j = 1}^N d_j w_j^2$. Moreover, 
	\begin{flalign*} 
	\displaystyle\sum_{j = 1}^N a_j^2 = \Tr \textbf{C} = \Tr \textbf{D} = \displaystyle\sum_{j = 1}^N d_j, \qquad \displaystyle\sum_{1 \le j, k \le N} c_{jk}^2 = \Tr \textbf{C}^2 = \Tr \textbf{D}^2 = \displaystyle\sum_{j = 1}^N d_j^2,
	\end{flalign*}
	
	\noindent so that 
	\begin{flalign}
	\label{probabilitydjwj1}
	\begin{aligned}
	\mathbb{P} & \Bigg[ \big| \| \textbf{g} \|_2^2 - \| \textbf{a} \|_2^2 \big| \ge 50 (\log N)^{10} \bigg( \displaystyle\sum_{1 \le j, k \le N} c_{jk}^2 \bigg)^{1 / 2} \Bigg] \\
	& \qquad = \mathbb{P} \Bigg[ \bigg| \displaystyle\sum_{j = 1}^N d_j (w_j^2 - 1) \bigg| \ge 50 (\log N)^{10} \bigg( \displaystyle\sum_{j = 1}^N d_j^2 \bigg)^{1 / 2} \Bigg]. 
	\end{aligned}
	\end{flalign}
	
	Now, since the $\{ w_j^2 - 1 \}$ are mutually independent, the fact that the right side of \eqref{probabilitydjwj1} is bounded by $\exp \big( - (\log N)^2 \big)$ is standard. For instance, it can be deduced by truncating each $d_j (w_j^2 - 1)$ at $4 d_j \log N$ and then applying the Azuma--Hoeffding inequality. 
\end{proof}

\begin{proof}[Proof of \Cref{lowersum1}]

	First observe that 
	\begin{flalign}
	\label{wjpwja}
	\left( \displaystyle\frac{1}{N} \displaystyle\sum_{j = 1}^N |w_j|^r \right)^{1 / p} \left( \displaystyle\frac{1}{N} \displaystyle\sum_{j = 1}^N |w_j|^a \right)^{1 / q} \ge \displaystyle\frac{1}{N} \displaystyle\sum_{j = 1}^N |w_j|^2. 
	\end{flalign} 
	
	We must therefore provide an upper bound on the $a$-th moments of the $w_j$ and a lower bound on the second moments. To that end, observe that since each $w_j$ is a Gaussian random variable of variance $V_j$, we have that 
	\begin{flalign}
	\label{wja}
	\mathbb{P} \left[ \displaystyle\frac{1}{N} \displaystyle\sum_{j = 1}^N |w_j|^a \ge 16 \mathcal{X} (\log N)^8 \right] \le \displaystyle\sum_{j = 1}^N \mathbb{P} \big[ |w_j| \ge  2 (\log N)^2 V_j^{1 / 2} \big] \le C N \exp \big( -(\log N)^2 \big). 
	\end{flalign}
	
	\noindent Furthermore, by \Cref{sumlower}, we have that 
	\begin{flalign}
	\label{wj2uv}
	\mathbb{P} \Bigg[ \bigg| \displaystyle\frac{1}{N} \displaystyle\sum_{j = 1}^N |w_j|^2 - V \bigg| \ge 50 (\log N)^{10} U^{1 / 2} \Bigg] \le \exp \big( - (\log N)^2 \big).
	\end{flalign}
	
	Now the lemma follows from combining \eqref{wjpwja}, \eqref{wja}, \eqref{wj2uv}, and the assumption that $V > 100 (\log N)^{10} U^{1 / 2}$. 
\end{proof}

\subsection{Establishing \Cref{localz1z0}}

\label{61proof}

In this section we prove \Cref{localz1z0}. We first establish \Cref{approxfixedpoint} in \Cref{Equations12Proof}. Then, we will show that \Cref{localz1z0} holds when $|z|$ is sufficiently large in \Cref{LargeE}; we will establish \Cref{localz1z0} for more general $z$ in \Cref{SmallE}.

\subsubsection{Proof of \Cref{approxfixedpoint}}

\label{Equations12Proof}

In this section we establish \Cref{approxfixedpoint}. To that end, denote  
 \begin{flalign*} 
 J(z) = \E \Big[ \big( -\mathrm{i} z   -\mathrm{i} S_j (z) \big)^{-1} \Big], \quad I(z) = \E \Big[ \big(  - \mathrm{i} z - \mathrm{i} S_j (z) \big)^{-\alpha/2} \Big].
 \end{flalign*}
 
 \noindent We begin by showing that $Y(z)$ is approximately equal to $I(z)$ and that $X(z)$ is approximately equal to $J(z)$ (recall \eqref{xy12}), assuming that $\mathbb{P} \big[ \Lambda (z)^c \big]$ is small. 

\begin{lem}

\label{diagonal}

Let $z \in \mathcal{D}_{K, \varpi, \mathfrak{B}}$ for some compact interval $K \subset \mathbb{R}$ and some $\mathfrak{B} > 0$. If $\mathbb{P}\big[ \Lambda (z)^c \big] < \frac{1}{N^{10}}$, then there exists a large constant $C = C(\alpha, b, \delta) > 0$ such that 
\begin{flalign}
\label{yixjerror}
 \big| Y(z) - I(z) \big| & \le \displaystyle\frac{C  (\log N)^{70 /(\alpha-1)}}{(N\eta^2)^{\alpha / 8}}, \quad  \big| X(z) - J(z) \big| \le \displaystyle\frac{C  (\log N)^{70 /(\alpha-1)}}{(N\eta^2)^{\alpha / 8}}. 
  \end{flalign}

\end{lem}

\bp 

In this proof, we will abbreviate $S_1 = S_1 (z)$, $T_1 = T_1 (z)$, and $R_{11} = R_{11} (z)$. To bound $|Y(z) - I(z)|$, we apply \eqref{gii1}, \eqref{powerdiff} (with $x = z + S_1$, $y = z + S_1 - T_1$, and $p = -\frac{\alpha}{2}$), and \eqref{gijeta} to obtain for any $v > 0$ that
\begin{flalign}
\label{s1r11error}
\begin{aligned}
 \Big| & \big(  - \mathrm{i} z - \mathrm{i} S_1  \big)^{-\alpha / 2} - \big( - \mathrm{i} R_{11}  \big)^{\alpha / 2} \Big|   \\
 & \le \displaystyle\frac{\alpha}{2} \big| T_1  \big| \Bigg( \bigg| \frac{1}{z+ S_1}  \bigg|^{\alpha / 2 +1} + \bigg| \frac{1}{z+ S_1 - T_1} \bigg|^{ \alpha / 2 +1}  \Bigg) \one_{|T_1 | \le v }  \one_{\Lambda (z)}  \\
& \quad +\left( \left| \frac{1}{z+ S_1}  \right|^{\alpha / 2} + \left| \frac{1}{z+ S_1 - T_1} \right|^{\alpha / 2}  \right)\one_{|T_1 | > v }  \one_{\Lambda (z)}  + \left( \left| \frac{1}{z+ S_1}  \right|^{\alpha / 2} + \left| \frac{1}{z+ S_1 - T_1} \right|^{\alpha / 2}  \right)  \one_{\Lambda (z)^c} \\
& \le \alpha v (\log N)^{60 / (\alpha - 1)} \one_{\Lambda (z)}  + 2 (\log N)^{30 / (\alpha - 1)} \one_{|T_1 | > v } \one_{\Lambda (z)} + \displaystyle\frac{2 }{\eta}  \one_{\Lambda (z)^c}.
\end{aligned}
\end{flalign}

 Setting $v = (N \eta^2)^{- 1 / 4}$ in \eqref{s1r11error}, taking expectations, using \Cref{tiestimateprobability2} to bound $\mathbb{P} \big[ |T_1| > v \big]$, and applying our assumed estimate $\mathbb{P}\big[ \Lambda (z)^c \big] < \frac{1}{N^{10}}$ yields 
\begin{flalign*}
\mathbb{E} \bigg[ \Big|  \big( & - \mathrm{i} z - \mathrm{i} S_1  \big)^{-\alpha / 2} - \big( - \mathrm{i} R_{11}  \big)^{\alpha / 2} \Big|  \one_{\Lambda (z)} \bigg] \le \displaystyle\frac{6 (\log N)^{70 / (\alpha - 1)} }{(N \eta^2)^{\alpha / 8}},
\end{flalign*}
 	
\noindent from which we deduce the first estimate in \eqref{yixjerror}. The proof of the second estimate in \eqref{yixjerror} is entirely analogous and therefore omitted. 
\ep

We now estimate the error resulting in replacing the entries of $\textbf{X}$ with those of $\textbf{H}$. 

\bel

\label{sdiff} 

There exists a large constant $C = C(\alpha) > 0$ such that 
\begin{flalign*}
\P \big[ |S_i - \mathfrak{S}_i| \ge N^{- 4 \theta} \big] < C N^{-4 \theta} \Big( 1  + \E \big[ | R_{11}| \big] \Big).
\end{flalign*} 

\eel

\bp Let $q>0$ be a real parameter, which will be chosen later. Fix $i$, and let $\mathcal B$ denote the event that for every $1\le j \le N$ with $j\neq i$,  $| H_{ij} |  <  N^q$ and $|Z_{ij} | < N^q$. By the hypotheses on the tail behavior of the $H_{ij}$  stated in \Cref{momentassumption} and a union bound,
\begin{equation}\label{e:pb}
\P\left[ \mathcal B \right] \ge 1 -  C N^{ - q \alpha},
\end{equation}

\noindent for some constant $C = C(\alpha) > 0$. We now work on the set $\mathcal B$. Due to the coupling between $\X$ and $\textbf{H}$ (of \Cref{abremovedmatrix}), \begin{flalign}
\label{siestimatesi}
\E \big[ \one_{\mathcal B} | S_i - \mathfrak{S}_i | \big] \le \mathbb{E} \left[ \one_{\mathcal B} \sum_{j\neq i}  | Z_{ij}^2 - X_{ij}^2 | \big| R^{(i)}_{jj} \big| \right] \le  \sum_{j\neq i} \E \left[\one_{\mathcal B} |Z_{ij}^2 - H_{ij}^2 | + \one_{\mathcal B} | H_{ij}^2 - X_{ij}^2 | \right] \E \Big[ \big| R_{jj}^{(i)} \big| \Big]. 
\end{flalign}

\noindent In this calculation, we used the independence of $H_{ij}$, $Z_{ij}$, and $X_{ij}$ from $R^{(i)}_{jj}$. To estimate the right side of \eqref{siestimatesi}, we take expectations in \eqref{gijgijr1} applied with $r = 1$ to obtain  
\begin{flalign*} 
\bigg| \E \Big[ \big| R^{(i)}_{jj} \big| \Big]  - \E \big[ | R_{jj} | \big]  \bigg| \le \frac{10}{N\eta},
\end{flalign*}

\noindent where we used the exchangeability of the $R_{jj}$ and \eqref{gijeta}.  Also, from \eqref{probabilityxij} and \Cref{abremovedmatrix} we compute (after increasing $C$ if necessary)
\begin{equation*}
\E \left[| H_{ij}^2 - X_{ij}^2 | \right]  = \E \left[ H_{ij}^2 \one_{ H_{ij} < N^{b - 1/\alpha}}  \right] \le C N^{2b  - 2/\alpha - b\alpha}.
\end{equation*}
Similarly, we compute (again after increasing $C$ if necessary)  
\begin{flalign*}
\E \left[\one_{\mathcal B} |Z_{ij}^2 - H_{ij}^2 | \right] \le \E \Big[\one_{\mathcal B} \big( 2 |Z_{ij}| |J_{ij}| + |J_{ij}|^2 \big) \Big] & \le C \left( N^{q - q \alpha /2 - 1/2 - 1/\alpha} + N^{-2/\alpha} \right) \\
& \le 2 C N^{(1 - \alpha / 2) (q - 1 / \alpha) - 1},
\end{flalign*}
where in the second inequality we used $\E\left[\one_{\mathcal B} |Z| |J| \right] \le \sqrt{ \E\left[\one_{\mathcal B} |Z|^2 \right]  \E \left[\one_{\mathcal B}|J|^2 \right]} $. We therefore deduce from \eqref{deltaomega} and \eqref{siestimatesi} that, after gaining a factor of $N$ due to the sum over $j$ and choosing $q = \frac{1}{4}$, 
\begin{flalign*}
\E \big[ \one_{\mathcal{B}} | S_i - \mathfrak{S}_i | \big]  \le C N^{-10 \theta} \left( \displaystyle\frac{10}{N \eta} + \E \big[ | R_{11}| \big] \right ).
\end{flalign*}

\noindent We conclude from a Markov estimate that 
\begin{equation}\label{e:goodpb} \P \big[ \one_\mathcal{B} |S_i - \mathfrak{S}_i| \ge N^{- 4 \theta} \big] < C N^{-4 \theta} \Big( 1  + \E \big[ | R_{11}| \big] \Big),\end{equation}

\noindent for sufficiently large $N$. The claim now follows from \eqref{e:pb} and \eqref{e:goodpb}.
\ep

Given \Cref{sdiff}, the proof of the following corollary is very similar to that of \Cref{diagonal} given \Cref{tiestimateprobability2}. Therefore, we omit its proof. 

\begin{cor}
	
	\label{replacedenom}  
	
	Let $p \in (0, 1]$ and $z \in \mathcal{D}_{K, \varpi, \mathfrak{B}}$ for some compact interval $K \subset \mathbb{R}$ and some $\mathfrak{B} > 0$. If $\mathbb{P}\big[ \Lambda (z)^c \big] < \frac{1}{N^{10}}$, then there exists a large constant $C = C(\alpha, b, \delta, \varepsilon, p) >0$ such that 
	\begin{flalign*}
	\E  \bigg[ \Big| \big( -\mathrm{i} z - \mathrm{i} S_1 (z) \big)^{-p}  -  \big( -\mathrm{i} z - i \mathfrak{S}_1 (z) \big)^{-p} \Big| \bigg]   \le \displaystyle\frac{C (\log N)^{100 / (\alpha-1)}}{N^{2 \theta}}.
	  \end{flalign*}

\end{cor}

\noindent We can now establish \Cref{approxfixedpoint}. 

\begin{proof}[Proof of \Cref{approxfixedpoint}]
	
	Given what we have done, the proof of this proposition will be similar to that of Proposition 3.1 in \cite{bordenave2013localization}. Specifically, defining
	\begin{flalign*}
	\gamma(z) = \E [(  -\mathrm{i} z - i \mathfrak{S}_1)^{-\alpha/2}], \qquad   \Xi (z) = \E[( -\mathrm{i} z   -i \mathfrak{S}_1 )^{-1}],
	\end{flalign*}
	
	\noindent we have from Corollary B.2 of (see in particular equation (31) of) \cite{bordenave2013localization} that
	\begin{flalign}
	\label{xigamma}
	\begin{aligned}
	\gamma (z) & = \E \Bigg[ \varphi_{\alpha, z} \bigg(  \frac{1}{N} \sum_{j=2}^N \big( - \mathrm{i} R^{(1)}_{jj} \big)^{\alpha/2} \frac{|g_j|^\alpha }{ \E \big[ |g_j|^{\alpha} \big]}  \bigg)  \Bigg], \\
	\Xi(z) & = \E \Bigg[ \psi_{\alpha,z} \bigg(  \frac{1}{N} \sum_{j=2}^N \big( - \mathrm{i} R^{(1)}_{jj} \big)^{\alpha/2} \frac{|g_j|^\alpha }{ \E \big[ |g_j|^{\alpha} \big]}   \bigg)  \Bigg] ,
	\end{aligned}
	\end{flalign}
	
	\noindent where $\textbf{g} = (g_2, g_3, \ldots , g_N)$ denotes an $(N - 1)$-dimensional centered Gaussian random variable with covariance matrix given by $\Id$ that is independent from $\textbf{H}$, and $\mathbb{E}$ denotes the expectation with respect to both $\textbf{H}$ and $\textbf{g}$. 
	
	We will only establish the first estimate in \eqref{yyxyestimate} (on $\big| Y(z) - \varphi_{\alpha, z} \big (Y (z) \big) \big|$). The proof of the second is entirely analogous and is therefore omitted. To that end, set $\rho_j = \big( - \mathrm{i} R_{jj}^{(1)} (z) \big)^{\alpha/2}$ for each $2 \le j \le N$. 
	
	We will show that $\gamma (z) \approx \varphi_{\alpha, z} \big( \mathbb{E} [\rho_2] \big)$ and $\mathbb{E} [\rho_2] \approx Y$, and then use \Cref{replacedenom} and \Cref{diagonal} to deduce that $I(z) \approx \gamma (z)$ and $Y(z) \approx I(z)$, respectively. To implement the first task, observe from \eqref{xigamma} that
	\begin{flalign}
	\label{xizrhoestimate}
	\begin{aligned} 
	\Big| & \gamma(z)  - \varphi_{\alpha,z} \big( \E [\rho_2] \big) \Big| \\
	& = \Bigg| \E \bigg[ \varphi_{\alpha, z} \Big(  \frac{1}{N} \sum_{j=2}^N \big( - \mathrm{i} R^{(1)}_{jj} \big)^{\alpha/2} \frac{|g_j|^\alpha }{ \E \big[|g_j|^\alpha \big]}  \bigg)  \Bigg]   - \varphi_{\alpha,z} \big( \mathbb{E} [\rho_2] \big) \Bigg| \\
	& \le   \E \Bigg[ c_{\varphi} \bigg|  \frac{1}{N} \sum_{j=2}^N \big( - \mathrm{i} R^{(1)}_{jj} \big)^{\alpha/2} \frac{|g_j|^\alpha }{ \E \big[|g_j|^\alpha \big]} - \mathbb{E} [\rho_2] \bigg| \Bigg] \\
	& \le \displaystyle\frac{c_{\varphi}}{\mathbb{E} \big[ |g_j|^{\alpha} \big]} \E \Bigg[ \bigg|  \frac{1}{N-1} \sum_{j=2}^N \rho_j |g_j|^\alpha   -  \frac{1}{N-1} \E \Big[ \sum_{j=2}^N \rho_j |g_j|^\alpha \Big] \bigg| \Bigg]  + \frac{ c_{\varphi} \E \big[ | \rho_2 | \big]	}{N}\\ 
	& \le \displaystyle\frac{ 2 c_{\varphi}}{N \mathbb{E} \big[ |g_j|^{\alpha} \big]} \Bigg(   \E \bigg[ \Big|  \sum_{j=2}^N \rho_j |g_j|^\alpha   -   \sum_{j=2}^N \rho_j \E \big[|g_j|^\alpha \big] \Big| \bigg]    + \mathbb{E} \big[ |g_j|^{\alpha} \big] \E \bigg[ \Big|  \sum_{j=2}^N \big( \rho_j   -  \mathbb{E} [\rho_j] \big) \Big| \bigg] \Bigg)  + \frac{c_{\varphi}}{N \eta},
	\end{aligned}	
	\end{flalign}
	
	\noindent where to deduce the first estimate we used the fact (from \Cref{mapping}) that $\varphi_{\alpha, z}$ is Lipschitz with constant $c_{\varphi}$ and the fact that $\mathbb{E} [\rho_j]$ is independent of $j \in [2, N]$, and to deduce the third estimate we used \eqref{gijeta}. 	
	
	Now observe that 
	\begin{flalign}
	\label{xizrhoestimate1}
	\displaystyle\frac{1}{N} \Bigg(   \E \bigg[ \Big|  \sum_{j=2}^N \rho_j |g_j|^\alpha   -   \sum_{j=2}^N \rho_j \E \big[|g_j|^\alpha \big] \Big| \bigg] \Bigg) \le \displaystyle\frac{1}{N} \Bigg(   \bigg|  \sum_{j=2}^N |\rho_j|^2 \mathbb{E} \big[ |g_j|^{2 \alpha} \big]   -   \sum_{j=2}^N |\rho_j|^2 \E \big[|g_j|^\alpha \big]^2 \bigg| \Bigg)^{1 / 2}. 
	\end{flalign}
	
	\noindent Furthermore, \Cref{expectationfnear2} with $t$ replaced by $(N \eta^2)^{- \alpha / 4} t (\log N)^2$ yields the existence of a large constant $C = C(\alpha) > 0$ such that
	\begin{flalign}
	\label{rhojrhojerror}
	\mathbb{P} \Bigg[ \displaystyle\frac{1}{N} \displaystyle\sum_{j = 2}^N \big( \rho_j - \mathbb{E} [\rho_j] \big) > \displaystyle\frac{t (\log N)^2}{(N \eta^2)^{\alpha / 4}} \Bigg] \le C  \exp \left( - \displaystyle\frac{t^2 (\log N)^2}{C} \right),
	\end{flalign}
	
	\noindent for each $t \ge 1$. Integrating \eqref{rhojrhojerror} yields
	\begin{flalign}
	\label{rhojrhojerror2}
	\mathbb{E} \Bigg[ \displaystyle\frac{1}{N} \displaystyle\sum_{j = 2}^N \big( \rho_j - \mathbb{E} [\rho_j] \big) \Bigg] \le \displaystyle\frac{C (\log N)^2}{(N \eta^2)^{\alpha / 4}}, 
	\end{flalign}
	
	\noindent after increasing $C$ if necessary. Combining \eqref{xizrhoestimate}, \eqref{xizrhoestimate1}, and \eqref{rhojrhojerror2} yields (again upon increasing $C$ if necessary)
	\begin{flalign}
	\label{xizrho2estimate} 
	\Big| \gamma(z)  - \varphi_{\alpha,z} \big( \E [\rho_2] \big) \Big| \le \displaystyle\frac{C c_{\varphi} \mathbb{E} \big[ |\rho_2|^2 \big]^{1 / 2}}{N^{1 / 2}} + \displaystyle\frac{C c_{\varphi} (\log N)^2}{(N \eta^2)^{\alpha / 4}} + \displaystyle\frac{c_{\varphi}}{N \eta} \le \displaystyle\frac{2 C c_{\varphi} (\log N)^2}{(N \eta^2)^{\alpha / 4}}, 
	\end{flalign}
	
	\noindent where in the second estimate we used the fact that $\big| \rho_2 \big|^2 \le \eta^{-\alpha} \le \eta^{-2}$ (due to \eqref{gijeta}). 
	
	To show that $\mathbb{E} [\rho_2] \approx Y(z)$, we apply \eqref{gijeta}, \eqref{gijgijr1} with $r = \frac{\alpha}{2}$, and the exchangeability of the entries of $\textbf{X}$, and then take expectations to find 
	\begin{flalign}
	\label{rho2y}
	\big|\E [\rho_2] - Y(z) \big| \le  \displaystyle\frac{5}{(N \eta)^{\alpha/2}}.
	\end{flalign}
	
	\noindent From \eqref{xizrho2estimate}, \eqref{rho2y}, and the fact that $\varphi_{\alpha, z}$ is Lipschitz with constant $c_{\varphi}$, we deduce that 
	\begin{flalign}
	\label{gammarhoestimate}
	\Big| \gamma(z) - \varphi_{\alpha,z} \big( Y(z) \big) \Big| \le \displaystyle\frac{c_{\varphi} C (\log N)^2}{(N \eta^2)^{\alpha / 8}},
	\end{flalign}
	
	\noindent upon increasing $C$ if necessary. 
	
	Now, by \Cref{replacedenom} (with $p = \frac{\alpha}{2}$) and \Cref{diagonal} we have (again after increasing $C$ if necessary) that 
	\begin{flalign} 
	\label{gammazyestimate}
	\big| I(z) - \gamma(z) \big| \le   \displaystyle\frac{C (\log N)^{100 / (\alpha-1)}}{N^{4\theta}}, \qquad \big| Y(z) - I(z) \big| \le \displaystyle\frac{C (\log N)^{70/(\alpha-1)}}{(N\eta^2)^{\alpha/4}}. 
	\end{flalign}
	
	\noindent Now the first estimate in \eqref{yyxyestimate} follows from \eqref{gammarhoestimate} and \eqref{gammazyestimate}. 
\end{proof}

\subsubsection{Proof of \Cref{localz1z0} for Large \texorpdfstring{$|z|$}{}}

\label{LargeE}

In this section we establish \Cref{localz1z0} if $|z|$ is sufficiently large. We begin by addressing the case of large $\eta$, given by the following lemma. 

\begin{lem}
	
	\label{largezc}
	
	Adopt the notation of \Cref{localalpha12}. There exist constants $C = C(\alpha, b) > 0$ and $\mathfrak{B} = \mathfrak{B} (\alpha) > 0$ such that \eqref{omegacestimate} holds for some $\varkappa > 0 $.
\end{lem}

\begin{proof}
	
	From the definition \eqref{stieltjespsi} of $m_{{\alpha}} (z)$, we deduce that 
	\begin{flalign}
	\label{mnmmuzc}
	\Big| \E \big[ m_N(z) \big] - m_\alpha(z) \Big| \le \Big| X(z)  - \psi_{\alpha,z} \big( Y(z) \big) \Big| + \Big| \psi_{\alpha, z} \big( Y(z) \big) - \psi_{\alpha, z} \big( y(z) \big) \Big|.
	\end{flalign}
	
	In view of \Cref{mapping}, there exists a large constant $\mathfrak{B} = \mathfrak{B} (\alpha) > 1$ such that for any $z \in \mathbb{H}$ with $|z| \ge \mathfrak{B}$ we have that $\max \{ c_{\varphi}, c_{\psi} \} < \frac{1}{2}$. Thus, let $E \in \mathbb{R}$ and let $z = E + \mathrm{i} \mathfrak{B}$. Then, 
	\begin{flalign*}
	\big| Y(z) - y(z) \big| & \le \big| Y(z) - \varphi_{\alpha,z}(Y(z)) \big| + \big| \varphi_{\alpha,z}(Y(z)) - \varphi_{\alpha,z}(y(z)) \big| \\
	&  \le \big| Y(z) - \varphi_{\alpha,z}(Y(z)) \big| + \displaystyle\frac{\big| Y(z) - y(z) \big|}{2},
	\end{flalign*}	
	
	\noindent which implies that 
	\begin{flalign}
	\label{yyestimatec}
	\big| Y(z) - y(z) \big| \le 2 \big| Y(z) - \varphi_{\alpha,z} \big( Y(z) \big) \big|.
	\end{flalign}
		
	\noindent By \eqref{siui}, $\Lambda (z)$ holds deterministically. Thus we can apply \Cref{approxfixedpoint} (and \eqref{yyestimatec}) to bound the right side of \eqref{mnmmuzc}. This yields the existence of a large constant $C = C(\alpha, b, \varkappa) > 0$ such that
	\begin{flalign}
	\label{1yyestimatec}
	\begin{aligned}
		\big| Y(z) - y(z) \big| & \le C (\log N)^{100 / (\alpha - 1)} \left( \displaystyle\frac{1}{(N \eta^2)^{\alpha / 8}} + \displaystyle\frac{1}{N^{2 \theta}} \right) \le \displaystyle\frac{2 C(\log N)^{100 / (\alpha - 1)}}{N^{\varkappa}}, \\
	\Big| \E \big[ m_N(z) \big] - m_\alpha(z) \Big| & \le C (\log N)^{100 / (\alpha - 1)} \left( \displaystyle\frac{1}{(N \eta^2)^{\alpha / 8}} + \displaystyle\frac{1}{N^{2 \theta}} \right) \le \displaystyle\frac{2 C(\log N)^{100 / (\alpha - 1)}}{N^{\varkappa}}.
	\end{aligned}
	\end{flalign}

	\noindent Now the lemma follows from \eqref{1yyestimatec}, the first estimate in \eqref{mnexpectationnearmnimaginary}, and the deterministic estimate $\big| R_{ij} (z) \big| \le \frac{1}{\eta} < 1$. 
\end{proof}

 The following proposition analyzes the case when $\Re E$ is large.

\begin{prop}

\label{weaklaw} 

Let $\mathfrak{B}$ be as in \Cref{largezc}. There exists a large constant $E_0 = E_0 (\alpha) > 0$ such that, for any compact interval $K = [u, v]$ disjoint from $[-E_0, E_0]$, there exists a large constant $C = C(\alpha, b, u, v, \delta) > 0$ and absolute constant $c>0$ such that the following holds. Suppose $E \in [u, v]$ and $z_0, z \in \mathcal{D}_{[u, v], \varpi, \mathfrak{B}}$ satisfy $\Re z_0 = E = \Re z$ and $\Im z_0  - \frac{1}{N^5} \le \Im z \le \Im z_0$, and $\varkappa < c$. If $\mathbb{P} \big[ \Omega( z_0 )^c \big] < \frac{1}{N^{20}}$, then
\begin{flalign}
\label{omegaestimatee0}
\mathbb{P} \big[ \textbf{\emph{1}}_{\Omega (z)} < \textbf{\emph{1}}_{\Omega (z_0)}\big] \le C \exp \left( - \displaystyle\frac{(\log N)^2}{C} \right)
\end{flalign} 
for large enough $N$.
\end{prop}

\bp 

First, recall that, since $m_{{\alpha}} (z)$ is the Stieltjes transform of a probability measure $\mu_{\alpha}$ whose density is bounded and whose support is $\mathbb{R}$ (see Proposition 1.1 of \cite{belinschi2009spectral}), for any $\mathfrak{B} > 0$ there exists a small constant $\varepsilon = \varepsilon (u, v, \mathfrak{B}) > 0$ such that 
\begin{flalign}
\label{mmuzbounded}
\varepsilon < \displaystyle\sup_{w \in \mathcal{D}_{[u, v], \varpi, \mathfrak{B}}} \Im m_{{\alpha}} (w	) < \displaystyle\frac{1}{\varepsilon}.
\end{flalign}

\noindent Now, we claim that 
\begin{flalign}
\label{mnestimatemmualpha}
\mathbb{P} \bigg[ \one_{\Omega (z_0)} \Big|  \Im m_N (z)  - \Im m_{{\alpha}} (z) \Big| > \displaystyle\frac{2}{N^{\varkappa }} \bigg] \le 2 \exp \big( - (\log N)^2 \big).
\end{flalign}

 Indeed, \eqref{mnestimatemmualpha} follows from the fact that $\one_{\Omega (z_0)} \big| m_N (z_0) - \Im m_{{\alpha}} (z_0) \big| < N^{- \varkappa }$, the fact that $\big| m_N (z) - m_N (z_0) \big| < \frac{2}{N}$ since $|z - z_0| < \frac{1}{N^5}$ (from \eqref{rijrijestimate}), the fact  that $\big|  m_{{\alpha}} (z) - m_{{\alpha}} (z_0) \big| \le \frac{2}{N}$ (since $m_{{\alpha}}$ is the Stieltjes transform of the probability measure $\mu_{\alpha}$),  and the second estimate in \eqref{mnexpectationnearmnimaginary}.

In particular, \eqref{mmuzbounded} and \eqref{mnestimatemmualpha} imply that
\begin{flalign} 
\label{mnmuestimatelowerupper}
\mathbb{P} \left[ 2 \varepsilon \one_{\Omega (z_0)} \le \one_{\Omega (z_0)} \Im m_N (z) < \displaystyle\frac{1}{2 \varepsilon} \right] \ge 1  - 2 \exp \big( - (\log N)^2 \big)
\end{flalign}
	
Next, as in the proof of \Cref{largezc}, \Cref{mapping} implies the existence of a large constant $E_0 = E_0 (\alpha) > 0$ such that if $z \in \mathbb{H}$ satisfies $|z| > E_0$ then $\max \{ c_{\varphi}, c_{\psi} \} < \frac{1}{2}$. Recalling $X(z)$ and $Y(z)$ from \eqref{xy12} and following the proof of \Cref{largezc}, we deduce that 
\begin{flalign}
\label{mnmmuze0}
\begin{aligned}
\Big| \E \big[ m_N(z) \big] - m_\alpha(z) \Big| & \le \Big| X(z)  - \psi_{\alpha,z} \big( Y(z) \big) \Big| + \Big| \psi_{\alpha, z} \big( Y(z) \big) - \psi_{\alpha, z} \big( y(z) \big) \Big|, \\
\big| Y(z) - y(z) \big| & \le 2 \big| Y(z) - \varphi_{\alpha,z} \big( Y(z) \big) \big|.
\end{aligned} 
\end{flalign}

Observe that the hypotheses of  \Cref{approxfixedpoint} are satisfied for $z$; this is because  \eqref{mnmuestimatelowerupper} and \Cref{lambdaestimate}, together with the trivial bound $| m_N (z) | \le \eta^{-1}$ on the set $\Omega(z_0)^c$, imply that $\mathbb{P} [ \Lambda (z)^c ] < \mathbb{P} \big[ \Omega (z_0) \big] + \frac{1}{N^{20}} < \frac{1}{N^{10}}$ for large enough $N$. Then we can use  \Cref{approxfixedpoint} to estimate the terms appearing on the right side of \eqref{mnmmuze0}. 
Since $c_{\psi} < \frac{1}{2}$, this yields the existence of a large constant $C = C(\alpha, b, u, v, \varkappa) > 0$ such that 
\begin{flalign}
\label{e01yyestimate}
\begin{aligned}	
\one_{\Omega (z_0)} \big| Y(z) - y(z) \big| & \le C (\log N)^{100 / (\alpha - 1)} \left( \displaystyle\frac{1}{(N \eta^2)^{\alpha / 8}} + \displaystyle\frac{1}{N^{2 \theta}} \right) \le \displaystyle\frac{2 C(\log N)^{100 / (\alpha - 1)}}{N^{\varkappa}}, \\
\one_{\Omega (z_0)} \Big| \E \big[ m_N(z) \big] - m_\alpha(z) \Big| & \le C (\log N)^{100 / (\alpha - 1)} \left( \displaystyle\frac{1}{(N \eta^2)^{\alpha / 8}} + \displaystyle\frac{1}{N^{2 \theta}} \right) \le \displaystyle\frac{2 C(\log N)^{100 / (\alpha - 1)}}{N^{\varkappa}}.
\end{aligned}
\end{flalign}

\noindent Therefore, the first estimate in \eqref{mnexpectationnearmnimaginary} and the second estimate in \eqref{e01yyestimate} together imply that 
\begin{flalign}
\label{e0mnexpectationestimate}
\mathbb{P} \bigg[ \one_{\Omega (z_0)} \Big| m_N (z) - m_{{\alpha}} (z)  \Big| > \displaystyle\frac{1}{N^{\varkappa}} \bigg] \le 2 \exp \left( - \displaystyle\frac{(\log N)^2}{8} \right). 
\end{flalign}

\noindent Furthermore, observe that \eqref{gii1}, \eqref{mnmuestimatelowerupper}, and \Cref{lambdaestimate} together yield
\begin{flalign}
\label{e0rjjestimate}
\mathbb{P} \left[ \one_{\Omega (z_0)} \displaystyle\max_{1 \le j \le N} \big| R_{jj} (z) \big| > (\log N)^{30 / (\alpha - 1)} \right] < C \exp \left( - \displaystyle\frac{(\log N)^2}{C} \right).
\end{flalign}

\noindent Now \eqref{omegaestimatee0} follows from the first estimate in \eqref{e01yyestimate}, \eqref{e0mnexpectationestimate}, and \eqref{e0rjjestimate}.
\ep

\subsubsection{Bootstrap for Small Energies}

\label{SmallE}

Let $E_0$ and $\mathfrak{B}$ be as in \Cref{weaklaw}; in this section we establish the analog of that proposition when $|E| \le E_0 + 1$. To that end, let $\mathcal{S} = \mathcal{S}_{\alpha}$ denote the set of $x \in \mathbb{C}$ with $\Re x \in K$ and $\Im x \in [0, \mathfrak{B}]$ such that $\varphi_{\alpha, z}' (x) - 1 = 0$. Recall from either Lemma 6.2 of \cite{arous2008spectrum} or equation (3.17) of \cite{belinschi2009spectral} that if $z \ne 0$ there exists an entire function $g(x) = g_{\alpha} (x)$ such that $\varphi_{\alpha, z} (x) = C z^{-\alpha} g(x)$. Therefore, since $K$ is a compact interval that does not contain $0$, we have that $\mathcal{S}_{\alpha}$ is finite. 

Thus the implicit function theorem yields the existence of some integer $M = M(\alpha, K) > 0$ (corresponding to the order of the largest zero of $\varphi_{\alpha, z}' - 1$ in $\mathcal{S}_{\alpha}$), a small constant $c = c(\alpha, K, \mathfrak{B}) > 0$, and a large constant $C = C (\alpha, K, \mathfrak{B}) > 0$ such that the following holds. If $z \in \mathbb{H}$ satisfies $\Re z \in K$ and $\Im z \le \mathfrak{B}$, then for any $t > 0$ and $w \in \mathbb C$,
\begin{flalign} 
\label{ywestimate} 
\big| w  - y (z) \big| \le c \quad \text{and} \quad \big| w - \varphi_{\alpha, z}(w) \big|\le t \quad \text{together imply} \quad \big| w - y(z) \big| \le C  t^{1 / M}.
\end{flalign}

Now we can establish the following proposition that establishes \eqref{omegaz1z} when $|\Re z| \le E_0 + 1$.

\begin{prop} 
	
	\label{omegageneral}

Let $\varkappa = \varkappa (\alpha, \delta, K) = \frac{\delta}{20 M}$. For any compact interval $K = [u, v] \subset \mathbb{R}$ that does not contain $0$, there exists a large constant $C = C(\alpha, b, u, v, \varkappa) > 0$ such that the following holds. Suppose $E \in [u, v]$ and $z_0, z \in \mathcal{D}_{[u, v], \varpi, \mathfrak{B}}$ satisfy $\Re z_0 = E = \Re z$ and $\Im z_0  - \frac{1}{N^5} \le \Im z \le \Im z_0$. If $\mathbb{P} \big[ \Omega (z_0) \big] < \frac{1}{N^{20}}$, then 
\begin{flalign*} 
\mathbb{P} \big[ \textbf{\emph{1}}_{\Omega (z)} < \textbf{\emph{1}}_{\Omega (z_0)}\big] < C \exp \left( - \displaystyle\frac{(\log N)^2}{C} \right).
\end{flalign*}

\end{prop}

\begin{proof}

Since $\Re z = \Re z_0$, $\Im z_0 - \frac{1}{N^5} \le \Im z \le \Im z_0$, continuity estimates for $Y(z)$ and $y(z)$ (see, for instance, equation (39) of \cite{bordenave2013localization}) imply that $\big| Y(z) - Y(z_0) \big| + \big| y(z) - y(z_0) \big| \le \frac{1}{N}$. Therefore, since $\big| Y(z_0) - y (z_0) \big| \one_{\Omega (z_0)} \le N^{-\varkappa }$, it follows that $\big| Y(z) - y(z) \big| \one_{\Omega (z_0)} \le 2 N^{-\varkappa }$ for $N$ sufficiently large.

Thus \eqref{ywestimate} implies the existence of a large constant $C = C(\alpha, u, v) > 0$ such that 
\begin{flalign}
\label{yzyzgeneral}
\one_{\Omega (z_0)} \big| Y(z) - y(z) \big| \le C \Big| Y(z) - \varphi_{\alpha, z} \big( Y(z) \big) \Big|^{\frac{1}{M}} \one_{\Omega (z_0)}. 
\end{flalign}

 Following very similar reasoning as used to establish \eqref{mnmmuze0} in the proof of \Cref{weaklaw}, we have that 
\begin{flalign}
\label{mnmmuzgeneral}
\begin{aligned}
\one_{\Omega (z_0)} \Big| \E \big[ m_N(z) \big] - m_\alpha(z) \Big| & \le \bigg( \Big| X(z)  - \psi_{\alpha,z} \big( Y(z) \big) \Big| + \Big| \psi_{\alpha, z} \big( Y(z) \big) - \psi_{\alpha, z} \big( y(z) \big) \Big| \bigg) \one_{\Omega (z_0)} \\
& \le \bigg( \Big| X(z)  - \psi_{\alpha,z} \big( Y(z) \big) \Big| + c_{\psi} \big| Y(z) - y(z) \big| \bigg) \one_{\Omega (z_0)}.
\end{aligned} 
\end{flalign}

Having established the estimates \eqref{yzyzgeneral} and \eqref{mnmmuzgeneral}, the remainder of the proof of this proposition is very similar to that of \Cref{weaklaw} after \eqref{mnmmuze0} and is therefore omitted. 
\end{proof}

 Using the results above, we can now establish \Cref{localz1z0}.

\begin{proof}[Proof of \Cref{localz1z0}]
	
	The estimate \eqref{omegacestimate} follows from \Cref{largezc}. Furthermore, \Cref{weaklaw} establishes the existence of a large constant $E_0 = E_0 (\alpha)$ such that \eqref{omegaz1z} holds when $|\Im z| = |\Im z_0| > E_0$ and $\varkappa<c$. Then \Cref{omegageneral} implies \eqref{omegaz1z} when $|\Re z| = |\Re z_0| \le E_0 + 1$ and $\varkappa = \frac{\delta}{20 M}$. Together these yield \Cref{localz1z0}.
\end{proof}

	\section{Intermediate Local Law for Almost All \texorpdfstring{$\alpha \in (0, 2)$}{} at Small Energies}
	
	\label{LocalTail2} 
	
	In this section and in \Cref{EstimatesSmall} we establish \Cref{localsmallalpha3} (in fact the slightly more general \Cref{localsmallalpha} below), which provides a local law at sufficiently small energies for the removal matrix $\textbf{X}$ for almost all $\alpha \in (0, 2)$. In \Cref{Local02Estimate} we state the local law (given by \Cref{localsmallalpha} below) and an estimate (\Cref{sspzrrpzgamma}) that implies the local law. We will then establish \Cref{sspzrrpzgamma} in \Cref{LocalProof1}. 
		
	However, before doing this, let us recall some notation. In what follows we fix parameters $\alpha \in (1, 2)$ and $0 < b < \frac{1}{\alpha}$; we recall the removal matrix $\textbf{X}$ and its resolvent $\textbf{R}$ from \Cref{abremovedmatrix}; we recall $m_N (z) = N^{-1} \Tr \textbf{R}$; and we recall the domain $\mathcal{D}_{C; \delta}$ from \eqref{dcdelta2}. Furthermore, we denote by $\mathbb{K}$ the set of $z \in \mathbb{C}$ with $\Re z > 0$, and we set $\mathbb{K}^+ = \overline{\mathbb{K} \cap \mathbb{H}}$ to be the closure of the positive quadrant of the complex plane. We also let $\mathbb{S}^1$ be the unit circle, consisting of all $z \in \mathbb{C}$ with $|z| = 1$, and we define the closure $\mathbb{S}_+^1 = \overline{\mathbb{K}^+ \cap \mathbb{S}}$.

	\subsection{An Estimate for the Intermediate Local Law} 
	
	\label{Local02Estimate}

	In this section we state the local law for $\textbf{X}$ on scales $N^{\delta - 1 / 2}$ (\Cref{localsmallalpha} below) and an estimate (\Cref{sspzrrpzgamma}) that implies it; this will be done in \Cref{LocalTheorem}. However, we will first define a certain inner product and metric in \Cref{MetricInner} that will be required to define a family of fixed point equations in \Cref{Fixed02}.

	\subsubsection{Inner Product and Metric}
	
	\label{MetricInner}
	
	In order to establish a convergence result for $m_N (z)$ (which is approximately equal to $\mathbb{E} [R_{ii}]$), we in fact must understand the convergence of more general expectations, including the fractional moments $\mathbb{E} \big[ (- \mathrm{i} R_{jj})^p \big]$, the absolute moments $\mathbb{E} \big[ |R_{jj}|^p \big]$, and the imaginary moments $\mathbb{E} \big[ | \Im R_{jj}|^p \big]$. To facilitate this, we define for any $u, v \in \mathbb{C}$, the inner product 
	\begin{flalign*}
	(u \b| v) = u \Re v + \overline{u} \Im v = \Re u \big( \Re v + \Im v \big) + \mathrm{i} \Im u \big( \Re v - \Im v \big). 
	\end{flalign*}
	
	\noindent In particular, for any $u, v \in \mathbb{C}$, we have that   
	 \begin{flalign}
	 \label{uinner}
	 ( u \b| 1) = u, \qquad  ( - \mathrm{i} u \b| e^{\pi \mathrm{i} / 4}) = \Im u \sqrt{2}, \qquad \Big| \big( u \b| v \big) \Big| \le 2 |u| |v|.
	 \end{flalign}

	 We will attempt to simultaneously understand the quantities $A_z (u) = \mathbb{E} \Big[ \big( (-\mathrm{i} R_{ii} )^{\alpha / 2} \b| u \big) \Big]$ for all $u \in \mathbb{K}^+$. Our reason for this (as opposed to only considering the cases $u = 1$ and $u = e^{\pi \mathrm{i} / 4}$) is that the absolute moments $\mathbb{E} \big[ |R_{jj}|^p \big]$ will be expressed as an integral of a function of $A_z (u)$ over $u$ (see the definitions \eqref{ipjp} of $J_p$ and $r_{p, z}$ and also the second estimate in \eqref{spzrpz} below); this was implemented in \cite{bordenave2017delocalization}. 

	To explain this fixed point equation further, we require a metric space of functions. To that end, for any $w \in \mathbb{C}$, we let $\mathcal{H}_w$ denote the space of $\mathcal{C}^1$ functions $g: \mathbb{K}^+ \rightarrow \mathbb{C}$ such that $g (\lambda u) = \lambda^w g(u)$ for each $\lambda \in \mathbb{R}_{\ge 0}$. Following equation (10) of \cite{bordenave2017delocalization}, we define for any $r \in [0, 1)$ a norm on $\mathcal{H}_r$ by 
	\begin{flalign*}
	\| g \|_{\infty} =  \sup_{u \in \mathbb{S}_+^1} \big| g(u) \big|, \qquad \| g \|_r = \| u \|_{\infty} + \sup_{u \in \mathbb{S}_+^1} \sqrt{\big| (\mathrm{i} \b| u)^r \partial_1 g (u) \big|^2 + \big| ( \mathrm{i} \b| u )^r \partial_2 g (u) \big|^2 },
	\end{flalign*}
	
	\noindent where $\partial_1 g (x + \mathrm{i} y) = \partial_x g (x + \mathrm{i} y)$ and $\partial_2 g (x + \mathrm{i} y) = \partial_y g (x + \mathrm{i} y)$. Observe in particular that 
	\begin{flalign}
	\label{ffnorm}
	\displaystyle\sup_{u \in \mathbb{S}_+^1} \big| g(u) \big| \le \big\| g \|_r, \qquad \text{for any $r > 0$.} 
	\end{flalign} 

	We let $\mathcal{H}_{w, r}$ be the completion of $\mathcal{H}_w$ with respect to the $\| g \|_r$ norm. Further define for any $\delta > 0$ the subset $\mathcal{H}_{w, r}^{\delta} \subset \mathcal{H}_{w, r}$ consisting of all $g \in \mathcal{H}_{w, r}$ such that $\Re g(u) > \delta$ for all $u \in \mathbb{S}_+^1$, and define $\mathcal{H}_{w, r}^{0} = \bigcup_{\delta > 0} \mathcal{H}_{w, r}^{\delta}$. Further abbreviate $\mathcal{H}_w^{\delta} = \mathcal{H}_{w, 0}^{\delta}$.

	The following stability lemma, which appears as Lemma 5.2 of \cite{bordenave2017delocalization}, will be useful to us.

	\begin{lem}[{\cite[Lemma 5.2]{bordenave2017delocalization}}]
		
		\label{normnear}
		
		Assume that $r \in (0, 1)$ and $u \in \mathbb{S}_+^1$. Let $x_1, x_2 \in \mathbb{K}^+$, and let $a \in (0, 1)$ be such that $|x_1|, |x_2| \le a^{-1}$. Set $F_k (u) = \big( x_k \b| u \big)^r$ for each $k \in \{ 1, 2 \}$. Then, there exists a constant $C = C(r) > 0$ such that for any $s \in (0, r)$, we have that 
		\begin{flalign}
			\label{f1f2near}
			\| F_k \|_{1 - r + s} \le C |x_k|^r, \quad \text{for any $k \in \{ 1, 2 \}$}, \qquad \| F_1 - F_2 \|_{1 - r + s} \le C a^{-r} \big( |x_1 - x_2|^r + a^s |x_1 - x_2|^s \big). 
		\end{flalign}
		
		\noindent If we further assume that $\Re x_1, \Re x_2 \ge t$ for some $t > 0$, and we set $G_k (u) = \big( x_k^{-1} \b| u \big)^r$ for each $k \in \{ 1, 2 \}$, then there exists a constant $C = C(r) > 0$ such that 
		\begin{flalign}
			\label{g1g2near}
			\| G_1 - G_2 \|_{1 - r + s} \le C t^{r - 2} a^{2r - 1} |x_1 - x_2|. 
		\end{flalign}
		
	\end{lem}

	\subsubsection{Equations for \texorpdfstring{$m$}{}}
	
	\label{Fixed02}
	
	Following Section 3.2 of \cite{bordenave2017delocalization} (or Section 5.1 of \cite{bordenave2013localization}), define for any complex numbers $u \in \mathbb{S}_+^1$ and $h \in \overline{\mathbb{K}}$, and any function $g \in \mathcal{H}_{\alpha / 2}$, the function 	
	\begin{flalign*}
	F_{h, g} (u) & = \displaystyle\int_0^{\pi / 2} \Bigg( \displaystyle\int_0^{\infty} \bigg( \displaystyle\int_0^{\infty} \Big( e^{ -r^{\alpha / 2} g (e^{\mathrm{i} \theta}) - (rh \b| e^{\mathrm{i} \theta})} - e^{ -r^{\alpha / 2} g (e^{\mathrm{i} \theta} + uy) - (yrh \b| u) - (rh \b| e^{\mathrm{i} \theta})}  \Big) r^{\alpha / 2 - 1} dr \bigg)  \\
	& \qquad \qquad \qquad \qquad \quad \times y^{- \alpha / 2 - 1} dy \Bigg) (\sin 2 \theta)^{\alpha / 2 - 1} d \theta.	
	\end{flalign*}	
	
	It was shown as Lemma 4.1 of \cite{bordenave2017delocalization} that $F_{h, g} \in \mathcal{H}_{\alpha / 2, r}$ if $g \in \mathcal{H}_{\alpha / 2, r}^0$, and also that it is in the closure $\overline{H}_{\alpha / 2, r}^0$ for any $g \in \overline{\mathcal{H}}_{\alpha / 2, r}^0$ if $\Re h > 0$. As in equation (13) of \cite{bordenave2017delocalization}, define the function 
	\begin{flalign}
	\label{functionf}
	\Upsilon_f (u) = \Upsilon_{z, f} (u) = c_{\alpha} F_{-\mathrm{i} z, f} (\widetilde{u}), \quad \text{where $c = c_{\alpha} = \frac{\alpha}{2^{\alpha / 2} \Gamma (\alpha / 2)^2}$ and $\widetilde{u} = \mathrm{i} \overline{u}$.}
	\end{flalign}
	
	\noindent Observe that $\big( \overline{u}, v \big) = \overline{u} \Re v + u \Im v = \big( u, \widetilde{v} \big)$. Now, for any $u \in \mathbb{C}$, define 
	\begin{flalign}
	\label{expectationgamma}
	\vartheta_z (u) = \Gamma \left( 1 - \frac{\alpha}{2} \right) (- \mathrm{i} R_{jj} \b| u )^{\alpha / 2}	, \qquad \gamma_z (u) = \mathbb{E} \big[ \vartheta_z (u) \big] = \Gamma \left( 1 - \frac{\alpha}{2} \right) \mathbb{E} \big[ (- \mathrm{i} R_{jj} \b| u )^{\alpha / 2} \big],
	\end{flalign}
	
	\noindent for any $j \in [1, N]$; observe that $\gamma_z (u)$ does not depend on $j$ due to the fact that the entries of $\textbf{X}$ are identically distributed.
	
	Furthermore, for any $p > 0$ and $f \in \mathcal{H}_{\alpha / 2}$, define $I_p, J_p, r_{p, z} (f) \in \mathbb{C}$ and $s_{p, z}: \mathbb{K} \rightarrow \mathbb{C}$ by
	\begin{flalign}
	\label{ipjp}
	\begin{aligned}
	& I_p = I_p (z) = \mathbb{E} \big[ (-\mathrm{i} R_{jj} )^p \big], \quad s_{p, z} (x) = \displaystyle\frac{1}{\Gamma (p)} \displaystyle\int_0^{\infty} y^{p - 1} e^{- \mathrm{i} y z - x y^{\alpha / 2} } dy,  \\
	& J_p = J_p (z) =	 \mathbb{E} \big[ |\mathrm{i} R_{jj} |^p \big], \quad r_{p, z} (f) = \displaystyle\frac{2^{1 - p / 2}}{\Gamma (p / 2)^2} \displaystyle\int_0^{\pi / 2} \displaystyle\int_0^{\infty} y^{p - 1} e^{(\mathrm{i} yz \b| e^{\mathrm{i} \theta} ) - y^{\alpha / 2} f (e^{\mathrm{i} \theta})} (\sin 2 \theta)^{p / 2 - 1} dy d \theta,
	\end{aligned}
	\end{flalign} 
	
	\noindent for any $x \in \mathbb{K}$. The convergence of these integrals can quickly be deduced from the fact that $\Re (\mathrm{i} z) < 0$. 
	
	We now state four lemmas that can be found in \cite{bordenave2017delocalization}. The first two provide existence, stability, and estimates for the solution to a certain fixed point equation, while the latter two provide bounds and stability estimates for the functions $F$, $s_{p, z}$, $r_{p, z}$, and $\Upsilon$.

			\begin{lem}[{\cite[Proposition 3.3]{bordenave2017delocalization}} ,{\cite[Lemma 4.3]{bordenave2011spectrum}}] 
		
		\label{estimategzfomega} 
		
		There exists a countable subset $\mathcal{A} \subset (0, 2)$ with no accumulation points on $(0, 2)$ such that, for any $r \in (0, 1]$ and $\alpha \in (0, 2) \setminus \mathcal{A}$, there exists a constant $c = c(\alpha, r) > 0$ with the following property. 
		
		There exists a unique function $\Omega_0 \in \mathcal{H}_{\alpha / 2}$ such that $\Omega_0 = \Upsilon_{0; \Omega_0}$. Additionally, if $\Im z >0$ and $|z| \le c$, then there is a unique function $f = \Omega_z \in \mathcal{H}_{\alpha / 2, r}$ that solves $f = \Upsilon_{z; f}$ with $\| f - \Omega_0 \|_r \le c$. Moreover, this function satisfies $\Omega_z (e^{\pi \mathrm{i} / 4}) \ge c$ and, for any $p > 0$, there exists a constant $C = C(\alpha, p) > 0$ such that $r_{p, z} (\Omega_z) \le C$.

	\end{lem}

\begin{lem}[{\cite[Proposition 3.4]{bordenave2017delocalization}}]
	
	\label{gzfestimate} 
	
	Adopt the notation of \Cref{estimategzfomega}. After decreasing $c$ if necessary, there exists a constant $C > 0$ such that the following holds. If $\Im z>0$, $|z| \le c$, and $\| f - \Omega_z \|_r \le c$, then
	\begin{flalign*}
	\| f - \Omega_z \|_r \le C \big\| f - \Upsilon_{z; f} \big\|_r. 
	\end{flalign*} 
	
\end{lem}

The following stability properties of $F_h$ and $\Upsilon$ will be useful to us later. 

\begin{lem}[{\cite[Lemma 4.1]{bordenave2017delocalization}}]
	
	\label{festimate1}
	
	Let $r \in (0, 1)$ and $p > 0$. There exists a constant $C = C(\alpha, p, r) > 0$ such that, for any $g \in \overline{\mathcal{H}}_{\alpha / 2, r}^0$ and $h \in \mathbb{K}$, we have that 
	\begin{flalign}
	\label{festimate11}
	\begin{aligned}
	 \| & F_h (g) \|_r \le C (\Re h)^{- \alpha / 2} + C \| g \|_r (\Re h)^{-\alpha / 2}, \\
	 \big|  r_{p, \mathrm{i} h} & (g) \big| \le C (\Re h)^{-p}, \qquad \big| s_{p, \mathrm{i} h} \big( g (1) \big| \le C (\Re h)^{-p}.
	\end{aligned} 
	\end{flalign}
\end{lem}

\begin{lem}[{\cite[Lemma 4.3]{bordenave2017delocalization}}]
	
	\label{functionalnearfg}
	
	For any fixed $a, r > 0$, there exists a constant $C = C(\alpha, a, r)$ such that for any $f, g \in \mathcal{H}_{\alpha / 2, r}^a$ and $z \in \mathbb{C}$, we have that
	\begin{flalign}
	\label{functionalnearfg1}
	\| \Upsilon_f - \Upsilon_g \|_r \le C \| f - g \|_r + \| f - g \|_{\infty} \big( \|f \|_r + \|g \|_r \big).
	\end{flalign}
	
	\noindent Furthermore, for any $p > 0$ there exists a constant $C' = C' (\alpha, a, r, p)$ such that for any $f, g \in \mathcal{H}_{\alpha / 2, r}^a$ and any $z \in \mathbb{C}$ and $x, y \in \mathbb{K}$ with $\Re x, \Re y \ge a$, we have that
	\begin{flalign}
	\label{functionalnearfg2} 
	\big| r_{p, z} (f) - r_{p, z} (g) \big| \le C' \| f - g \|_{\infty}, \qquad \big| s_{p, z} (x) - s_{p, z} (y) \big| \le C' |x - y|. 
	\end{flalign} 
\end{lem}

\subsubsection{An Intermediate Local Law for \texorpdfstring{$\textbf{\emph{X}}$}{}}

\label{LocalTheorem}

The following theorem provides a local law for $\textbf{X}$. 

		\begin{thm}
		
		\label{localsmallalpha}
		
		There exists a countable set $\mathcal{A} \subset (0, 2)$, with no accumulation points in $(0, 2)$, such that the following holds. Fix $\alpha \in (0, 2) \setminus \mathcal{A}$ and $0 < b < \frac{1}{\alpha}$. Denote $\theta = \frac{(b - 1 / \alpha) (2 - \alpha)}{20}$ and fix some $\delta \in (0, \theta)$ with $\delta < \frac{1}{2}$. Then, there exists a constant $C = C(\alpha, b, \delta, p) > 0$ such that 
		\begin{flalign}
		\label{mnzs1zomega}
		\begin{aligned} 
		& \mathbb{P} \left[ \displaystyle\sup_{z \in \mathcal{D}_{C, \delta}} \Big| m_N (z) - \mathrm{i} s_{1, z} \big(\Omega_z (1) \big) \Big| > \displaystyle\frac{1}{N^{\alpha \delta / 8}} \right] < C \exp \left( - \displaystyle\frac{(\log N)^2}{C} \right),
		\end{aligned}
		\end{flalign}
		
		\noindent where we recall the definition of $\Omega_z$ from \Cref{estimategzfomega}. Furthermore, we have that 
		\begin{flalign}
		\label{spzrpz2}
		\begin{aligned}
		& \displaystyle\sup_{u \in \mathbb{S}_+^1} \big| \gamma_z (u) - \Omega_z (u) \big| \le \displaystyle\frac{C}{N^{\alpha \delta / 8}}, \qquad \big| J_2 - r_{2, z} (\Omega_z) \big| \le \displaystyle\frac{C}{N^{\alpha \delta / 8}},
		\end{aligned}	
		\end{flalign} 
		
		\noindent and  
		\begin{flalign}
		\label{rijestimate}
		\mathbb{P} \left[ \displaystyle\sup_{z \in \mathcal{D}_{C, \delta}} \displaystyle\max_{1 \le j \le N} \big| R_{jj} (z) \big| > (\log N)^C \right] < C \exp \left( -\displaystyle\frac{(\log N)^2}{C} \right). 
		\end{flalign}
		
	\end{thm}
	
	\begin{rem}	
		
		\label{s1zomegammualpha} 
		
		One can show that the fixed point equations \eqref{stieltjespsi} and \Cref{estimategzfomega} defining $m_{{\alpha}} (z)$ and $\Omega_z$, respectively, are equivalent when $u = 1$; this implies that $ \mathrm{i} s_{1, z} \big( \Omega_z (1) \big) = m_{{\alpha}} (z)$. 
		
	\end{rem}

	 \Cref{localsmallalpha} is a consequence of the following theorem (whose proof will be given in \Cref{LocalProof1} below), which is similar to Proposition 3.2 of \cite{bordenave2017delocalization} but two main differences. The first is that \Cref{sspzrrpzgamma} below establishes estimates on the scale $\eta \gg N^{-1 / 2}$, while the corresponding estimate in \cite{bordenave2017delocalization} was shown with $\eta \gg N^{-\alpha / (2 + \alpha)}$. The second is that we also establish estimates on each $\big| R_{jj} \big|$, which was not pursued in \cite{bordenave2017delocalization}. In fact, these bounds on the resolvent entries (which follow as consequences of \Cref{sestimate} and \Cref{stestimate} below) are partially what allow us to improve the scale from $\eta \gg N^{-\alpha / (2 + \alpha)}$ in \cite{bordenave2017delocalization} to $\eta \gg N^{ - 1 / 2}$ here.

\begin{thm}
	
	\label{sspzrrpzgamma} 
	
	Fix $\alpha \in (0, 2)$, $0 < b < \frac{1}{\alpha}$, $s \in \big( 0, \frac{\alpha}{2} \big)$, $p > 0$, $\varepsilon \in (0, 1]$, and a positive integer $N$. Define $\theta = \frac{( b - 1 / \alpha) (2 - \alpha)}{10}$, and suppose that $z = E + \mathrm{i} \eta \in \mathbb{H}$ with $E, \eta \in \mathbb{R}$. Assume that 
	\begin{flalign}
	\label{rr11alpha2r112} 
	\eta \ge N^{\varepsilon - s / \alpha}, \quad |z| < \frac{1}{\varepsilon}, \quad \mathbb{E} \big[ (\Im R_{11} )^{\alpha / 2} \big] \ge \varepsilon, \quad \mathbb{E} \big[ | R_{11} |^2 \big] \le \varepsilon^{-1}.
	\end{flalign}

	\noindent Then, there exists a constant $C = C(\alpha, \varepsilon, b, s, p) > 0$ such that
	\begin{flalign}
	\label{gammazgammaz}
	& \| \gamma_z - \Upsilon_{\gamma_z} \|_{1 - \alpha / 2 + s} \le C (\log N)^C \left( \displaystyle\frac{1}{(N \eta^2)^{\alpha / 8}} + \displaystyle\frac{1}{N^{\theta}} + \displaystyle\frac{1}{N^s \eta^{\alpha / 2}}  \right),
	\end{flalign}
	
	\noindent and
	\begin{flalign}
	\label{spzrpz}
	\begin{aligned}
	& \big| I_p - s_{p, z} \big( \gamma_z (1) \big) \big| \le C (\log N)^C \left( \displaystyle\frac{1}{(N \eta^2)^{\alpha / 8}} + \displaystyle\frac{1}{N^{\theta}} + \displaystyle\frac{1}{N^s \eta^{\alpha / 2}}  \right), \\
	& \big| J_p - r_{p, z} (\gamma_z) \big| \le C (\log N)^C  \left( \displaystyle\frac{1}{(N \eta^2)^{\alpha / 8}} + \displaystyle\frac{1}{N^{\theta}} + \displaystyle\frac{1}{N^s \eta^{\alpha / 2}} \right).
	\end{aligned}	
	\end{flalign}  
	
	\noindent Furthermore, 
	\begin{flalign}
	\label{gammazlower}
	\displaystyle\inf_{u \in \mathbb{S}_+^1} \Im \gamma_z (u) \ge \displaystyle\frac{1}{C},
	\end{flalign}

	\noindent and
	\begin{flalign}
	\label{rijestimatetheta}
	\mathbb{P} \left[ \displaystyle\max_{1 \le j \le N} \big| R_{jj} \big| > C (\log N)^C \right] < C \exp \left( -\displaystyle\frac{(\log N)^2}{C} \right). 
	\end{flalign}

\end{thm}

Given \Cref{estimategzfomega}, \Cref{gzfestimate}, and \Cref{sspzrrpzgamma}, the proof of \Cref{localsmallalpha} is very similar to the proof of Theorem 5.11 in Section 5.4 of \cite{bordenave2017delocalization} and is therefore omitted. However, let us briefly explain the idea of the proof, referring to \cite{bordenave2017delocalization} for the remaining details. 

To that end, after proving \Cref{localsmallalpha} in the case when $\eta = \frac{1}{C}$ is of order $1$, one first observes by \eqref{rijrijestimate} that it suffices to establish \Cref{localsmallalpha} for any individual $z$ on a certain lattice. In particular, for a constant $C > 0$, let $A = A (C) = \lfloor \frac{2 N^C}{C} \rfloor$ and $B = B (C) = \lfloor \frac{N^C - N^{C + \delta - 1 / 2}}{C} \rfloor$, and define $z_{jk} = z_{j, k} = \frac{j}{N^C} - \frac{1}{C} + \mathrm{i} \big( \frac{1}{C} - \frac{k}{N^C} \big)$ for each $0 \le j \le A$ and $0 \le k \le B$. 

If $C$ is sufficiently large, it suffices to verify \eqref{mnzs1zomega} and \eqref{rijestimate} for each $z_{jk}$. We will induct on $k$; the initial estimate states that they are true for $k = 0$. So let $M \in [1, B]$ be an integer, assume that the theorem holds for $k \le M - 1$, and let us establish it for $k = M$. To that end, we will apply \Cref{sspzrrpzgamma} with $s = \frac{\alpha - \alpha \delta}{2}$ and $\varepsilon \le \frac{\delta}{2}$.

To apply this theorem, we must verify \eqref{rr11alpha2r112}. The first estimate there holds since $\eta \ge N^{\delta - 1 / 2}$, and the second holds for sufficiently small $\varepsilon$ if $z \in \mathcal{D}_{C, \delta}$. The third follows from \eqref{rijrijestimate}, the second statement of \eqref{uinner}, the first estimate in \eqref{spzrpz2} (applied with $z = z_{j, M - 1}$ on the previous scale), and the lower bound on $\Omega_z (e^{\pi \mathrm{i} / 4})$ provided by \Cref{estimategzfomega}. The fourth estimate in \eqref{rr11alpha2r112} similarly follows from \eqref{rijrijestimate}, the second estimate in \eqref{spzrpz2} (applied with $z = z_{j, M - 1}$ on the previous scale), and the upper bound on $r_{2, z} (\Omega_z)$ given by \Cref{estimategzfomega}.

Thus, applying \Cref{sspzrrpzgamma} yields that \eqref{gammazgammaz}, \eqref{spzrpz}, \eqref{gammazlower}, and \eqref{rijestimatetheta} all hold for $z = z_{jM}$; the last estimate implies \eqref{rijestimate}. Furthermore, \eqref{gammazgammaz}, \eqref{ffnorm}, \eqref{rijrijestimate}, the estimate \eqref{mnzs1zomega} (applied with $z = z_{j, M - 1}$ on the previous scale), and \Cref{gzfestimate} together imply the first estimate in \eqref{spzrpz2} for $z = z_{jM}$. Now \eqref{mnzs1zomega} for $z = z_{jM}$ follows from the first estimate in \eqref{spzrpz} (applied with $p = 1$), the first estimate in \eqref{spzrpz2}, the first identity in \eqref{uinner}, \eqref{gammazlower}, the second estimate in \eqref{functionalnearfg2}, and the first estimate in \eqref{mnexpectationnearmnimaginary}. The second estimate in \eqref{spzrpz2} for $z = z_{jM}$ follows from the second estimate in \eqref{spzrpz} (applied with $p = 2$), the first estimate in \eqref{spzrpz2}, \eqref{gammazlower}, and the first estimate in \eqref{functionalnearfg2}.

\subsection{Establishing \Cref{sspzrrpzgamma}}

\label{LocalProof1}

In this section we establish \Cref{sspzrrpzgamma} assuming \Cref{sestimate}, \Cref{stestimate}, and \Cref{zgamma}; the latter results will be proven later, in \Cref{EstimatesSmall}. 

Define 
\begin{flalign}
\label{sitiui}
S_i = \displaystyle\sum_{j \ne i} X_{ij}^2 R_{jj}^{(i)}, \quad \text{and} \quad  T_i = X_{ii} - U_i, \quad \text{where} \quad U_i = \displaystyle\sum_{\substack{j, k \ne i \\ j \ne k}} X_{ij} R_{jk}^{(i)} X_{ki},
\end{flalign}

\noindent and observe that $R_{ii}$ can be expressed in terms of $T_i$, $z$, and $S_i$ through \eqref{gii1}. 

We begin in \Cref{RemoveT02} by ``removing $T_i$'' from the equations defining $R_{ii}$ by approximating functions of the $R_{ii}$ by analogous functions of $\big( - z - S_i \big)^{-1}$. Next, in \Cref{ReplaceXH} we analyze the error in replacing all of the removal entries $X_{ij}$ in the expression defining $S_i$ with the original $\alpha$-stable entries $Z_{ij}$ (recall \Cref{abremovedmatrix}). This will be useful for deriving approximate fixed point equations in \Cref{FixedApproximate}, which we will use to conclude the proof of \Cref{sspzrrpzgamma} in \Cref{FixedConvergence}.

	\subsubsection{Removing \texorpdfstring{$T_i$}{}}

	\label{RemoveT02}
	
	Denoting 
	\begin{flalign}
	\label{expectationomega}
	\omega_z (u) = \big( (- \mathrm{i} z - \mathrm{i} S_i)^{-1} \b| u \big)^{\alpha / 2}, \qquad \varpi_z (u) = \mathbb{E} \Big[ \big( (- \mathrm{i} z - \mathrm{i} S_i)^{-1} \b| u \big)^{\alpha / 2} \Big],
	\end{flalign}

	\noindent we would like to show that $\gamma_z \approx \varpi_z$ and that other similar approximations hold; see \Cref{expectationsst} below. Such estimate which would follow from \Cref{tiestimateprobability2} if one could show that $\Im (S_i - T_i)$ and $\Im S_i$ could be bounded from below with overwhelming probability. The following two propositions, which will be proven in \Cref{EstimateS02} and \Cref{STEstimate02}, establish the latter statement.

	\begin{prop}
		
	\label{sestimate}
		
	Adopt the notation of \Cref{sspzrrpzgamma}. There exists a large constant $C = C(\alpha, \varepsilon, b) > 1$ such that 
	\begin{flalign}
	\label{siprobability} 
	\mathbb{P} \left[ \Im S_i < \displaystyle\frac{1}{C (\log N)^C} \right] < C \exp \left( - \displaystyle\frac{(\log N)^2}{C} \right).
	\end{flalign}
	\end{prop}

	\begin{prop}
	
	\label{stestimate} 
	
	Adopt the notation of \Cref{sspzrrpzgamma}. There exists a large constant $C = C(\alpha, \varepsilon, b) > 1$ such that 
	\begin{flalign} 
	\label{tisiprobability}  
	\mathbb{P} \left[ \Im (S_i - T_i) < \displaystyle\frac{1}{C (\log N)^C} \right] < C \exp \left( - \displaystyle\frac{(\log N)^2}{C} \right). 
	\end{flalign} 
	
	\noindent In particular, we have that 
	\begin{flalign}
	\label{riiestimateprobability}
	\mathbb{P} \left[ \displaystyle\max_{1 \le j \le  N} |R_{jj}| < C (\log N)^{C}\right] < C \exp \left( - \displaystyle\frac{ (\log N)^2}{2 C} \right). 
	\end{flalign}
	
	\end{prop}

	The following proposition is a consequence of \Cref{tiestimateprobability2}, \Cref{sestimate}, and \Cref{stestimate}; its proof will be similar to that of \Cref{diagonal}.

	\begin{prop}
		
	\label{expectationsst}
	
	Adopt the notation of \Cref{sspzrrpzgamma}. Then, there exists a large constant $C = C(\alpha, \varepsilon, b, s, p) > 0$ such that 
		\begin{flalign}
	\label{riizsipestimate1}
	\mathbb{E} \bigg[ \Big| & |R_{ii}|^p - \big| (-z - S_i)^{-1} \big|^p \Big| \bigg] \le \displaystyle\frac{C (\log N)^C }{(N \eta^2)^{\alpha / 8}}, \quad \mathbb{E} \bigg[ \big| (-\mathrm{i} R_{ii})^p - (- \mathrm{i} z - \mathrm{i} S_i)^{-p}  \big| \bigg] \le \displaystyle\frac{C (\log N)^C }{(N \eta^2)^{\alpha / 8}},
	\end{flalign}
	
	\noindent and 
	\begin{flalign}	
	\label{expectationriizsiu}  
	\big\| \gamma_z - \varpi_z \big\|_{1 - \alpha / 2 + s} < \displaystyle\frac{C (\log N)^C}{(N \eta^2)^{\alpha / 8}},  
	\end{flalign}
	
	\noindent where $\gamma_z$ and $\varpi_z$ are defined in \eqref{expectationgamma} and \eqref{expectationomega}, respectively. 	
	\end{prop} 
	
	\begin{proof}
		
		Let us first establish the first estimate in \eqref{riizsipestimate1}. The proof of the second is entirely analogous and is therefore omitted. To that end, observe from \eqref{gii2} and \eqref{powerdiff} that, for any $v > 0$, we have that
		\begin{flalign}
		\label{risi1}
		\begin{aligned}
		\Big| |R_{ii}|^p - \big| (-z - S_i)^{-1} \big|^p \Big| & \le (p - 1) v \bigg( \Big| \displaystyle\frac{1}{\Im (S_i - T_i + z )}\bigg|^{p + 1} + \bigg| \displaystyle\frac{1}{\Im (z + S_i)} \Big|^{p + 1} \bigg) \one_{|T_i| < v} \\
		& \qquad + \bigg( \Big| \displaystyle\frac{1}{\Im (S_i - T_i + z )}\bigg|^p + \bigg| \displaystyle\frac{1}{\Im (z + S_i)} \Big|^p \bigg) \one_{|T_i| \ge v} .  
		\end{aligned}
		\end{flalign}
		
		We will use \Cref{tiestimateprobability2}, \Cref{sestimate}, and \Cref{stestimate} to bound the expectation of the right side of \eqref{risi1}. Let $C_1$, $C_2$, and $C_3$ denote the constants $C$ from \Cref{tiestimateprobability2}, \Cref{sestimate}, and \Cref{stestimate} respectively. Also let $E_1$ denote the event on which $\inf_{1 \le i \le N} \Im S_i < C_2^{-1} (\log N)^{-C_2}$, let $E_2$ denote the event on which $\inf_{1 \le i \le N} \Im (S_i - T_i) < C_3^{-1} (\log N)^{-C_3}$, and let $E = E_1 \cup E_2$. 
		
		Now, using the deterministic estimate \eqref{siui} and the fact that $\Im z = \eta \le N^{- 1 / 2}$ to estimate the expectation of the right side of \eqref{risi1} on $E$, and using \Cref{sestimate} and \Cref{stestimate} to estimate it off of $E$, yields 
		\begin{flalign}
		\label{riizsipestimate}
		\begin{aligned}
		\mathbb{E} \bigg[ \Big| & |R_{ii}|^p - \big| (-z - S_i)^{-1} \big|^p \Big| \bigg] \\
		& \le (p - 1) v \big( C_2^{p + 1} (\log N)^{(p + 1) C_2} + C_3^{p + 1} (\log N)^{(p + 1) C_3}  \big) \\ 
		& \quad + (p - 1) v \big( C_2^{p + 1} (\log N)^{(p + 1) C_2} + C_3^{p + 1} (\log N)^{(p + 1) C_3}  \big) \mathbb{P} \big[ |T_i| \ge v \big] \\
		& \quad  + \big( N^{p / 2} + (p - 1) v N^{(p + 1) / 2} \big) \Bigg( \exp \bigg( - \displaystyle\frac{(\log N)^2}{C_2} \bigg) + \exp \bigg( - \displaystyle\frac{(\log N)^2}{C_3} \bigg) \Bigg).  
		\end{aligned}
		\end{flalign}
		
		Setting $v = (N \eta^2)^{- 1 / 4}$ in \eqref{riizsipestimate} together with the estimate on $\mathbb{P} \big[ |T_i| \ge s \big]$ given by \eqref{tiprobability} (applied with $t = (N \eta^2)^{1 / 4}$) yields \eqref{riizsipestimate1}. 
		
		The proof of \eqref{expectationriizsiu} is similar, except we now use \Cref{normnear}. Recall the functions $\vartheta_z$ and $\omega_z$ from \eqref{expectationgamma} and \eqref{expectationomega}, respectively. Furthermore, recall the event $E$ from above, and let $F$ denote the complement of $E$. On $F$, we apply \eqref{g1g2near} with $x_1 = \mathrm{i} T_i - \mathrm{i} z - \mathrm{i} S_i$, $x_2 = - \mathrm{i} z - \mathrm{i} S_i$, $r = \alpha / 2$, and $t = a = (C_2 + C_3)^{-1} (\log N)^{-C_2 - C_3}$ to obtain that 
		\begin{flalign}
		\label{risi2e}
		\begin{aligned}
		\big\| \vartheta_z - \omega_z \big\|_{1 - \alpha / 2 + s} \one_F \one_{|T_i| \le v} & \le C (C_2 + C_3)^{3 - 3 \alpha / 2} (\log N)^{(3 - 3 \alpha / 2)(C_2 + C_3)} \left| \displaystyle\frac{1}{z + S_i - T_i} - \displaystyle\frac{1}{z + S_i} \right| \one_F \\
		&  \le  C (C_2 + C_3)^{5 - 3 \alpha / 2} (\log N)^{(5 - 3 \alpha / 2)(C_2 + C_3)} v \one_F. 
		\end{aligned}
		\end{flalign}
		
		\noindent Similarly, applying the first estimate in \eqref{f1f2near} with $x_1 = (\mathrm{i} T_i - \mathrm{i} z - \mathrm{i} S_i)^{-1}$, $x_2 = (- \mathrm{i} z - \mathrm{i} S_i)^{-1}$, $r = \alpha / 2$, and $a = (C_2 + C_3)^{-1} (\log N)^{-C_2 - C_3}$ yields the existence of a constant $C = C(\alpha) > 0$ such that	
		\begin{flalign}
		\label{risi3e}
		\begin{aligned}
		& \big\| \vartheta_z - \omega_z \big\|_{1 - \alpha / 2 + s} \one_F \one_{|T_i| > v} \\
		& \quad \le C \Big( \big| z + S_i - T_i \big|^{-1} + \big| z + S_i \big|^{-1} \Big) \one_F \one_{|T_i| > v}  \le  2 C (C_2 + C_3)^{\alpha / 2} (\log N)^{C_2 + C_3} \one_F \one_{|T_i| > v}. 
		\end{aligned}
		\end{flalign}
		
		\noindent Moreover, again using the first estimate in \eqref{f1f2near} with the same $x_1$, $x_2$, and $r$ as above, but now with $a = \eta \ge N^{-1 / 2}$, we obtain that 
		\begin{flalign}
		\label{risi4e}
		\begin{aligned}
		& \big\| \vartheta_z - \omega_z \big\|_{1 - \alpha / 2 + s} \one_E  \le  2 C N^{\alpha / 4} \one_E.
		\end{aligned}
		\end{flalign}
		
		\noindent Now \eqref{expectationriizsiu} follows similarly to \eqref{riizsipestimate1} as explained above. Set $v = (N \eta^2)^{-1 / 4}$ and sum \eqref{risi2e}, \eqref{risi3e}, and \eqref{risi4e}. Then apply  \eqref{tiprobability} (with $t = (N \eta^2)^{1 / 4}$), \Cref{sestimate}, and \Cref{stestimate}. Finally, use the facts that $\varpi_z = \mathbb{E} \big[ \omega_z (u) \big]$ and $\gamma_z (u) = \mathbb{E} \big[ \vartheta_z (u) \big]$. 
	\end{proof}

	\subsubsection{Replacing \texorpdfstring{$\textbf{\emph{X}}$}{}} 
	
	\label{ReplaceXH}	
	
	To facilitate the proof of \Cref{sspzrrpzgamma}, it will be useful to replace all of the $X_{ij}$ with $Z_{ij}$ (which we recall are coupled from \Cref{abremovedmatrix}). To that end, we define 
	\begin{flalign}
	\label{si} 
	\begin{aligned}
	\mathfrak{S}_i  = & \displaystyle\sum_{j \ne i} Z_{ij}^2 R_{jj}^{(i)}, \qquad \Psi_z (u) = \Gamma \left( 1 - \displaystyle\frac{\alpha}{2} \right) \big( (- \mathrm{i} z - \mathrm{i} \mathfrak{S}_i)^{-1} \b| u \big)^{\alpha / 2} , \\
	&   \psi_z (u) = \mathbb{E} [\Psi_z] = \Gamma \left( 1 - \displaystyle\frac{\alpha}{2} \right) \mathbb{E} \Big[ \big( (- \mathrm{i} z - \mathrm{i} \mathfrak{S}_i)^{-1} \b| u \big)^{\alpha / 2} \Big].
	\end{aligned}  
	\end{flalign}
	
	We now have the following lemma that compares $S_i$ and $\mathfrak{S}_i$. It is a quick consequence of \Cref{sdiff} in \Cref{Equations12Proof} and our assumption that $\mathbb{E} \big[ |R_{jj}| \big| \le \mathbb{E} \big[ |R_{jj}|^2 \big]^{1 / 2} < \varepsilon^{-1 / 2}$. 
	
	\begin{lem} 
		
		\label{siestimatesi2}
		
		Adopt the notation of \Cref{sspzrrpzgamma}. There exists a large constant $C = C(\alpha, b, \varepsilon) > 0$ such that 
		\begin{flalign}
		\label{probabilitysi}
		\mathbb{P} \big[ \big| \mathfrak{S}_i - S_i| > N^{-4 \theta} \big] < C N^{-4 \theta}. 
		\end{flalign}
		
	\end{lem}

	The proof of the following proposition, which lower bounds $\Im \mathfrak{S}_i$, is very similar to that of \Cref{sestimate} and is therefore omitted. 
	
	\begin{prop}
		
		\label{siestimate4} 
		
		Adopt the notation of \Cref{sspzrrpzgamma}. There exists a large constant $C = (\alpha, \varepsilon) > 0$ such that 
		\begin{flalign}
		\label{probabilitysi1}
		\mathbb{P} \left[ \Im \mathfrak{S}_i < \displaystyle\frac{1}{C (\log N)^C} \right] \le C \exp \left(- \displaystyle\frac{(\log N)^2}{C} \right). 
		\end{flalign}
		
	\end{prop}

	Given \Cref{sestimate}, \Cref{siestimatesi2}, and \Cref{siestimate4}, the proof of the following proposition is similar to that of \Cref{expectationsst} and is therefore omitted.
	
	\begin{prop}
		
		\label{siestimate2} 
		
			Adopt the notation of \Cref{sspzrrpzgamma}. Then, there exists a constant $C = C(\alpha, \varepsilon, b, s, p) > 0$ such that 
		\begin{flalign}
		\label{estimatesi3}
		\begin{aligned}
		\mathbb{E} \bigg[ \Big| \big| (-z - \mathfrak{S}_i)^{-1} \big|^p - \big| (-z - S_i)^{-1} \big|^p \Big| \bigg] & \le C (\log N)^C N^{-4 \theta}, \\
		 \mathbb{E} \bigg[ \big| (- \mathrm{i} z - \mathrm{i} \mathfrak{S}_i)^{-p} - (- \mathrm{i} z - \mathrm{i} S_i)^{-p}  \big| \bigg] & \le C (\log N)^C N^{-4 \theta},
		 \end{aligned}
		\end{flalign}
		
		\noindent and 
		\begin{flalign}
		\label{psiestimate1}  
		\big\| \psi_z - \varpi_z \big\|_{1 - \alpha / 2 + s} < C (\log N)^C N^{- 4 \theta},  
		\end{flalign}
		
		\noindent where $\varpi_z$ and $\psi_z$ are defined in \eqref{expectationomega} and \eqref{si}, respectively. 	
	\end{prop}

	\subsubsection{Approximate Fixed Point Equations} 
	
	\label{FixedApproximate}
	
	In this section we establish several approximate fixed point equations for $\psi_z$. To that end, we begin with the following lemma, which appears as Corollary 5.8 of \cite{bordenave2017delocalization}.

		\begin{lem}[{\cite[Corollary 5.8]{bordenave2017delocalization}}]
		
		\label{fsr}
		
		Fix $\sigma > 0$, $\alpha \in (0, 2)$, $p > 0$, and a positive integer $N$. Let $Z$ be a $(0, \sigma)$ $\alpha$-stable law, and let $h_1, h_2, \ldots , h_N$ be mutually independent, identically distributed random variables with laws given by $N^{-1 / \alpha} Z$. Suppose that $A_1, A_2, \ldots , A_N \in \mathbb{C}$ are complex numbers with nonnegative real part. Then, denoting 
		\begin{flalign*}
		& \mathcal{F} (u) = \Gamma \left( 1 - \frac{\alpha}{2} \right) \mathbb{E} \Bigg[ \bigg( \Big( \displaystyle\sum_{j = 1}^N h_j^2 A_j - \mathrm{i} z \Big)^{-1} \bigg| u \bigg)^{\alpha / 2} \Bigg], \\
		S_p = \mathbb{E} & \left[ \bigg(  \displaystyle\sum_{j = 1}^N h_j^2 A_j - \mathrm{i} z \bigg)^{-p} \right],  \qquad R_p = \mathbb{E} \left[ \bigg|  \displaystyle\sum_{j = 1}^N h_j^2 A_j - \mathrm{i} z \bigg|^{-p} \right],
		\end{flalign*}
		
		\noindent we have that $\mathcal{F} (u) = \mathbb{E} \big[ \Upsilon_{\mathfrak{Z}} \big]$, $S_p = \mathbb{E} \big[ s_{p, z} (\mathscr{Z}) \big]$, and $R_p = \mathbb{E} \big[ r_{p, z} (\mathfrak{Z}) \big]$, where $\Upsilon$ is given by \eqref{functionf} and
		\begin{flalign*}
		\mathfrak{Z} = \mathfrak{Z} (u) = \displaystyle\frac{2^{\alpha / 2} \sigma^{\alpha}}{N} \displaystyle\sum_{j = 1}^N \big( A_j \b| u 	 \big)^{\alpha / 2} |y_j|^{\alpha}, \qquad \mathscr{Z} = \mathfrak{Z} (1) = \displaystyle\frac{2^{\alpha / 2} \sigma^{\alpha}}{N} \displaystyle\sum_{j = 1}^N A_j^{\alpha / 2} |y_j|^{\alpha}, 
		\end{flalign*}
		
		\noindent where $(y_1, y_2, \ldots , y_N)$ is an $N$-dimensional centered Gaussian random variable whose covariance matrix is $\Id$. 
		
	\end{lem}

	Using \Cref{fsr}, we can express a number of quantities of interest in terms of the function $\mathfrak{Z}$ above.

	\begin{cor}
		
	\label{psiidentity} 
	
	Recalling the definition of $\Upsilon$ from \eqref{functionf} and $\Psi_z$ from \eqref{si}, we have that 
	\begin{flalign}
	\label{psizu2}
	\Psi_z (u) = \mathbb{E}_{\mathfrak{Y}} \big[ \Upsilon_{\mathfrak{Z}} \big], 	
	\end{flalign}
	
	\noindent where
	\begin{flalign}
	\label{z}
	\mathfrak{Z} = \mathfrak{Z} (u) = \displaystyle\frac{\Gamma \big( 1 - \frac{\alpha}{2} \big)}{N - 1} \displaystyle\sum_{j \ne i} \big( - \mathrm{i} R_{jj}^{(i)} \b| u \big)^{\alpha / 2} \displaystyle\frac{|y_j|^{\alpha}}{\mathbb{E} \big[ |y_j|^{\alpha} \big]},
	\end{flalign}
	
	\noindent where $\mathfrak{Y} = (y_j)_{j \ne i}$ is an $(N - 1)$-dimensional centered real Gaussian random variable with covariance matrix given by $\Id$. In \eqref{psizu2}, the expectation is with respect to $\mathfrak{Y}$. 
	
	Moreover, denoting $\mathscr{Z} = \mathfrak{Z} (1)$, we have that
	\begin{flalign}
	\label{momentsi} 
	\mathbb{E} \Big[ \big( -\mathrm{i} z - \mathrm{i} \mathfrak{S}_i \big)^{-p} \Big] = \mathbb{E}_{\mathfrak{Y}} \big[ s_{p, z} (\mathscr{Z}) \big], \qquad \mathbb{E} \Big[ \big| - z - \mathfrak{S}_i \big|^{-p} \Big] = \mathbb{E}_{\mathfrak{Y}} \big[ r_{p, z} (\mathfrak{Z}) \big].
	\end{flalign}
	
	\end{cor} 

	\begin{proof}
	The identity \eqref{psizu2} follows from the first statement of \Cref{fsr}, applied with $h_j = X_{ij}$ and $A_j = - \mathrm{i} R_{jj}^{(i)}$, and also the fact that 
	\begin{flalign}
	\label{sigmaexpectationidentity} 
	2^{\alpha / 2} \sigma^{\alpha} = \displaystyle\frac{2^{\alpha / 2 - 1} \pi}{\sin \big( \frac{\pi \alpha}{2} \big) \Gamma (\alpha)} = \displaystyle\frac{\pi}{\sin \big( \frac{\pi \alpha}{2} \big) \Gamma \big( \frac{\alpha}{2} \big) \mathbb{E} \big[ |y_j|^{\alpha} \big]} = \displaystyle\frac{\Gamma \big(1 - \frac{\alpha}{2} \big)}{\mathbb{E} \big[ |y_j|^{\alpha} \big]}.
	\end{flalign}
	
	\noindent To establish the first identity in \eqref{sigmaexpectationidentity} we used the definition \eqref{stable} of $\sigma$, and to establish the second and third we used \eqref{yjalphaidentity}. The proof of \eqref{momentsi} is entirely analogous, as a consequence of the second and third statements of \Cref{fsr}, as well as \eqref{sigmaexpectationidentity}. 
\end{proof}

The following proposition, which will be proven in \Cref{GammaNear}, states that $\mathfrak{Z}$ is approximately equal to $\gamma_z$. Thus, taking the expectation of both sides of \eqref{psizu2}, using the facts that $\psi_z = \mathbb{E} \big[ \Psi_z \big]$ (recall \eqref{si}) and that $\gamma_z$ is approximately equal to $\psi_z$ (recall \eqref{expectationriizsiu} and \eqref{psiestimate1}), \eqref{psizu2} yields an approximate fixed point equation for $\psi_z$.

\begin{prop}
	
	\label{zgamma}
	
	Adopt the notation of \Cref{sspzrrpzgamma}. There exists a constant $C = C(\alpha, \varepsilon, s) > 1$ such that
	\begin{flalign} 
	\label{zgammaprobability}
	\mathbb{P} \left[ \| \mathfrak{Z} - \gamma_z \|_{1 - \alpha / 2 + s} > \displaystyle\frac{C (\log N)^C}{N^{s / 2} \eta^{\alpha / 2}} \right] < C \exp \left( - \displaystyle\frac{(\log N)^2}{C} \right). 
	\end{flalign} 
\end{prop}

\subsubsection{Convergence to Fixed Points} 

\label{FixedConvergence}

In this section we establish \Cref{sspzrrpzgamma}. To that end, recall that \eqref{psizu2} can be viewed as a fixed point equation for $\psi_z$. In order to analyze this fixed point equation, we require the following lemma. 

\begin{lem}
	
	\label{zgammazestimates}
	
	Adopt the notation and assumptions of \Cref{sspzrrpzgamma}. There exists a constant $C = C(\alpha, \varepsilon, s) > 1$ such that 
	\begin{flalign} 
	\label{gammazgammazz}
	\begin{aligned}
	\| \gamma_z \|_{1 - \alpha / 2 + s} < C, \qquad \displaystyle\inf_{u \in \mathbb{S}_+^1} \Re \gamma_z (u) > \displaystyle\frac{1}{C}, \qquad \mathbb{P} \bigg[ & \displaystyle\inf_{u \in \mathbb{S}_+^1} \Re \mathfrak{Z} (u) < \displaystyle\frac{1}{C} \bigg]  < C \exp \left( - \displaystyle\frac{(\log N)^2}{C} \right).
	\end{aligned}
	\end{flalign}
\end{lem}

\begin{proof}
	
	In view of \eqref{ffnorm} and \Cref{zgamma}, it suffices to only establish the first two estimates in \eqref{gammazgammazz} on $\gamma_z$. Let us first establish the upper bound. To that end, observe that the first statement of \eqref{f1f2near} implies the existence of a constant $C = C(s)$ such that
	\begin{flalign} 
	\label{riestimate}
	\Big\| \big( - \mathrm{i} R_{ii} \b| u \big)^{\alpha / 2} \Big\|_{1 - \alpha / 2 + s}  \le C \big| R_{ii} \big|^{\alpha / 2}.
	\end{flalign} 
	
	Taking expectations in \eqref{riestimate}, using the definition \eqref{expectationgamma} of $\gamma_z$, and using the fact that $\mathbb{E} \big[ |R_{ii}|^2 \big] < \varepsilon^{-1}$, we deduce that 
	\begin{flalign*}
	\| \gamma_z \|_{1 - \alpha / 2 + s} & \le \Gamma \left( 1 - \frac{\alpha}{2} \right) \mathbb{E} \Big[ \big\| \big( - \mathrm{i} R_{ii} \b| u \big)^{\alpha / 2} \big\|_{1 - \alpha / 2 + s} \Big] \\
	& \le C \Gamma \left(1 - \frac{\alpha}{2} \right) \mathbb{E} \big[ |R_{ii}|^{\alpha / 2} \big] \le C \Gamma \left( 1 - \frac{\alpha}{2} \right) \mathbb{E} \big[ |R_{ii}|^2 \big]^{\alpha / 2} \le C \Gamma \left( 1 - \frac{\alpha}{2} \right) \varepsilon^{- \alpha / 2},
	\end{flalign*} 
	
	\noindent from which we deduce the first estimate in \eqref{gammazgammazz}. 
	
	Now let us verify the lower bound on $\Re \gamma_z$. In that direction, observe that for any $u \in \mathbb{S}_+^1$, we have that 
	\begin{flalign*}
	\Re \gamma_z (u) & = \Gamma \left( 1 - \frac{\alpha}{2} \right) \mathbb{E} \big[ \Re \big( - \mathrm{i} R_{ii} \b| u \big)^{\alpha / 2} \big] \\
	& \ge \Gamma \left( 1 - \frac{\alpha}{2} \right) \mathbb{E} \big[  \big( \Re ( - \mathrm{i} R_{jj}\b| u)  \big)^{\alpha / 2} \big] \ge \Gamma \left( 1 - \frac{\alpha}{2} \right) \mathbb{E} \big[ (\Im R_{jj})^{\alpha / 2} \big] \ge \Gamma \left( 1 - \frac{\alpha}{2} \right) \varepsilon. 
	\end{flalign*}
	
	\noindent The first identity above follows from the definition \eqref{expectationgamma} of $\gamma_z$; the second follows from the fact that $\Re a^r \ge (\Re a)^r$ for any $a \in \mathbb{K}$ and $r \in (0, 1)$ (see Lemma 5.10 of \cite{bordenave2017delocalization}); the third follows from the fact that $\Re \big( a \b| u \big) \ge \Re a$ for any $u \in \mathbb{S}_+^1$ and $a \in \mathbb{K}^+$; and the fourth follows from our assumed lower bound on $\mathbb{E} \big[ (\Im R_{jj})^{\alpha / 2} \big]$. 
\end{proof}

Now we can deduce the following consequence of \eqref{psizu2}. 

\begin{cor}
	
	\label{zgammazfunctional}
	
	Adopt the notation of \Cref{sspzrrpzgamma}. There exists a constant $C = C(\alpha, \varepsilon, s) > 0$ such that
	\begin{flalign}
	\label{psizgammazequation}
	\mathbb{P} \left[ \| \psi_z - \Upsilon_{\gamma_z} \|_{1 - \alpha / 2 + s} < \displaystyle\frac{C (\log N)^C}{N^{s / 2} \eta^{\alpha / 2}} \right] < C \exp \left( - \displaystyle\frac{ (\log N)^2}{C} \right). 
	\end{flalign} 
\end{cor}

\begin{proof} 
	
	Let us first show that $\Upsilon_{\gamma_z}$ is approximately equal to $\Upsilon_{\mathfrak{Z}}$ using \Cref{functionalnearfg}. To verify the conditions of that lemma, first observe that $\gamma_z, \mathfrak{Z} \in \mathcal{H}_{\alpha / 2}$ since the inner product $( x \b| y)$ is bilinear. Furthermore, let $C_1$ denote the constant $C$ from \Cref{zgamma}, and let $C_2$ denote the constant $C$ from \Cref{zgammazestimates}. Define the events 
	\begin{flalign}
	\label{e1e2}
	\begin{aligned}
	E_1 & = \left\{ \| \mathfrak{Z}  - \gamma_z \|_{1 - \alpha / 2 + s} \ge \displaystyle\frac{C_1 (\log N)^{C_1}}{N^{s / 2} \eta^{\alpha / 2}} \right\} \\
	E_2 & =  \left\{ \displaystyle\inf_{u \in \mathbb{S}_+^1} \Re \mathfrak{Z} (u) \le \displaystyle\frac{1}{C_2} \right\} \cup \left\{ \displaystyle\inf_{u \in \mathbb{S}_+^1} \Re \gamma_z (u) \le \displaystyle\frac{1}{C_2} \right\} \cup \big\{ \| \gamma_z \|_{1 - \alpha / 2 + s} \ge C_2 \big\}.
	\end{aligned}
	\end{flalign}
	
	Denoting $E = E_1 \cup E_2$, \Cref{zgamma} and \Cref{zgammazestimates} together imply that 
	\begin{flalign}
	\label{probabilitye1}
	\mathbb{P} [E] \le (C_1 + C_2) \exp \left( - \displaystyle\frac{(\log N)^2}{C_1 + C_2} \right).
	\end{flalign} 
	
	Therefore, denoting the complement of $E$ by $F$ and applying  \eqref{functionalnearfg1} and \eqref{ffnorm} yields a constant $C > 1$ (only dependent on $C_2$ and $s$) such that
	\begin{flalign}
	\label{efestimate}
	\begin{aligned}
	\one_F \| \Upsilon_{\mathfrak{Z}} - \Upsilon_{\gamma_z} \big\|_{1 - \alpha / 2 + s} & \le C \| \mathfrak{Z} - \gamma_z \|_{1-  \alpha / 2 + s} + \| \mathfrak{Z} - \gamma_z \|_{\infty} \big( \| \mathfrak{Z}  \|_{1 - \alpha / 2 + s} + \| \gamma_z \|_{1 - \alpha / 2 + s} \big) \\
	& \le \displaystyle\frac{C C_1 (\log N)^{C_1}}{N^{s / 2} \eta^{\alpha / 2}} \left( 1 + C_2 + \displaystyle\frac{C_1 C_2 (\log N)^{C_1}}{N^{s / 2} \eta^{\alpha / 2}} \right), 
	\end{aligned}
	\end{flalign}
	
	\noindent where we have used the fact that $\one_F \mathfrak{Z}$ and $\gamma_z$ are in $\mathcal{H}_{\alpha / 2, 1 - \alpha / 2 + s}^{1 / C_2}$. 
	
	The estimate \eqref{efestimate} bounds $\| \Upsilon_{\mathfrak{Z}} - \Upsilon_{\gamma_z} \|_{1 - \alpha / 2 + s}$ away from the event $E$; now let us bound it on $E$ through a deterministic estimate. To that end, observe that \eqref{gijeta} implies that $\gamma_z \in \mathcal{H}_{\alpha / 2, r}^{\eta / 2}$. Using the first bound in \Cref{festimate1} and the definition \eqref{functionf} of $\Upsilon$ in terms of $F$, we deduce that 
	\begin{flalign}
	\label{gammazestimatenormlarge}
	\| \Upsilon_{\gamma_z} \|_{1 - \alpha / 2 + s} \le C \eta^{- \alpha / 2} \big( 1 + \| \gamma_z \|_{1 - \alpha / 2 + s} \big),
	\end{flalign}
	
	\noindent after enlarging $C$ if necessary. Now, applying the first statement of \eqref{f1f2near} (with $x_1 = R_{ii}$, $r = \alpha / 2$, and $a = \eta$) and \eqref{gijeta}, we have that 
	\begin{flalign}
	\label{gammazlargeestimate}
	\| \gamma_z \|_{1 - \alpha / 2 + s} = C \Gamma \left( 1 - \displaystyle\frac{\alpha}{2} \right) \big| R_{jj} \big|^{\alpha / 2} \le C \Gamma \left( 1 - \displaystyle\frac{\alpha}{2} \right) \eta^{-\alpha / 2}. 
	\end{flalign}
	
	\noindent Inserting \eqref{gammazlargeestimate} into \eqref{gammazestimatenormlarge} yields 
	\begin{flalign}
	\label{gammazestimatenormlarge2}
	\| \Upsilon_{\gamma_z} \|_{1 - \alpha / 2 + s} \le 2 C^2 \eta^{- \alpha} \Gamma \left( 1 - \frac{\alpha}{2} \right). 
	\end{flalign}
	
	\noindent Furthermore, applying the definition \eqref{si} of $\Psi_z$, \eqref{gijeta}, and the first statement of \eqref{f1f2near} (now with $x_1 = -\mathrm{i} z - \mathrm{i} \mathfrak{S}$, $r = \alpha / 2$, and $a = \eta$) yields that 
	\begin{flalign}
	\label{estimatenormlarge}
	\|\Psi_z \|_{1 - \alpha / 2 + s} \le C \Gamma \left( 1 - \displaystyle\frac{\alpha}{2} \right) \eta^{- \alpha / 2}.
	\end{flalign}
	
	\noindent Combining \eqref{psizu2}, \eqref{efestimate}, \eqref{gammazestimatenormlarge2}, and \eqref{estimatenormlarge} yields 
	\begin{flalign}
	\label{psiequation1}
	\begin{aligned}
	\| \psi_z - \Upsilon_{\gamma_z} \|_{1 - \alpha / 2 + s} & \le \mathbb{E} \big[ \| \Psi_z - \Upsilon_{\gamma_z} \|_{1 - \alpha / 2 + s} \big] \\
	& \le \mathbb{E} \big[ \one_F \| \Psi_z - \Upsilon_{\gamma_z} \|_{1 - \alpha / 2 + s} \big] + \mathbb{E} \big[ \one_E \| \Psi_z \|_{1 - \alpha / 2 + s} \big] + \mathbb{E} \big[ \one_E \| \gamma_z \|_{1 - \alpha / 2 + s} \big] \\
	& \le \displaystyle\frac{C C_1^2 (C_2 + 2) (\log N)^{C_1}}{N^{s / 2} \eta^{\alpha / 2}} + C \Gamma \left( 1 - \frac{\alpha}{2} \right) \eta^{-\alpha} (2 C + 1) \mathbb{P} [E] .
	\end{aligned} 
	\end{flalign}
	
	\noindent Now \eqref{psizgammazequation} follows from \eqref{probabilitye1} and \eqref{psiequation1}. 
\end{proof}

Now we can establish \Cref{sspzrrpzgamma}. 

\begin{proof}[Proof of \Cref{sspzrrpzgamma}]
	The first estimate \eqref{gammazgammaz} follows from \eqref{expectationriizsiu}, \eqref{psiestimate1}, and \eqref{psizgammazequation}. Furthermore, the fourth estimate \eqref{gammazlower} follows from the second estimate in \eqref{gammazgammazz}; the fifth estimate \eqref{rijestimatetheta} follows from \eqref{riiestimateprobability}. 
	
	The proofs of the two estimates given in \eqref{spzrpz} are similar, so let us only establish the latter. To that end, recall the notation from the proof of \Cref{zgammazfunctional}, and define the events $E_1$ and $E_2$ as in \eqref{e1e2}. As in the proof of \Cref{zgammazfunctional}, we let $E = E_1 \cup E_2$ and $F$ be the complement of $E$.
	
	Then, $\gamma_z$ and $\one_F \mathfrak{Z}$ are both in $\mathcal{H}_{\alpha / 2, 1 - \alpha / 2 + s}^{1 / C_2}$, so applying the first estimate in \eqref{functionalnearfg2} and \eqref{ffnorm} yields a constant $C'$ (only dependent on $C_2$, $s$, and $p$) such that 
	\begin{flalign}
	\label{rpzgammazz}
	\one_F \big| r_{p, z} (\gamma_z) - r_{p, z} (\mathfrak{Z}) \big| \le C' \displaystyle\sup_{u \in \mathbb{S}_+^1} \big| \gamma_z - \mathfrak{Z} \big| \one_F \le C' \big\| \gamma_z - \mathfrak{Z} \big\|_{1 - \alpha / 2 + s} \one_F \le \displaystyle\frac{C' C_1 (\log N)^{C_1}}{N^{s / 2} \eta^{\alpha / 2}}.
	\end{flalign}
	
	The estimate \eqref{rpzgammazz} bounds $\big| r_{p, z} (\gamma_z) - r_{p, z} (\mathfrak{Z}) \big|$ off of $E$. To bound it on $E$, we use the deterministic estimate given by the second inequality in \eqref{festimate11}. This yields the existence of a constant $C = C(\alpha, p, s)$ such that 
	\begin{flalign}
	\label{rpzgammazz2} 
	\one_E \big| r_{p, z} (\gamma_z) - r_{p, z} (\mathfrak{Z}) \big| \le \one_E \big| r_{p, z} (\gamma_z) \big| + \big| r_{p, z} (\mathfrak{Z}) \big| \le 2 C \eta^{-p} \one_E.
	\end{flalign}
	
	Combining the second equality in \eqref{momentsi}, \eqref{rpzgammazz}, and \eqref{rpzgammazz2} yields 
	\begin{flalign}
	\label{rpzgammazzsi}
	\begin{aligned} 
	\Big| r_{p, z} (\gamma_z) -  \mathbb{E} \big[ |-z - \mathfrak{S}_i|^p \big] \Big| & \le \mathbb{E}_{\mathfrak{Y}} \Big[ \big| r_{p, z} (\gamma_z) - r_{p, z} (\mathfrak{Z}) \big| \Big] \\
	& = \mathbb{E}_{\mathfrak{Y}} \Big[ \one_F \big| r_{p, z} (\gamma_z) - r_{p, z} (\mathfrak{Z}) \big| \Big] + \mathbb{E}_{\mathfrak{Y}} \Big[ \one_E \big| r_{p, z} (\gamma_z) - r_{p, z} (\mathfrak{Z}) \big| \Big] \\
	& \le \displaystyle\frac{C' C_1 (\log N)^{C_1}}{N^{s / 2} \eta^{\alpha / 2}} + 2 C \eta^{-p} \mathbb{P} [E]. 
	\end{aligned}
	\end{flalign}
	
	The second statement of \eqref{spzrpz} now follows from the first statement of \eqref{riizsipestimate1}, the first statement of \eqref{estimatesi3}, \eqref{probabilitye1}, and \eqref{rpzgammazzsi}. 	
\end{proof}

\section{Estimates for the Fixed Point Quantities}

\label{EstimatesSmall}

In this section we establish the estimates stated in the proof of \Cref{sspzrrpzgamma} in \Cref{LocalProof1}. To that end, we first require some concentration estimates, which will be given in \Cref{Near}. We will then establish \Cref{sestimate}, \Cref{stestimate}, and \Cref{zgamma} in \Cref{EstimateS02}, \Cref{STEstimate02}, and \Cref{GammaNear}, respectively.

\subsection{Concentration Results}

\label{Near} 

In this section, we collect concentration statements that will be used in the proofs of the estimates stated in \Cref{LocalProof1}. The first (which is an analog of \Cref{expectationfnear2}) is Lemma 5.3 of \cite{bordenave2017delocalization}, applied with their $\beta$ equal to our $\frac{\alpha}{2}$ and their $\delta$ equal to our $s$.

\begin{lem}[{\cite[Lemma 5.3]{bordenave2017delocalization}}]
	
	\label{expectationfnear}
	
	Let $N$ be a positive integer, let $r$ and $s$ be positive real numbers, and let $\textbf{\emph{A}} = \{ a_{ij} \}_{1 \le i, j \le N}$ be an $N \times N$ symmetric random matrix such that the $i$-dimensional vectors $A_i = (a_{i1}, a_{i2}, \ldots , a_{ii} )$ are mutually independent for $1 \le i \le N$. Let $z = E + \mathrm{i} \eta \in \mathbb{H}$, and denote $\textbf{\emph{B}} = \{ B_{ij} \} = (\textbf{\emph{A}} - z)^{-1}$. Fix $u \in \mathbb{S}_+^1$, $\alpha \in (0, 2)$, and $s \in \big( 0, \frac{\alpha}{2} \big)$. 
	
	Then, if we denote $f = f_u: \mathbb{C} \rightarrow \mathbb{C}$ by $f_u (z) = (\mathrm{i} z \b| u)^{\alpha / 2}$, there exists a constant $C = C(\alpha) > 0$ such that
	\begin{flalign*}
	\mathbb{P} \Bigg[ \bigg\| \displaystyle\frac{1}{N} \displaystyle\sum_{j = 1}^N f (B_{jj}) - \displaystyle\frac{1}{N} \displaystyle\sum_{j = 1}^N \mathbb{E} \big[ f(B_{jj}) \big] \bigg\|_{1 - \alpha / 2 + s} \ge t \Bigg] \le C (\eta^{\alpha / 2} t)^{- 1 / s} \exp \left( - \displaystyle\frac{ N (\eta^{\alpha / 2} t)^{2 / s}}{C} \right). 
	\end{flalign*}
	
\end{lem}

The following  (which is analog of \Cref{randomvariablenearexpectation2}) is a special case of Lemma 5.4 of \cite{bordenave2017delocalization}, applied with their $\{ g_j \}$ equal to our $\{ y_j \}$; their $\{ h_j \}$ equal to our $-\mathrm{i} R_{jj}$; their $\beta$ equal to our $\frac{\alpha}{2}$; their $\delta$ equal to our $s$; and their $t$ equal to $C N^{-s / 2} \eta^{-\alpha / 2} (\log N)^s$.

\begin{lem}[{\cite[Lemma 5.4]{bordenave2017delocalization}}] 
	
	\label{randomvariablenearexpectation}

	Let $(y_1, y_2, \ldots , y_N)$ be a Gaussian random vector whose covariance matrix is given by $\Id$, let $s \in \big( 0, \frac{\alpha}{2} \big)$, and for each $1 \le j \le N$ let 
	\begin{flalign*}
	f_j (u) = \big( - \mathrm{i} R_{jj}^{(i)} \b| u \big)^{\alpha / 2} |y_j|^{\alpha}, \qquad g_j (u) = \big( - \mathrm{i} R_{jj}^{(i)} \b| u \big)^{\alpha / 2} \mathbb{E} \big[ |y_j|^{\alpha} \big]. 
	\end{flalign*}
	
	\noindent Then, there exists a constant $C = C(\alpha) > 0$ that
	\begin{flalign}
	\label{sumgjjialpha}
	\begin{aligned}
	\mathbb{P} & \left[ \bigg\| \displaystyle\frac{1}{N} \displaystyle\sum_{j = 1}^N (f_j - g_j) \bigg\|_{1 - \alpha / 2 + s} > \displaystyle\frac{C (\log N)^s}{N^{s / 2} \eta^{\alpha / 2}} \right] < \displaystyle\frac{C N^{1 / 2}}{\log N} \exp \left( - \displaystyle\frac{	(\log N)^2}{C} \right),
	\end{aligned}
	\end{flalign}
	
	\noindent where the expectation is with respect to $(y_1, y_2, \ldots , y_N)$ and conditional on $\textbf{\emph{X}}^{(i)}$. 
	
\end{lem}

\subsection{Proof of \Cref{sestimate}}

\label{EstimateS02}

In this section we establish \Cref{sestimate}. Its proof will be similar to that of \Cref{lambda1lambda2} in \Cref{ProofSLower}.

\begin{proof}[Proof of \Cref{sestimate}]

	Since all entries of $\textbf{R}$ are identically distributed, we may assume that $i = N$. In what follows, let $\mathcal{E}$ denote the event on which 
	\begin{flalign}
	\label{eevent}
	\big| \Tr \Im \textbf{R}^{(N)} - \mathbb{E} [ \Im R_{11} ] \big| \le \displaystyle\frac{4 \log N}{(N \eta^2)^{1 / 2}} + \displaystyle\frac{8}{N \eta}.
	\end{flalign} 
	
	\noindent In view of \Cref{rankperturbation} (applied with $r = 1$) and the second estimate in \eqref{mnexpectationnearmnimaginary}, we deduce that $\mathbb{P} \big[ \mathcal{E}^c \big] \le 2 \exp \big( - (\log N)^2 \big)$, where $\mathcal{E}^c$ denotes the complement of $\mathcal{E}$.

	We now apply \Cref{quadraticlaw2} with $X = \big( X_{Nj} \big)_{j \ne N}$ and $\textbf{A} = \{ A_{ij} \}$ equal to the $(N - 1) \times (N - 1)$ diagonal matrix with $A_{jj} = \Im R_{jj}^{(N)}$. Then, $\Im S_N = \langle \textbf{A} X, X \rangle$. Inserting $t = (\log N)^{2 / \alpha} (2 \log 2)^{1 / 2}$ into \Cref{quadraticlaw2}, we find from a Markov estimate that
	\begin{flalign}
	\label{probabilitysi2}
	\begin{aligned}
	& \mathbb{P} \big[ \Im S_N < \one_{\mathcal{E}} (\log N)^{-4 / \alpha} \big] \\
	& \quad \le 2 \mathbb{E} \Bigg[ \one_{\mathcal{E}} \exp \bigg( -\displaystyle\frac{t^2}{2} \langle \textbf{A} X, X \rangle \bigg) \Bigg] \\
	& \quad \le 2 \mathbb{E} \Bigg[ \one_{\mathcal{E}} \exp \bigg( - \displaystyle\frac{\sigma^{\alpha} (2 \log 2)^{\alpha / 2} (\log N)^2  \| \textbf{A}^{1 / 2} Y \|_{\alpha}^{\alpha}}{N - 1} \bigg)  \exp \bigg( O \Big( (\log N)^{4 / \alpha + 1} N^{- 10 \theta - 1} \Tr \textbf{A} \Big) \bigg) \Bigg] \\
	& \qquad + 2 N \exp \left( - \displaystyle\frac{( \log N )^2}{4}\right) + 2 \mathbb{P} \big[ \mathcal{E}^c \big],
	\end{aligned} 
	\end{flalign}
	
	\noindent where $Y = (y_1, y_2, \ldots , y_{N - 1})$ is a Gaussian random variable whose covariance matrix is given by $\Id$.
	
	Now, in view of the definition \eqref{eevent} of the event $\mathcal{E}$ and our assumption that $\mathbb{E} [\Im R_{11}] < \mathbb{E} \big[ |R_{11}|^2 \big]^{-1 / 2} \le \varepsilon^{-1 / 2}$, we have that $\one_{\mathcal{E}} \big| \Tr \textbf{A} \big| < 2 \varepsilon^{-1 / 2}$ for sufficiently large $N$. This (and our previous estimate $\mathbb{P} \big[ \mathcal{E}^c \big] \le 2 \exp \big( - (\log N)^2 \big)$) guarantees the existence of a constant $C = C(\alpha, b, \varepsilon) > 0$ such that
		\begin{flalign}
	\label{probabilitysi3}
	\begin{aligned}
	& \mathbb{P} \big[ \Im S_i < (\log N)^{-4 / \alpha} \big] \le C \mathbb{E} \Bigg[ \exp \bigg( - \displaystyle\frac{(\log N)^2  \| \textbf{A}^{1 / 2} Y \|_{\alpha}^{\alpha}}{C N} \bigg) \Bigg] + C \exp \left( - \displaystyle\frac{(\log N)^2}{C}\right). 
	\end{aligned} 
	\end{flalign}
	
	\noindent Thus, to provide a lower bound on $\Im S_N$, it suffices to establish a lower bound on
	\begin{flalign}
	\label{aalphaynorm}
	\displaystyle\frac{\| \textbf{A}^{1 / 2} Y \|_{\alpha}^{\alpha}}{N}  = \displaystyle\frac{1}{N} \displaystyle\sum_{j = 1}^{N - 1} \big| \Im R_{jj}^{(N)} \big|^{\alpha / 2} |y_j|^{\alpha} .
	\end{flalign}
	
	\noindent To that end, we apply \Cref{randomvariablenearexpectation2} (with $\textbf{A} = \textbf{H}^{(N)}$ and $t = (\log N)^{\alpha / 2} (N \eta^2)^{\alpha / 4}$) to obtain that
	\begin{flalign}
	\label{sumgjjialpha1}
	\begin{aligned}
	\mathbb{P} & \left[ \bigg| \displaystyle\frac{1}{N} \displaystyle\sum_{j = 1}^{N - 1} \big| \Im R_{jj}^{(N)} \big|^{\alpha / 2} |y_j|^{\alpha} - \displaystyle\frac{1}{N} \displaystyle\sum_{j = 1}^{N - 1} \big| \Im R_{jj}^{(N)} \big|^{\alpha / 2} \mathbb{E} \big[ |y_j|^{\alpha} \big] \bigg| > \displaystyle\frac{C (\log N)^4}{N^{\alpha / 4} \eta^{\alpha / 2}} \right] \\
	& \qquad \qquad \qquad \qquad \qquad \qquad \qquad \qquad \qquad \qquad < C \exp \left( - \displaystyle\frac{(\log N)^2}{C}\right),  
	\end{aligned}
	\end{flalign}
	
	\noindent after increasing $C$ if necessary. Next, applying \Cref{rankperturbation} with $r = \frac{\alpha}{2}$ yields the deterministic estimate 
	\begin{flalign}
	\label{gjjgjjialpha2}
	\displaystyle\frac{1}{N} \displaystyle\sum_{j = 1}^N \Big| \big( \Im R_{jj} \big)^{\alpha / 2} - \big( \Im R_{jj}^{(N)} \big)^{\alpha / 2} \Big| < \displaystyle\frac{4}{(N \eta)^{\alpha / 2}}.
	\end{flalign}
	
	\noindent The estimate \eqref{ffnorm} and \Cref{expectationfnear} (applied with $\textbf{A} = \textbf{X}^{(N)}$, $s = \frac{\alpha}{2}$, and $t = (N \eta^2)^{- \alpha / 4} (\log N)^{\alpha / 2}$), yields, after increasing $C$ if necessary, that 
	\begin{flalign}
	\label{gjjexpectationnearsi}
	\begin{aligned}
	\mathbb{P} & \left[ \bigg| \displaystyle\frac{1}{N} \displaystyle\sum_{j = 1}^{N - 1} \big| \Im R_{jj} \big|^{\alpha / 2} - \displaystyle\frac{1}{N} \displaystyle\sum_{j = 1}^{N - 1} \mathbb{E} \big[ | \Im R_{jj} |^{\alpha / 2} \big] \bigg| > \displaystyle\frac{(\log N)^{\alpha / 2}}{N^{\alpha / 4} \eta^{\alpha / 2}} \right]  < C \exp \left( - \displaystyle\frac{(\log N)^2}{C} \right). 
	\end{aligned}
	\end{flalign}
	
	 Combining the lower bound $\mathbb{E} \big[ |\Im R_{jj}|^{\alpha / 2} \big] \ge \varepsilon$ (see the second estimate in \eqref{rr11alpha2r112}), \eqref{gijeta}, \eqref{aalphaynorm}, \eqref{sumgjjialpha1}, \eqref{gjjgjjialpha2}, \eqref{gjjexpectationnearsi}, and the fact that all entries of $\textbf{R}$ are identically distributed yields (again, after increasing $C$ if necessary) 
	\begin{flalign*} 
	\mathbb{P} \left[ \displaystyle\frac{\| \textbf{A}^{1 / 2} Y \|_{\alpha}^{\alpha}}{N} \le \displaystyle\frac{\varepsilon}{C} \right]  \le C \exp \left( - \displaystyle\frac{(\log N)^2}{C} \right),
	\end{flalign*}
	
	\noindent from which we deduce the lemma upon insertion into \eqref{probabilitysi}.
\end{proof}

\subsection{Proof of \Cref{stestimate}}

\label{STEstimate02} 

In this section we establish \Cref{stestimate}. Its proof will be similar to that of \Cref{lambda3} in \Cref{ProofSTLower}.

\begin{proof}[Proof of \Cref{stestimate}]
	
	Since all entries of $\textbf{R}$ are identically distributed, we may assume that $i = N$. 
	
	As in the proof of \Cref{sestimate}, we begin by applying \Cref{quadraticlaw2}, now with $\textbf{A} = \Im \textbf{R}^{(N)}$, $X = \big( X_{Nj} \big)_{1 \le j \le N - 1}$, and $t = (\log N)^{2 / \alpha} (2 \log 2)^{1 / 2}$. Then, $\Im (S_N - T_N	) = \langle \textbf{A} X, X \rangle$. Following the proof of \Cref{sestimate} yields a constant $C = C(\alpha, b, \varepsilon) > 0$ such that 

	\begin{flalign}
	\label{probabilitysti3}
	\begin{aligned}
	& \mathbb{P} \big[ \Im (S_N - T_N) < (\log N)^{-4 / \alpha} \big] \le C \mathbb{E} \Bigg[ \exp \bigg( - \displaystyle\frac{C (\log N)^2  \| \textbf{A}^{1 / 2} Y \|_{\alpha}^{\alpha}}{N} \bigg) \Bigg] + C \exp \left( - \displaystyle\frac{(\log N)^2}{C}\right), 
	\end{aligned} 
	\end{flalign}
	
	\noindent where $Y = (y_1, y_2, \ldots , y_{N - 1})$ is a Gaussian random variable whose covariance is given by $\Id$. Thus, it again suffices to establish a lower bound on $N^{-1} \| \textbf{A} Y \|_{\alpha}^{\alpha}$. 
	
	To that end, we apply \Cref{lowersum1} with $w_i = (\textbf{A}^{1 / 2} Y)_i$, $r = \alpha$, and $a = 2 + \varepsilon$. Then we find that $V_j = \Im R_{jj}^{(N)} (z)$, and $U_{jk} = \Im R_{jk}^{(N)} (z)$ for each $1 \le j, k \le N - 1$. We must next estimate the quantities $V$, $\mathcal{X}$, and $U$ from that lemma. These are given by $V = (N - 1)^{-1} \sum_{i = 1}^{N - 1} V_j$, $\mathcal{X} = (N - 1)^{-1} \sum_{i = 1}^{N - 1} V_j^{a / 2}$, and $U = (N - 1)^{-2} \sum_{1 \le j, k \le N - 1} c_{jk}$.
	
	To do this, observe from \eqref{gij2} and \eqref{gijeta} that 
	\begin{flalign}
	\label{cjk2estimate}
	U \le \displaystyle\frac{4}{N^2} \displaystyle\sum_{j = 1}^{N - 1} \displaystyle\sum_{k = 1}^{N - 1} c_{jk}^2  = \displaystyle\frac{4}{N^2} \displaystyle\sum_{j = 1}^{N - 1} \displaystyle\sum_{k = 1}^{N - 1} \big| \Im R_{jk}^{(N)} \big|^2 & \le \displaystyle\frac{4}{N^2 \eta} \displaystyle\sum_{j = 1}^{N - 1} \Im R_{jj}^{(N)} \le \displaystyle\frac{4}{N \eta^2}.
	\end{flalign} 
	
	To bound $V$, we apply the first estimate in \eqref{mnexpectationnearmnimaginary} to deduce that 
	\begin{flalign}
	\label{gjjexpectationestimate1}
	\mathbb{P} \left[ \bigg| \displaystyle\frac{1}{N} \displaystyle\sum_{j = 1}^N \Im R_{jj} - \mathbb{E} \big[ \Im R_{jj} \big] \bigg| > \displaystyle\frac{4 \log N}{(N \eta^2)^{1 / 2}} \right] < 2 \exp \left( - \displaystyle\frac{(\log N)^2}{8} \right). 
	\end{flalign}

	\noindent Therefore, \Cref{rankperturbation} (applied with $r = 1$), \eqref{gjjexpectationestimate1}, and the assumption \eqref{rr11alpha2r112} that $\mathbb{E} \big[ \Im R_{jj}  \big] \ge \mathbb{E} \big[ (\Im R_{jj})^{\alpha / 2} \big]^{2 / \alpha} \ge \varepsilon^{2 / \alpha}$ together imply that 
	\begin{flalign}
	\label{estimatev} 
	\mathbb{P} \left[ |V| < \displaystyle\frac{1}{C} \right] < C \exp \left( -\displaystyle\frac{(\log N)^2}{C} \right),
	\end{flalign} 
	\noindent after increasing $C$ if necessary. In particular, $\mathbb{P} \big[ |V| \le 100 (\log N)^{10} U^{1 / 2} \big] < 2 C \exp \big( - C^{-1} (\log N)^2 \big)$ for sufficiently large $N$. 
	
	Now let us estimate $\mathcal{X} = (N - 1)^{-1} \sum_{j = 1}^N V_j^{a / 2}$. To that end, observe by \eqref{gijeta} and \Cref{rankperturbation2} (applied with $r = \frac{a}{2} \le 2$), we find that
	\begin{flalign}
	\label{gjjgjjiestimate2}
	\begin{aligned}
	\bigg| \displaystyle\frac{1}{N} \displaystyle\sum_{j = 1}^{N - 1} \big| \Im R_{jj} \big|^{a / 2} - \displaystyle\frac{(N - 1) \mathcal{X}}{N}\bigg| & \le \bigg| \displaystyle\frac{1}{N} \displaystyle\sum_{j = 1}^N \big| \Im R_{jj} \big|^{a / 2} - \displaystyle\frac{1}{N} \displaystyle\sum_{j = 1}^N \big| \Im R_{jj}^{(i)} \big|^{a / 2}  \bigg| + \displaystyle\frac{4}{N \eta^{a / 2}} \\
	& \le \displaystyle\frac{12}{N \eta^{a / 2}}. 
	\end{aligned}
	\end{flalign}

	Now let $f(y) = \one_{|\Im y| \le \eta^{-1}} |\Im y|^{a / 2} + \one_{|\Im y| > \eta^{-1}} (2 \eta)^{-a / 2}$, and observe that $f$ is Lipschitz with constant $L = a \eta^{1 - a / 2}$. Applying \Cref{fgjjnear} with $t = N^{-1 / 2} \eta^{-a / 2} \log N$ and using \eqref{gijeta} yields
	\begin{flalign}
	\label{gjjexpectationestimate2} 
	\mathbb{P} \left[ \bigg| \displaystyle\frac{1}{N} \displaystyle\sum_{j = 1}^N \big| \Im R_{jj} \big|^{a / 2} - \mathbb{E} \big[ \big| \Im R_{jj} \big|^{a / 2}  \big] \bigg|  \ge \displaystyle\frac{\log N}{N^{1 / 2} \eta^{a / 2}} \right] \le 2 \exp \left( - \displaystyle\frac{(\log N)^2}{8 a^2}\right). 
	\end{flalign}
	
	\noindent Combining \eqref{gjjgjjiestimate2}, \eqref{gjjexpectationestimate2}, the fact that $\eta \ge N^{\varepsilon - s / \alpha} \ge N^{\varepsilon - 1 / 2}$, and the fact (due to \eqref{rr11alpha2r112}) that $\mathbb{E} \big[ |R_{jj}|^{a / 2} \big] \le \mathbb{E} \big[ |R_{jj}|^2 \big]^{a / 4} \le \varepsilon^{-a / 4}$ yields that 
	\begin{flalign}
	\label{estimatevs}
	\mathbb{P} \big[ |\mathcal{X}| > C \big] < C \exp \left( - \displaystyle\frac{ (\log N)^2}{C}  \right),
	\end{flalign}
	
	\noindent after increasing $C$ if necessary. Now \Cref{lowersum1} with \eqref{cjk2estimate}, \eqref{estimatev}, and \eqref{estimatevs} together yield that 
	\begin{flalign}
	\label{walpha} 
	\mathbb{P} \left[ \displaystyle\frac{\| \textbf{A}^{1 / 2} Y  \|_{\alpha}^{\alpha}}{N}  < (\log N)^{-C} \right] < C \exp \left( - \displaystyle\frac{(\log N)^2}{C} \right), 
	\end{flalign}
	
	\noindent after increasing $C$ if necessary. Now the lemma follows from combining \eqref{probabilitysti3} and \eqref{walpha}. 
\end{proof}

\subsection{Proof of \Cref{zgamma}}	

\label{GammaNear}

In this section we establish \Cref{zgamma}.

\begin{proof}[Proof of \Cref{zgamma}]
	
	Let us define 
	\begin{flalign*}
	& \mathcal{Z} = \mathcal{Z} (u) = \mathbb{E} \big[ \mathfrak{Z} \big] = \displaystyle\frac{\Gamma \big( 1 - \frac{\alpha}{2} \big)}{N} \displaystyle\sum_{j \ne i} \big( - \mathrm{i} R_{jj}^{(i)} \b| u \big)^{\alpha / 2}, \\
	&  \Phi_z = \Phi_z (u) = \displaystyle\frac{\Gamma \big( 1 - \frac{\alpha}{2} \big)}{N} \displaystyle\sum_{j \ne i} \mathbb{E} \Big[ \big( - \mathrm{i} R_{jj}^{(i)} \b| u \big)^{\alpha / 2} \big], \\
	& \xi_z = \xi_z (u) = \displaystyle\frac{\Gamma \big( 1 - \frac{\alpha}{2} \big)}{N} \displaystyle\sum_{j \ne i} \mathbb{E} \Big[ \big( - \mathrm{i} R_{jj} \b| u \big)^{\alpha / 2} \Big].
	\end{flalign*}
	
	To establish this proposition, we will first show that $\mathfrak{Z}$, $\mathcal{Z}$, $\xi_z$, and $\gamma_z$ are all approximately equal. To that end, first observe that \Cref{randomvariablenearexpectation} implies the existence of a constant $C = C(\alpha) > 0$ such that   
	\begin{flalign}
	\label{znearz}
	\mathbb{P} \left[ \big\| \mathcal{Z} - \mathfrak{Z} \big\|_{1 - \alpha / 2 + s} \ge \displaystyle\frac{(\log N)^s}{N^{s / 2} \eta^{\alpha / 2}} \right] \le C \exp \left( - \displaystyle\frac{(\log N)^2}{C} \right).
	\end{flalign}	
	
	\noindent Next, applying \Cref{expectationfnear} with $\textbf{A} = \textbf{X}^{(i)}$ and $t = N^{-s / 2} \eta^{-\alpha / 2} (\log N)^s$ yields (after increasing $C$ if necessary) 
	\begin{flalign}
	\label{znearexpectation}
	\mathbb{P} \left[ \big\| \mathcal{Z} - \Phi_z \big\|_{1 - \alpha / 2 + s} \ge \displaystyle\frac{(\log N)^s}{N^{s / 2} \eta^{\alpha / 2}} \right] \le C \exp \left( -\displaystyle\frac{(\log N)^2}{C} \right).
	\end{flalign}
	
	\noindent Now we apply the second estimate in \eqref{f1f2near} with $x_1 = R_{jj}$, $x_2 = R_{jj}^{(i)}$, $r = \frac{\alpha}{2}$, and $a = \eta$ to obtain (again, after increasing $C$ if necessary)
	\begin{flalign}
	\label{gjjiuneargjju} 
	\Big\| \big( R_{jj} \b| u \big)^{\alpha / 2} - \big( R_{jj}^{(i)} \b| u \big)^{\alpha / 2} \Big\|_{1 - \alpha / 2 + s} \le C \eta^{-\alpha / 2} \Big( \big| R_{jj} - R_{jj}^{(i)} \big|^{\alpha / 2} + \eta^s \big| R_{jj} - R_{jj}^{(i)} \big|^s \Big).
	\end{flalign}
	
	\noindent To estimate the right side of \eqref{gjjiuneargjju} we apply apply \Cref{rankperturbation} to deduce that 
	\begin{flalign}
	\label{gjjineargjj}
	\displaystyle\frac{1}{N} \displaystyle\sum_{j = 1}^N \big| R_{jj} - R_{jj}^{(i)} \big|^{\alpha / 2} \le \displaystyle\frac{4}{(N \eta)^{\alpha / 2}}, \qquad \displaystyle\frac{1}{N} \displaystyle\sum_{j = 1}^N \big| R_{jj} - R_{jj}^{(i)} \big|^s \le \displaystyle\frac{4}{(N \eta)^s}. 
	\end{flalign}
	
	\noindent Summing \eqref{gjjiuneargjju} over all $j \ne i$, taking expectations, applying \eqref{gjjineargjj}, and \eqref{gijeta} yields (after increasing $C$ if necessary) that
	\begin{flalign} 
	\label{nearxiz} 
	\| \Phi_z - \xi_z \|_{1 - \alpha / 2 + s} \le C \left( \displaystyle\frac{1}{N^{\alpha / 2} \eta^{\alpha}} + \displaystyle\frac{1}{N^s \eta^{\alpha / 2}} \right). 
	\end{flalign}
	
	\noindent Furthermore, since the entries of $\textbf{R}$ are identically distributed, we have (after increasing $C$ if necessary) that 
	\begin{flalign}
	\label{xizneargammaz}
	\big\| \xi_z - \gamma_z \big\|_{1 - \alpha / 2 + s} = \displaystyle\frac{\Gamma \big( 1 - \frac{\alpha}{2} \big)}{N} \bigg\| \mathbb{E} \Big[ \big( - \mathrm{i} R_{jj} \b| u \big)^{\alpha / 2} \Big] \bigg\|_{1 - \alpha / 2 + s} \le \displaystyle\frac{C}{N \eta^{\alpha / 2}}, 
	\end{flalign}
	
	\noindent where we have used \eqref{gijeta} and the first estimate in \eqref{f1f2near}.
	
	Now the proposition follows from \eqref{znearz}, \eqref{znearexpectation}, \eqref{nearxiz}, \eqref{xizneargammaz}, and the fact that $N > \eta^{-2}$. 
\end{proof}

\appendix

\section{Estimating the Entries of \texorpdfstring{$\textbf{G}_t$}{}} 

\label{GtEstimate} 

In this section we establish \Cref{mestimategijestimate}. To that end, we first require some additional notation. Recalling the definitions of $\textbf{H}_s$ and $\textbf{G}_s$ from the beginning of \Cref{Eigenvectors}, let $\big\{ \lambda_j (s) \big\}_{j \in [1, N]}$ denote the $N$ eigenvalues of $\textbf{H}_s$, and define $m_s = m_s (z) = N^{-1} \Tr \textbf{G}_s = N^{-1} \sum_{j = 1}^N \big( \lambda_j (s) - z \big)^{-1}$. 

Further let $m_{\text{fc}, s} (z) \in \mathbb{H}$ denote the unique solution in the upper half plane to the equation 
\begin{flalign}
\label{mgi} 
m_{\text{fc}, s} (z) = m_0 \big( z + t m_{\text{fc}, s} (z) \big) = \displaystyle\frac{1}{N} \displaystyle\sum_{j = 1}^N g_j (s, z), \quad \text{where} \quad g_j (s, z) = \displaystyle\frac{1}{\lambda_j - z - t m_{\text{fc}, s} (z)}.  
\end{flalign}

The quantity $m_{\text{fc}, s}$ denotes the Stieltjes transform of the free convolution (see directly before \Cref{universalityperturbation2}) of the empirical spectral distribution of $\textbf{H}_0$ with a suitable multiple of the semicircle law \cite{biane1997free}.

We require the following two results, which appear as Theorem 2.1 and Theorem 2.2 of \cite{bourgade2017eigenvector}.

\begin{prop}[{\cite[Theorem 2.1]{bourgade2017eigenvector}}]
	
	\label{gijt1}
	
	Adopt the notation of \Cref{eta0rregular}, and further assume that $\textbf{\emph{H}}_0$ is $(\eta_0, \gamma, r)$-regular with respect to $E_0$. Let $\textbf{\emph{U}} = \{ u_{ij} \}$ and $\textbf{\emph{D}} = \{ d_{jj} = d_j \}$ denote orthogonal and diagonal matrices, respectively, so that $\textbf{\emph{H}}_0 = \textbf{\emph{U}} \textbf{\emph{D}} \textbf{\emph{U}}^{-1}$. For each $j \in [1, N]$, let $\textbf{\emph{u}}_j = (u_{1j}, u_{2j}, \ldots , u_{Nj}) \in \mathbb{R}^N$ denote the $j$-th column of $\textbf{\emph{U}}$. 
	
	Fix $s \in [0, 1]$ satisfying $N^{\delta} \eta \le s \le N^{-\delta} \gamma$. Then, for any $D > 1$ and $\kappa \in (0, 1)$, there exists a constant $C = C(\delta, \kappa, D, A) > 0$ such that 
	\begin{flalign}
	\label{qgktz} 
	\mathbb{P} \Bigg[ \displaystyle\sup_{z \in \mathcal{D}} \bigg( \Big| \big\langle \textbf{\emph{q}} \textbf{\emph{G}}_s (z), \textbf{\emph{q}} \big\rangle - \displaystyle\sum_{j = 1}^N \langle \textbf{\emph{u}}_j, \textbf{\emph{q}} \rangle^2 g_j (s, z) \Big| - \displaystyle\frac{N^{2 \delta}}{(N \eta)^{1 / 2}} \Im \Big( \displaystyle\sum_{j = 1}^N \langle \textbf{\emph{u}}_j, \textbf{\emph{q}} \rangle^2 g_j (s, z) \Big) \bigg) > 0 \Bigg] < C N^{-D},
	\end{flalign} 
	
	\noindent for any vector $\textbf{\emph{q}} \in \mathbb{R}^N$ such that $\| \textbf{\emph{q}} \|_2 = 1$. In \eqref{qgktz}, we have abbreviated $\mathcal{D} = \mathcal{D} (E_0, r, N^{4 \delta - 1}, 1 - \kappa r, \kappa)$ (recall \eqref{definitionde0reta0}).

\end{prop}

\begin{prop}[{\cite[Proposition 2.2]{bourgade2017eigenvector}}]
	
	\label{gijtgi}
	
	Adopt the notation and assumptions of \Cref{gijt1}. Then, there exists a constant $C  = C (\delta, \kappa, D, A) > 0$ such that 
	\begin{flalign}
	\label{sumgjtzmtestimate}
	\big| m_{\text{\emph{fc}}, s} (z) \big| \le \displaystyle\frac{1}{N} \displaystyle\sum_{j = 1}^N \big| g_j (s, z) \big| \le C \log N, \qquad \displaystyle\frac{1}{C} \le \Im m_{\text{\emph{fc}}, s} (z) \le C, 
	\end{flalign}
	
	\noindent for any $z \in \mathcal{D}$.  
	
\end{prop} 

Now we can establish \Cref{mestimategijestimate}.

\begin{proof}[Proof of \Cref{mestimategijestimate}]
	
	Applying \eqref{qgktz} with $\textbf{q} = (q_1, q_2, \ldots , q_N)$ satisfying $q_k = \one_{k = j}$ for each $k \in [1, N]$ yields the existence of a constant $C = C(\delta, \kappa, D, A) > 0$ such that 
	\begin{flalign}
	\label{qgktzgjj} 
	\mathbb{P} \left[ \displaystyle\sup_{z \in \mathfrak{D}} \left( \bigg| G_{jj} (s, z) - \displaystyle\sum_{k = 1}^N u_{jk}^2 g_k (s, z) \bigg| - \displaystyle\frac{N^{\delta / 2}}{(N \eta)^{1 / 2}} \Im \bigg( \displaystyle\sum_{k = 1}^N u_{jk}^2 g_k (s, z) \bigg) \right)> 0 \right] < C N^{- 10 D}.
	\end{flalign} 
	
	Let us estimate the terms $g_k (s, z)$ appearing in \eqref{qgktzgjj}. To that end, we define $\mathcal{A}_0 = \mathcal{A}_0 (E_0) = [E_0 - \eta_0, E_0 + \eta_0]$ and $\mathcal{A}_m = \mathcal{A}_m (E_0) = \big[ E_0 - 2^m \eta_0, E_0 - 2^{m - 1} \eta_0 \big] \cup \big[ E_0 + 2^{m - 1} \eta_0, E_0 + 2^m \eta_0 \big]$, for each integer $m \ge 1$. Since \eqref{sumgjtzmtestimate} implies the existence a constant $C = C(\delta, \kappa, E_0, D, A) > 1$ such that $\big| m_{\text{fc}, s} (z) \big| \le C \log N$ and $\frac{1}{C} \le \Im m_{\text{fc}, s} (z) \le C$, the definition \eqref{mgi} of the $g_k$ implies that
	\begin{flalign}
	\label{gjtzr}
	\displaystyle\max_{\lambda_k \in \mathcal{A}_m} \big| g_k (s, E_0 + \mathrm{i} \eta) \big| \le \left( \displaystyle\frac{C^2}{\big( \min \{ 2^{m - 1} \eta_0 - C^2 s \log N, 0 \} \big)^2 + s^2} \right)^{1 / 2}.
	\end{flalign}
	
	\noindent for any integer $m \ge 1$.
	
	Next let us estimate the entries of $\textbf{U}$, where we recall from \Cref{gijt1} that $\textbf{H}_0 = \textbf{U} \textbf{D} \textbf{U}^{-1}$. The assumed bound on the entries of $\textbf{G}_0 (z)$ implies 
	\begin{flalign}
	\label{gijuij} 
	\displaystyle\sup_{z \in \mathcal{D} (E_0, r, \eta_0, \gamma, 0)} \left| \displaystyle\sum_{k = 1}^N \displaystyle\frac{u_{jk}^2}{z - \lambda_k} \right| = \displaystyle\sup_{z \in \mathcal{D} (E_0, r, \eta_0, \gamma, 0)} \big| G_{jj} (z) \big| \le B, 
	\end{flalign}
	
	\noindent where we have denoted $\lambda_j = \lambda_j (0)$ as the eigenvalues of $\textbf{H}_0$. Thus, setting $z = E_0 + \mathrm{i} \eta_0$ in \eqref{gijuij} yields 
	\begin{flalign}
	\label{sumuija}
	\displaystyle\max_{1 \le j \le N} \displaystyle\sum_{\lambda_k \in \mathcal{A}_m} u_{jk}^2 \le \displaystyle\min \big\{ 2^m \eta_0 B, 1 \big\}, \quad \text{for any integer $m \ge 0$.} 
	\end{flalign}
	
	Now we can bound the terms appearing in \eqref{qgktzgjj}. We define 
	\begin{flalign*}
	M = \Bigg\lceil \log_2 \bigg( \frac{s (\log N)^2}{\eta_0} \bigg) \Bigg\rceil, 
	\end{flalign*} 
	
	\noindent and write 	
	\begin{flalign*}
	\left| \displaystyle\sum_{k = 1}^N u_{jk}^2 g_k (s, z) \right| & \le \displaystyle\sum_{m = 0}^{\infty} \displaystyle\sum_{\lambda_k \in \mathcal{A}_m (E)} u_{jk}^2 \big| g_k (s, z) \big|  \\
	& \le  \displaystyle\sum_{m = 0}^M \displaystyle\sum_{\lambda_k \in \mathcal{A}_m (E)} u_{jk}^2 \big| g_k (s, z) \big| + \displaystyle\sum_{m = M + 1}^{\lceil 4 \log N \rceil} \displaystyle\sum_{\lambda_k \in \mathcal{A}_m (E)} u_{jk}^2 \big| g_k (s, z) \big| \\
	& \qquad + \displaystyle\sum_{m = \lceil 4 \log N \rceil}^{\infty} \displaystyle\sum_{\lambda_k \in \mathcal{A}_m (E)} \big| g_k (s, z) \big|.
	\end{flalign*}
	
\noindent We bound these three sums by combining \eqref{gjtzr}, \eqref{sumuija}, and the facts that $\eta_0 > N^{-1}$, $1 < B < N$, and $s \in \big( \eta_0, N^{-\delta} \big)$. For the first sum, we apply \eqref{gjtzr} --- noting the minimum in the denominator of the right side takes the value $0$ --- and the first argument of the minimum in the left side of \eqref{sumuija}. For the second sum, we apply  \eqref{gjtzr}, with the minimum on the right side taking the nonzero value, and \eqref{sumuija}. The third sum is bounded using \eqref{gjtzr} only.
	
 We deduce for sufficiently large $N$ that 
\begin{flalign}
\label{ujkgkaeta0t}
\left| \displaystyle\sum_{k = 1}^N u_{jk}^2 g_k (s, z) \right| \le C s^{-1} 2^M \eta_0 B + 	C B \log N + \displaystyle\frac{C}{N} \le C B (\log N)^3, 
\end{flalign}
	
	\noindent after increasing $C$ (in a way that only depends on $\delta$, $\kappa$, $D$, and $A$) if necessary. 
	
	Therefore, combining \eqref{qgktzgjj}, \eqref{ujkgkaeta0t}, the fact that $N \eta \ge N^{\delta}$, and a union bound over $j \in [1, N]$ yields (again after increasing $C$ if necessary, in a way that only depends on $\delta$, $\kappa$, $D$, and $A$) 
	\begin{flalign}
	\label{gjjestimateaeta}
	\mathbb{P} \bigg[ \displaystyle\sup_{z \in \mathfrak{D}} \displaystyle\max_{1 \le j \le N} \big| G_{jj} (s, z) \big| > C B (\log N)^3 \bigg] < C N^{- 5 D}. 
	\end{flalign}
	
	\noindent To estimate the remaining entries of $\textbf{G}_s$, we apply \eqref{qgktz} with $\textbf{q} = (q_1, q_2, \ldots , q_N)$ satisfying $q_k = 2^{- 1 / 2} \left( \one_{k = i} + \one_{k = j} \right)$ for some fixed $i, j \in [1, N]$. Using \eqref{ujkgkaeta0t}, this yields (after increasing $C$ if necessary, in a way that only depends on $\delta$, $\kappa$, $D$, and $A$) 
	\begin{flalign}
	\label{giigjjgijestimate} 
	\mathbb{P} \Bigg[ \displaystyle\sup_{z \in \mathfrak{D}} \big| G_{jj} (s, z) + G_{ii} (s, z) + 2 G_{ij} (s, z) \big| > C B (\log N)^3 \Bigg] < C N^{- 5 D}. 
	\end{flalign}
	
	\noindent Now the corollary follows from combining \eqref{gjjestimateaeta}, \eqref{giigjjgijestimate}, and a union bound over all $i, j \in [1, N]$. 	
\end{proof}

\section{Comparing Deformed Stable Laws to Their Removals} 

\label{StableCompare}

In this section we establish \Cref{quadraticlaw2}. However, we first require the following lemma that estimates the characteristic functions of removals of stable laws.

\begin{lem}
	
	\label{exponentialexpectationx}
	
	Fix $\sigma > 0$, $\alpha \in (0, 2)$, a positive integer $N$, and $0 < b < \frac{1}{\alpha}$. Let $X$ denote the random variable given by the $b$-removal of a deformed $(0, \sigma)$ $\alpha$-stable law, as in \Cref{partialstable}. Let $X_1, X_2, \ldots , X_N$ be mutually independent random variables, each with law $N^{-1 / \alpha} X$, and let $c_1, c_2, \ldots , c_N \in \mathbb{R}$ be constants. 
	
	Then, for any $t \in \mathbb{R}$, we have that
	\begin{flalign*}
	\mathbb{E} \Bigg[ \exp \bigg( \mathrm{i} t \displaystyle\sum_{j = 1}^N c_j X_j \bigg) \Bigg] & = \exp \Bigg( - \displaystyle\frac{\sigma^{\alpha} |t|^{\alpha}}{N} \displaystyle\sum_{j = 1}^N |c_j|^{\alpha} \Bigg) \exp \Bigg( O \bigg( t^2 N^{ (2 - \alpha) (b - 1 / \alpha) - 1} \displaystyle\sum_{j = 1}^N |c_j|^2 \bigg) \Bigg),
	\end{flalign*}
	
	\noindent where the implicit constant on the right side only depends on $\alpha$. 
\end{lem}

\begin{proof} 
	
	Let $Z$ be a $(0, \sigma)$ $\alpha$-stable law and $J$ be a random variable satisfying \Cref{momentassumption}. Let $Y = (Z+J) \one_{|Z+J| < N^b}$, so that $X = Z - Y$. Let $Y_1, Y_2, \ldots , Y_N$ be mutually random variables with law $N^{-1 / \alpha} Y$, let $Z_1, Z_2, \ldots , Z_N$ be mutually independent random variables with law $N^{-1 / \alpha} Z$, and let $J_1, J_2, \ldots, J_N$ be mutually independent variables with law $N^{-1/\alpha}J$. Then the random variables $X_j$ have laws $N^{-1 / \alpha} X$, where we assume that the $X_j$, $Y_j$, $Z_j$, and $J_j$ are coupled so that $X_j = Z_j +J _j - Y_j$ for each $1 \le j \le N$. 
	
	Observe that, for any $t \in \mathbb{R}$, we have that
	\begin{flalign}
	\label{expectationx}
	\begin{aligned}
	\mathbb{E} \big[ e^{\mathrm{i} t X} \big] & = \mathbb{E} \big[ e^{\mathrm{i} t (Z+J)}  \big] + \mathbb{E} \big[ e^{\mathrm{i} t (Z+J)} ( e^{- \mathrm{i} t Y} - 1) \big] \\
	&  = \mathbb{E} \big[ e^{\mathrm{i} t (Z+J)}  \big] - \mathrm{i} t \mathbb{E} \big[ e^{\mathrm{i} t (Z+J) } Y \big] + O \Big( \mathbb{E} \big[ t^2 Y^2 \big] \Big) \\
	& = \mathbb{E} \big[ e^{\mathrm{i} t (Z+J)}  \big] - \mathrm{i} t \mathbb{E} \big[ e^{\mathrm{i} t (Z+J)} (Z+J) \one_{ |Z+J | < N^b} \big] + O \Big( \mathbb{E} \big[ t^2 Y^2 \big] \Big) \\
	& = \mathbb{E} \big[ e^{\mathrm{i} t (Z+J)}  \big] - \mathrm{i} t \mathbb{E} \big[ (Z+J) \one_{|Z+J| < N^b} \big] + O \Big( \mathbb{E} \big[ t^2 Y^2 \big] \Big) = \mathbb{E} \big[ e^{\mathrm{i} t (Z+J)}  \big] + O \Big( \mathbb{E} \big[ t^2 Y^2 \big] \Big), 
	\end{aligned}
	\end{flalign} 
	
	\noindent where the second equality above follows from a Taylor expansion, the third from the definition of $Y$, the fourth from another Taylor expansion, and the fifth from the fact that $Z+J$ is symmetric. A similar argument shows that 
	\begin{flalign}
	\label{expectationzj}
	\mathbb{E} \big[ e^{\mathrm{i} t (Z + J)} \big]  = \mathbb{E} \big[ e^{\mathrm{i} t Z}  \big] + O \Big( \mathbb{E} \big[ t^2 J^2 \big] \Big).
	\end{flalign}
	
	Replacing $t$ with $c_j N^{-1 / \alpha} t$ in \eqref{expectationx} and \eqref{expectationzj}, we find that
	\begin{flalign*}
	\mathbb{E} \big[ e^{\mathrm{i} c_j t X_j} \big] & = \mathbb{E} \big[ e^{\mathrm{i} c_j t N^{-1 / \alpha} Z} \big] + \displaystyle\frac{c_j^2 t^2}{N^{2 / \alpha}} O \Big( \mathbb{E} \big[ |Z+J |^2 \one_{ |Z+J | \le N^b} \big] + \E [J^2] \Big) \\
	& = \exp \left( - \displaystyle\frac{\sigma^{\alpha} |c_j t|^{\alpha}}{N} \right) +  O \big( N^{ (2 - \alpha) (b - 1 / \alpha) - 1 } |c_j t|^2 \big) , 
	\end{flalign*}
	
	\noindent where in the second estimate above we used \eqref{betasigmaalphalaw} and integrated \eqref{probabilityxij}. Now, let $R = R_j = N^{-1} |c_j t|^{\alpha}$. Then, we find that 
	\begin{flalign}
	\label{expectationxr} 
	\mathbb{E} \big[ e^{\mathrm{i} c_j t X_j} \big] \le \exp \left( - \sigma^{\alpha} R \right) +  O \big( N^{ (2 - \alpha) b} R^{2 / \alpha}  \big) \le \exp \big( -\sigma^{\alpha} R \big) \exp \big( O (N^{(2 - \alpha) b} R^{2 / \alpha}) \big). 
	\end{flalign}
	
	\noindent Indeed, if $R \le 1$ then \eqref{expectationxr} follows from the estimate $y \le e^y - 1$. Otherwise, if $R > 1$ and $N$ is sufficiently large, we have that $N^{(2 - \alpha) b} R^{2 / \alpha} > 2 \sigma^{\alpha} R$ (since $\alpha < 2$), from which we again deduce \eqref{expectationxr} from the estimate $y \le e^y - 1$. Inserting the definition of $R = N^{-1} |c_j t|^{\alpha}$ into \eqref{expectationxr} yields 
	\begin{flalign}
	\label{expectationx2}
	\mathbb{E} \big[ e^{\mathrm{i} c_j t X_j} \big] \le \exp \left( - \displaystyle\frac{\sigma^{\alpha} |c_j t|^{\alpha}}{N} \right) \exp \Big( O \big( N^{(2 - \alpha) (b - 1 / \alpha) - 1} |c_j t|^2 \big) \Big). 
	\end{flalign}

	Now the lemma follows from taking the product of \eqref{expectationx2} over all $j \in [1, N]$. 
\end{proof}

Now we can establish \Cref{quadraticlaw2}.

\begin{proof}[Proof of \Cref{quadraticlaw2}]
	
	The proof of this lemma will follow a similar method as the one used to establish Lemma B.1 of \cite{bordenave2013localization}. To that end, observe that 
	\begin{flalign*}
	\mathbb{E} \bigg[ \exp \Big( - \displaystyle\frac{t^2}{2} \langle \textbf{A} X, X \rangle \Big) \bigg] = \mathbb{E} \bigg[ \exp \Big( - \displaystyle\frac{t^2}{2} \langle \textbf{B} X, \textbf{B} X \rangle \Big) \bigg] & = \mathbb{E} \bigg[ \exp \Big( - \mathrm{i} t \langle \textbf{B} X, Y \rangle \Big)  \bigg] \\
	& = \mathbb{E} \bigg[ \exp \Big( - \mathrm{i} t \langle X, \textbf{B} Y \rangle \Big) \bigg]. 
	\end{flalign*}
	
	\noindent Denote $W = \textbf{B} Y = (w_1, w_2, , \ldots , w_N)$. In view of \Cref{exponentialexpectationx}, we have the conditional expectation estimate 
	\begin{flalign} 
	\label{expectationxw1}
	\mathbb{E} \bigg[ \exp \Big( - \mathrm{i} t \langle X, W \rangle \Big) \bigg| W \bigg]  = \exp \Bigg( - \displaystyle\frac{\sigma^{\alpha} |t|^{\alpha}}{N} \displaystyle\sum_{j = 1}^N |w_j|^{\alpha} \Bigg) \exp \Bigg( O \bigg( t^2 N^{ (2 - \alpha) (b - 1 / \alpha) - 1} \displaystyle\sum_{j = 1}^N |w_j|^2 \bigg) \Bigg). 
	\end{flalign}
	
	\noindent Now, observe that since each $w_j$ is a Gaussian random variable with variance $\sum_{i = 1}^N b_{ij}^2$, we have from a union bound that 
	\begin{flalign}
	\label{probabiltiywj2sum}
	\mathbb{P} \Bigg[ \displaystyle\sum_{j = 1}^N w_j^2 > (\log N) \Tr \textbf{A} \Bigg] \le \displaystyle\sum_{j = 1}^N \mathbb{P} \Bigg[ w_j^2 > (\log N) \displaystyle\sum_{i = 1}^N b_{ij}^2 \Bigg] \le N e^{-(\log N)^2 / 2}, 
	\end{flalign}
	
	\noindent where in the first estimate we used the fact that $\Tr \textbf{A} = \Tr \textbf{B}^2 = \sum_{1 \le i, j \le N} b_{ij}^2$. 
	
	Taking the expectation on both sides of \eqref{expectationxw1} over the events where $\sum_{j = 1}^N |w_j|^2$ is at most or at least $(\log N) \Tr \textbf{A}$ and further using the fact that the exponential inside the expectation on the left side of \eqref{expectationxw1} is bounded by $1$, we deduce that
	\begin{flalign*}
	\mathbb{E} \bigg[ \exp \Big( - \mathrm{i} t \langle X, W \rangle \Big) \bigg] & = \mathbb{E} \Bigg[ \exp \bigg( - \displaystyle\frac{\sigma^{\alpha} |t|^{\alpha}}{N} \displaystyle\sum_{j = 1}^N |w_j|^{\alpha} \bigg) \Bigg] \exp \bigg( O \Big( t^2 N^{ (2 - \alpha) (b - 1 / \alpha) - 1} (\log N) \Tr \textbf{A} \Big) \bigg) \\
	& \qquad + N e^{- (\log N)^2 / 2},
	\end{flalign*}
	
	\noindent from which we deduce the lemma. 
\end{proof}

\end{document}